\documentclass [12pt]{book}
\usepackage{amsfonts}
\usepackage{amssymb}
\usepackage{bbm}
\usepackage{amsmath}
\usepackage{mathrsfs}
\usepackage{amsthm}
\usepackage{titlesec}
\usepackage[utf8]{inputenc}
\usepackage{txfonts}
\usepackage{yfonts}
\usepackage[english]{babel}
\usepackage{bbm}
\usepackage{mathtools}
\usepackage{changepage,lipsum,titlesec}

\titlespacing*{\subsection} {1.5em}{3.25ex plus 1ex minus .2ex}{1.5ex plus .2ex}
\titlespacing*{\subsubsection} {1.5em}{3.25ex plus 1ex minus .2ex}{1.5ex plus .2ex}

\titleformat{\subsubsection}[runin]
 {\normalfont\bfseries}{\thesubsubsection}{em}{}

\titleformat{\section}
  {\normalfont\Large\bfseries}{\S\thesection}{1em}{}

\def\thechapter{\Roman{chapter}}

\usepackage{chngcntr}
\counterwithout{section}{chapter}

\usepackage[T1]{fontenc}
\usepackage{titlesec, blindtext, color}
\definecolor{gray75}{gray}{0.75}
\newcommand{\hsp}{\hspace{20pt}}
\titleformat{\chapter}[hang]{\Huge\bfseries}{\thechapter\hsp\textcolor{black}{|}\hsp}{0pt}{\Huge\bfseries}

\begin{document}

\setcounter{tocdepth}{4}
\tableofcontents

\def\vphi{\varphi}
\def\rphi{\vphi\in[0,2\pi)}
\def\vtheta{\vartheta}
\def\SCPM{{\SCPZ\over\SCPN}}
\def\SCPZ{{\xi^2+\eta^2}}
\def\SCPN{{\xi^2\eta^2}}

\def\ralpha{\alpha\in(\i K',\i K'+2K)}
\def\rPalpha{\alpha\in(\i K',\i K'+K)}
\def\rbeta{\beta\in[0,4K')}
\def\rPbeta{\beta\in(0,K')}
\def\rmu{\mu\in(\i K',\i K'+2K)}
\def\rPmu{\mu\in(\i K',\i K'+K)}
\def\reta{\eta\in[0,4K')}
\def\rPeta{\eta\in(0,K')}

\newdimen\theight
\newcommand{\x}{{\ensuremath{\times}}}
\newcommand{\bb}[1]{\makebox[16pt]{{\bf#1}}}

\newtheorem{theorem}{Theorem}[section]
\newtheorem{definition}{Definition}
\newtheorem{proposition}[theorem]{Proposition}
\newtheorem{axiom}{Axiom}
\newtheorem*{corollary}{Corollary}
\newtheorem{example}{Example}[section]
\newtheorem{examples}{Examples}[section]

\newtheorem*{cat}{Definition}
\newtheorem*{AOC}{Axiom of Choice}
\newtheorem*{lemma1}{Lemma}
\newtheorem*{lemma2.1}{Lemma I}
\newtheorem*{lemma2.2}{Lemma II}
\newtheorem*{lemma3.1}{Lemma I}
\newtheorem*{lemma3.2}{Lemma II}
\newtheorem*{lemma3.3}{Lemma III}
\newtheorem*{lemma4.1}{Lemma}
\newtheorem*{Zlemma}{Zorn's Lemma}
\newtheorem*{lemma f1}{Lemma I}
\newtheorem*{lemma f2}{Lemma II}
\newtheorem*{lemma iso}{Lemma}
\newtheorem*{lemma func par}{Lemma}
\newtheorem*{lemma op func 1}{Lemma I}
\newtheorem*{lemma op func 2}{Lemma II}
\newtheorem*{lemma dns1}{Lemma}
\newtheorem*{lemma order funct1}{Lemma I}
\newtheorem*{lemma order funct2}{Lemma II}
\newtheorem*{lemma commute funct}{Lemma}
\newtheorem*{lemma duality n}{Lemma}
\newtheorem*{lemma intlaw}{Lemma}
\newtheorem*{Y Lemma}{Yoneda Lemma}
\newtheorem*{lemma nat t1}{Lemma I}
\newtheorem*{lemma nat t2}{Lemma II}

\newcommand{\arcsinh}{\mbox{ArcSinh}}
\newcommand{\arccosh}{\mbox{ArcCosh}}
\newcommand{\sign}{\mbox{sign}}
\newcommand{\diag}{\mathop{\mathrm{diag}}}

\def \Column{%
             \vadjust{\setbox0=\hbox{\sevenrm\quad\quad tcol}%
             \theight=\ht0
             \advance\theight by \dp0    \advance\theight by \lineskip
             \kern -\theight \vbox to \theight{\rightline{\rlap{\box0}}%
             \vss}%
             }}%

\catcode`\@=11
\def\qed{\ifhmode\unskip\nobreak\fi\ifmmode\ifinner\else\hskip5\p@\fi\fi
 \hbox{\hskip5\p@\vrule width4\p@ height6\p@ depth1.5\p@\hskip\p@}}
\catcode`@=12 

\def\cents{\hbox{\rm\rlap/c}}
\def\miss{\hbox{\vrule height2pt width 2pt depth0pt}}

\def\vvert{\Vert}                

\def\tcol#1{{\baselineskip=6pt \vcenter{#1}} \Column}

\def\dB{\hbox{{}}}                 
\def\mB#1{\hbox{$#1$}}             
\def\nB#1{\hbox{#1}}               

\newpage

\chapter*{Introduction}

The attempt is to give a formal concpet of \textit{system}, and with this provide a definition of category, that will also satisfy the definition of a system. An axiomatic base is given, for constructing the group of integers. In the process, we define a group of automorphisms; we are defining an ordered group of functors with a natural transformation between any two. We give an isomorphism from the group of integers into the group of automorphisms, as guaranteed by Cayley's Theorem. The ultimate aim is to use these definitions and concepts, of system and category, to give a general description of mathematics.

When studying a system we will identify two kinds of components, as enough to define a system. We shall call the one kind \textit{objects} and the other kind \textit{relations}. In essence, it is being argued that anything can be thought as being completely described by 1) the things it is composed of, and  2) the characteristics about it and the things that compose it.

There are three properties that will be regarded as of absolute importance in the study of systems. These are three characteristics that appear throughtout systems of all sorts. We speak of Structure, Symmetry, and Inheritance. This last takes the form of order, as well. The inheritance principal is viewed as a principal of conservation. In considering inheritance, we wish to establish formal observations to the following questions, regarding formation of  new systems from old ones: 

\begin{itemize}\item[1)]What basic properties are passed on from one system to another? \item[2)]If a new system has been formed, what systems could it come from (who are the possible parents)? \item[3)]If two systems are formed from one, what properties will these two new systems have in common? 
\end{itemize}

It was necessary to first describe a system, in order to later describe relations between systems. The relations, in turn, were used for the description of inheritance principles. There is a parallelism with categories, functors and natural transformations. The system takes the place of category, while the functor takes the place of relations between systems. Finally, the natural transformation describes an inheritance principle. In the development of category theory, the same line of thought was followed, as was pointed out by [I]:\\

\textit{\indent``...`category' has been defined in order to define `functor' and `functor' has been defined \indent\indent in order to define `natural transformation'."}\\

Inheritance can be studied in terms of \textit{society}; a system that satisifes a property, that is also satisifed by all the objects in the system. So we may say a human society is a society, since the interactions of different societies are more or less like human interactions. One can say a human (or any living organism, for example) is a society of cells. Cells themselves are born, interacting, reproducing in one form or another, and dying. Do cells inherit these properties from particles?

In mathematics this situation is encountered. One can say certain collections of groups are also groups. We can consider the collection of collections as a society. We would like to verify if filter, topology, vector space, and many other concepts can be seen as a society. Can we define some system, whose objects are vector spaces, so that this new system is itself in one way or another like a vector space? Notice we are not asking that the system be a vector space, just that it has some property that vector spaces have. We can also consider the collection of filters, on a set, and see how inheritance plays out in this context. 

There is a difference between the definition of category given here, and the definition given in [I]:

\begin{cat}A metagraph consists of objects a,b,c,..., arrows f,g,h,..., and two operations as follows:

\begin{itemize}\item[]\textbf{Domain}, which assigns to each arrow f an object a=dom f

\item[]\textbf{Codomain}, which assigns to each arrow f an object b=cod f\end{itemize}

These operations on f are best indicated by displaying f as an actual arrow starting at its domain (or ``source'') and ending at its codomain (or ``target''):$f:a\rightarrow b$...A metacategory is a metagraph with two additional operations:

\begin{itemize}\item[]\textbf{Identity}, which assigns to each object a an arrow $id_a=1_a:a\rightarrow a$

\item[]\textbf{Composition}, which assigns to each pair $\langle g,f\rangle$ of arrows with $dom~g= cod~f$ an arrow $g\circ f$, called their composite, with $g\circ f:dom~f\rightarrow cod~g$... These operations in a metacategory are subject to the two following axioms:\end{itemize}

\begin{itemize}\item[]\textbf{Associativity.} For given objects and arrows in the configuration $a\stackrel{f}\rightarrow b\stackrel{g}\rightarrow c\stackrel{k}\rightarrow d$ one always has the equality $k\circ(g\circ f)=(k\circ g)\circ f$. This axiom asserts that the associative law holds for the operation of composition whenever it makes sense (i.e., whenever the composites on either side of (1) are defined).

\item[]\textbf{Unit law.}For all arrows $f:a\rightarrow b$ and $g:b\rightarrow c$ composition with the identity arrow $1_b$ gives $1_b\circ f=f$ and $g\circ 1_b=g$. This axiom asserte that the identity arrow $1_b$ of each object b acts as identity for the operation of composition, whenever this makes sense.\end{itemize}\end{cat}

Contrary to this definition, the definition of category we will provide, in terms of systems, consists of two kinds of objects. We call the ordinary objects, c-objects. Arrows will be the other kind of object, of the system. This process of considering arrows as object is a key concept that is present throughout. This is how generalizations and complexity arises: by taking relations, and making relations about those. For example, relations of the form $a\rightarrow c\Rightarrow b\rightarrow d$ can be turned into objects of arrows $a\rightarrow c\Rightarrow b\rightarrow d\Longrightarrow e\rightarrow g\Rightarrow f\rightarrow h$. Arrows are one particular way of seeing relations, in mathematics. The definition of category will be a system with relations regarding the arrows and objects. Arrows can be $\circlearrowleft$ (reflexive) or $\leftrightarrow$ (symmetric). The following diagram is very important and we will try to see the underlying properties, uses and consequences of it.  $$_{_a\stackrel{\nearrow^b\searrow}\longrightarrow_c}$$For example, in topology the diagram expresses concepts of separability; the diagram will be used to give a description of supremum, and density. Curiously, in the scope of these definitions, the supremum is a particular case of density. Of course, this is also the diagram that expresses transitivity. We can also think of this diagram when we are developing products of functors, on product categories. It describes three main forms that inheritance will take (each one of these forms for inheritance can be seen reflected in one of our three questions regarding inheritance):

\begin{itemize}

\item[1)] Generational Inheritance

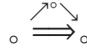
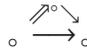
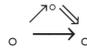
\[_{\circ}\stackrel{\nearrow^
{\circ}\searrow}\Longrightarrow_{\circ}\]

\item[2)] Common Descendency

\[_{\circ}\stackrel{\Nearrow^
{\circ}\searrow}\longrightarrow_{\circ}\]

\item[3)] Common Origin

\[_{\circ}\stackrel{\nearrow^
{\circ}\Searrow}\longrightarrow_{\circ}\]

\end{itemize}

We first speak of systems and immediately after, introduce relations, arrows,  and equivalence. Orders are given in three main types; preorder, partial order, natural order. After making such differentiations, we see basic definitions and observations. We end this subsection with the construction (rather, the assumption of existence) of a natural order that is not trivial, the order of integers. We move on to functions, on collections, and provide the basic types and definitions of functions. We include a definition of two-part functions, which are crucial in the definition of functor. Here we introduce the concept of selection function to describe the Axiom of Choice, for the first time. Also, we define order preserving functions and duality is brought into play for the first time, in the form of dual orders. At the end we see an important description of how two functions can be asked to behave well, in joint. We first study Galois Connections, and then we study the concept \textit{natural pair of functions}. This last, will be the main idea behind the concpet of functor; we will give the defnition of functor, in terms of this concept. A notation is defined and presented in the context of arrows, functions, operations, and generalizations of these. This will enable us to give mechanical proofs in many circumstances.

There is a special kind of function, that we define as operation. That is, next we define an operation as a function that sends objects into functions. We define the concept of left and right operations, for a certain type of operations. When we operate two objects, we can take the view that either of the two objects, was acting on the other; a function is associated to one object and this applied to the other object is the result of the operation. Consider an object in the collection, for which a closed operation is defined, call the object $x$. We can say, on the one hand, \textit{object x acts, on other objects of the collection in such manner...} To verify this, we would provide a list of pairs of objects; the first of each pair is the object operated with $x$. The second of each pair is the corresponding result of the operation, between $x$, and the first of the pair. One can also say, \textit{the objects of $\mathcal O$ act on $x$, according to...} To verify this, we would give another list of pairs. Here the first of each pair is the object that acts on $x$ and the second is the result of the operation. The first list is the right operation of $x$ while the second is the left operation.

In chapter two, we introduce categories and the main properties of functor, along with a definition of natural transformation. Here, we start by giving a basic classification of arrows; isomorphism, left/right cancellable, left/right invertible. After this, we show that a partal order is category and we also define algebraic categories, motivated by considering a category whose arrows are functors from one category to itself. This is the concept of a monoid, as given in [I]. We also give the definition of a group and prove very basic properties regarding unit and inverse. The functor is next, and we begin by giving a definiton of functor that will provide easy to relate to the concept of natural pair of functions. After giving a brief classification of some functor types, we define opposite category and functor. We prove that there is an opposite functor on the category of all categories. This serves as bases for a discussion on contravariant functor. We see a general description of a simple type of functor, the algebraic functor. The next subsection provides a construction of product category and we try to see what functors look like in this tyoe of category. We then define bifunctors with one component contravariant, the other covariant. This is followed by a theorem that allows decomposition of some bifunctors, into two simpler bifunctors. After studying the functor, we go into natural transformation and we see how this relates to the constructions given for the product of two functors, with common domain. Throughout, we will apply the concept of using arrows as objcets, so as to generalize the concept of arrows between arrows. The natural transofrmation is great example of how this and for that we show that natural transformations can compose in two manners. We show the interchange law for the vertical and horizontal compositions of natural transformations, and after this we show that there are two categories to be considered. These are 1) the category of functors $\mathcal Cat(\mathcal C,\mathcal D)$, and 2) the 2-category of all categories which we will denote $\vec{\mathcal C}at$.

The section of integers is central, it begins retaking the discrete number system. A discrete number system is one whose objects and arrows are arranged in the form $\cdots\rightarrow a\rightarrow b\rightarrow c\rightarrow\cdots\rightarrow x\rightarrow y\rightarrow z\rightarrow\cdots$. We show that this system generates a group of automorphisms, on the partial order which corresponds to the discrete number system. This group of automorphisms is generated using one functor, the functor determined by the arrows of the discrete number system. Using this group of automorphisms we build the group of integers with addition. We study the operation as functor, and as a natural transfomation. In short, a discrete number system contains all the information needed to construct a group on itself. Next, a second operation is defined for integers. This is of course the product, another functor. Rational systems are introduced through a construction that involves dual orders. We see that a rational system is an involution. Then, the operations are defined for the rationals and we give an embedding of the integers into the rationals, as order and operation.

In chapter III, we start with an axiomatic base for sets that is quite similar to the one provided in [I]; in fact only minor changes have been made. After the basic properties and operations of sets, we study sets as categories. The first view is to see a set as collection category, then we see it as a partial order under inclusion. Here, we will also introduce concrete category and give a proof of the Yoneda lemma. We see how this relates to the construction provided for the integers through Cayley's theorem.

 We continue to set functions and we introduce here the concept of image and inverse image for a family of sets. The general results that relate image, inclusion and set operations are presented. After this, we present the concept of fiber and we use it to give a decomposition for set functions. The next subsection for set functions, is sequence. In the paragraph Sequence of Objects for an Operation, we introduce the concept of series by defining a sequence of functions that can be applied to a sequence, and that results in the sequence of partial sums. There is a paragraph on compositions of a sequence of functions and here we introduce, among other things, the concept of invariant objects and sets. From this we derive the concept of once-effective functions. Once-effective functions are important in the study of the closure functor which is later used as a starting point in topology. We end the discussion on sequences with a description of sequence in terms of cartesian product. The concept of sequence is generalized in the next subsection. Here we start by defining bounds and supremum. After this we start discussion on directed sets and end with the definition of net. This is used for defining matrices.

The next section is on families of sets, starting with the power set and functor before defining direct image and direct inverse image for families. We then define the forward and backward image of sets $f^\rightarrow A$, $f^\leftarrow B$. After this, we establish a characterization for direct and direct inverse image in terms of $f^\rightarrow A$ and $f^\leftarrow B$. We also note the duality between image and fiber, using the forward and backward images. We see how to express the power set as another kind of functor. Some properties and relations for special types of families are given and then we start with the study of filters. We give a brief introduction to ultrafilters. Ultrafilters will be taken up again in the next chapter; when we study lattices; after we have proven Zorn's Lemma. In fact, The proof of Zorn's Lemma is what follows next.

In Chapter IV, we study lattices. First, we revisit the supremum, then define lattice in terms of supremum and order. The concept of semilattice is defined algebraically in terms of an operation. We show that sublattices and lattices are closely connected and then we show that a lattice can also be viewed as an algebraic structure (with two operations) and a sublattice can also be viewed as an order. After introducing lattices we see properties of completeness. Filters, in the general sense, and boolean algebras are to be included here, but do not yet make appearence in this document.

After Lattices, we devote a chapter to group theory. The last chapter is a brief description of topological systems. We describe them as functors on algebraic categories.

Let us take this oppostunnity to understand the notation that is used through to the end. Firstly, the notation is presented for a general case of arrows. That is, if we have an arrow $a\rightarrow c\longrightarrow b\rightarrow d$ we express this by $a,c;b,d$. Next, we say that $a;fx,f$ is equivalent to the statement \textit{the function f applied to the object x, results in fx}. Finally, we use the notation to express operations by saying $a,c;b,d$ is equivalent to the statement $a*d=b*c$. 
This way, if we define $a,g;b,f$ as the statement saying $fa=gb$ we can express $a,*c;b,*d$ or $a,c;b,d$ as saying the same thing. For example, when defining the integers we give the operation in terms of comparability of arrows. That is, we give  a system of arrows between arrows and the operation is defined so that it coincides with the notation. We define an order on top of an existing order and this gives rise to strong arrows.

\chapter*{List of symbols (Incomplete)}

$\mathcal Sys$~~~~~~~~~~~~~~~~~~~~~~~~~~~~~~~~~~~~~~~~~~~~~~~Universal system\newline
$S,S',S'',S''',...$~~~~~~~~~~~~~~~~~~~~~~~~~~~System\newline
$\emptyset$~~~~~~~~~~~~~~~~~~~~~~~~~~~~~~~~~~~~~~~~~~~~~~~~~~Empty system\newline
$a,b,c,...,x,y,z,...$~~~~~~~~~~~~~~~~~~~~~~~~Object\newline
$\mathcal O$, $\mathcal Q$,...~~~~~~~~~~~~~~~~~~~~~~~~~~~~~~~~~~~~~~~~~Collection of objects\newline
$SX$~~~~~~~~~~~~~~~~~~~~~~~~~~~~~~~~~~~~~~~~~~~~~~~~Remove all objects, except $X$, and all relations that do not include $X$\newline
$f:a\rightarrow b$~~~~~~~~~~~~~~~~~~~~~~~~~~~~~~~~~~~~~Arrow\newline
$\{a\rightarrow b\}$~~~~~~~~~~~~~~~~~~~~~~~~~~~~~~~~~~~~~~~~Collection of arrows from $a$ into $b$\newline
$\{a\rightarrow\}$~~~~~~~~~~~~~~~~~~~~~~~~~~~~~~~~~~~~~~~~~~~Collection of arrows with $a$ as source\newline
$\{\rightarrow b\}$~~~~~~~~~~~~~~~~~~~~~~~~~~~~~~~~~~~~~~~~~~~Collection of arrows with $b$ as target\newline
$\mathcal O_1\rightarrow_\times\mathcal O_2$ or $\mathcal O_1\times\mathcal O_2$~~~~~~~~~~~~~~~~~Cartesian product\newline
$a\rightarrow_\times b$~~~~~~~~~~~~~~~~~~~~~~~~~~~~~~~~~~~~~~~~~Object of cartesian product\newline
$x\leq y$ or $y\geq x$~~~~~~~~~~~~~~~~~~~~~~~~~~~~~~~Comparability in partial order $\mathcal P$\newline
$x=y$~~~~~~~~~~~~~~~~~~~~~~~~~~~~~~~~~~~~~~~~~~~~~$x\leq y$ and $y\leq x$\newline
$x< y$ or $y>x$~~~~~~~~~~~~~~~~~~~~~~~~~~~~~~~Irreflexive comparability in partial order\newline
$a,c;b,d$~~~~~~~~~~~~~~~~~~~~~~~~~~~~~~~~~~~~~~~~Comparability between arrows $a\rightarrow c$ and $b\rightarrow d$\newline
$f:\mathcal O_1\rightarrow\mathcal O_2$~~~~~~~~~~~~~~~~~~~~~~~~~~~~~~~~~Function with $\mathcal O_1$ as domain and range is $\mathcal O_2$\newline
$Dom~f$~~~~~~~~~~~~~~~~~~~~~~~~~~~~~~~~~~~~~~~~~~Domain of the function $f$\newline
$Range~f$~~~~~~~~~~~~~~~~~~~~~~~~~~~~~~~~~~~~~~~Range of the function $f$\newline
$Im~f$~~~~~~~~~~~~~~~~~~~~~~~~~~~~~~~~~~~~~~~~~~~~~Image of the function\newline
$fa,~af,~f(a),~(a)f$~~~~~~~~~~~~~~~~~~~~~~~Image of object $a$, under $f$\newline
$a;fa,f$~~~~~~~~~~~~~~~~~~~~~~~~~~~~~~~~~~~~~~~~~Image of object $a$, under $f$, is $fa$\newline
$a,g;b,f$~~~~~~~~~~~~~~~~~~~~~~~~~~~~~~~~~~~~~~~~Image of $a$, under $f$, is the same object as image of $b$, under $g$\newline
$f\mathcal Q_1$ or $f[\mathcal Q_1]$~~~~~~~~~~~~~~~~~~~~~~~~~~~~~~~~Image of subcollection	$\mathcal Q_1$\newline
$f^{-1}\mathcal Q_2$ or $f^{-1}[\mathcal Q_2]$~~~~~~~~~~~~~~~~~~~~~~~~~Preimage of subcollection $\mathcal Q_2$\newline
$g\circ f$~~~~~~~~~~~~~~~~~~~~~~~~~~~~~~~~~~~~~~~~~~~~~Composition of functions\newline
$I:\mathcal O\rightarrow\mathcal O$~~~~~~~~~~~~~~~~~~~~~~~~~~~~~~~~~~~~Identity function\newline
$f^{-1}$~~~~~~~~~~~~~~~~~~~~~~~~~~~~~~~~~~~~~~~~~~~~~~~Inverse function\newline
$f|_{\mathcal Q_1}$~~~~~~~~~~~~~~~~~~~~~~~~~~~~~~~~~~~~~~~~~~~~~~~Function with domain restricted to $\mathcal Q_1$\newline
$\iota$~~~~~~~~~~~~~~~~~~~~~~~~~~~~~~~~~~~~~~~~~~~~~~~~~~~~Inmersion\newline
\textbf{Sel}~~~~~~~~~~~~~~~~~~~~~~~~~~~~~~~~~~~~~~~~~~~~~~~Selection function\newline
$\mathcal P^{op}$~~~~~~~~~~~~~~~~~~~~~~~~~~~~~~~~~~~~~~~~~~~~~~~Opposite partial  order\newline
$x\leq_{op}y$~~~~~~~~~~~~~~~~~~~~~~~~~~~~~~~~~~~~~~~~~~Comparability in opposite order\newline
$\mathcal P,\mathcal Q;\mathcal P^{op},\mathcal Q^{op}$~~~~~~~~~~~~~~~~~~~~~~~~~~~~~~~$\mathcal P$, $\mathcal Q$ are order bijective\newline
$\mathcal P,\mathcal Q^{op};\mathcal P^{op},\mathcal Q$~~~~~~~~~~~~~~~~~~~~~~~~~~~~~~~$\mathcal P$, $\mathcal Q$ are dual\newline
$\mathcal P,\mathcal P^{op};\mathcal P^{op},\mathcal P$~~~~~~~~~~~~~~~~~~~~~~~~~~~~~~~$\mathcal P$ is self dual\newline
$\mathcal P,\mathcal Q;\mathcal Q,\mathcal P$~~~~~~~~~~~~~~~~~~~~~~~~~~~~~~~~~~~~~~~~Galois connection for two partial orders $\mathcal P$, $\mathcal Q$\newline
$\{\mathcal O_1f\mathcal O_2\}$~~~~~~~~~~~~~~~~~~~~~~~~~~~~~~~~~~~~~~~~~~~All functions $\mathcal O_1\rightarrow\mathcal O_2$\newline
$\mathcal O_1f\mathcal O_2$~~~~~~~~~~~~~~~~~~~~~~~~~~~~~~~~~~~~~~~~~~~~~All functions $\mathcal Q_a\rightarrow\mathcal Q_b$, for all subcollections $\mathcal Q_a$, $\mathcal Q_b$ of $\mathcal O_a$, $\mathcal O_b$\newline
$*,\oplus,\cdot,\circ,...$~~~~~~~~~~~~~~~~~~~~~~~~~~~~~~~~~~~~~~~Operation\newline
$*x$~~~~~~~~~~~~~~~~~~~~~~~~~~~~~~~~~~~~~~~~~~~~~~~~~~~~Right operation\newline
$x*$~~~~~~~~~~~~~~~~~~~~~~~~~~~~~~~~~~~~~~~~~~~~~~~~~~~~Left operation\newline
$a;a*x,x$~~~~~~~~~~~~~~~~~~~~~~~~~~~~~~~~~~~~~~~~~The action of $x$ on $a$ results in $a*x$\newline
$a,y;b,x$~~~~~~~~~~~~~~~~~~~~~~~~~~~~~~~~~~~~~~~~~~~The action of $x$ on $a$ is the same object as the action of $y$ on $b$\newline
$x;x,e$~~~~~~~~~~~~~~~~~~~~~~~~~~~~~~~~~~~~~~~~~~~~~~~The object $e$ is a unit of the operation\newline
$x;e,e$~~~~~~~~~~~~~~~~~~~~~~~~~~~~~~~~~~~~~~~~~~~~~~~The object $e$ absorbs with the operation\newline
$x;e,x^{-1}$~~~~~~~~~~~~~~~~~~~~~~~~~~~~~~~~~~~~~~~~~~~~The objects $x,x^{-1}$ are dual to the unit\newline
$x,x;y,y$~~~~~~~~~~~~~~~~~~~~~~~~~~~~~~~~~~~~~~~~~~~~The operation commutes\newline
$x*,x*;*y,*y$ or $a,c;a*b,b*c$~~~~~~~The operation is associative\newline
$f,f;*fx,*x$~~~~~~~~~~~~~~~~~~~~~~~~~~~~~~~~~~~~~Function and operation commute\newline
$a\oplus b;(a*x)\oplus(b*x),x$~~~~~~~~~~~~~~~~~~Two operations commute\newline
$\mathcal C$, $\mathcal D,...$
~~~~~~~~~~~~~~~~~~~~~~~~~~~~~~~~~~~~~~~~~~~~Category\newline
$\mathcal A$~~~~~~~~~~~~~~~~~~~~~~~~~~~~~~~~~~~~~~~~~~~~~~~~~~~~~~Collection of arrows\newline
$\mathcal{O|C}$~~~~~~~~~~~~~~~~~~~~~~~~~~~~~~~~~~~~~~~~~~~~~~~~~~~Collection of objects of the category\newline
$\mathcal{A|C}$~~~~~~~~~~~~~~~~~~~~~~~~~~~~~~~~~~~~~~~~~~~~~~~~~~Collection of arrows of the category\newline
$b\rightarrow c*a\rightarrow b$~~~~~~~~~~~~~~~~~~~~~~~~~~~~~~~~~Composition of $b\rightarrow c$ with $a\rightarrow b$\newline
$1_x$~~~~~~~~~~~~~~~~~~~~~~~~~~~~~~~~~~~~~~~~~~~~~~~~~~~~~Reflexive arrow for $x$\newline
$\textgoth{1}_\mathcal C$~~~~~~~~~~~~~~~~~~~~~~~~~~~~~~~~~~~~~~~~~~~~~~~~~~~~~Unit function of the category; \newline
$\{a\rightarrow_{iso}b\}$~~~~~~~~~~~~~~~~~~~~~~~~~~~~~~~~~~~~~~~~Collection of isomorphisms for $a,b$\newline
$\textgoth{F}:\mathcal{C}\rightarrow\mathcal{D}$~~~~~~~~~~~~~~~~~~~~~~~~~~~~~~~~~~~~~
Functor\newline
$\textgoth{F}_\mathcal O$~~~~~~~~~~~~~~~~~~~~~~~~~~~~~~~~~~~~~~~~~~~~~~~~~~
Object function of the functor $\textgoth F$\newline
$\textgoth{F}_\mathcal A$~~~~~~~~~~~~~~~~~~~~~~~~~~~~~~~~~~~~~~~~~~~~~~~~~
Arrow function of the functor $\textgoth F$\newline
$\mathcal Cat$~~~~~~~~~~~~~~~~~~~~~~~~~~~~~~~~~~~~~~~~~~~~~~~~~~Category of cateories as c-objects and functors as arrows\newline
$\textgoth I_\mathcal C$~~~~~~~~~~~~~~~~~~~~~~~~~~~~~~~~~~~~~~~~~~~~~~~~~~~Identity functor for $\mathcal C$\newline
$\textgoth I_{Cat}$~~~~~~~~~~~~~~~~~~~~~~~~~~~~~~~~~~~~~~~~~~~~~~~~~Identity functor for $\mathcal Cat$\newline
$\{\mathcal C\textgoth F\mathcal D\}$~~~~~~~~~~~~~~~~~~~~~~~~~~~~~~~~~~~~~~~~~~~Collection of functors\newline
$\mathcal C\textgoth F\mathcal D$~~~~~~~~~~~~~~~~~~~~~~~~~~~~~~~~~~~~~~~~~~~~~Collection of functors for subcategories\newline
$\{\mathcal C\textgoth F_{iso}\mathcal D\}$~~~~~~~~~~~~~~~~~~~~~~~~~~~~~~~~~~~~~~~Collection of all isomorphisms $\mathcal C\rightarrow\mathcal D$\newline
$\mathcal C^{op}$~~~~~~~~~~~~~~~~~~~~~~~~~~~~~~~~~~~~~~~~~~~~~~~~~Opposite category\newline
$f^{op}:b\rightarrow_{op}a$~~~~~~~~~~~~~~~~~~~~~~~~~~~~~~~~~Opposite arrow of $f$\newline
$\textgoth F^{op}:\mathcal C^{op}\rightarrow\mathcal D^{op}$~~~~~~~~~~~~~~~~~~~~~~~~~~~Opposite functor of $\textgoth F:\mathcal C\rightarrow\mathcal D$\newline
$\textbf{op}:\mathcal Cat\rightarrow\mathcal Cat$~~~~~~~~~~~~~~~~~~~~~~~~~~~~~~Functor that sends a categories and functors to their opposite\newline
$\textgoth F^{\ltimes}$~~~~~~~~~~~~~~~~~~~~~~~~~~~~~~~~~~~~~~~~~~~~~~~~~Contravariant functor\newline
$\textgoth f\times\textgoth g:\mathcal C\rightarrow\mathcal D_1\times\mathcal D_2$~~~~~~~~~~~~~~~~~~~~~Product functor of common domain\newline
$\textgoth f\times\textgoth g:\mathcal C_1\times\mathcal C_2\rightarrow\mathcal D\times\mathcal D$~~~~~~~~~~~~~~Product functor of common range\newline
$(h*,*g)$~~~~~~~~~~~~~~~~~~~~~~~~~~~~~~~~~~~~~~~~~~Left-right operation function\newline
$\textgoth B:\mathcal C_1\times\mathcal C_2\rightarrow\mathcal D$~~~~~~~~~~~~~~~~~~~~~~~~~~Bifunctor\newline
$\tau:\mathcal{O|C}\rightarrow\mathcal{A|D}$~~~~~~~~~~~~~~~~~~~~~~~~~~~~ Bridge or natural transformation\newline
$\mathcal Z$~~~~~~~~~~~~~~~~~~~~~~~~~~~~~~~~~~~~~~~~~~~~~~~~~~~Discrete number system\newline
$\mathcal Z_<$~~~~~~~~~~~~~~~~~~~~~~~~~~~~~~~~~~~~~~~~~~~~~~~~~Discrete number system with transitive arrows\newline
$\mathcal Z_{\leq}$~~~~~~~~~~~~~~~~~~~~~~~~~~~~~~~~~~~~~~~~~~~~~~~~~Discrete number system with transitive and reflexive arrows\newline
$\{+x\}_\leq$~~~~~~~~~~~~~~~~~~~~~~~~~~~~~~~~~~~~~~~~~~~~~~~Order of automorphisms $+x$, for $\mathcal Z_\leq$\newline
$\mathbb Z^\dagger$~~~~~~~~~~~~~~~~~~~~~~~~~~~~~~~~~~~~~~~~~~~~~~~~~~~~Group of automorphisms $+x:\mathcal Z_\leq\rightarrow\mathcal  Z_\leq$\newline
$\mathbb Z$~~~~~~~~~~~~~~~~~~~~~~~~~~~~~~~~~~~~~~~~~~~~~~~~~~~~~Integer number system\newline
$\mathbb N_0$~~~~~~~~~~~~~~~~~~~~~~~~~~~~~~~~~~~~~~~~~~~~~~~~~~~System of natural numbers\newline
$\mathbb N$~~~~~~~~~~~~~~~~~~~~~~~~~~~~~~~~~~~~~~~~~~~~~~~~~~~~System of natural numbers, without the object of operation $0$\newline
$\mathbb N_\dagger$~~~~~~~~~~~~~~~~~~~~~~~~~~~~~~~~~~~~~~~~~~~~~~~~~~A partial order $\mathbb N\times\{c\}$, isomorphic to $\mathbb N$\newline
$\mathbb N^\dagger$~~~~~~~~~~~~~~~~~~~~~~~~~~~~~~~~~~~~~~~~~~~~~~~~~~A partial order $\mathbb N\times\{c\}$, dual to $\mathbb N$\newline
$\frac{a}{b}$~~~~~~~~~~~~~~~~~~~~~~~~~~~~~~~~~~~~~~~~~~~~~~~~~~~~Arrow $a\rightarrow_\times b$, of the collection $\mathbb N_0\rightarrow_\times\mathbb N$ or $-\mathbb N_0\rightarrow_\times\mathbb N$\newline
$\mathbb Q$~~~~~~~~~~~~~~~~~~~~~~~~~~~~~~~~~~~~~~~~~~~~~~~~~~~~Rational number system\newline
$\textbf{Obj}$~~~~~~~~~~~~~~~~~~~~~~~~~~~~~~~~~~~~~~~~~~~~~~~~Collection of all objects\newline
$\textbf{Obj}_\in$~~~~~~~~~~~~~~~~~~~~~~~~~~~~~~~~~~~~~~~~~~~~~~System of all objects and of arrows $x\in A$\newline
$\mathcal U$~~~~~~~~~~~~~~~~~~~~~~~~~~~~~~~~~~~~~~~~~~~~~~~~~~~Universe of sets; subcollection of $\textbf{Obj}$\newline
$\mathcal V$~~~~~~~~~~~~~~~~~~~~~~~~~~~~~~~~~~~~~~~~~~~~~~~~~~~Add the object $\mathcal U$ to the collection $\mathcal U$\newline
$\mathcal R$~~~~~~~~~~~~~~~~~~~~~~~~~~~~~~~~~~~~~~~~~~~~~~~~~~~~Collection of all normal objects\newline
\textbf{Set}~~~~~~~~~~~~~~~~~~~~~~~~~~~~~~~~~~~~~~~~~~~~~~~~~~Category of small sets; arrows are functions\newline
$A,B,...,X,Y,...$~~~~~~~~~~~~~~~~~~~~~~~~~~~~~~~~Set\newline
$\emptyset$~~~~~~~~~~~~~~~~~~~~~~~~~~~~~~~~~~~~~~~~~~~~~~~~~~~~~Empty set\newline
$x\in A$~~~~~~~~~~~~~~~~~~~~~~~~~~~~~~~~~~~~~~~~~~~~~~$x$ is an element of the set $A$\newline
$\notin$~~~~~~~~~~~~~~~~~~~~~~~~~~~~~~~~~~~~~~~~~~~~~~~~~~~~~Not an element of\newline
$A=\{x\}_{x\in A}$~~~~~~~~~~~~~~~~~~~~~~~~~~~~~~~~~~~~~~~Set $A$ is represesnted as,\newline
$A\subseteq B$~~~~~~~~~~~~~~~~~~~~~~~~~~~~~~~~~~~~~~~~~~~~~~Set $A$ is a subset of set $B$\newline
$A\subset B$~~~~~~~~~~~~~~~~~~~~~~~~~~~~~~~~~~~~~~~~~~~~~~Set $A$ is a subset of set $B$, but they are not the same object\newline
$A=B$~~~~~~~~~~~~~~~~~~~~~~~~~~~~~~~~~~~~~~~~~~~~~~Set $A$ is subset of set $B$, and $B$ is subset of $A$\newline
$\mathcal X$, $\mathcal Y$,...~~~~~~~~~~~~~~~~~~~~~~~~~~~~~~~~~~~~~~~~~~~Family of sets\newline
$\bigcup\mathcal X$ or $\bigcup\limits_{A\in\mathcal X}A$~~~~~~~~~~~~~~~~~~~~~~~~~~~~~~~~~~~Union of all sets in the family\newline
$A\cup B$~~~~~~~~~~~~~~~~~~~~~~~~~~~~~~~~~~~~~~~~~~~~~~Union of sets $A$ and $B$\newline
$\bigcap\mathcal X$ or $\bigcap\limits_{A\in\mathcal X}A$~~~~~~~~~~~~~~~~~~~~~~~~~~~~~~~~~~~Intersection of all sets in the family\newline
$A\cap B$~~~~~~~~~~~~~~~~~~~~~~~~~~~~~~~~~~~~~~~~~~~~~Intersection of sets $A$ and $B$\newline
$A-B$~~~~~~~~~~~~~~~~~~~~~~~~~~~~~~~~~~~~~~~~~~~~~Difference of set $A$ and $B$\newline
$A^c$~~~~~~~~~~~~~~~~~~~~~~~~~~~~~~~~~~~~~~~~~~~~~~~~~~Complement of set $A$\newline
$\textgoth{P}A$~~~~~~~~~~~~~~~~~~~~~~~~~~~~~~~~~~~~~~~~~~~~~~~~~Set of all subsets of set $A$\newline
$\textbf A,\textbf B,...$~~~~~~~~~~~~~~~~~~~~~~~~~~~~~~~~~~~~~~~~~~Collection/set category\newline
$\textbf PA$~~~~~~~~~~~~~~~~~~~~~~~~~~~~~~~~~~~~~~~~~~~~~~~~~Set category of the set $\textgoth PA$\newline
$\textbf{Set}_{\subseteq}$~~~~~~~~~~~~~~~~~~~~~~~~~~~~~~~~~~~~~~~~~~~~~~
Category of small sets; ordered under inclusion\newline
$\mathcal PA$~~~~~~~~~~~~~~~~~~~~~~~~~~~~~~~~~~~~~~~~~~~~~~~~Power category of $A$; ordered under inclusion\newline
$f[[\mathcal X]]$~~~~~~~~~~~~~~~~~~~~~~~~~~~~~~~~~~~~~~~~~~~Image of family $\mathcal X$\newline
$f^{-1}[[\mathcal Y]]$~~~~~~~~~~~~~~~~~~~~~~~~~~~~~~~~~~~~~~~~Inverse image of family $\mathcal Y$\newline
$(fx)\in fA$~~~~~~~~~~~~~~~~~~~~~~~~~~~~~~~~~~~~~Object in the image of $A$; there is $x\in A$ such that $(fx)=fx$\newline
$(f^{-1}B)\in f^{-1}[[\mathcal Y]]$~~~~~~~~~~~~~~~~~~~~~~~~~~~~Object in the family $f^{-1}[[\mathcal Y]]$; there is $B\in\mathcal Y$ such that $(f^{-1}B)=f^{-1}B$\newline
$f^{-1}[z],~f^{-1}\{z\}$~~~~~~~~~~~~~~~~~~~~~~~~~~~~~~~Fiber of object $z$\newline
$\{\{A\}\}$ or $\{\{x\}\}_{x\in A}$~~~~~~~~~~~~~~~~~~~~~~~~~~~~Family of singletons of $A$\newline
$f|^{Im~f}$~~~~~~~~~~~~~~~~~~~~~~~~~~~~~~~~~~~~~~~~~~~~Function onto the image\newline
$\iota_f$~~~~~~~~~~~~~~~~~~~~~~~~~~~~~~~~~~~~~~~~~~~~~~~~~~Inmersion of image into range\newline
$E_f$~~~~~~~~~~~~~~~~~~~~~~~~~~~~~~~~~~~~~~~~~~~~~~~~Image equivalence of $f$\newline
$f/E_f$~~~~~~~~~~~~~~~~~~~~~~~~~~~~~~~~~~~~~~~~~~~Bijective function, $f$ module $E_f$\newline
$Dom~f/E_f$~~~~~~~~~~~~~~~~~~~~~~~~~~~~~~~~~~~Inverse image of the family $\{\{Im~f\}\}$\newline
$p_f$~~~~~~~~~~~~~~~~~~~~~~~~~~~~~~~~~~~~~~~~~~~~~~~~Projection into fiber\newline
$\textbf{n}$~~~~~~~~~~~~~~~~~~~~~~~~~~~~~~~~~~~~~~~~~~~~~~~~~All numbers $x\in\mathbb N$ such that $1\leq x\leq n$\newline
$\mathcal I$~~~~~~~~~~~~~~~~~~~~~~~~~~~~~~~~~~~~~~~~~~~~~~~~~Index set\newline
$s,t,...,(x_i)_i,(y_n)_n,...$~~~~~~~~~~~~~~~~~~~~~Sequence\newline
$(x_i)_{i=1}^n$~~~~~~~~~~~~~~~~~~~~~~~~~~~~~~~~~~~~~~~~~~Finite sequence\newline
$\mathcal O^{\mathbb N}$~~~~~~~~~~~~~~~~~~~~~~~~~~~~~~~~~~~~~~~~~~~~~~Collection of all sequence of objects in $\mathcal O$\newline
$\mathcal O^{\textbf n}$~~~~~~~~~~~~~~~~~~~~~~~~~~~~~~~~~~~~~~~~~~~~~~~Collection of all finite  sequence of the form $\textbf n\rightarrow\mathcal O$\newline
$\textgoth{S},\textgoth{T},...$~~~~~~~~~~~~~~~~~~~~~~~~~~~~~~~~~~~~~~~Order preserving functor $\mathbb N\rightarrow\mathbb N$\newline
$\Lambda_{i=1}^{n+1}$~~~~~~~~~~~~~~~~~~~~~~~~~~~~~~~~~~~~~~~~~~~~Finite operator; function that sends $(x_i)_i$ into $\Lambda_{j=1}^n(x_i)_i\oplus x_{n+1}$\newline
$\Lambda_{i=1}^{n}(x_i)_i$~~~~~~~~~~~~~~~~~~~~~~~~~~~~~~~~~~~~~~Finite operator applied to the sequence $(x_i)_i$\newline
$\Lambda$~~~~~~~~~~~~~~~~~~~~~~~~~~~~~~~~~~~~~~~~~~~~~~~~General operator; sequence of functions $\Lambda_{j=1}^n$\newline
$\Lambda_\mathcal O$~~~~~~~~~~~~~~~~~~~~~~~~~~~~~~~~~~~~~~~~~~~~~~Union of the family that consists of the image of every finite operator\newline
$\Lambda_{i\in\mathbb N}x_i$ or $\Lambda_{i=1}^\infty x_i$~~~~~~~~~~~~~~~~~~~~~~~~~~Series of $(x_i)_i$, under a general operation $\oplus$\newline
$\sum\limits_{i=1}^nx_i$~~~~~~~~~~~~~~~~~~~~~~~~~~~~~~~~~~~~~~~~~~~Partial sum; image of $(x_i)_i$, under the finite operator $\sum_{j=1}^n$\newline
$\sum\limits_{i=1}^\infty x_i$~~~~~~~~~~~~~~~~~~~~~~~~~~~~~~~~~~~~~~~~~~~Series for operation $+$; image of $(x_i)_i$ under the function $\sum\mathbb N$\newline
$\prod\limits_{i=1}^nx_i$~~~~~~~~~~~~~~~~~~~~~~~~~~~~~~~~~~~~~~~~~~~Partial product; image of $(x_i)_i$, under the finite operator $\prod_{j=1}^n$\newline
$\prod\limits_{i=1}^\infty x_i$~~~~~~~~~~~~~~~~~~~~~~~~~~~~~~~~~~~~~~~~~~~Series for operation $\cdot$; image of $(x_i)_i$ under the function $\prod\mathbb N$\newline
$\Phi$~~~~~~~~~~~~~~~~~~~~~~~~~~~~~~~~~~~~~~~~~~~~~~~Sequence of functions\newline
$\phi$~~~~~~~~~~~~~~~~~~~~~~~~~~~~~~~~~~~~~~~~~~~~~~~~Sequence function associated to the sequence of functions $\Phi$\newline
$\phi(x_i)_i$ or $(\phi_i)_i$~~~~~~~~~~~~~~~~~~~~~~~~~~~~~Sequence of partial operations for $(x_i)_i$; every $\phi_n$ is $\Lambda_{j=1}^n(x_i)_i$\newline
$(\sigma_i)_i$~~~~~~~~~~~~~~~~~~~~~~~~~~~~~~~~~~~~~~~~~~~Sequence of partial sums for some sequence $(x_i)_i$\newline
$(\pi_i)_i$~~~~~~~~~~~~~~~~~~~~~~~~~~~~~~~~~~~~~~~~~~~Sequence of partial products for some sequence $(x_i)_i$\newline
$a^n$~~~~~~~~~~~~~~~~~~~~~~~~~~~~~~~~~~~~~~~~~~~~~~Partial product for the constant sequence $(a)_i$\newline
$(\oplus x_i)_i$~~~~~~~~~~~~~~~~~~~~~~~~~~~~~~~~~~~~~~~~Sequence of right operations for $(x_i)_i$ and $\oplus$\newline
$\bigcirc$~~~~~~~~~~~~~~~~~~~~~~~~~~~~~~~~~~~~~~~~~~~~~General operator for composition\newline
$f^n$~~~~~~~~~~~~~~~~~~~~~~~~~~~~~~~~~~~~~~~~~~~~~Composition of the sequence $(f)_i$\newline

\chapter{Systems}

A system is a collection of objects and a collection of relations. When considering a system, it will be decided, for any concept that makes appearence, whether it be an object or a description. We can form collections of systems, just as we can with any other type of object. Relations are those components of the system that come in the form of assertions about the system, objects of the system, and other systems. All relations refer to, at least, the system or objects of the system. Additionally, it is required that every object appear in at least one relation. Therefore, if a certain relation is an object of the system, we are saying it is the object of some relation regarding it. This is kept in mind, when category is defined as a system.

	\section{Generating New Systems}

		\subsection{Taking/Adding Objects/Relations}

We will now discuss the concept of making new system(s), from given system(s). Suppose you have a system $\mathcal{S}$, then the only actions permitted in the forming of new systems are 1) Taking objects or relations from a given system 2) Adding objects or relations to a given system. We will give a restriction for taking or adding objects or relations: no object can be left without a relation. A collection is called a subcollection of another, if it has been obtained by taking objects/relations from the other.

If $\mathcal{S}'$ is obtained from $\mathcal{S}$ by adding only relations, we will say $\mathcal{S}'$ is a \textit{detailed version} of $\mathcal{S}$. Consider any object $x$ that has any relation in $\mathcal{S}ys$, with some object of $\mathcal{S}'$. Then the system $\mathcal{S}''$ formed by adding the object $x$ and any collection of relations, that mention $x$ or an object of $\mathcal{S}'$, is a \textit{complication} of $\mathcal{S}'$. We will say that $\mathcal{S}'$ is a \textit{simplified version} of $\mathcal{S}''$.

\begin{axiom}We will suppose the existence of one system, $\mathcal{S}ys$. It is defined by one relation: All system $S$ is obtained by taking objects or arrows from $\mathcal{S}ys$.\end{axiom}

Another way of stating this axiom is that every $thing$ is an object of $\mathcal{S}ys$; in particular, every system is an object of it. Consequently, $\mathcal Sys$ is an object of $\mathcal Sys$.

\begin{axiom}Given any object $X$, in a system $\mathcal{S}$, we can form the new system $\mathcal{S}X$, whose (only) object is $X$ and whose relations are the relations of $\mathcal{S}$, pertaining to $X$. For every system, there is a system $\mathcal{S}ysX$.
\end{axiom}

The twin system of $\mathcal{S}ys$, is the \textit{empty set}, denoted by $\emptyset$, and constructed from $\mathcal{S}ys$ by taking away all objects and all relations. Any system can be obtained from $\emptyset$ by adding objects or relations. The emptyset is a \textit{thing} in $\mathcal Sys$; the thing that consists of nothing.

		\subsection{Separating and Combining Subsystems}

When studying a system, we can decide it convenient to see which objects are not related amongst themselves, in any chosen manner. To be more explicit, consider $\mathcal{S}$ to be such that its objects can be distinguished as two kinds $a,b,c,..;x,y,z,...$ and suppose there is not a single relation in $\mathcal{S}$ for which one of each kind of object is present. We consider two, evident, simplified versions of $\mathcal{S}$.

Let us now consider the case in which we may want to separate $\mathcal S$ into two systems, but there is no guarantee as too whether we can differentiate objects as before (there may be relations, in $\mathcal S$, that involve objects of both collections). We will form the two simplifications by defining the two kinds of objects. Relations will be two kinds also: the one kind of relations are those that do not refer to the second kind of object, and the second kinds of relation are those that do not refer to the first kind of object. In this case we have a \textit{separation that loses information}. We accomplish a separation that does not lose information by defining the first kind of relations to be those referring to objects of the first kind, and the second kind of relations are defined as those that mention the second kind of objects.

	\section{Relations and Objects}

A relation is a collection of statements referring to systems and objects. If all the statements of a relation can be obtained by changing the objects (subjects of the statement) in one statement, the statements are \textit{components} of the relation. Arrows will be used to represent relations; a collection of arrows can be the simplfied form for certain relations. The main idea behind the use of an arrow is to distinguish two objects, in a given scenario. Thus, $a\rightarrow b$ means \textit{first is a, second is b}. The object to the left of the arrow is the \textit{source object} and the one to the right is the \textit{target object}. We can represent an arrow by saying $f:a\rightarrow b$. It is reasonable that for some relations, the collection of arrows is \textit{discernible}. This means, that there can be more than one arrow between two objects. A pair of arrows is \textit{parallel} if they have the same source and same target object. For objects $a,b$ in $\mathcal C$, define $\{a\rightarrow b\}$ as the collection of all arrows with $a$ as source and $b$ as target. We also define $\{a\rightarrow\}$ as the collection of arrows with $a$ as source and $\{\rightarrow b\}$ as the collection that has $b$ in the target.

\textit{Binary relations} are those such that every statement refers to only two objects $a,b$. An \textit{ordered pair relation} is any relation whose components are arrows between two objects of the system, and two components are the same if they have the same objects in the same position with respect to the arrow. This last means that the arrows are non-discernible (parallel arrows are the same arrow).

An ordered pair relation is said to be \textit{reflexive} if for every object $a$ of the system we have the component $a\rightarrow a$; these arrows will be called reflexive. In case there are no reflexive arrows, the relation is $irreflexive$. It will be called \textit{symmetrical} if all arrows are double ended; for every arrow $a\rightarrow b$ we have the arrow $b\rightarrow a$. The relation is \textit{anti-symmetric} if given any  choice of distinct objects $a,b$ there is at most one arrow between them (only reflexive arrows are allowed to be double ended). Suppose for every every pair of components $a\rightarrow b$ and $b\rightarrow c$ the component $a\rightarrow c$ exists in the collection of statements, then the relation will be called \textit{transitive}. If there is an arrow $a\rightarrow b$ or $b\rightarrow a$, we will say $a,b$ are $\rightarrow$-\textit{comparable}. We say that a relation is trivial if all the arrows are reflexive.

Throughout mathematics we will also use the equivalence relation, amongst others, with great frequency. It is desireable to regard certain objects as the same with respect to certain criteria. We can determine objects by the relations it has in other systems that are not $\mathcal{S}ys$. We wish to say when two systems are the same, in a lesser sense. An ordered pair relation is called an \textit{equivalence relation} if it is reflexive, symmetric and transitive.

		\subsection{Equivalence}

The collection $\mathcal{O}_{1}\rightarrow_{\times}\mathcal{O}_{2}$ defined in terms of the collections $\mathcal{O}_{1},\mathcal{O}_{2}$ will be called the \textit{cartesian product} and it consists of all arrows $a\rightarrow_{\times}b$, where $a,b$ are in $\mathcal{O}_{1},\mathcal{O}_{2}$, respectively. Given a collection $\mathcal{O}$, we immediately get an equivalence relation $\mathcal{O}\rightarrow_{\times}\mathcal{O}$. Such equivalence relations are called \textit{simple}.

Take an object $x$ of an equivalence relation and let $X$ be the collection of objects of the equivalence that are comparable with $x$. We can carry out a separation of the equivalence, so as to consider a system that, for convenience, is also denoted by $X$. When giving the separation, into the collection $X$, it does not matter if you carry out the separation that loses information, or the separation that does not lose information. We invariably get a separation that does not lose information and is one of the reasons that simple equivalence relations are so important. For objects $x,y$ in $X$, we shall verify that the system $Y$, correspoding to $y$, this is the same system as $X$. Let $a$ be an object of $X$ then we can suppose the existence of an arrow $a\rightarrow x$ because of symmetry. This together with the existence of an arrow $x\rightarrow y$ imply the existence of an arrow $a\rightarrow y$ which means $a$ is an object of $Y$. A similar argument can be given to verify that all objects of $Y$ also belong to $X$. Therefore, the systems $X,Y$ are the same because they both consist of the same objects and both consist of all arrows amongst these objects.

		\subsection{Order}

A $preorder$ is a reflexive and transitive ordered pair relation. A preorder that is anti-symmetric will be called $partial~order$. Not all the objects are ordered amongst each other and when there is an arrow amongst two objects $x,y$ we say they are \textit{order comparable}, for example the arrow $x\leq y$. If an arrow for an order is not reflexive we may express $x<y$. Arrows $\leq$ may be called inequalities, while an arrow $<$ is a strict inequailty. Anti-symmetry holds if $x\leq y$ and $y\leq x$, imply $x=y$. An object $M$ of a partial order is said to be \textit{maximal} if for every $x$ in the partial order such that $M\leq x$, we have $x=M$. We say that an object $m$ is minimal if $x\leq m$ implies $x=m$. Maximality of $M$ means that it appears as source object in one arrow only, the reflexive arrow. Of course, minimality means $m$ appears as target, only in the reflexive arrow.

A \textit{natural order} is a transitive binary relation, in which every object of the collection appears in exactly one arrow with every other object. We can see that a natural order is a partial order, however not all partial orders are necesarilly natural. Of course, given a collection of partial orders, the system formed by including all the objects of all the orders with their respective arrows, is again a partial order. 

\begin{proposition}Given a partial order $\mathcal P$, and a subcollection of it, we can specify a simplified version of $\mathcal P$, that is itself a partial order on the objects of the subcollection.\end{proposition}

Informally, the  axiomatic base for numbers is given by 1) an object exists, 2) there is one arrow on each side of every object, and 3) an order relation is given.
\begin{axiom}We suppose the existence of a system such that 

\begin{itemize}\item[1)] The system is non-empty.\item[2)] Every object is source of exactly one arrow and target of exactly one arrow.\item[3)] The relations of the system are components of a non-trivial natural order.\end{itemize}\end{axiom}

We are able to ask that 2) holds because we know we can take arrows away to form a new system and we are not breaching our only rule: no object is left without a relation (arrow).

We have objects 0 and 1, such that $0\rightarrow1$. We are unable to apply 2) to conlcude $1\rightarrow0$, so we have $-1\rightarrow0\rightarrow1$, for some object $-1$ that is not $0,1$. We are now unable to say $1\rightarrow-1$, because transitivity would lead to a contradiction with the fact that the order is irreflexive. We now know there are objects $-2,2$ such that $-2\rightarrow-1\rightarrow0\rightarrow1\rightarrow2$. We continue construction in this manner and call such a system a \textit{discrete number system}.

		\subsection{Comparability}

It is desireable to have relations amongst relations themselves. Some orders arise in the form of relations amongst relations, as will be seen in building rational systems. Consider systems $\mathcal{S}_{1},\mathcal{S}_{2}$, with respective arrows $\rightarrow_{1}$ and $\rightarrow_{2}$ in the systems. Suppose we have formed a system whose objects are the arrows of $\mathcal{S}_{1},\mathcal{S}_{2}$. For an ordered pair relation, with a component being an arrow of the form $a\rightarrow_{1}c\longrightarrow b\rightarrow_{2}d$ or $a\rightarrow_{2}c\longrightarrow b\rightarrow_{1}d$, we will express $a,c;b,d$. We will in such cases say $a,c$ and $b,d$ are comparable. This kind of comparability will give us some operations and it will allow many constructions for finding the integers and rationals and it will also arise in categories, where we define functors and natural transformations in terms of comparability properties.

		\subsection{Function}

\subsubsection{Definition} 

Consider a system $f$ with two kinds of objects. The first collection of objects called the \textit{domain} and the second kind will be called \textit{range}; unless a more convenient notation is taken up in particular cases, these will be represented by $Dom~f$ and $Range~f$, respectively. If the system has an ordered pair relation such that the collection of target objects is a subcollection of $Range~f$ and any source object appears to the left of exactly one arrow, we will say the system is a function from $Dom~f$ into $Range~f$. This will be expressed by $f:Dom~f\rightarrow Range~f$. For the most part, except in cases where this is not convenient, the components of a function $f$ will be denoted by $a\mapsto_{f}fa$ or $a\rightarrow_ffa$ or $a;fa,f$. We can, in place of $fa$, write $af$ or $f(a)$. If we are to consider two functions $f,g:\mathcal{O}_{1}\rightarrow\mathcal{O}_{2}$, such that $fa$ and $gb$ are the same object, we will write $a,g;b,f$. It will be common to name functions by $f,g:Dom~f,Dom~g\rightarrow Range~f,Range~g$, meaning $f:Dom~f\rightarrow Range~f$ and $g:Dom~g\rightarrow Range~g$.

\subsubsection{Image} The collection of all objects that appear to the right of an arrow is represented by $Im~f$. Let $f:\mathcal{O}_{1}\rightarrow\mathcal{O}_{2}$, and for any subcollection $\mathcal{Q}_{2}$ of $\mathcal{O}_{2}$, we will write $f^{-1}\mathcal{Q}_{2}$ to represent the subcollection of objects $x$ in $\mathcal O_1$, such that $fx$ is in $\mathcal{Q}_{2}$. This collection is the preimage of $\mathcal Q_2$. We say $f\mathcal{Q}_{1}$ is the image of $\mathcal Q_1$, and it is defined as the collection of objects $fx$ in $\mathcal{O}_{2}$, for every $x$ in $\mathcal{Q}_{1}$. If there exists the possibility of confusion, the image and preimage will be written as $f[\mathcal{Q}_{1}]$ and $f^{-1}[\mathcal{Q}_{2}]$.

\subsubsection{Classifications} If the function is such that all the objects of $Im~f$ appear in exactly one relation, it will be called \textit{monic}. If every object in $Range~f$ appears in at least one relation, then it will be said the function is from $Dom~f$ onto $Range~f$, or if all else is clear from the context, we will simply say $f$ is  $onto$. If every object in $Range~f$ appears in exactly one arrow, the function is monic and onto; the function is said to be \textit{bijective}.

One would think that we are able to restrict the function to a certain subcollection, in the $Im~f$ and obtain a function that is onto. Also, we expect to restrict $Dom~f$ in such a way that we obtain a monic function. These two instances will be studied and used.

\subsubsection{Composition} Given two functions $f,g:\mathcal{O}_{1},\mathcal{O}_{2}\rightarrow\mathcal{O}_{2},\mathcal{O}_{3}$, we form $g\circ f:\mathcal{Q}_{1}\rightarrow\mathcal{O}_{3}$, where $a\mapsto_{g\circ f}c$ if there exists an object $b$ in $\mathcal{O}_{2}$ such that $a\mapsto_{f}b\mapsto_{g}c$. This is a principal of transitivity because we are assigning an arrow from $a$ to $c$, given there are arrows from $a$ to $b$ and from $b$ to $c$. Any function $g\circ f$ will be called a \textit{composition}. The composition is associative:\begin{eqnarray}\nonumber f;f(g\circ h)a,(g\circ h)a\\\nonumber f;f(g(ha)),(g\circ h)a\\\nonumber f;(f\circ g)(ha),(g\circ h)a\\\nonumber f;((f\circ g)\circ h)a,(g\circ h)a\end{eqnarray}

\subsubsection{Inverse and Unit}Let $I:\mathcal O\rightarrow\mathcal O$ such that $x\mapsto_I x$; this function will be denoted by $I_\mathcal O$ and if no confusion is possible we simply write $I$. Given a bijective function $f:\mathcal{O}_{1}\rightarrow\mathcal{O}_{2}$, we will say the function $f^{-1}:\mathcal{O}_{2}\rightarrow\mathcal{O}_{1}$ is inverse of $f$, if $f^{-1}\circ  f$ and $I_{\mathcal{O}_{1}}$ are the same function. 

\begin{proposition}Let $f$ be a bijective function and $f^{-1}$ an inverse of $f$. Then $f\circ f^{-1}$ is the function $I_{\mathcal O_2}$.

The function $f^{-1}:\mathcal O_2\rightarrow\mathcal O_1$, that is obtained by reversing all arrows of $f$, is an inverse of $f$. The definition of equality, in the next paragraph, makes $f^{-1}$ unique.\end{proposition}

\begin{proof}We know that for every object in $\mathcal O_1$, we have $x\mapsto_f fx\mapsto_{f^{-1}}x$, so that $(f^{-1}\circ f)x$ is $x$.

Now, let $f^{-1}$ be the function obtained by reversing all arrows of $f$. If $x$ is an object in $\mathcal O_1$, then $f^{-1}$ is the function that sends $fx\mapsto_{f^{-1}}x$ because we have $x\mapsto_f fx$.\end{proof}

The inverse of a composition $g\circ f:\mathcal O_{1}\rightarrow\mathcal O_3$ is $f^{-1}\circ g^{-1}:\mathcal O_3\rightarrow\mathcal O_1$. This is easily figured out by observing the bijectivity of $f,g$ and $f^{-1},g^{-1}$.

\subsubsection{Two Part Function and the Meaning of Equality for Functions} 

Our first equivalence relation is stated so as to determine an equality criteria for functions. Two functions will be the same if they both have the same domain, range, and components. If a function can be seperated into two systems $f_{\mathcal{A}}$ and $f_\mathcal{B}$, without losing information, we say $f$ is a two part function $f_{\mathcal{A}},f_{\mathcal{B}}$. We shall say $\mathcal{A},\mathcal{B}$ are the objects of the first and second kind given by the separation; notice that each of these collections consists of objects in the domain and range. We are considering simplified versions of the function $f$. This is done by taking away arrows and objects in the domain, in such a way that the new system is also a function. 

If $f:\mathcal{O}_{1}\rightarrow\mathcal{O}_2$ and $\mathcal{Q}_{1}$ is the subcollection of $\mathcal{O}_1$ considered, then $f|_{\mathcal{Q}_1}:\mathcal{Q}_{1}\rightarrow\mathcal{O}_2$ is called the restriction of $f$ to $\mathcal{Q}_1$. Notice that we are carrying out the construction of a simplified version of $f$. If $f$ is monic, then any restricted function of $f$ is also monic.  A similar remark does not hold if $f$ is onto. A function $\iota$ is called an \textit{inmersion} if $Dom~\iota\subseteq Range~\iota$ and $x\mapsto_{\iota}x$. Of course, the identity function $I$ is the special case of $\iota$, when equality holds. Also, if we restrict $I$ to a subcollection of its domain, then it becomes an inmersion into its own domain.

\subsubsection{Selection Function}

In this section, we give an axiom that allows us to identify any object with a reflexive arrow. We are looking for a simple way of characterizing the selection of an object, amongst all others in some collection $\mathcal O$. A function $\{x\}\rightarrow Range~f$, where the domain consists of one object, is called \textit{selection function}. For any $x$ in $\mathcal O$, we have the selection function $\iota_x:\{x\}\rightarrow\mathcal O$ that is an inmersion. Let $\mathcal X$ represent a collection of collections and write $\bigcup\mathcal X$ for the collection of objects that belong to some $\mathcal O$ in $\mathcal X$. A \textit{selection function for $\mathcal X$} is a function $\textbf{Sel}:\mathcal X\rightarrow \bigcup\mathcal X$ such that $\textbf{Sel}\mathcal O\in\mathcal O$.

\begin{AOC}Given a non-empty family of functions $\mathcal X$, there exists a selection function.\end{AOC}

We are stating that objects can be chosen one by one, from arbitrary families. This axiom will take another interesting form, in a later chapter. There, we will discuss the axiom of choice in a slightly different context. We will be able to choose elements from partial orders.

\subsubsection{Order Preserving Function and Galois Connections [VIII]}

There are certain functions from an order into another, that can transform the objects in such a manner that the function suits the orders well. We can say that the function recognizes the order of both. Let $\mathcal{P}$ and $\mathcal Q$ be two partial orders where arows are $\leq$ and $\preceq$, respectively.

\begin{definition} A function $\mathcal P\rightarrow\mathcal Q$ defined on the objects of these categories is said to preserve the order if for every arrow $a\leq b$ we also have $fa\preceq fb$; if instead of this, we have $fb\preceq fa$, then the function is said to be order reversing.

We say an order preserving function is an order embedding if for every $fa\preceq fb$ we have $a\leq b$. An order embedding that is onto, is called order bijectivity.\end{definition}

It is easy to see an order bijectivity is indeed a bijective function. For this, we only need to verify that every order embedding is monic. To prove this, take $a,b$ in $\mathcal P$ such that $fa=fb$, and verify $a=b$.

The opposite of a partial order is denoted by $\mathcal P^{op}$ and it is defined as the system whose objects are those from $\mathcal P$ and whose arrows are the reversed arrows of $\mathcal P$. It will be seen that there is not always an order bijectivity for a partial order and its opposite. When it does exist, we say \textit{$\mathcal P$ is self dual}.

When there is an order bijectivity $\mathcal P\rightarrow \mathcal Q$ we write $\mathcal P,\mathcal Q;\mathcal P^{op},\mathcal Q^{op}$. If there is an order bijectivity into the opposite of $\mathcal Q$, then we say $\mathcal P$ and $\mathcal Q$ are dual orders and we represent this with $\mathcal P,\mathcal Q^{op};\mathcal P^{op},\mathcal Q$. Notice that self dualtiy is expressed as $\mathcal P,\mathcal P^{op};\mathcal P^{op},\mathcal P$.

We shall address the issue of whether there exists an order bijectivity $\mathcal P^{op}\rightarrow\mathcal Q^{op}$ given there is an order bijectivity $\mathcal P\rightarrow\mathcal Q$. Let $f:\mathcal P\rightarrow\mathcal Q$ be an order bijectivity, we wish to prove that $f$ is also an order bijectivity $\mathcal P^{op}\rightarrow\mathcal Q^{op}$. We know $x\leq_{op}y$ if and only if $y\leq x$ which is true if and only if $fy\preceq fx$ and this holds if and only if $fx\preceq_{op}fy$.

If we consider an order bijectivity $f:\mathcal P\rightarrow\mathcal Q$, and its inverse $f^{-1}$, we have two order preserving functions; one from $\mathcal P$ to $\mathcal Q$ and the other from $\mathcal Q$ to $\mathcal P$. There is a similar situation with more generality than this last example. Let $f,g:\mathcal P,\mathcal Q\rightarrow\mathcal Q,\mathcal P$. We will ask that $f,g$ are order preserving but we will also ask that the compositions be order preserving, in a different sense. This last condition is a less restrictive condition than the situation with $f,f^{-1}$ in which we have $ff^{-1}q=q$ and $f^{-1}fp=p$.

\begin{definition}A pair of functions $f,g$, as mentioned above, is called a \textit{Galois connection} if 1) $f,g$ are order preserving, and 2) for every object $p$ in $\mathcal P$ and $q$ in $\mathcal Q$ we have $p\leq gfp$ and $fgq\leq q$. \end{definition}

Here is a characterization of Galois connections.
Suppose that the functions $f,g:\mathcal P,\mathcal Q\rightarrow\mathcal Q,\mathcal P$ are such that for every $p,q$ in $\mathcal P,\mathcal Q$ we have $fp\leq q$ if and only if $p\leq gq$. We will sometimes say $f,g$ are comparable functions and we will write $f;g$ instead of using the notation for arrows $\mathcal P,\mathcal Q;\mathcal Q,\mathcal P$.

\begin{proposition}A pair of functions is a Galois connection if and only if it is comparable.\end{proposition}

\begin{proof}Suppose the pair $f,g$ is comparable. For $p$, we know that $fp\leq fp$ and this implies $p\leq gfp$. We can similarly prove this for $q$ in $\mathcal Q$ and that proves 2). Now, we shall prove $f,g$ are order preserving. Take $p,r$, both in $\mathcal P$, such that $p\leq r\leq gfr$. This is equivalent to stating that $fp\leq fr$. If $fgr\leq r\leq q$ are in $\mathcal Q$, we have $gr\leq gq$.

Now, suppose $f,g$ is a Galois connection. If we suppose $fp\leq q$, then $p\leq gfp\leq gq$. Let $p\leq gq$, then $fp\leq fgq\leq q$.\end{proof}

\subsubsection{Natural Pair of Functions} 

Let $*a:\mathcal{O}_1\rightarrow\mathcal{Q}_1$ and $*b:\mathcal{O}_2\rightarrow\mathcal{Q}_2$ be onto functions. Consider the collection $\mathcal{O}$ of all functions of the form $O_1\rightarrow\mathcal{O}_2$. Let $y$ in $\mathcal{Q}_1$ and $f$ a function in $\mathcal{O}$, then we have $x$ in $\mathcal{O}_1$ such that $x\mapsto_{*a}y$. Define $\mathcal{F}f:\mathcal{Q}_1\rightarrow\mathcal{Q}_2$ such that $ x*a\mapsto_{\mathcal{F}f}fx*b$; this simply means $fx,\mathcal{F}f;x*a,*b$. The situation described here is of great importance. If we have functions $F,G:\mathcal{O}_1,\mathcal{Q}_1\rightarrow\mathcal{O}_2,\mathcal{Q}_2$, then we would like to know when it is possible to give functions $*a:\mathcal{O}_1\rightarrow\mathcal{Q}_1$ and $*b:\mathcal{O}_2\rightarrow\mathcal{Q}_2$ that will allow us to describe the actions of applying $G$, in terms of the functions $*b$ and $F$. That is, we are giving a function $F$ and we also give rules for transforming the domain and range of $F$ (that is the role of $*a$ and $*b$ respectively), so as to generate a new function $G$ that is defined for the transformed domain and range. The way this function $G$ works is by sending objects $*a(x)$ in the transfromed image, into the object that results from transforming the orginal object, first under $F$, and then under $*b$. Explicitly we have $Fx,G;x*a,*b$. We will say that $F,G$ are a natural pair of functions, under $*a$ and $*b$ and this will be expressed by $F,G;*a,*b$.

		\subsection{Operation}

Sometimes it will be convenient to use the collection $\{\mathcal O_1\rightarrow\mathcal O_2\}$ of all functions from $\mathcal O_1$ into $\mathcal O_2$; we will use the notation $\{\mathcal O_1f\mathcal O_2\}$. However, in the definition of operation, this collection is not large enough to be useful. Let $\mathcal O_1f\mathcal O_2$ represent the collection of all functions of the form $\mathcal{Q}_1\rightarrow\mathcal{Q}_2$, where $\mathcal{Q}_1,\mathcal{Q}_2$ are any subcollections of $\mathcal{O}_1,\mathcal{O}_2$, respectively. 

An operation is a function $*:\mathcal{O}\rightarrow\mathcal{O}_{1}f\mathcal{O}_{2}$. We say the objects in $\mathcal O$ are the \textit{actions} of the operation. The image of $a$ in $\mathcal O$, under $*$, is a function $a*$ called \textit{left operation}. The image of $x$ in $\mathcal O_1$, under $a*$ is $a*x$ and we express $a;a*x,x$ instead of the function notation $x;a*(x),a*$. Just as we have characterized the operation in terms of functions associated to the objects of $\mathcal O$, we can also determine a function in terms of functions associated to the objects of $\mathcal O_1$. Let $x$ be any object in $\mathcal O_1$; we wish to build a general function $*x$. An object is in the domain of $*x$ if and only if there exists an object $a$, in $\mathcal O$, such that $a*x$ is defined, i.e., $x$ is in the domain of $a*$. We define the transformation as $a\mapsto_{*x}a*(x)$; this is expressed with $a;a*x,x$ instead of the function notation $a;*x(a),*x$. the objects of $\mathcal O_1,\mathcal O_2$ are called \textit{origin} and \textit{target} objects, respectively. An operation $*$ is \textit{full} if $*:\mathcal O\rightarrow\{\mathcal{O}_{1}f\mathcal{O}_{2}\}$. This means $\mathcal O_1$ is  the domain of every left operation. Equivalently, an operation is complete if the domain of every right operation is $\mathcal O$. 

We will frequently encounter functions of the form $*:\mathcal O_1\rightarrow\mathcal O_1f\mathcal O_2$; we call them \textit{operation from $\mathcal O_1$ to $\mathcal O_2$}. For this kind of operation, we can define the concept of commutativity. We say $x$ in $\mathcal O_1$ commutes if $*x,x*$ are the same function. We say $*$ commutes if every action (and thus target) object $x$ commutes. Other times we will consider operations $*:\mathcal O_1\rightarrow\mathcal O_2f\mathcal O_2$ and we say \textit{the objects of $\mathcal O_1$ are actions for $\mathcal O_2$}. If the objects for $\mathcal O_1$ are actions for $\mathcal O_2$. Suppose we have such an operation, and suppose there exists an action object $e$ such that $e*$ is the identity operation for $\mathcal O_2$; we say $e$ is a \textit{left unit of *}. If $*:\mathcal O_1\rightarrow\mathcal O_2f\mathcal O_1$, and $*e$ is $I_{\mathcal O_1}$, for some $e$ in $\mathcal O_2$, we say $e$ is a \textit{right unit}. We can speak of a \textit{unit} if $*:\mathcal O\rightarrow\mathcal Of\mathcal O$ and there is an object $e$ that commutes; $e$ is a unit if the left/right operations are the identity function. When there is a unit, two objects can be related in a special manner and we say $x^{-1}$ is right inverse of $x$ if $x;e,x^{-1}$. We are justified in saying $x$ is left inverse of $x^{-1}$.

Form the system whose weak arrows are the arrows of the right operation functions $*x$, for every $x$ in $\mathcal{O}$. Two weak arrows $a\rightarrow_{*x}c$ and $b\rightarrow_{*y}d$ will be related if $c$ and $d$ are the same object; in such case $a,y;b,x$ will be expressed instead of $a,c;b,d$. This relation is an equivalence relation so we can equally say $b,x;a,y$. We choose to do this because in terms of the notation for functions we can express $a,*y;b,*x$ or $b,*x;a,*y$. Two objects $x,y$ commute if $x,x;y,y$.

If $*:\mathcal O\rightarrow\{\mathcal Of\mathcal O\}$, then the composition is defined for the functions which are left/right operations. We want to define an operation $\circ$ such that the actions, source, and target objects are the left and right operations of $*$. The operation in question is $\circ:\{\mathcal Of\mathcal O\}\rightarrow\{\mathcal Of\mathcal O\}f\{\mathcal Of\mathcal O\}$, where $\circ$ is composition of functions. We will say the operation is associative if $x*,x*;*y,*y$, for every choice of obects $x,y$. This means that $x*\circ*y$ is the same function as $*y\circ x*$; the functions $x*$ and $*y$ commute. Associativity can also be seen as a transitive property regarding the operation, for the following reason. We have assigned $a\rightarrow_{y*\circ*x} c$ to any arrows $a\rightarrow_{y*}b\rightarrow_{*x}c$. In terms of the objects and the original operation, associativity is expressed by $a,c;a*b,b*c$, for every objects $a,b,c$. This last form of writing the associativity gives intuition for why we say it is a transitive principle: on one side we have $a,c$, while on the other we have $a,b$ and $b,c$, in that order. In conclusion, $*$ is associative if for every pair $x,y$ of objects, $x*$ is a natural pair of functions with itself, under $*y$, $*y$. Or, we can say $*y$ is a natural pair with itself under $x*$, $x*$.

An operation and a function can commute. This happens for $f:\mathcal O\rightarrow\mathcal O$ and $*:\mathcal O\rightarrow\mathcal Of\mathcal O$ if $f,f;*fx,*x$, for every $x$; put differently, $f\circ *x$ is the same function as $*fx\circ f$. Given two operations $\oplus:\mathcal O_1\rightarrow\{\mathcal O_1f\mathcal O_1\}$ and $*:\mathcal O_2\rightarrow\{\mathcal O_1f\mathcal O_1\}$, we say \textit{$*$ distributes over $\oplus$ on the left} if for all $b$ in $\mathcal O_1$ and all $x$ in $\mathcal O_2$ we have $x*,x*;\oplus(x*b),\oplus b$. Another way of saying this is all the left operations of $*$ commute with $\oplus$. If $*:\mathcal O_1\rightarrow\{\mathcal O_2f\mathcal O_1\}$ and instead we verify the above with $*x,*x;\oplus(b*x),\oplus b$ then we say $*$ \textit{distributes on the right}. Distribution on the left can be written as $x;(x*a)\oplus(x*b),a\oplus b$ and distribution on the right is $a\oplus b;(a*x)\oplus(b*x),x$.

Suppose we have a full operation $*:\mathcal O\rightarrow\{\mathcal O_1f\mathcal O_2\}$. Propose a function $*:(\mathcal O_1\rightarrow_\times\mathcal O)\longrightarrow\mathcal O_2$ which is called a \textit{binary operation}. Let $a\rightarrow_\times x$ be an object in $\mathcal O_1\rightarrow_\times\mathcal O$. We define $*(a\rightarrow_\times x):=a*x$. Let us consider the construction of the operation, provided we have a binary operation. We wish to give a function $*x$ for every object in $\mathcal O$; this function must send objects of $\mathcal O_1$ into $\mathcal O_2$. The components of the function are $a\rightarrow^{*x}*(a\rightarrow_\times x)$. Given a function $((\mathcal O_2\rightarrow_\times\mathcal O_1)\rightarrow_\times\mathcal O)\longrightarrow\mathcal O_3$, there is an operation $\mathcal O\rightarrow\{(\mathcal O_2\rightarrow_\times\mathcal O_1)f\mathcal O_3\}$. Every function in the image can be turned into an operation of the form $\mathcal O_1\rightarrow\{\mathcal O_2f\mathcal O_3\}$. It is then possible to build an operation $\mathcal O\rightarrow\{\mathcal O_1f\{\mathcal O_2f\mathcal O_3\}\}$.

\chapter{Category}

Here, we will treat objects and arrows as two different kinds of objects. The relations in a category are not the arrows. The relations, in this definition of category, are statements about arrows and objects. It will be clear, from the definition, that a category is indeed a system whose objects consist of relations also. This is important to consider an environment where a functor is to be treated as a function.

A collection of arrows is said to have the reflexive property if for all object $x$ there exists at least one reflexive arrow. It will be said that the collection of arrows has the transitive property if the composition for the arrows (adequately seen as functions) defines an operation $*:\mathcal{A|C}\rightarrow\mathcal{A\|C}f\mathcal{A|C}$. Clearly, the composition is not always a full operation. A collection of arrows that satisfies the last two properties, for all objects of some collection $\mathcal{O}$, will be called a \textit{collection of order arrows} for $\mathcal{O}$.

\begin{definition}A category is a system with objects of two kinds; the one kind are called c-objects, while the other kind is a collection of order arrows, for the collection of c-objects. Additionally, for any choice of c-objects $x,y$, and arrows $f:a\rightarrow x$ and $g:b\rightarrow y$, there is a reflexive arrow $1_{b},1_x$ such that \begin{itemize}\item[1)] $1_{x};f,f$~~~~~~~~~and~~~~~~~~$g;g,1_{b}$~~~~~~~~~~~~~~~~~~~~~
~~~~~~~~~~~~~~~~~~~~~~~~~~~~~~~~~~~~~~~~~~Unit\item[2)] $g*,g*;*f,*f$~~~~~~~~~~~~~~~~~~~~~~~~~~~~~~~~~~~~~~~~~~~~~~~~~~~~
~~~~~~~~~~~~~~~~~~~~~~~~~~~~~~~~~~~~Associativity\end{itemize}\end{definition}

The two kinds of objects are represented by $\mathcal{O|C},\mathcal{A|C}$, respectively. In 2) we are stating that for every $h:x\rightarrow b$, the arrows $g*(h*f)$ and $(g*h)*f$ are the same. All we are saying is that both, the collection of arrows and the operation for the arrows, satisfy their own reflexive and transitive properties.

	\section{Arrows}

		\subsection{Defining Equality of Arrows} 

Given two arrows we need to provide a definition of equality between them. That is, we will establish an equivalence relation for arrows of any given category. Since categories may have arrows that are discernible, an arrow is not necessarily determined by the objects and the position with respect to the arrow, as is the case with non-discernible arrows. 

\begin{definition}Two arrows $f,g:x\rightarrow b$, of the same category, are the same arrow if for every arrow $h:a\rightarrow x$ and $i:b\rightarrow y$ we verify $f,h;g,h$ and $i,g;i,f$.\end{definition}

What this means is that an arrow is defined to be the same as another in terms of the result of the operation $*$. We say $f,g$ are the same when $f*h$ is the same as $g*h$ and $i*f$ is the same as $i*g$, for every $h,i$ composable with $f,g$.

\newpage

It remains to be proven that this is indeed an equivalence relation, defined on the arrows of the category. Reflexivity is trivially verified. Symmetry is not hard to prove; we know $g,h;f,h$ is a notation to replace $f,h;g,h$ and $i,f;i,g$ replaces $i,g;i,f$. Suppose $e$ is the same as $f$ and $f$ is the same as $g$. Since $e,h;f,h$ and $f,h;g,h$ we can say that $e,h;g,h$. Similarly, one finds $i,g;i,e$.

		\subsection{Classifying Arrows [I]}

\subsubsection{Invertible}Defining types of arrows is important for classifying functions, and more generally, functors. We say that $f:a\rightarrow b$ is \textit{invertible} if there exists $f^{-1}:b\rightarrow a$ such that $f^{-1}*f$ and $f*f^{-1}$ are $1_a$ and $1_b$, respectively. If there is an invertible arrow between two objects, we say they are \textit{isomorphic}. So, we say an arrow that makes two object isomorphic, is an \textit{isomorphism}. The collection of isomorphisms from for any pair of objects is represented as $\{a\rightarrow_{iso}b\}$.

\begin{proposition}The relation of isomorphism, is an equivalence relation on the objects of the category.\end{proposition}

\begin{proof}First, we know that any object $a$, of the category, is isomorphic to itself because $1_a$ is invertible and $1_a;1_a,1_a$. By definition, symmetry holds; if $f$ is invertible, then so is $f^{-1}$. Now, let $a,b$ be isomorphic with $f$ and $b,c$ isomorphic with $g$. Consider the arrow $f^{-1}*g^{-1}$, and it is easy to verify $g*f;1_c,f^{-1}*g^{-1}$ and $f^{-1}*g^{-1};1_a,g*f$.\end{proof}

Say $f^{-1}$ and $g$ both make $f$ invertible, we shall prove they are the same arrow. We know they are both functions of the form $b\rightarrow a$, so now it must be shown that the operations coincide. Let $h$ be an arrow composable with $f^{-1},g$.
\begin{eqnarray}\nonumber g&;&g*h,h\\\nonumber g&;&1_a*(g*h),h\\\nonumber g&;&(1_a*g)*h,h\\\nonumber g&;&[(f^{-1}*f)*g]*h,h\\\nonumber g&;&[f^{-1}*(f*g)]*h,h\\\nonumber g&;&[f^{-1}*1_b]*h,h\\\nonumber g&;&f^{-1}*h,h\\\nonumber g,h&;&f^{-1},h\end{eqnarray}

With this we prove that $f^{-1},g$ have the same left operation for composition. We can just as easily show that they have the same right operation. 
\begin{proposition}The inverse of an invertible arrow is unique.\end{proposition}

\subsubsection{Left and Right Inverse}If two discernible arrows have the same source and target, we say they are \textit{parallel}. An arrow $f:x\rightarrow b$ is \textit{left cancellable} if for every pair of parallel arrows $g,h:a\rightarrow x$, the expression $f,g;f,h$ implies $1_x,g;1_x,h$. Suppose for every $g,h:b\rightarrow y$, the expression $g,f;h,f$ implies $g,1_b;h,1_b$. Then we say $f$ is \textit{right cancellable}.

If there is an arrow $l:b\rightarrow a$ such that $l*f$ is $1_a$, we say $f$ has \textit{left inverse}. In a similar manner, $f$ has \textit{right inverse} if there is $r:b\rightarrow a$ such that $f;1_b,r$.

\begin{proposition}If $f$ has left inverse, then it is left cancellable because we can apply $f^{-1}$ to $f,g;f,h$. An arrow with right inverse is right cancellable.
\end{proposition}

Now we give a result that characterizes monic functions as functions with left inverse. It will later be proven that onto functions are characterized as functions with right inverse.

\begin{lemma f1}A function has left inverse if and only if it is monic.\end{lemma f1}

\begin{proof}If the function has left inverse, and $a,b$ are different objects, then $(l\circ f)a$ and $(l\circ f)b$ are also different. Therefore, $fa$ and $fb$ must be different, so that $l$ is a function. This means $f$ is monic.

Suppose$f$ is monic; we will give $l:Range~f\rightarrow Dom~f$ such that $a;a,l\circ f$. Take $y$ in the image of $f$ so that $x;y,f$ for some $x$ in the domain. We define $y;x,l$. If, however, $y$ is not in the image, we can take any $a$ in the domain and define $y;a,l$. The reader can prove this is a left inverse.\end{proof}

\begin{lemma f2}A function has right inverse if and only if it is onto.\end{lemma f2}

\begin{proof}Let $f$ be a function with right inverse $r:Range~f\rightarrow Dom~f$. This means that for every $y$ in the range of $f$, we verify $y;y,f\circ r$. It is tue that $f$ is onto because to every $y$, in the range, we assign an object, in the domain, whose image, under $f$, is $y$. 

If the function is onto, every element of the range has non-empty inverse image. Let $f^{-1}[[\{\{Im~f\}\}]]$ represent the collection of inverse images, of objects in the range; this notation will later be justified when we study set functions. We are giving a collection whose objects are $f^{-1}[\{y\}]$, where $\{y\}$ can be any collection of one object, in $Range~f$. There is a selection function $S:f^{-1}[[\{\{Im~f\}\}]]\rightarrow Dom~f$, for the family $f^{-1}[[\{\{Im~f\}\}]]$. Let $T:Range~f\rightarrow f^{-1}[[\{\{Im~f\}\}]]$ be the function that sends $y\mapsto_Tf^{-1}[\{y\}]$. It is now necessary to verify that $S\circ T$ acts as right inverse of $f$.\end{proof}

\begin{theorem}\label{th bij}A function has inverse if and only if it is bijective.\end{theorem}

uodate nexr appt send miles cb

	\section{Some Categories}

We see that a category is an efficient way of describing partial orders, collections of functions, and operations. These descriptions are given in the present section.

		\subsection{Partial Order}

A partial order can be seen as a category. We shall prove that the objects and relations of a partial order define c-objects and arrows of an \textit{order category}. The relations of the partial order are a collection of order arrows for the objects because the order is reflexive and transitive. Anti-symmetry means there is at most one relation for two objects. The reflexive arrow is the unit and we now only need to verify associativity for the order relations. We see that the operation for two relations $b\leq c*a\leq b$, results in $a\leq c$; recall that this operation works as composition for functions. It is therefore clear that associativity holds, since we get the same result  from $c\leq d*(b\leq c*a\leq b)$ and $(c\leq d*b\leq c)*a\leq b$.

		\subsection{Algebraic Category}

An important type of category is that which consists of domains and functions. That is, we consider categories whose c-objects are collections and whose arrows are functions. When we study sets, we will see that functions are arrows in the category of sets. For now, we study a particular case. Consider a system that consists of one c-object $\mathcal O$ and the arrows are functions $\mathcal O\rightarrow\mathcal O$. We want to prove that this system is a category. The collection of functions is a collection of order arrows for $\mathcal O$ because they are distinguishable and we have defined the function $I_\mathcal O$ and composition for functions. The function $I$ acts as the unit. Associativity has already been proven for functions. When cosidering these types of categories, we will usually consider the category $\{\mathcal Of_{iso}\mathcal O\}$ of bijective functions. Notice that this category consists of one c-object and the arrows have inverse. We now move on to a type of category that generalizes these categories of functions.

An \textit{algebraic category} is a category $\mathcal{C}$ with one object in $\mathcal{O|C}$; it consists of one c-object, $e$. In this particular instance, the arrows, all of which are reflexive arrows of $e$, are called \textit{objects of operation}. Notice that the operation is full. We are using the concept of category to simplify our view of a system that has an operation defined. We have chosen a unit object for the operation, by choosing to have one c-object. Our unit is clearly the arrow that appears in $1)$ of the definition to a category. Moreover, we also represent this arrow with $e$. We may say that a system is an \textit{associative category} if it satisfies everything except 1) in the definition of category (there is no unit).

A \textit{group} is an algebraic category such that its c-objects all have inverse to the unit. This form of reasoning will also help build the integers with sum. For all objects of operation $f,g,h$, of a group,

\begin{itemize}

\item[1)]$f;f,e$~~~~~~~~~~~~~~~~~~~~~~~~~~~~~~~~~ ~~~~~~~~~~~~~~~~~~~~~~~~~~~~~~~~~~~~~~~~~~~~~~~~~~~~~~~~~
~~~~~~~~~~~~~$Ref.~(Unit)$

\item[2)] $f;e,f^{-1}$~~~~~~~~~~~~~~~~~~~~~~~~~~~~~~~~~~~~~~~~~~~~~~~
~~~~~~~~~~~~~~~~~~~~~~~~~~~~~~~~~~~~~~~~~~~~~~~~~~~~~$Symm.~(duality)$

\item[3)] $g*,g*;*f,*f$~~~~~~~$or$~~~~~~~~~~~$g,f;g*h,h*f$~~~~~~~~~~~~~~~~~~
~~~~
~~~~~~~~~~~~~~~~~~~~~~~$Trans.~(Associativity)$,

\end{itemize}We will prove $f^{-1};e,f$ and $e;f,f$ also hold.
\begin{eqnarray}\nonumber f^{-1}&;&f^{-1}*f,f\\\nonumber f^{-1}&;&(f^{-1}* f)*e,f\\\nonumber f^{-1}&;&(f^{-1}* f)* [f^{-1}*(f^{-1})^{-1}],f\\\nonumber f^{-1}&;&[(f^{-1}* f)* f^{-1}]*(f^{-1})^{-1},f\\\nonumber f^{-1}&;&[f^{-1}* (f* f^{-1})]*(f^{-1})^{-1},f\\\nonumber f^{-1}&;&f^{-1}*(f^{-1})^{-1},f\\\nonumber f^{-1}&;&e,f\end{eqnarray}
\begin{eqnarray}\nonumber e&;&e*f,f\\\nonumber e&;&(f*f^{-1})*f,f\\\nonumber e&;&f*(f^{-1}*f),f\\\nonumber e&;&f*e,f\end{eqnarray}
What is more, $f$ and $(f^{-1})^{-1}$ are the same object.
\begin{eqnarray}\nonumber f&;&f,e\\\nonumber f&;&f*e,e\\\nonumber f&;&f*[f^{-1}*(f^{-1})^{-1}],e\\\nonumber f&;&(f* f^{-1})*(f^{-1})^{-1},e\\\nonumber f&;&(f^{-1})^{-1},e\end{eqnarray}

One of the important properties of a group is that there always is a unique solution to $f;g,x$, for $x$. It is straightforward to verify that we have one solution to $f;g,x$:
\begin{eqnarray}\nonumber f&;&f*(f^{-1}*g),f^{-1}*g\\\nonumber f&;&(f^{-1}*f)\cdot g,f^{-1}*g\\\nonumber f&;&g,f^{-1}*g.\end{eqnarray}

We use associativity to prove $*$ is left cancellable; this last means $f$ is left cancellable, for every $f$ in the category that is not $e$. If $f,g;f,h$, then
\begin{eqnarray}\nonumber f^{-1}&;&f^{-1}*(f*g),f*h\\\nonumber f^{-1}&;&(f^{-1}*f)*g,f*h\\\nonumber f^{-1}&;&g,f*h\\\nonumber f^{-1}*f&;&g,h\\\nonumber f^{-1}*f,e&;&g,h\\\nonumber e,e&;&g,h\\\nonumber e,g&;&e,h.\end{eqnarray}

Let us now suppose we have two solutions $x_{1},x_{2}$. That is, $f;g,x_{1}$ and $f;g,x_{2}$. We may conclude $e,x_{1};e,x_{2}$ which means they are the same object of operation. An important consequence of this is that the unit and inverse are unique. A similar exposition shows the only solutions to $x;g,f$ is $g*f^{-1}$:
\begin{eqnarray}\nonumber g*f^{-1}&;&(g*f^{-1})*f,f\\\nonumber g*f^{-1}&;&g*(f^{-1}*f),f\\\nonumber g*f^{-1}&;&g,f.\end{eqnarray}

Also, given $g,f;h,f$ we have $g,e;h,e$. A group is said to be \textit{abelian} if the operation is commutative.

\begin{proposition}If $*$ is left cancellable and $x$ is not a unit, then the left operation functions $x*$ are monic. If the operation is right cancellable, the right operation functions $*x$ are onto.\end{proposition}

\begin{proof}Take a left operation $f*$, and let $g,h$ be objects of the category such that $f,g;f,h$ which is true if and only if $f*(h)$ is the same object as $f*(g)$. We know that $e,g;e,h$ holds, so $f*$ is monic.

A similar proof is valid for functions $*f$, given $*$ is right cancellable.\end{proof}

\begin{proposition}\label{prop lr can grp}The left and right operation functions, for a group $G$, are invertible.\end{proposition}

\begin{proof}In a group, the inverse objects of operation commute. It is therefore clear that $x*\circ(x^{-1})*$ is the same function as $(x^{-1})*\circ x*$.\end{proof}

	\section{Functor}

At various points in our studies, we will come across a concept of generalized functions; we will consider functions from a category into another. This means that we need to define a function on the collection of objects of a category. It will be necessary to define two part functions for categories, so that c-objects will be sent into c-objects, and arrows will be sent into arrows. In this section we discuss the concept of functor; a two part function between categories. There is an important aspect about the functor, that will arise in the definition, and we mention it briefly. When defining a functor, we request three conditions, for the two part function; natural pair of functions are behind each of these.

Let $\textgoth{1}_{\mathcal{C}}$ be the unit function of $\mathcal{C}$, defined by $x\mapsto_{\textgoth{1}_{\mathcal{C}}}1_{x}$. The first condition that a functor satisfies means that we invariably obtain the same result by i) sending an object into its reflexive arrow and then transforming the arrow, or ii) transforming the object first, and associating the reflexive arrow second. In short, $\textgoth F1_x$ and $1_{\textgoth Fx}$ are the same object. The condition is expressed in terms of the notation for composition of functions as operation.

A functor contains an arrow function and this arrow function will preserve arrows in the following, intuitive sense. If $f:a\rightarrow b$, then the functor will send $f$ into an arrow from the transformed object of $a$, into the transformed object of $b$. In other words, the arrow function sends arrows, to arrows between corresponding objects. In condition 2) we will use the notation in terms of comparability, where weak arrows are the arrows of $\mathcal C$ and $\mathcal D$ and strong arrows are the components of $\textgoth F_\mathcal A$.

We will finally request that functors preserve the operation of arrows. This means that we will obtain the same result after applying the functor to a composition of arrows, as we would from composing the corresponding transformed arrows. This means that the notation is used in terms of the composition operation for functions.

\begin{definition}A functor $\textgoth{F}:\mathcal{C}\rightarrow\mathcal{D}$ is a two part function $\textgoth{F}_{\mathcal{O}},\textgoth{F}_{\mathcal{A}}:\mathcal{O|C},\mathcal{A|C}\rightarrow\mathcal{O|D},\mathcal{A|D}$ such that for every arrow $f:a\rightarrow b$, we have \begin{itemize}
\item[1)] $\textgoth{1}_{\mathcal{D}},\textgoth{1}_{\mathcal{C}};\textgoth{F}_{\mathcal{A}},\textgoth{F}_{\mathcal{O}}$~~~~~~~~~~~~~~~~~~~~~~~~~~~~~~~~~~~~~
~~~~~~~~~~~~~~~~~~~~~~~~~~~~~~~~~~~~~~~~~Preserves Unit
\item[2)]$a,b;\textgoth{F}_\mathcal Oa,\textgoth{F}_\mathcal Ob$~~~~~~~~~~~~~~~~~~~~~~~~~~~~~~~~~~~~~~~~~~~~~~~~~~
~~~~~~~~~~~~~~~~~~~~~~~~~
~~~~~Preserves Objects\item[3)] $\textgoth{F}_\mathcal A,\textgoth{F}_\mathcal A;*\textgoth{F}_\mathcal Of,*f$~~~~~~~~~~~~~~~~~~~~~~~~~~~~~~~~~~~~~~
~~~~~~~~~~~~~~~~~~~~~~~~~~~~~~~~~~~~~~~~Preserves Composition\end{itemize}\end{definition}

So, a functor is a function that preserves: 1) the unit property for the operation, 2) arrows amongst respective pairs of objects, and 3) the transitive property for the collection of arrows.

The composition of a functor is a functor. Let $\textgoth{F,G}:\mathcal{C,D}\rightarrow\mathcal{D,E}$. We have $\textgoth{1}_{\mathcal{D}},\textgoth{1}_{\mathcal{C}};\textgoth{F}_{\mathcal{A}},\textgoth{F}_{\mathcal{O}}$ and $\textgoth{1}_{\mathcal{E}},\textgoth{1}_{\mathcal{D}};\textgoth{G}_{\mathcal{A}},\textgoth{G}_{\mathcal{O}}$. Let $\textgoth{G}\circ\textgoth{F}$ be the two part function that consists of $\textgoth{G}_{\mathcal O}\circ\textgoth{F}_{\mathcal O}$ and $\textgoth{G}_{\mathcal A}\circ\textgoth{F}_{\mathcal A}$.
 \begin{eqnarray}\nonumber\textgoth{1}_{\mathcal{E}}&;&(\textgoth{1}
_{\mathcal{E}}\circ\textgoth{G}_{\mathcal{O}})\circ\textgoth{F}_{\mathcal{O}},\textgoth{G}_{\mathcal{O}}\circ
\textgoth{F}_{\mathcal{O}}\\\nonumber\textgoth{1}_{\mathcal{E}}&;
&(\textgoth{G}_{\mathcal{A}}\circ\textgoth{1}_{\mathcal{D}})\circ\textgoth{F}
_{\mathcal{O}},\textgoth{G}_{\mathcal{O}}\circ
\textgoth{F}_{\mathcal{O}}\\\nonumber\textgoth{1}_{\mathcal{E}}&;
&\textgoth{G}_{\mathcal{A}}\circ(\textgoth{1}_{\mathcal{D}}\circ\textgoth{F}
_{\mathcal{O}}),\textgoth{G}_{\mathcal{O}}\circ
\textgoth{F}_{\mathcal{O}}\\\nonumber\textgoth{1}_{\mathcal{E}}&;
&\textgoth{G}_{\mathcal{A}}\circ(\textgoth{F}
_{\mathcal{A}}\circ \textgoth{1}_{\mathcal{C}}),\textgoth{G}_{\mathcal{O}}\circ
\textgoth{F}_{\mathcal{O}}\\\nonumber\textgoth{1}_{\mathcal{E}}&;
&(\textgoth{G}_{\mathcal{A}}\circ\textgoth{F}
_{\mathcal{A}})\circ \textgoth{1}_{\mathcal{C}},\textgoth{G}_{\mathcal{O}}\circ
\textgoth{F}_{\mathcal{O}}\\\nonumber\textgoth{1}_{\mathcal{E}},\textgoth{1}_{\mathcal{C}}&;
&\textgoth{G}_{\mathcal{A}}\circ\textgoth{F}
_{\mathcal{A}},\textgoth{G}_{\mathcal{O}}\circ
\textgoth{F}_{\mathcal{O}}.\end{eqnarray}

Next, we verify condition 2) is satisfied.
\begin{eqnarray}\nonumber f&;&\textgoth{G}[\textgoth{F}(a\rightarrow b)],\textgoth{G}\circ\textgoth{F}\\\nonumber f&;&\textgoth{G}(\textgoth{F}a\rightarrow\textgoth{F}b),\textgoth{G}\circ\textgoth{F}\\\nonumber f&;&\textgoth{G}(\textgoth{F}a)\rightarrow\textgoth{G}(\textgoth{F}b),\textgoth{G}\circ\textgoth{F}\\\nonumber f&;&(\textgoth{G}\circ\textgoth{F})a\rightarrow(\textgoth{G}\circ\textgoth{F})b,\textgoth{G}\circ\textgoth{F}.\end{eqnarray}

Finally, we see 3) holds:
\begin{eqnarray}\nonumber a*b&;&(\textgoth{G}\circ\textgoth{F})(a*b),\textgoth{G}\circ\textgoth{F}\\\nonumber a*b&;&\textgoth{G}[\textgoth{F}(a*b)],\textgoth{G}\circ\textgoth{F}\\\nonumber a*b&;&\textgoth{G}[\textgoth{F}a*\textgoth{F}b],\textgoth{G}\circ\textgoth{F}\\\nonumber a*b&;&(\textgoth{G}\circ\textgoth{F})a*(\textgoth{G}\circ\textgoth{F})b,\textgoth{G}\circ\textgoth{F}.\end{eqnarray}

The following result establishes that functors send ismorphisms into isomorphisms. 

\begin{proposition}If $\textgoth F:\mathcal C\rightarrow\mathcal D$ is a functor and $f:a\rightarrow b$ is an isomorphism in $\mathcal C$, then $\textgoth Ff$ is an isomorphism in $\mathcal D$.\label{iso map}\end{proposition}

\begin{proof}The arrow $f^{-1}:b\rightarrow a$ has image $\textgoth Ff^{-1}:\textgoth Fb\rightarrow\textgoth Fa$; we wish to show $\textgoth Ff^{-1}$ and $\textgoth Ff$ are inverse.\begin{eqnarray}\nonumber\textgoth Ff^{-1}&;&\textgoth Ff^{-1}*\textgoth Ff,\textgoth Ff\\\nonumber\textgoth Ff^{-1}&;&\textgoth F(f^{-1}*f),\textgoth Ff\\\nonumber\textgoth Ff^{-1}&;&\textgoth F1_a,\textgoth Ff\\\nonumber\textgoth Ff^{-1}&;&1_{\textgoth Fa},\textgoth Ff\end{eqnarray}\end{proof}

\subsection{Natural Pair of Functions in the Definition of Functor}

We now study the role of natural pair of functions in the definition of functor. The concept appears in very clear form when we give 1). In defining a functor we consider two categories $\mathcal C$ and $\mathcal D$. But, we have also brought into play $\textgoth1_\mathcal C:\mathcal{O|C}\rightarrow\mathcal{A|C}$. So, let $\textgoth F_\mathcal O$ and $\textgoth F_\mathcal A$ be the transformation of the domain and image, respectively. We conlcude that in order for $\textgoth F$ to be a functor, $\textgoth1_\mathcal C$ and $\textgoth1_\mathcal D$ have to be a natural pair of functions under the two simplified versions of $\textgoth F$.

Given an arrow $f:a\rightarrow b$ in $\mathcal C$, we will consider it to be a one component function. Also, the arrow $\textgoth Ff$ is a one component function $\textgoth Fa\rightarrow\textgoth Fb$. In 2), we are stating that $f$ is sent into some function $\textgoth Ff$, with which it is a natural pair of functions, under $\textgoth F_\mathcal O$. That is, the one component function $\textgoth F_\mathcal O:a\rightarrow b$.

Now we consider that $*f$ and $*\textgoth Ff$ transform $\mathcal{A|C}$ and $\mathcal{A|D}$, respectively. The third statement in the definition of functor states that $\textgoth F$ can be applied to an arrow, after applying $*f$, or one can first apply $\textgoth F$ and then $*\textgoth Fa$. Thus, we can restate 3) by saying $\textgoth F_\mathcal A$ is a natural pair with itself, under $*f$ and $*\textgoth Ff$.

		\subsection{Isomorphism}

The purpose of this subsection is to show that we already have a criteria of isomorhpism defined for functors. We will also describe weaker forms of isomorphism that make sense, for functors. 

Let $\mathcal Cat$ represent the system with all categories as objects, and functors as arrows. It is not difficult to see that functors are a collection of order arrows for categories. They are discernible arrows and composition is well defined. Consider the \textit{identity functor} $\textgoth{I}_\mathcal C$ which acts as unit on objects and arrows of the category $\mathcal C$. We know $\mathcal Cat$ is associative because functions are associative. Since identity functions act as unit under composition, so does the identity functor. This means that $\mathcal Cat$ is a category. The functor $\textgoth I_{\mathcal Cat}$ is an arrow of $\mathcal Cat$; proving this is a functor is not difficult.

Since we are dealing with a category, and functors are the arrows, it is oportune to mention that $\{\mathcal C\textgoth F\mathcal D\}$ represents the collection of functors from $\mathcal C$ into $\mathcal D$. As one may expect, we use $\mathcal C\textgoth F\mathcal D$ to represent the collection of all functors from any subcategory of $\mathcal C$ into any subcategory of $\mathcal D$.

\begin{definition}A functor $\textgoth F$ is an isomorphism between $\mathcal C$ and $\mathcal D$ if $\textgoth F_\mathcal O$ and $\textgoth F_\mathcal A$ are bijections. An isomorphism from one category to itself is called an automorphism.\end{definition}

Following our conventions, we agree to use $\{\mathcal C\textgoth F_{iso}\mathcal D\}$ in representation of the collection that consists of all isomorphisms from $\mathcal C$ to $\mathcal D$. From theorem \ref{th bij} we know the isomorphism has inverse functions $\textgoth F_{\mathcal O}^{-1}$ and $\textgoth F_{\mathcal A}^{-1}$, to the object and arrow functions, respectively. We may now prove the following result.

\begin{lemma iso}The definition of isomorphism between categories, coincides with the definition of isomorphic arrows in the category $\mathcal Cat$. That is to say, all isomorphisms are invertible arrows.\end{lemma iso}

\begin{theorem}The system of one c-object, $\mathcal C$, and automorphisms, is a group.\label{auto 1}\end{theorem}

\begin{proof}We know that the composition of an isomorphism is again an isomorphism; the composition of functors is a functor, and the composition of bijective functions is bijective. Also, $\textgoth I_\mathcal C$ is an isomorphism, so that isomorphisms are a collection of order arrows for categories. Furthermore, the system of categories and isomorphisms is a category because functions are associative.

The system of one c-object, $\mathcal C$, and automorphisms for arrows, is indeed  an algebraic category because it consists of one c-object, namely $\mathcal C$. It is a group because of the lemma.\end{proof}

\begin{corollary}The collection of bijective functions $\{\mathcal Of_{iso}\mathcal O\}$ is the $\textit{the group of transformations of X}$. If $H$ is a group such that the group of transformations is a detailed version of $H$, we say $H$ is a \textit{group of transformations of X}.\end{corollary}

		\subsection{Classification of Other Functor Types}

We have no guarantee that, given $a,b$ in $\mathcal C$ and an arrow $h:\textgoth Fa\rightarrow\textgoth Fb$ in $\mathcal D$, there is an arrow $f:a\rightarrow b$ in $\mathcal C$ such that $f;h,\textgoth F$. When this is true for every $a,b$ in $\mathcal C$ and arrow $h:\textgoth Fa\rightarrow\textgoth Fb$, we say the functor is \textit{full}. It is \textit{faithful} if for every $f,g:a\rightarrow b$, the expression $f,\textgoth F;g,\textgoth F$ implies $f,\textgoth I_{\mathcal{C}};g,\textgoth I_{\mathcal{C}}$.

Consider the function $\textgoth F_\mathcal A|_{\{a\rightarrow b\}}^{\{\textgoth Fa\rightarrow\textgoth Fb\}}$; we restrict the domain, of $\textgoth F_\mathcal A$, to $\{a\rightarrow b\}$, and then we retrict the range of $\textgoth F_\mathcal A|_{\{a\rightarrow b\}}$ to its image, $\{\textgoth Fa\rightarrow\textgoth Fb\}$. We are making the observation $\textgoth F\{a\rightarrow b\}:=\{\textgoth Fa\rightarrow\textgoth Fb\}$.

\begin{proposition}\makebox[5pt][]{}\mbox {}\begin{itemize}\item[1)]A functor $\textgoth F:\mathcal C\rightarrow\mathcal D$ is full if and only if for every pair of objects $a,b$, in the domain category $\mathcal C$, the function $\textgoth F_{\mathcal A}|_{\{a\rightarrow b\}}^{\{\textgoth Fa\rightarrow\textgoth Fb\}}$ is onto.\item[2)]A functor $\textgoth F:\mathcal C\rightarrow\mathcal D$ is faithful if and only if for every pair of objects $a,b$, in the domain category $\mathcal C$, the function $\textgoth F_{\mathcal A}|_{\{a\rightarrow b\}}^{\{\textgoth Fa\rightarrow\textgoth Fb\}}$ is monic.\end{itemize}\end{proposition}Of course, a functor that is full and faithful is not necessarily an isomorphism because we have no knowledge regarding the object function $\textgoth F_\mathcal O$.

			\subsubsection{Embedding}

An embedding is a faithful functor with monic object function. We have already studied an instance of this, when we introduced order preserving functions. We have left clear that an order embedding is a monic function. We make this the monic function  of a functor $\mathcal P\rightarrow\mathcal Q$. We need to define the arrow function and this could not be easier. We already know that the arrows of a patial order are non-discernible. Since $a\leq b$ if and only if $fa\leq fb$, the arrow function is defined.

Embedding will be a concept used throughout in order to show that one category can be seen as an extension of another. Let us consider this in terms of systems. Take a system $S$, and give a separation of it that loses information. We keep a certain subcollection $S$, and we take away the arrows that do not pertain to the objects we are left with. This is a one way of giving an idea of embedding in a system. We are attributing an object in one system to an object in another system. Different objects are assigned different objects. The relations in one system are relations in another.

			\subsubsection{Functor for Partial Orders}

We are now able to speak of functors for partial orders; these functors are order preserving functions.

\begin{lemma func par}Given an order preserving function $f:\mathcal{O|P}\rightarrow\mathcal{O|Q}$, we have a functor $\textgoth F:\mathcal P\rightarrow\mathcal Q$ whose object function is $f$.\end{lemma func par}

\begin{proof}We have already seen that partial orders are categories, so it is left to show that the properties of functor are satisfied. The arrow function is determined by property 2) of functors; recall that there is at most one arrow in a partial order. This means that a relation $a\leq b$, in $\mathcal P$, is sent into the relation $fa\leq fb$ for objects of $\mathcal Q$; we know this relation exists in $\mathcal Q$ because the function is order preserving.

Take an object $a$ in $\mathcal P$ and apply the functor to it and then we send that object to the unit arrow; the result is $\textgoth Fa\leq\textgoth Fa$. If, on the contrary we apply $\textgoth1_\mathcal P$ to $a$ and then send $a\leq a$ to its corresponding arrow, the result is $\textgoth Fa\leq\textgoth Fa$.

We are left to prove 3)\begin{eqnarray}\nonumber b\leq c*a\leq b&;&\textgoth F(b\leq c*a\leq b),\textgoth F\\\nonumber b\leq c*a\leq b&;&\textgoth F(a\leq c),\textgoth F\\\nonumber b\leq c*a\leq b&;&\textgoth Fa\leq\textgoth Fc,\textgoth F\\\nonumber b\leq c*a\leq b&;&\textgoth Fb\leq\textgoth Fc*\textgoth Fa\leq\textgoth F b,\textgoth F\\\nonumber b\leq c*a\leq b&;&\textgoth F(b\leq c)*\textgoth F(a\leq b),\textgoth F.\end{eqnarray}\end{proof}

\begin{theorem}If $f$ is an order bijectivity, then $\textgoth F$, with $f$ as object function, is an isomorphism.\end{theorem}

\begin{proof}We only need to prove that the arrow function is bijective, because $f$ is bijective. The arrows in the partial orders are non-discernible and $f$ sends every $a\leq b$ into a unique arrow $fa\leq fb$. 

If we have an arrow $x\leq y$, in $\mathcal Q$. Then, there are $a,b$ in $\mathcal P$ such that $a;x,\textgoth F$ and $b;y,\textgoth F$, and an arrow $a\leq b$. The arrow $\textgoth Fa\leq\textgoth Fb$ is the arrow assigned to $a\leq b$.\end{proof}

\subsubsection{Opposite Category and Contravariant Functor}

The topic of this paragraph is inspired on previous developments for partial orders. Given a category $\mathcal C$, we construct a system consisting of the same objects, and for every arrow $f:a\rightarrow b$ we give an arrow $f^{op}:b\rightarrow_{op}a$. The system that results is called the \textit{opposite of $\mathcal C$} and we write it with $\mathcal C^{op}$. 

Of course, $\textgoth1_{\mathcal C^{op}}$ is defined as the function that acts as $x;1_x^{op},\textgoth1_{\mathcal C^{op}}$. Now it is on us to define the operation in $\mathcal C^{op}$. Let us consider a composable pair of arrows, say $f:a\rightarrow b$ and $g:b\rightarrow c$. Then we define the arrow $f^{op}*g^{op}$ as the arrow $c\rightarrow_{op}a$ that is opposite of $b\rightarrow c*a\rightarrow b$. In other words, $f^{op}*g^{op}$ is defined as $(g*f)^{op}$.

\begin{lemma op func 1}The opposite, of a category $\mathcal C$, is also a category.\end{lemma op func 1}

\begin{proof}We would like to show that $1_a^{op}$ is indeed the unit of $a$ in the opposite category. If $f^{op}:x\rightarrow_{op}a$ is an arrow in $\mathcal C^{op}$, then we have\begin{eqnarray}\nonumber1^{op}_a&;&1^{op}_a*f^{op},f^{op}\\\nonumber
1^{op}_a&;&(f*1_a)^{op},f^{op}\\\nonumber
1^{op}_a&;&f^{op},f^{op}\end{eqnarray}

If $g^{op}:y\rightarrow_{op}b$, then\begin{eqnarray}\nonumber g^{op}&;&g^{op}*1_y^{op},1_y^{op}\\\nonumber
g^{op}&;&(1_y*g)^{op},1_y^{op}\\\nonumber
g^{op}&;&g^{op},1_y^{op}\end{eqnarray}

To show that associativty holds, consider $h^{op}:b\rightarrow_{op}x$:\begin{eqnarray}\nonumber f^{op}&;&f^{op}*(h^{op}*g^{op}),h^{op}*g^{op}\\\nonumber f^{op}&;&f^{op}*(g*h)^{op},h^{op}*g^{op}\\\nonumber f^{op}&;&[(g*h)*f]^{op},h^{op}*g^{op}\\\nonumber f^{op}&;&[g*(h*f)]^{op},h^{op}*g^{op}\\\nonumber f^{op}&;&(h*f)^{op}*g^{op},h^{op}*g^{op}\\\nonumber f^{op}&;&(f^{op}*h^{op})*g^{op},h^{op}*g^{op}\end{eqnarray}\end{proof}

\begin{lemma op func 2}Given a functor $\textgoth F:\mathcal C\rightarrow\mathcal D$, there is a functor $\textgoth F^{op}:\mathcal C^{op}\rightarrow\mathcal D^{op}$.\end{lemma op func 2}

\begin{proof}We define $\textgoth F^{op}$ as the functor with $\textgoth F_\mathcal O$ as object function. Suppose $f^{op}:b\rightarrow_{op}a$ is an arrow in $\mathcal C^{op}$; it is the opposite of the arrow $f:a\rightarrow b$ and we define $f^{op};(\textgoth Ff)^{op},\textgoth F^{op}$.

To prove 1) of functors, let $x$ be a c-object in $\mathcal C^{op}$

\begin{eqnarray}\nonumber x&;&(\textgoth1_{\mathcal D^{op}}\circ\textgoth F_\mathcal O)x,\textgoth1_{\mathcal D^{op}}\circ\textgoth F^{op}_\mathcal O\\\nonumber x&;&\textgoth1_{\mathcal D^{op}}(\textgoth Fx),\textgoth1_{\mathcal D^{op}}\circ\textgoth F^{op}_\mathcal O\\\nonumber x&;&1_{\textgoth Fx}^{op},\textgoth1_{\mathcal D^{op}}\circ\textgoth F^{op}_\mathcal O\\\nonumber x&;&\textgoth F^{op}(1_x^{op}),\textgoth1_{\mathcal D^{op}}\circ\textgoth F^{op}_\mathcal O\\\nonumber x&;&(\textgoth F^{op}\circ\textgoth1_{\mathcal C^{op}})x,\textgoth1_{\mathcal D^{op}}\circ\textgoth F^{op}_\mathcal O.\end{eqnarray}We know that $(\textgoth Ff)^{op}$ is of the form $\textgoth Fb\rightarrow_{op}\textgoth Fa$. This means that 2) is satisified. To verify 3), let $g^{op}:c\rightarrow_{op}b$, so that \begin{eqnarray}\nonumber f^{op}*g^{op}&;&\textgoth F^{op}(f^{op}*g^{op}),\textgoth F^{op}\\\nonumber f^{op}*g^{op}&;&\textgoth F^{op}(g*f)^{op},\textgoth F^{op}\\\nonumber f^{op}*g^{op}&;&[\textgoth F(g*f)]^{op},\textgoth F^{op}\\\nonumber f^{op}*g^{op}&;&(\textgoth Fg*\textgoth Ff)^{op},\textgoth F^{op}\\\nonumber f^{op}*g^{op}&;&(\textgoth Ff)^{op}*(\textgoth Fg)^{op},\textgoth F^{op}\\\nonumber f^{op}*g^{op}&;&\textgoth F^{op}f^{op}*\textgoth F^{op}g^{op},\textgoth F^{op}\end{eqnarray}\end{proof}

\begin{theorem}There is a functor $\textbf{op}:\mathcal Cat\rightarrow\mathcal Cat$ that sends a category to its opposite category, and a functor to its opposite functor.\end{theorem}

\begin{proof}The functor $\textbf{op}$, acts as $\textgoth F;\textgoth F^{op},\textbf{op}$ and $\mathcal C;\mathcal C^{op},\textbf{op}$. It is not difficult to see $(\textgoth1_{\mathcal Cat}\circ\textbf{op}_\mathcal O)\mathcal C$ is the same as $(\textbf{op}_\mathcal A\circ\textgoth1_{\mathcal Cat})\mathcal C$, which verifies 1) for functors. 2) is true because $\mathcal C,\mathcal D;\textbf{op}\mathcal C,\textbf{op}\mathcal D$ is the same as saying $\textgoth F:\mathcal C\rightarrow\mathcal D$ is sent into $\textgoth F^{op}:\mathcal C^{op}\rightarrow\mathcal D^{op}$. Given $\textgoth G:\mathcal D\rightarrow\mathcal E$, the relation $\textgoth G\circ\textgoth F;\textbf{op}\textgoth G\circ\textbf{op}\textgoth F,\textbf{op}$ holds; $(\textgoth G\circ\textgoth F)^{op}:\mathcal C^{op}\rightarrow\mathcal E^{op}$ and $\textgoth G^{op}\circ\textgoth F^{op}:\mathcal C^{op}\rightarrow\mathcal D^{op}\rightarrow\mathcal E^{op}$ are equal. Let $f^{op}$ be an arrow in $\mathcal C^{op}$,

\begin{eqnarray}\nonumber f&;&(\textgoth G\circ\textgoth F)^{op}f^{op},(\textgoth G\circ\textgoth F)^{op}\\\nonumber f&;&[(\textgoth G\circ\textgoth F)f]^{op},(\textgoth G\circ\textgoth F)^{op}\\\nonumber f&;&[\textgoth G(\textgoth Ff)]^{op},(\textgoth G\circ\textgoth F)^{op}\\\nonumber f&;&\textgoth G^{op}(\textgoth Ff)^{op},(\textgoth G\circ\textgoth F)^{op}\\\nonumber f&;&\textgoth G^{op}(\textgoth F^{op}f^{op}),(\textgoth G\circ\textgoth F)^{op}\\\nonumber f&;&(\textgoth G^{op}\circ\textgoth F^{op})f^{op},(\textgoth G\circ\textgoth F)^{op}.\end{eqnarray}\end{proof}

We say $\textgoth F^{\ltimes}$ is \textit{contravariant} if instead of satisfying condition 2), it satisfies $2)'~a,b;\textgoth F^\ltimes_{\mathcal O}b,\textgoth F^\ltimes_\mathcal Oa$. This means that arrows are reversed. In terms of natural pair of functions, we are saying that $f:a\rightarrow b$ is sent into some function $\textgoth F^\ltimes f:\textgoth F^\ltimes b\rightarrow\textgoth F^\ltimes a$ such that $f$ is a natural pair of functions with $(\textgoth F^\ltimes f)^{op}$, under $\textgoth F^\ltimes_\mathcal O$. This has implications for condition 3); let $g:b\rightarrow c$:\begin{eqnarray}\nonumber b\rightarrow c*a\rightarrow b&;&\textgoth F^\ltimes(b\rightarrow c*a\rightarrow b),\textgoth F^\ltimes\\\nonumber b\rightarrow c*a\rightarrow b&;&\textgoth F^\ltimes(a\rightarrow c),\textgoth F^\ltimes\\\nonumber b\rightarrow c*a\rightarrow b&;&\textgoth F^\ltimes c\rightarrow \textgoth F^\ltimes a,\textgoth F^\ltimes\\\nonumber b\rightarrow c*a\rightarrow b&;&\textgoth F^\ltimes b\rightarrow \textgoth F^\ltimes a*\textgoth F^\ltimes c\rightarrow\textgoth F^\ltimes b,\textgoth F^\ltimes \\\nonumber b\rightarrow c*a\rightarrow b&;&\textgoth F^\ltimes (a\rightarrow b)*\textgoth F^\ltimes (b\rightarrow c),\textgoth F^\ltimes \end{eqnarray}Thus, a contravariant functor satisfies $3)'~\textgoth F^\ltimes ,\textgoth F^\ltimes ;*\textgoth F^\ltimes g,g*$ instead of $3)$. Equivalently, $\textgoth F^\ltimes (g*f)$ is the same object as $\textgoth F^\ltimes f*\textgoth F^\ltimes g$; in notation for function, we have $g*f;\textgoth F^\ltimes f*\textgoth F^\ltimes g,\textgoth F^\ltimes $. This can be resolved in a simple way, using the opposite category of the range category. Let us be clear, a functor is \textit{covariant} when it satisfies $1),2),3)$, and \textit{contravariant} when it satisfies $1),2)',3)'$.

\begin{theorem}\makebox[5pt][]{}\mbox {}\begin{itemize}\item[1)]Given a contravariant functor $\textgoth F^\ltimes:\mathcal C\rightarrow\mathcal D$, there is a covariant functor $\textgoth F:\mathcal C\rightarrow\mathcal D^{op}$.\item[2)]Given a covariant functor $\textgoth F:\mathcal C\rightarrow\mathcal D$, there is a contravariant functor $\textgoth F^\ltimes:\mathcal C\rightarrow\mathcal D^{op}$.\end{itemize}\end{theorem}

\begin{proof}\makebox[5pt][]{}\mbox {}\begin{itemize}\item[1)]The covariant functor $\textgoth F$ is defined to have the same object function but the arrow function is $\textgoth F_\mathcal A$ which sends $f:a\rightarrow b$ into $\textgoth F_\mathcal Af:=(\textgoth F_\mathcal A^\ltimes f)^{op}:\textgoth Fa\rightarrow\textgoth Fb$. 

To show that 1) holds,\begin{eqnarray}\nonumber x&;&(\textgoth1_{\mathcal D^{op}}\circ\textgoth F^\ltimes_\mathcal O)x,\textgoth1_{\mathcal D^{op}}\circ\textgoth F_\mathcal O\\\nonumber x&;&\textgoth1_{\mathcal D^{op}}(\textgoth F_\mathcal O^\ltimes x),\textgoth1_{\mathcal D^{op}}\circ\textgoth F_\mathcal O\\\nonumber x&;&(1_{\textgoth F_\mathcal O^\ltimes x})^{op},\textgoth1_{\mathcal D^{op}}\circ\textgoth F_\mathcal O\\\nonumber x&;&(\textgoth F^\ltimes_\mathcal A1_x)^{op},\textgoth1_{\mathcal D^{op}}\circ\textgoth F_\mathcal O\\\nonumber x&;&\textgoth F_\mathcal A1_x,\textgoth1_{\mathcal D^{op}}\circ\textgoth F_\mathcal O\\\nonumber x&;&(\textgoth F_\mathcal A\circ\textgoth1_\mathcal C)x,\textgoth1_{\mathcal D^{op}}\circ\textgoth F_\mathcal O.\end{eqnarray}This means $\textgoth1_{\mathcal D^{op}},\textgoth1_\mathcal C;\textgoth F_\mathcal A,\textgoth F_\mathcal O$. The definition of $\textgoth F_\mathcal A$ assures us that $2)$ is satisfied. Moving on to $3)$, we make $f:a\rightarrow b$ and $g:b\rightarrow c$.\begin{eqnarray}\nonumber g*f&;&\textgoth F(g*f),\textgoth F\\\nonumber g*f&;&[\textgoth F^\ltimes(g*f)]^{op},\textgoth F\\\nonumber g*f&;&[\textgoth F^\ltimes f*\textgoth F^\ltimes g]^{op},\textgoth F\\\nonumber g*f&;&(\textgoth F^\ltimes g)^{op}*(\textgoth F^\ltimes f)^{op},\textgoth F\\\nonumber g*f&;&\textgoth F g*\textgoth Ff,\textgoth F\end{eqnarray}\item[2)]We will carry out the proof of the second part of this theorem, although it is similar to the first. Here, we define $\textgoth F^\ltimes_\mathcal Af:=(\textgoth F_\mathcal Af)^{op}$\begin{eqnarray}\nonumber x&;&(\textgoth1_{\mathcal D^{op}}\circ\textgoth F_\mathcal O)x,\textgoth1_{\mathcal D^{op}}\circ\textgoth F^\ltimes_\mathcal O\\\nonumber x&;&\textgoth1_{\mathcal D^{op}}(\textgoth F_\mathcal O x),\textgoth1_{\mathcal D^{op}}\circ\textgoth F^\ltimes_\mathcal O\\\nonumber x&;&(1_{\textgoth F_\mathcal O x})^{op},\textgoth1_{\mathcal D^{op}}\circ\textgoth F^\ltimes_\mathcal O\\\nonumber x&;&(\textgoth F_\mathcal A1_x)^{op},\textgoth1_{\mathcal D^{op}}\circ\textgoth F^\ltimes_\mathcal O\\\nonumber x&;&\textgoth F^\ltimes_\mathcal A1_x,\textgoth1_{\mathcal D^{op}}\circ\textgoth F^\ltimes_\mathcal O\\\nonumber x&;&(\textgoth F^\ltimes_\mathcal A\circ\textgoth1_\mathcal C)x,\textgoth1_{\mathcal D^{op}}\circ\textgoth F^\ltimes_\mathcal O.\end{eqnarray}Since $(\textgoth F_\mathcal Af)^{op}$ is $\textgoth Fb\rightarrow_{op}\textgoth Fa$, we know $\textgoth F^\ltimes$ is contravariant. Now, we prove $3)'$ holds:\begin{eqnarray}\nonumber g*f&;&\textgoth F^\ltimes(g*f),\textgoth F^\ltimes\\\nonumber g*f&;&[\textgoth F(g*f)]^{op},\textgoth F^\ltimes\\\nonumber g*f&;&[\textgoth Fg*\textgoth Ff]^{op},\textgoth F^\ltimes\\\nonumber g*f&;&(\textgoth Ff)^{op}*(\textgoth Fg)^{op},\textgoth F^\ltimes\\\nonumber g*f&;&\textgoth F^\ltimes f*\textgoth F^\ltimes g,\textgoth F^\ltimes.\end{eqnarray}\end{itemize}\end{proof}

\begin{corollary}If $f^\ltimes:\mathcal{O|P}\rightarrow\mathcal{O|Q}$ is an order reversing function, then it is the object function of a covariant functor $\textgoth F:\mathcal P\rightarrow\mathcal Q^{op}$.\end{corollary}

\begin{proof}All that we need to prove is that there is a contravariant functor $\mathcal P\rightarrow\mathcal Q$. The proof of 1) is the same as in theorem , and 2) is given by the fact that $f$ is order reversing. We are left to prove 3).\begin{eqnarray}\nonumber b\leq c*a\leq b&;&\textgoth F(b\leq c*a\leq b),\textgoth F\\\nonumber b\leq c*a\leq b&;&\textgoth F(a\leq c),\textgoth F\\\nonumber b\leq c*a\leq b&;&\textgoth Fc\leq\textgoth Fa,\textgoth F\\\nonumber b\leq c*a\leq b&;&\textgoth Fb\leq\textgoth Fa*\textgoth Fc\leq\textgoth Fb,\textgoth F\\\nonumber b\leq c*a\leq b&;&\textgoth F(a\leq b)*\textgoth F(b\leq c),\textgoth F.\end{eqnarray}\end{proof}

\begin{corollary}\makebox[5pt][]{}\mbox {}\begin{itemize}\item[1)]Let $f$ be an order bijectivity $\mathcal P,\mathcal Q;\mathcal P^{op},\mathcal Q^{op}$. Then, there is a covariant isomorphism $\textgoth F:\mathcal P\rightarrow\mathcal Q$.\item[2)]Suppose $\mathcal P$ and $\mathcal Q$ are dual orders. Then, there is a covariant isomorphism $\textgoth F:\mathcal P\rightarrow\mathcal Q^{op}$.\end{itemize}\end{corollary}

There is another side to this. We can prove that given a covariant functor $\textgoth F:\mathcal C\rightarrow\mathcal D$, there is a contravariant functor $\textgoth F_\ltimes:\mathcal C^{op}\rightarrow\mathcal D$. As one may expect, we can also say that given a contravariant functor $\mathcal C\rightarrow\mathcal D$, there is a covariant $\mathcal C^{op}\rightarrow\mathcal D$.

\begin{proposition}If $\textgoth F:\mathcal C\rightarrow\mathcal D$ is a covariant functor, then so is $\textgoth F^{op}$. Likewise, $(\textgoth F^\ltimes)^{op}$ is contravariant, given $\textgoth F^\ltimes$ is contravariant.\end{proposition}

\subsubsection{Algebraic Functor}An algebraic functor is one such that the source and target are algebraic categories. Given two algebraic categories, it is trivial to define a functor from one category to the other. Suppose $e_1$, $e_2$ are the c-objects of $\mathcal C$, $\mathcal D$, respectively. This means that every arrow in $\mathcal C$ is a reflexive arrow for $e_1$, and similarly for $\mathcal D$ and $e_2$. With this in mind, we see that any functor $\mathcal C\rightarrow\mathcal D$ has the object function defined by $\textgoth F_\mathcal Oe_1:=e_2$. One can easily see that the composition of algebraic functors is again an agebraic functor. We already know that the composition of functors is a functor and it is easy to give the object function of a compostition of algebraic functors.

Let $*$ be the operation of composition for the objects of operation, in any category $\mathcal C$, and give a function $\odot:\mathcal O\rightarrow\mathcal C\textgoth F\mathcal C$, with a collection $\mathcal O$ in the domain. We are assigning every object in the domain, a functor, from some subcategory into the category itslef. This defines an operation $\cdot:\mathcal O\rightarrow\mathcal{A|C}f\mathcal{A|C}$. Since every object in the image of $\odot$ is a functor, we can say that $\odot x(f*g)$ is the same as $\odot xf*\odot xg$. In other words, $f*g;f\cdot x*g\cdot x,x$ in terms of the notation for $\cdot$.

\begin{proposition}If there is a function from some collection $\mathcal O$ into $\mathcal C\textgoth F\mathcal C$, then there is an operation $\cdot$ that distributes over the operation $*$, of the category.\end{proposition}

Recall theorem \ref{auto 1}; we sometimes consider an algebraic functor $\mathcal Z\rightarrow\mathcal C\textgoth F\mathcal C$, given $\mathcal Z$ is algebraic.
		
			\subsection{Product Category and Bifunctor}

Given two $\mathcal C_1$, $\mathcal C_2$ categories, form a system $\mathcal C$, with $\mathcal{O|C}:=\mathcal{O|C}_1\rightarrow_\times\mathcal{O|C}_2$ and $\mathcal{A|C}:=\mathcal{A|C}_1\rightarrow_\times\mathcal{A|C}_2$; the arrows of $\mathcal C$ are objects that consist of two arrows. We are using comparability for arrows. Weak arrows are those in $\mathcal C_1$ and $\mathcal C_2$. The arrows in $\mathcal C$ are strong (ordered pairs of arrows). In short, $\mathcal C$ is a category, with $\mathcal{O|C}:=\mathcal{O|C}_1\rightarrow_\times\mathcal{O|C}_2$ and $\mathcal{A|C}:=\mathcal{A|C}_1\rightarrow_\times\mathcal{A|C}_2$.

We define the arrows in a natural manner. Let $a\rightarrow_\times b$ a c-object, in $\mathcal C$. Take arrows $f:a\rightarrow_{\mathcal C_1}c$ and $g:b\rightarrow_{\mathcal C_2}d$; to avoid confusion, we will use the notation $a\rightarrow_{\mathcal C_1}c$ to express that the arrow is in category $\mathcal C_1$. The reader can see that $a\rightarrow_\times b$ and $c\rightarrow_\times d$ are a natural pair of functions, under $f,g$. Moreover, $a\rightarrow_\times b$ and $c\rightarrow_\times d$ are c-objects in $\mathcal C$. We define an arrow $f\rightarrow_\times g:a\rightarrow_\times b\longrightarrow c\rightarrow_\times d$.

\begin{theorem}The system $\mathcal C$, defined in terms of $\mathcal C_1$, $\mathcal C_2$, is a category and we say it is the product category of $\mathcal C_1$ and $\mathcal C_2$ (in that order). This category may be expressed as $\mathcal C_1\times\mathcal C_2$.\end{theorem}

\begin{proof}We first need to prove that $\mathcal{A|C}_1\rightarrow_\times\mathcal{A|C}_2$ is a collection of order arrows for $\mathcal{O|C}_1\rightarrow_\times\mathcal{O|C}_2$. Given a c-object $a\rightarrow_\times b$, in $\mathcal C$, we know the c-objects $a,b$ are in $\mathcal C_1$, $\mathcal C_2$, respectively. These satisfy $a;1_a,\textgoth1_{\mathcal C_1}$ and $b;1_b,\textgoth1_{\mathcal C_2}$, for their unit arrows because $\mathcal C_1$ and $\mathcal C_2$ are categories. We define the unit arrow $1_{a\rightarrow_\times b}$ as the arrow $1_a\rightarrow_\times1_b:a\rightarrow_\times b\longrightarrow a\rightarrow_\times b$.

Suppose we have two arrows in the product category; say $f\rightarrow_\times g:a\rightarrow_\times b\longrightarrow c\rightarrow_\times d$ and $h\rightarrow_\times i:c\rightarrow_\times d\longrightarrow x\rightarrow_\times y$. It is clear that there is an arrow $h*f$, in $\mathcal C_1$, and an arrow $i*g$, in $\mathcal C_2$. The composition $h\rightarrow_\times i*f\rightarrow_\times g:a\rightarrow_\times b\longrightarrow x\rightarrow_\times y$, is defined as the arrow $h*f\rightarrow_\times i*g$, so that $a\rightarrow_\times b$ and $x\rightarrow_\times y$ are a natural pair under $h*f:a\rightarrow x$ and $i*g:b\rightarrow y$. In other words, if we use the notation for the operation $*$ of $\mathcal C$, we have \begin{equation}h\rightarrow_\times i;h*f\rightarrow_\times i*g,f\rightarrow_\times g.\label{def prod cat}\end{equation}

We move on to show that the unit and associativity are true, for the operation $*$  that acts on the category $\mathcal C$. First, we will show that $1_{c\rightarrow_\times d}$ is the unit of the c-object $c\rightarrow_\times d$, in $\mathcal C$. \begin{eqnarray}\nonumber1_{c\rightarrow_\times d}&;&1_c\rightarrow_\times1_d*f\rightarrow_\times g,f\rightarrow_\times g\\\nonumber1_{c\rightarrow_\times d}&;&1_c*f\rightarrow_\times 1_d*g,f\rightarrow_\times g\\\nonumber1_{c\rightarrow_\times d}&;&f\rightarrow_\times g,f\rightarrow_\times g.\end{eqnarray} \begin{eqnarray}h\rightarrow_\times i&;&h\rightarrow_\times i*1_c\rightarrow_\times1_d,\nonumber1_{c\rightarrow_\times d}\\\nonumber h\rightarrow_\times i&;&h*1_c\rightarrow_\times i*1_d,1_{c\rightarrow_\times d}\\\nonumber h\rightarrow_\times i&;&h\rightarrow_\times i,1_{c\rightarrow_\times d}.\end{eqnarray}

Finally, we have to prove associativity. We have to consider two more arrows, let us use $j:x\rightarrow w$ and $k:y\rightarrow z$. We have a new arrow in $\mathcal C$; the arrow is $j\rightarrow_\times k$. In terms of the notation for $*$,\begin{eqnarray}\nonumber j\rightarrow_\times k&;&(j\rightarrow_\times k)*(h\rightarrow_\times i*f\rightarrow_\times g),h\rightarrow_\times i*f\rightarrow_\times g\\\nonumber j\rightarrow_\times k&;&(j\rightarrow_\times k)*(h*f\rightarrow_\times i*g),h\rightarrow_\times i*f\rightarrow_\times g\\\nonumber j\rightarrow_\times k&;&j*(h*f)\rightarrow_\times k*(i*g),h\rightarrow_\times i*f\rightarrow_\times g\\\nonumber j\rightarrow_\times k&;&(j*h)*f\rightarrow_\times(k*i)*g,h\rightarrow_\times i*f\rightarrow_\times g\\\nonumber j\rightarrow_\times k&;&(j*h\rightarrow_\times k*i)*(f\rightarrow_\times g),h\rightarrow_\times i*f\rightarrow_\times g\\\nonumber j\rightarrow_\times k&;&(j\rightarrow_\times k*h\rightarrow_\times i)*(f\rightarrow_\times g),h\rightarrow_\times i*f\rightarrow_\times g\\\nonumber j\rightarrow_\times k,f\rightarrow_\times g&;&j\rightarrow_\times k*h\rightarrow_\times i,h\rightarrow_\times i*f\rightarrow_\times g\end{eqnarray}\end{proof}

\begin{proposition}The opposite of a product category, is the product of the opposites; that is to say $(\mathcal C_1\times\mathcal C_2)^{op}$ and $\mathcal C_1^{op}\times\mathcal C_2^{op}$ are the same category.\end{proposition}

\begin{proof}A category and its opposite have the same c-objects, so we have not much to prove about the c-objects, we are interested in the arrows. Let $g\rightarrow_{\times~op} f$ represent the arrow $(f\rightarrow_\times g)^{op}$, in $(\mathcal C_1\times\mathcal C_2)^{op}$. The arrow $f\rightarrow_\times g$ is of the form $a\rightarrow_\times b\longrightarrow c\rightarrow_\times d$ because $f:a\rightarrow_{\mathcal C_1} c$ and $g:b\rightarrow_{\mathcal C_2}d$. We are able to express $g\rightarrow_{\times~op} f:c\rightarrow_\times d\longrightarrow_{op} a\rightarrow_\times b$ which means that we want to define  $(f\rightarrow_\times g)^{op}$ as an ordered pair, of $c\rightarrow_{\mathcal C_1}a$ and $d\rightarrow_{\mathcal C_2}b$. We define $(f\rightarrow_\times g)^{op}:=f^{op}\rightarrow_\times g^{op}$, which is in $\mathcal A|(\mathcal C_1^{op}\times\mathcal C_2^{op})$.

We have to show that this definition and the definition of product for opposite category are consistent. with each other. We reiterate the naming of arrows, as in the previous result; let $f:a\rightarrow_{\mathcal C_1}c$, $h:c\rightarrow_{\mathcal C_1}x$, $g:b\rightarrow_{\mathcal C_2}d$, and $i:d\rightarrow_{\mathcal C_2}y$. We know $f\rightarrow_\times g$ takes $a\rightarrow_\times b$ into $c\rightarrow_\times d$, and $h\rightarrow_\times i$ takes $c\rightarrow_\times d$ into $x\rightarrow_\times y$.\begin{eqnarray}\nonumber(f\rightarrow_\times g)^{op}&;&(f\rightarrow_\times g)^{op}*(h\rightarrow_\times i)^{op},(h\rightarrow_\times i)^{op}\\\nonumber(f\rightarrow_\times g)^{op}&;&f^{op}\rightarrow_\times g^{op}*h^{op}\rightarrow_\times i^{op},(h\rightarrow_\times i)^{op}\\\nonumber(f\rightarrow_\times g)^{op}&;&f^{op}*h^{op}\rightarrow_\times g^{op}*i^{op},(h\rightarrow_\times i)^{op}\\\nonumber(f\rightarrow_\times g)^{op}&;&(h*f)^{op}\rightarrow_\times (i*g)^{op},(h\rightarrow_\times i)^{op}\\\nonumber(f\rightarrow_\times g)^{op}&;&(h*f\rightarrow_\times i*g)^{op},(h\rightarrow_\times i)^{op}\\\nonumber(f\rightarrow_\times g)^{op}&;&(h\rightarrow_\times i*f\rightarrow_\times g)^{op},(h\rightarrow_\times i)^{op}\end{eqnarray}\end{proof}

			\subsubsection{Product Functor for Common Domain}It is interesting to see what happens when a functor is defined for product categories. What can be said about a functor $\textgoth F:\mathcal C_1\times\mathcal C_2\rightarrow\mathcal D_1\times\mathcal D_2$? We will carry out a detailed analysis, one step at a time. Begin, considering a functor into a product category.

\begin{proposition}Given functors $\textgoth f,\textgoth g:\mathcal C\rightarrow\mathcal D_1,\mathcal D_2$, there is a functor $\textgoth f\times\textgoth g:\mathcal C\rightarrow\mathcal D_1\times\mathcal D_2$, and we say it is the product functor of common domain.\label{common domain 1}\end{proposition}

\begin{proof}Let $\textgoth f:\mathcal C\rightarrow D_1$ and $\textgoth g:\mathcal C\rightarrow\mathcal D_2$ be two functors. We shall build a functor $\textgoth F:\mathcal C\rightarrow\mathcal D$, where the range is the product category. The object function, of $\textgoth F$, sends $a\mapsto\textgoth fa\rightarrow_\times\textgoth ga$. We make the arrow function send $f:a\rightarrow_{\mathcal C}c$ into $\textgoth ff\rightarrow_\times\textgoth gf$. To verify we have a functor, apply $\textgoth1_\mathcal D\circ\textgoth F$ to any c-object $a$ in $\mathcal C$, the result is $1_{\textgoth fa}\rightarrow_\times1_{\textgoth ga}$. On the other hand, $a\mapsto_{\textgoth1_\mathcal C}(1_a)\mapsto_\textgoth F(\textgoth f1_a\rightarrow_\times\textgoth g1_a)$; we can say 1) of functors is true because it is true for $\textgoth f$ and $\textgoth g$. We would like to see that $\textgoth Ff$ is an arrow of the form $\textgoth fa\rightarrow_\times\textgoth ga\longrightarrow\textgoth fc\rightarrow_\times\textgoth gc$, which is true because $\textgoth Ff$ is $\textgoth ff\rightarrow_\times\textgoth gf$. With this we have proven that an arrow from $a\rightarrow_\mathcal Cc$ is sent into an arrow $\textgoth Fa\rightarrow\textgoth Fc$. Lastly, we will show that 3) holds. Let $h:c\rightarrow x$, then\begin{eqnarray}\nonumber h*f&;&\textgoth F(h*f),\textgoth F\\\nonumber h*f&;&\textgoth f(h*f)\rightarrow_\times\textgoth g(h*f),\textgoth F\\\nonumber h*f&;&\textgoth fh*\textgoth ff\rightarrow_\times\textgoth gh*\textgoth gf,\textgoth F\\\nonumber h*f&;&\textgoth fh\rightarrow_\times\textgoth gh*\textgoth ff\rightarrow_\times\textgoth gf,\textgoth F\\\nonumber h*f&;&\textgoth Fh*\textgoth Ff,\textgoth F.\end{eqnarray}\end{proof}

One may now ask, given a functor $\textgoth F:\mathcal C\rightarrow\mathcal D_1\times\mathcal D_2$, is it possible to decompose it into two functors with common domain? We are saying that the functor sends c-objects $a\mapsto(\textgoth F_1a\rightarrow_\times\textgoth F_2a)$; since the image of a c-object is an order arrow $\rightarrow_\times$, we use $\textgoth F_1a$ to denote the source associated to $a$ and $\textgoth F_2a$ is the target. Also, the functor sends an arrow $f:a\rightarrow_\mathcal Cc$ into an arrow $\textgoth F_1a\rightarrow_\times\textgoth F_2a\longrightarrow\textgoth F_1c\rightarrow_\times\textgoth F_2c$ which means we have two arrows, $\textgoth F_1f:\textgoth F_1a\rightarrow\textgoth F_1c$ and $\textgoth F_2f:\textgoth F_2a\rightarrow\textgoth F_2c$. We can associate two arrows in the range, to every arrow in the domain; $f$ is assigned the arrows $\textgoth F_1f$, $\textgoth F_2f$.

\begin{proposition}Given a functor $\textgoth F:\mathcal C\rightarrow\mathcal D_1\times\mathcal D_2$, there are functors $\textgoth f:\mathcal C\rightarrow\mathcal D_1$ and $\textgoth g:\mathcal C\rightarrow\mathcal D_2$ such that $\textgoth f\times\textgoth g$ and $\textgoth F$ are the same functor.\label{common domain 2}\end{proposition}

We have proven a functor has a product category, in the image, if and only if it can be decomposed as two functors with common domain.

			\subsubsection{Bifunctor[IX]}It is possible to give a functor that acts on pairs of c-objects, and pairs of arrows; a product category is domain, instead of range. The first thing to be brought to our attention will be the limitations in dealing with these functors. One may expect to be able to form two functors with common range, by decomposing a functor $\textgoth F:\mathcal C_1\times\mathcal C_2\rightarrow\mathcal D$. This is not true, however. The reader may convince himself that there is no straightforward way of defining $\textgoth f:\mathcal C_1\rightarrow\mathcal D$ or $\textgoth g:\mathcal C_2\rightarrow\mathcal D$, given a functor $\mathcal C_1\times\mathcal C_2\rightarrow\mathcal D$.  Also, it is interesting to notice that the range is a simple category (not necessarily a product category) and we assign each pair of objects into a single object in $\mathcal D$. Again, we take two arrows $f,g$ (from $\mathcal C_1$ and $\mathcal C_2$) and we give a single arrow.

When we try giving the product of common range, for two functors $\mathcal C_1,\mathcal C_2\rightarrow\mathcal D$, we come across the following difficulty. Is it possible to build a functor from the product category of the domains into $\mathcal D$? Take functors $\textgoth f,\textgoth g:\mathcal C_1,\mathcal C_2\rightarrow\mathcal D$, and define $\textgoth F:\mathcal C_1\times\mathcal C_2\rightarrow\mathcal D$ such that the c-objects are transformed according to $(a\rightarrow_\times b)\mapsto_\textgoth F\textgoth F(a,b):=\textgoth fa$. Give arrows $f:a\rightarrow c$ and $g:b\rightarrow d$, then the action of $\textgoth F$ on the arrow $f\rightarrow_\times g$ is defined by $\textgoth F(f,g):=\textgoth ff:\textgoth fa\rightarrow\textgoth fc$. Let us prove $\textgoth F$ this is a functor. It is easily seen that $(\textgoth1_\mathcal D\circ\textgoth F)(a\rightarrow_\times b)$ is $1_{\textgoth fa}$, just as $(\textgoth F\circ\textgoth1_{\mathcal C_1\times\mathcal C_2})(a\rightarrow_\times b)$. The second condition of functors is satisfied; if we apply $\textgoth F$ to $f\rightarrow_\times g:a\rightarrow_\times b\longrightarrow c\rightarrow_\times d$, the result is an arrow of the form $\textgoth F(a,b)\rightarrow\textgoth F(c,d)$, because $\textgoth F(a,b)$ is $\textgoth fa$ and $\textgoth F(c,d)$ is $\textgoth fc$. To verify that 3) also holds, take a composable pair of arrows $f\rightarrow_\times g$ and $h\rightarrow_\times i$, and apply $\textgoth F$ to their composition, which has been defined as $h\rightarrow_\times i*f\rightarrow_\times g:=h*f\rightarrow_\times i*g$. This is expressed as $\textgoth F(h*f,i*g)$, and in reality it is the arrow $\textgoth f(h*f)$. Also, $\textgoth F(h,i)*\textgoth F(f,g)$ is the same as $\textgoth fh*\textgoth ff$. Of course, one can also build a functor $\textgoth G$ such that $a\rightarrow_\times b$ is sent into $\textgoth gb$ and $f\rightarrow_\times g$ is sent into $\textgoth gg$. Although we were able to find functors $\textgoth F$ and $\textgoth G$, given $\textgoth f,\textgoth g$ that have the same range, the functors we found are relatively simple. They are simple in the sense that we are losing all the infomation of one functor, or another. We are building functors whose transformation disregards the information regarding one functor. A solution to this, is proposed in the following proposition.

\begin{proposition}Let $\textgoth f,\textgoth g:\mathcal C_1,\mathcal C_2\rightarrow\mathcal D$ be functors, then we can define a functor $\textgoth f\times\textgoth g:\mathcal C_1\times\mathcal C_2\rightarrow\mathcal D\times\mathcal D$ called the product of common range.\label{bif prop1}\end{proposition}

\begin{proof}We define a c-object function, $(\textgoth f\times\textgoth g)(a,b):=\textgoth fa\rightarrow_\times\textgoth gb$. Let the arrow function be defined by that transformation $(\textgoth f\times\textgoth g)_\mathcal A(f,g):=\textgoth ff\rightarrow_\times\textgoth gg$. The reader may prove that this is a functor. If he runs across any difficulty, a similar construction and proof is carried out, in the theorem of this division.\end{proof}

Let $\textgoth B:\mathcal C_1\times\mathcal C_2\rightarrow\mathcal D$ be a function which has a product category in the domain, and suppose it sends a c-object $a\rightarrow_\times b$ into a c-object $\textgoth B(a,b)$. An arrow $f\rightarrow_\times g:a\rightarrow_\times b\longrightarrow c\rightarrow_\times d$, is sent into an arrow $\textgoth B(f,g):\textgoth B(c,b)\rightarrow\textgoth B(a,d)$. We know there are arrows $c\rightarrow_\times b$ and $a\rightarrow_\times d$ in $\mathcal C_1\times\mathcal C_2$. We will say \textit{$\textgoth B$ is contravariant with respect to the first argument and covariant with respect to the second argument} if the source and target objects of $\textgoth B(f,g)$, are the images of $c\rightarrow_\times b$ and $a\rightarrow_\times d$, respectively. Although we are sending pairs into singles, we do it in such a way that the source of $\textgoth B(f,g)$ depends on $c$ and $b$. Conversely, the target depends on $a$ and $d$. A function $\textgoth B$ can be suitably defined to be covariant in the first argument and contravariant in the second. There can also be functions covariant (or contravariant) in both arguments.

In (\ref {def prod cat}), we are giving an operation $*:\mathcal{A|C}\rightarrow\mathcal {A|C}f\mathcal{A|C}$. For every $f\rightarrow_\times g$, there is a right operation function $(*f,*g):\mathcal{A|C}\rightarrow\mathcal {A|C}$. The function $(*f,*g)$ sends $h\rightarrow_\times i$ into $h*f\rightarrow_\times i*g$. We have potential for more functions than just the right operation. The left operation of $h\rightarrow_\times i$ is defined as the function $(h*,i*)$ such that $f\rightarrow_\times g;h*f\rightarrow_\times i*g,(h*,i*)$. Still this is not all, since we can define the right-left operation which is a function $(*f,i*)$ that sends $h\rightarrow_\times g$ into $h*f\rightarrow i*g$. Finally, we define the following the left-right operation and use this to express contravariance in the first argument and covariance in the second. Given two arrows $h:c\rightarrow_{\mathcal C_1}x$ and $g:b\rightarrow_{\mathcal C_2}d$, we define the function $(h*,*g):\mathcal{A|C}\rightarrow\mathcal{A|C}$ that sends $(f\rightarrow_\times i)\mapsto(h*f\rightarrow_\times i*g)$, where $f:a\rightarrow_{\mathcal C_1}c$ and $i:d\rightarrow_{\mathcal C_2}y$. If we apply $\textgoth B$ to $h*f\rightarrow_\times i*g$, the result is the arrow $\textgoth B(h*f,i*g)$, which is of the form $\textgoth B(x,b)\rightarrow\textgoth B(a,y)$. On the other hand, applying $*\textgoth B(h,g)$ to $\textgoth B(f,i)$, we get \begin{eqnarray}\nonumber \textgoth B(f,i)&;&*\textgoth B(h,g)\textgoth B(f,i),*\textgoth B(h,g)\\\nonumber \textgoth B(f,i)&;&*[\textgoth B(x,b)\rightarrow\textgoth B(c,d)][\textgoth B(c,d)\rightarrow\textgoth B(a,y)],*\textgoth B(h,g)\\\nonumber \textgoth B(f,i)&;&[\textgoth B(c,d)\rightarrow\textgoth B(a,y)]*[\textgoth B(x,b)\rightarrow\textgoth B(c,d)],*\textgoth B(h,g)\\\nonumber \textgoth B(f,i)&;&\textgoth B(x,b)\rightarrow\textgoth B(a,y),*\textgoth B(h,g).\end{eqnarray} With this we show that $\textgoth B(f,i)*\textgoth B(h,g)$ is parallel to $\textgoth B(h*f,i*g)$, given $\textgoth B$ is contravariant in the first argument and covariant in the second. A function $\textgoth B:\mathcal C_1\times\mathcal C_2\rightarrow\mathcal D$ is a \textit{bifunctor} if\begin{itemize}\item[1)]$\textgoth1_\mathcal D,\textgoth1_{\mathcal C_1\times\mathcal C_2};\textgoth B_\mathcal A,\textgoth B_\mathcal O$\item[2)]$a\rightarrow_\times b,c\rightarrow_\times d;\textgoth B(c,b),\textgoth B(a,d)$\item[3)]$\textgoth B,\textgoth B;*\textgoth B(h,g),(h*,*g)$.\end{itemize}The first condition is stating that for any c-object, in $\mathcal C$, $\textgoth B(1_a,1_b)$, is the same arrow in $\mathcal D$, as $1_{\textgoth B(a,b)}$. It is the usual first condition for functors, where $\mathcal C$ is now a product category. The second request for a bifunctor is the statement that the arrow $f\rightarrow_\times g:a\rightarrow_\times b\longrightarrow c\rightarrow_\times d$ is sent into an arrow of the form $\textgoth B(c,b)\rightarrow\textgoth B(a,d)$. In 3) we are expressing that applying $\textgoth B\circ(h*,*g)$ or $*\textgoth B(h,g)\circ\textgoth B$ to $(f\rightarrow_\times i)$, yields the same result in both cases. If it is a bifunctor, then $h*f\rightarrow_\times i*g;\textgoth B(f,i)*\textgoth B(h,g),\textgoth B$ is true. Equivalently, to 3), we can say for every composable pair of arrows, $f\rightarrow_\times g$ and $h\rightarrow_\times i$, we verify condition $3)'~h\rightarrow_\times i*f\rightarrow_\times g;\textgoth B(f\rightarrow_\times i)*\textgoth B(h\rightarrow_\times g),\textgoth B$.

\begin{proposition}$\textgoth F:\mathcal C_1\times\mathcal C_2\rightarrow\mathcal D$ is a functor if and only if there is a bifunctor $\textgoth B:\mathcal C_1^{op}\times\mathcal C_2\rightarrow\mathcal D$.\label{bif prop2}\end{proposition}

\begin{proof}To verify that we have a bifunctor, we start with the property of unit arrows.\begin{eqnarray}\nonumber a\rightarrow_\times b&;&(\textgoth1_\mathcal D\circ\textgoth F)(a\rightarrow_\times b),\textgoth1_{\mathcal D}\circ\textgoth B_{\mathcal O}\\\nonumber a\rightarrow_\times b&;&\textgoth1_\mathcal D\textgoth F(a,b),\textgoth1_{\mathcal D}\circ\textgoth B_{\mathcal O}\\\nonumber a\rightarrow_\times b&;&1_{\textgoth F(a,b)},1_{\mathcal D}\circ\textgoth B_{\mathcal O}\\\nonumber a\rightarrow_\times b&;&\textgoth F_\mathcal A1_{(a\rightarrow_\times b)},1_{\mathcal D}\circ\textgoth B_{\mathcal O}\\\nonumber a\rightarrow_\times b&;&\textgoth F_\mathcal A(1_a\rightarrow_\times 1_b),1_{\mathcal D}\circ\textgoth B_{\mathcal O}\\\nonumber a\rightarrow_\times b&;&\textgoth B_\mathcal A[1_a^{op}\rightarrow_\times 1_b],1_{\mathcal D}\circ\textgoth B_{\mathcal O}\\\nonumber a\rightarrow_\times b&;&(\textgoth B_\mathcal A\circ\textgoth1_{\mathcal C_1^{op}\times\mathcal C_2})(a\rightarrow_\times b),1_{\mathcal D}\circ\textgoth B_{\mathcal O}\end{eqnarray}With this we are saying that $\textgoth B_\mathcal O(a,b)$ and $\textgoth F_\mathcal O(a,b)$ are the same, while $\textgoth B_\mathcal A(f^{op},g):=\textgoth F_\mathcal A(f,g)$. An arrow, in the domain of $\textgoth B$, say $f^{op}\rightarrow g:c\rightarrow_\times b\longrightarrow a\rightarrow_\times d$, is transformed into the arrow $\textgoth F(f,g):\textgoth F(a,b)\rightarrow\textgoth F(c,d)$. We prove that composition is preserved, with\begin{eqnarray}\nonumber f^{op}\rightarrow_\times i*h^{op}\rightarrow_\times g&;&\textgoth B(f^{op}*h^{op}\rightarrow_\times i*g),\textgoth B\\\nonumber f^{op}\rightarrow_\times i*h^{op}\rightarrow_\times g&;&\textgoth B(f^{op}*h^{op},i*g),\textgoth B\\\nonumber f^{op}\rightarrow_\times i*h^{op}\rightarrow_\times g&;&\textgoth B[(h*f)^{op},i*g],\textgoth B\\\nonumber f^{op}\rightarrow_\times i*h^{op}\rightarrow_\times g&;&\textgoth F(h*f,i*g),\textgoth B\\\nonumber f^{op}\rightarrow_\times i*h^{op}\rightarrow_\times g&;&\textgoth F(h\rightarrow_\times i)*\textgoth F(f\rightarrow_\times g),\textgoth B\\\nonumber f^{op}\rightarrow_\times i*h^{op}\rightarrow_\times g&;&\textgoth F(h,i)*\textgoth F(f,g),\textgoth B\\\nonumber f^{op}\rightarrow_\times i*h^{op}\rightarrow_\times g&;&\textgoth B(h^{op},i)*\textgoth B(f^{op},g),\textgoth B\end{eqnarray}

Given $\textgoth B$, we define a functor with $\textgoth F(a,b):=\textgoth B(a,b)$ and $\textgoth F(f,g):=\textgoth B(f^{op},g)$.\end{proof}

There is a bifunctor, for every functor with a product category in the domain. We will continue to give examples of bifunctors, from given functors. Given two functors, it shouldn't be difficult to build a bifunctor, since it is contravariant in the first argument and covariant in the second argument. In fact, the reader may verify that the following result is a corollary of the last two propositions. Nevertheless, we still give the proof, so as to help the reader with understanding the constructions better.

\begin{proposition}Given two functors $\textgoth f,\textgoth g:\mathcal C_1, \mathcal C_2\rightarrow\mathcal D$, it is possible to define a bifunctor of the form $\textgoth f\times\textgoth g:\mathcal C_1^{op}\times\mathcal C_2\rightarrow\mathcal D\times\mathcal D$. The object function is defined so that $(\textgoth f\times\textgoth g)(a,b)$ is $\textgoth fa\rightarrow_\times\textgoth gb$, and $(\textgoth f\times\textgoth g)(f^{op},g)$ is the same as $\textgoth ff\rightarrow_\times\textgoth gg$.\label{bif prop3}\end{proposition}

\begin{proof}To prove 1),\begin{eqnarray}\nonumber\textgoth B(a,b)&;&\textgoth1_{\mathcal D\times\mathcal D}\textgoth B(a,b),\textgoth1_{\mathcal D\times\mathcal D}\\\nonumber\textgoth B(a,b)&;&\textgoth1_{\mathcal D\times\mathcal D}(\textgoth fa\rightarrow_\times\textgoth gb),\textgoth1_{\mathcal D\times\mathcal D}\\\nonumber\textgoth B(a,b)&;&1_{(\textgoth fa\rightarrow_\times\textgoth gb)},\textgoth1_{\mathcal D\times\mathcal D}\\\nonumber\textgoth B(a,b)&;&1_{\textgoth fa}\rightarrow_\times1_{\textgoth gb},\textgoth1_{\mathcal D\times\mathcal D}\\\nonumber\textgoth B(a,b)&;&\textgoth f1_a\rightarrow_\times\textgoth g1_b,\textgoth1_{\mathcal D\times\mathcal D}\\\nonumber\textgoth B(a,b)&;&\textgoth B(1_a^{op},1_b),\textgoth1_{\mathcal D\times\mathcal D}\\\nonumber\textgoth B(a,b)&;&\textgoth B(1_a^{op}\rightarrow_\times 1_b),\textgoth1_{\mathcal D\times\mathcal D}\\\nonumber\textgoth B(a,b)&;&(\textgoth B\circ\textgoth1_{\mathcal C_1^{op}\times\mathcal C_2})(a\rightarrow_\times b),\textgoth1_{\mathcal D\times\mathcal D}.\end{eqnarray}

Let us now look at the second condition of bifunctors; an arrow $f^{op}\rightarrow_\times g:c\rightarrow_\times b\longrightarrow a\rightarrow_\times d$ is sent into $\textgoth ff\rightarrow_\times\textgoth gg:\textgoth fa\rightarrow_\times\textgoth gb\longrightarrow\textgoth fc\rightarrow_\times\textgoth gd$, which is of the form $\textgoth B(a,b)\rightarrow\textgoth B(c,d)$. Finally,\begin{eqnarray}\nonumber f^{op}\rightarrow_\times i*h^{op}\rightarrow_\times g&;&\textgoth B(f^{op}\rightarrow_\times i*h^{op}\rightarrow_\times g),\textgoth B\\\nonumber f^{op}\rightarrow_\times i*h^{op}\rightarrow_\times g&;&\textgoth B(f^{op}*h^{op}\rightarrow_\times i*g),\textgoth B\\\nonumber f^{op}\rightarrow_\times i*h^{op}\rightarrow_\times g&;&\textgoth B(f^{op}*h^{op},i*g),\textgoth B\\\nonumber f^{op}\rightarrow_\times i*h^{op}\rightarrow_\times g&;& \textgoth f(h*f)\rightarrow_\times\textgoth g(i*g),\textgoth B\\\nonumber f^{op}\rightarrow_\times i*h^{op}\rightarrow_\times g&;&\textgoth fh*\textgoth ff\rightarrow_\times\textgoth gi*\textgoth gg,\textgoth B\\\nonumber f^{op}\rightarrow_\times i*h^{op}\rightarrow_\times g&;&\textgoth fh\rightarrow_\times\textgoth gi*\textgoth ff\rightarrow_\times\textgoth gg,\textgoth B\\\nonumber f^{op}\rightarrow_\times i*h^{op}\rightarrow_\times g&;&\textgoth B(h^{op},i)*\textgoth B(f^{op},g),\textgoth B.\end{eqnarray}\end{proof}

\subsubsection{Decomposing Bifunctors}

We proceed with the study of a functor $\textgoth F:\mathcal C_1\times\mathcal C_2\rightarrow\mathcal D_1\times\mathcal D_2$. To begin with, it sends a c-object, say $a\rightarrow_\times b$ in $\mathcal C_1\times\mathcal C_2$, into a c-object $\textgoth F(a,b):=x\rightarrow_\times y$ in $\mathcal D_1\times\mathcal D_2$. Arrows have to be sent into arrows; give $f:a\rightarrow c$ and $g:b\rightarrow d$ in $\mathcal C_1$ and $\mathcal C_2$, respectively. Suppose $\textgoth F(b,d)=w\rightarrow_\times z$, then $f\rightarrow_\times g$ is an arrow of the form $a\rightarrow_\times b\longrightarrow c\rightarrow_\times d$, in the domain. The corresponding image is an arrow $\textgoth F(f\rightarrow_\times g):\textgoth F(a,b)\longrightarrow \textgoth F(c,d)$. Let us make observations based on the results that we already have. First, is an extension of propositions \ref{common domain 1} and \ref{bif prop1}.

\begin{proposition}Given two functors $\textgoth f,\textgoth g:\mathcal C_1,\mathcal C_2\rightarrow\mathcal D_1,\mathcal D_2$, there is a functor $\textgoth f\times\textgoth g:C_1\times\mathcal C_2\rightarrow\mathcal D_1\times\mathcal D_2$.\end{proposition}

\begin{proof}The product functor is defined as $\textgoth F_\mathcal O(a,b):=\textgoth fa\rightarrow_\times\textgoth gb$ and $\textgoth F_\mathcal A(f,g):=\textgoth ff\rightarrow_\times\textgoth gg$.\end{proof}

\begin{theorem}$\textgoth B:\mathcal C_1\times\mathcal C_2\rightarrow\mathcal D_1\times\mathcal D_2$ is a bifunctor if and only if $\textgoth B$ can be expressed as a product of common domain for two bifunctors $\textgoth p,\textgoth q:\mathcal C_1\times\mathcal C_2\rightarrow\mathcal D_1,\mathcal D_2$.\end{theorem}

\begin{proof}Proposition \ref{common domain 2} implies that the corresponding functor $\mathcal C_1^{op}\times\mathcal C_2\rightarrow\mathcal D_1\times\mathcal D_2$, associated to $\textgoth B$, can be decomposed as two functors of common domain $\textgoth f,\textgoth g:\mathcal C_1^{op}\times\mathcal C_2\rightarrow\mathcal D_1,\mathcal D_2$. We consider bifunctors $\textgoth p,\textgoth q:\mathcal C_1\times\mathcal C_2\rightarrow\mathcal D_1,\mathcal D_2$, by applying proposition \ref{bif prop2} a second time.

Let $\textgoth p,\textgoth q:\mathcal C_1\times\mathcal C_2\rightarrow\mathcal D_1,\mathcal D_2$ two bifunctors of common domain, then we can construct two functors $\textgoth f,\textgoth g:\mathcal C_1^{op}\times\mathcal C_2\rightarrow\mathcal D_1,\mathcal D_2$, again with result \ref{bif prop2}. From these two we form a functor of common domain, $\textgoth f\times\textgoth g:\mathcal C_1^{op}\times\mathcal C_2\rightarrow\mathcal D_1\times\mathcal D_2$. Now we can give a bifunctor $\textgoth B:\mathcal C_1\times\mathcal C_2\rightarrow\mathcal D_1\times\mathcal D_2$.

The reader is left to prove the eequality of the product with $\textgoth B$.\end{proof}

\begin{corollary}$\textgoth B:\mathcal C_1\times\mathcal C_2\rightarrow\mathcal D\times\mathcal D$ is a bifunctor if and only if there are bifunctors $\textgoth p,\textgoth q:\mathcal C_1\times\mathcal C_2\rightarrow\mathcal D$, parallel in $\mathcal Cat$.\end{corollary}


		\subsection{Natural Transformation}

Our ultimate goal is to generalize concepts, and the natural transformation is next, as exposed by [I]. We see that one way of giving meaning to the situations so far discussed is by interpreting the binary relations by stronger and weaker arrows. We will see the natural transformation as a strong arrow between two weak arrows, the functors. 

In the last section, given $\textgoth f,\textgoth g:\mathcal C\rightarrow\mathcal D_1,\mathcal D_2$, we were able to give a functor that was product for common domain. We did this by giving a category $\mathcal D_1\times\mathcal D_2$, and then we can assign a c-object in $\mathcal C$, into an arrow $\rightarrow_\times$, which is a c-object of the product category. Let us consider the situation where $\mathcal D_1$ and $\mathcal D_2$ are the same category. We give a functor from $\mathcal C$ into $\mathcal D\times\mathcal D$. Here, we take a turn. Because $\mathcal D\times\mathcal D$ is a collection of non-discernible arrows, instead of assigning an arrow $\rightarrow_\times$, to an object in $\mathcal C$, we will assign it an arrow in $\mathcal D$.

Let functors $\textgoth f,\textgoth g:\mathcal C\rightarrow\mathcal D$, and function $\tau:\mathcal{O|C}\rightarrow\mathcal{A|D}$ be such that $a\mapsto_\tau\tau a:\textgoth fa\rightarrow\textgoth ga$. We say $\tau$ is a bridge from $\textgoth f$ to $\textgoth g$, and we write $\tau:\textgoth f\rightarrow\textgoth g$.

\begin{proposition}Given a bridge, $\tau:\textgoth f\rightarrow\textgoth g$, we can give a category $\vec{\mathcal D}:=\vec{\mathcal D}(\tau;\textgoth f,\textgoth g)$, with $\mathcal{O|\vec{D}}:=Im~\tau$, and a functor $\textgoth T:\mathcal C\rightarrow\vec{\mathcal D}$ such that $\textgoth T_\mathcal O:=\tau$. \end{proposition}

\begin{proof}First, we must define the arrows of $\vec{\mathcal D}$. Given an arrow $f:a\rightarrow b$, define an arrow $\tau f:\tau a\rightarrow\tau b$, for the category $\vec{\mathcal D}$. The composition in $\vec{\mathcal D}$ is defined by $\tau g*\tau f:=\tau(f*g)$; we guarantee the existence of $\tau(f*g)$, for composable arrows in the category $\mathcal C$.

Given a c-object $\tau b$, we have a unit arrow $1_{\tau b}:=\tau1_b:\tau b\rightarrow\tau b$, such that\begin{eqnarray}\nonumber 1_{\tau b}&;&1_{\tau b}*\tau f,\tau f\\\nonumber 1_{\tau b}&;&\tau1_{b}*\tau f,\tau f\\\nonumber 1_{\tau b}&;&\tau(1_{b}*f),\tau f\\\nonumber 1_{\tau b}&;&\tau f,\tau f.\end{eqnarray}Similarly, we have $\tau f;\tau f,1_{\tau a}$ which proves $\vec{\mathcal D}$ has a unit arrow for every c-object $\tau a$. The reader should not have difficulty showing associativity holds for the composition, so we conlcude $\vec{\mathcal D}$ is a category.

If we define $\textgoth T_\mathcal Af:=\tau f$, we have a functor.\end{proof}

\begin{definition}A bridge $\tau:\textgoth f\rightarrow\textgoth g$ is a \textit{natural transformation} if for every arrow $f:a\rightarrow b$, in $\mathcal{C}$, we have $\tau b,\tau a;\textgoth g~f,\textgoth f~f$ in terms of composition for $\mathcal D$.\end{definition}

The expression $\tau b,\tau a;\textgoth gf,\textgoth ff$ means that $\tau b*\textgoth ff$ is the same arrow as $\textgoth gf*\tau a$. It is reasonable to request this because they are parallel, $\tau b*\textgoth ff:\textgoth fa\rightarrow\textgoth gb$ and $\textgoth gf*\tau a:\textgoth fa\rightarrow\textgoth gb$. Equivalently, one can say the bridge is a natural transformation if and only if $\textgoth ff,\textgoth gf$ are a natural pair of functions, under $\tau a,\tau b$.

\begin{proposition}Let $\tau$, a bridge for two functors $\textgoth f,\textgoth g:\mathcal C\rightarrow\mathcal D$, and suppose the arrows of $\mathcal D$ are non-discernible. Then $\tau$ is a natural transformation.\label{nat tran 1}\end{proposition}

\begin{corollary}Suppose there are two order preserving functions $f,g:\mathcal{O|P}\rightarrow\mathcal{O|Q}$ such that $fa\leq ga$ for every c-object in $\mathcal C$. Then, $\tau:\mathcal{O|P}\rightarrow\mathcal{A|Q}$ that sends $a\mapsto fa\leq ga$ is a natural transformation $\textgoth f\rightarrow\textgoth g$, where the functors satisfy $\textgoth f_\mathcal O:=f$ and $\textgoth g_\mathcal O:=g$.\end{corollary}

We are playing with an order defined on the functors of the partial orders. For any two order preserving $f,g:\mathcal{O|P}\rightarrow\mathcal{O|Q}$, we say $f\preceq g$ if and only if $fx\leq gx$, for every object in $\mathcal P$. Of course, if $\textgoth f$, $\textgoth g$ are the functors corresponding to $f,g$, respectively, then we say $\textgoth f\leq\textgoth g$.

\begin{proposition}\label{prop order funct order}Let $\textgoth f$ and $\textgoth g$ be two functors $\mathcal P\rightarrow\mathcal Q$, for partial orders. The relation $\textgoth f\leq\textgoth g$, is a partial order on the functors.\end{proposition}

\begin{proof}Let $f$ represent the object function of $\textgoth f$. We observe reflexivity holds in the fact that $\textgoth f\leq\textgoth f$ is equivalent to the statement $fx\leq fx$, for every $x$ in $\mathcal P$. Suppose $\textgoth f\leq\textgoth g\leq\textgoth h$. Transitivity holds because we can assert $fx\leq gx\leq hx$, for every $x$ in $\mathcal P$, where $g,h$ is the object function of $\textgoth g,~\textgoth h$. If $\textgoth f\leq\textgoth g$ and $\textgoth g\leq\textgoth f$, then we have $fx\leq gx$ and $gx\leq fx$ for every $x$. This means $fx=gx$.\end{proof}

We will build two different categories, using natural transformations as strong arrows for weak arrows, the functors. That is to say, our categories will have functors as c-objects and natural transformations as arrows. Consider the collection of natural transformations $\textgoth f\rightarrow\textgoth g$, for a pair of parallel functors $\textgoth f,\textgoth g:\mathcal C\rightarrow\mathcal D$; we define $Nat(\textgoth f,\textgoth g):=\{\textgoth f\rightarrow\textgoth g\}$. 

$\mathcal Cat(\mathcal C,\mathcal D)$ is a category if $\mathcal Cat(\mathcal C,\mathcal D)_\mathcal O:=\{\mathcal C\textgoth F\mathcal D\}$ and the natural transformations are arrows. First, every functor has a natural transformation $1_\textgoth f:\textgoth f\rightarrow\textgoth f$ so that $x\mapsto_{1_\textgoth f}1_{\textgoth fx}:=\textgoth fx\rightarrow\textgoth fx$; in other words, the unit arrow of $\textgoth f$ sends an object $x$, in $\mathcal C$, into the unit arrow $1_{\textgoth fx}$, in $\mathcal D$. Let use define the operation of arrows as $\cdot$ so that $x\mapsto_{\sigma\cdot\tau}\sigma x*\tau x$ for $\tau,\sigma:\textgoth f,\textgoth h\rightarrow\textgoth h,\textgoth j$. Of course, $\sigma x*\tau x:\textgoth fx\rightarrow\textgoth jx$ is the composition, in $\mathcal D$, of $\sigma x:\textgoth hx\rightarrow\textgoth jx$ and $\tau x:\textgoth fx\rightarrow\textgoth hx$. To prove this is indeed a natural transformation, use the arrow $\sigma b*\textgoth hf*\tau a$. Suppose we have a natural transformation $\tau:\textgoth f\rightarrow\textgoth h$, then $\tau\cdot1_{\textgoth f}$ and $1_{\textgoth h}\cdot\tau$ are the same as $\tau$. To prove associativity; show $\rho\cdot(\sigma\cdot\tau)$ and $(\rho\cdot\sigma)\cdot\tau$ are the same function. Applying $\rho\cdot(\sigma\cdot\tau)$ to $x$, \begin{eqnarray}\nonumber x&;&[\rho\cdot(\sigma\cdot\tau)]x,\rho\cdot(\sigma\cdot\tau)\\\nonumber x&;&\rho x*(\sigma\cdot\tau)x,\rho\cdot(\sigma\cdot\tau)\\\nonumber x&;&\rho x*(\sigma x*\tau x),\rho\cdot(\sigma\cdot\tau)\\\nonumber x&;&(\rho x*\sigma x)*\tau x,\rho\cdot(\sigma\cdot\tau)\\\nonumber x&;&(\rho\cdot\sigma)x*\tau x,\rho\cdot(\sigma\cdot\tau)\\\nonumber x&;&[(\rho\cdot\sigma)\cdot\tau]x,\rho\cdot(\sigma\cdot\tau)\end{eqnarray}

Now we study the category of all categories, $\vec{\mathcal C}at$. Let $\tau,\alpha:\textgoth f,\textgoth g\rightarrow\textgoth h,\textgoth i$ be natural transformations for the functors $\textgoth f,\textgoth h:\mathcal C\rightarrow\mathcal D$ and $\textgoth g,\textgoth i:\mathcal D\rightarrow\mathcal E$. If we apply $\tau$ to $x$, the reuslt is $\tau x:\textgoth fx\rightarrow\textgoth hx$. If we apply $\alpha$ to $\textgoth fx$, we get $\alpha\textgoth fx:\textgoth g(\textgoth fx)\rightarrow\textgoth i(\textgoth fx)$. In the same way, $\alpha\textgoth hx:\textgoth g(\textgoth hx)\rightarrow\textgoth i(\textgoth hx)$. We define $\alpha\circ\tau:\textgoth g\circ\textgoth f\rightarrow\textgoth i\circ\textgoth h$, a natural transformation such that \begin{equation}\label{define nat compose}(\alpha\circ\tau)x:=\alpha(\textgoth hx)*\textgoth g(\tau x).\end{equation} We can equivalently say $(\alpha\circ\tau)x:=\textgoth i(\tau x)*\alpha(\textgoth fx)$ because $\alpha(\textgoth hx),\alpha(\textgoth fx);\textgoth i(\tau x),\textgoth g(\tau x)$. To prove $(\alpha\circ\tau)$ is in $Nat(\textgoth g\circ\textgoth f,\textgoth i\circ\textgoth h)$, we must prove $(\alpha\circ\tau)b,(\alpha\circ\tau)a;(\textgoth i\circ\textgoth h)f,(\textgoth g\circ\textgoth f)f$ for any $f:a\rightarrow b$, in $\mathcal C$. Recall that for every arrow $\textgoth hf:\textgoth ha\rightarrow\textgoth hb$ in $\mathcal D$, we have $\alpha(\textgoth hb),\alpha(\textgoth ha);\textgoth i(\textgoth hf),\textgoth g(\textgoth hf)$ because $\alpha$ is in $Nat(\textgoth g,\textgoth i)$.\begin{eqnarray}\nonumber(\alpha\circ\tau)b&;&(\alpha\circ\tau)b*(\textgoth g\circ\textgoth f)f,(\textgoth g\circ\textgoth f)f\\\nonumber(\alpha\circ\tau)b&;&\alpha(\textgoth hb)*\textgoth g(\tau b)*\textgoth g\textgoth ff,(\textgoth g\circ\textgoth f)f\\\nonumber(\alpha\circ\tau)b&;&\alpha(\textgoth hb)*\textgoth g(\tau b*\textgoth ff),(\textgoth g\circ\textgoth f)f\\\nonumber(\alpha\circ\tau)b&;&\alpha(\textgoth hb)*\textgoth g(\textgoth hf*\tau a),(\textgoth g\circ\textgoth f)f\\\nonumber(\alpha\circ\tau)b&;&\alpha(\textgoth hb)*\textgoth g(\textgoth hf)*\textgoth g\tau a,(\textgoth g\circ\textgoth f)f\\\nonumber(\alpha\circ\tau)b&;&\textgoth i(\textgoth hf)*\alpha(\textgoth ha)*\textgoth g\tau a,(\textgoth g\circ\textgoth f)f\\\nonumber(\alpha\circ\tau)b&;&(\textgoth i\circ\textgoth h)f*(\alpha\circ\tau)a,(\textgoth g\circ\textgoth f)f\\\nonumber(\alpha\circ\tau)b,(\alpha\circ\tau)a&;&(\textgoth i\circ\textgoth h)f,(\textgoth g\circ\textgoth f)f.\end{eqnarray}To prove associativity, let $\mu:\textgoth l\rightarrow\textgoth m$, where $\textgoth l,\textgoth m:\mathcal E\rightarrow\mathcal F$. This implies $\mu\circ\alpha:\textgoth l\circ\textgoth g\rightarrow\textgoth m\circ\textgoth i$ such that $(\mu\circ\alpha)\textgoth hx:=\mu\textgoth i(\textgoth hx)*\textgoth l\alpha(\textgoth hx)$. Thus,\begin{eqnarray}\nonumber x&;&[\mu\circ(\alpha\circ\tau)]x,\mu\circ(\alpha\circ\tau)\\\nonumber x&;&\mu(\textgoth i\circ\textgoth h)x*\textgoth l(\alpha\circ\tau)x,\mu\circ(\alpha\circ\tau)\\\nonumber x&;&\mu\textgoth i(\textgoth hx)*\textgoth l(\alpha\textgoth hx*\textgoth g\tau x),\mu\circ(\alpha\circ\tau)\\\nonumber x&;&\mu\textgoth i(\textgoth hx)*\textgoth l\alpha(\textgoth hx)*(\textgoth l\circ\textgoth g)\tau x,\mu\circ(\alpha\circ\tau)\\\nonumber x&;&(\mu\circ\alpha)\textgoth hx*(\textgoth l\circ\textgoth g)\tau x,\mu\circ(\alpha\circ\tau)\\\nonumber x&;&[(\mu\circ\alpha)\circ\tau]x,\mu\circ(\alpha\circ\tau).\end{eqnarray}

\begin{lemma intlaw}Let $\alpha,\beta:\textgoth g,\textgoth i\rightarrow\textgoth i,\textgoth k$ be natual transformations of functors $\textgoth g,\textgoth i,\textgoth k:\mathcal D\rightarrow\mathcal E$. In terms of the notation operation $\circ$ of $\vec{\mathcal C}at$, we verify $\beta\cdot\alpha;(\beta\circ\sigma)\cdot(\alpha\circ\tau),\sigma\cdot\tau$. Equivalently, if the notation is used for $\cdot$, which is the operation in the categories $\mathcal Cat(\mathcal C,\mathcal D)$ and $\mathcal Cat(\mathcal D,\mathcal E)$, then $\beta\circ\sigma;(\beta\cdot\alpha)\circ(\sigma\cdot\tau),\alpha\circ\tau$.\end{lemma intlaw}

\begin{proof}Notice $\beta\cdot\alpha:\textgoth g\rightarrow\textgoth k$, and $\sigma\cdot \tau:\textgoth f\rightarrow\textgoth j$.\begin{eqnarray}\nonumber x&;&[(\beta\cdot\alpha)\circ(\sigma\cdot\tau)]x,(\beta\cdot\alpha)\circ(\sigma\cdot\tau)\\\nonumber x&;&(\beta\cdot\alpha)(\textgoth jx)*\textgoth g[(\sigma\cdot\tau)x],(\beta\cdot\alpha)\circ(\sigma\cdot\tau)\\\nonumber x&;&[\beta(\textgoth jx)*\alpha(\textgoth jx)]*\textgoth g(\sigma x*\tau x),(\beta\cdot\alpha)\circ(\sigma\cdot\tau)\\\nonumber x&;&\beta(\textgoth jx)*[\alpha(\textgoth jx)*\textgoth g(\sigma x)]*\textgoth g(\tau x),(\beta\cdot\alpha)\circ(\sigma\cdot\tau)\\\nonumber x&;&\beta(\textgoth jx)*[\textgoth i(\sigma x)*\alpha(\textgoth hx)]*\textgoth g(\tau x),(\beta\cdot\alpha)\circ(\sigma\cdot\tau)\\\nonumber x&;&[\beta(\textgoth jx)*\textgoth i(\sigma x)]*[\alpha(\textgoth hx)*\textgoth g(\tau x)],(\beta\cdot\alpha)\circ(\sigma\cdot\tau)\\\nonumber x&;&(\beta\circ\sigma)x*(\alpha\circ\tau)x,(\beta\cdot\alpha)\circ(\sigma\cdot\tau)\\\nonumber x&;&[(\beta\circ\sigma)\cdot(\alpha\circ\tau)]x,(\beta\cdot\alpha)\circ(\sigma\cdot\tau).\end{eqnarray}\end{proof}

\begin{theorem}Given categories $\mathcal C,\mathcal D,\mathcal E$, there is a functor $\textgoth F:\mathcal Cat(\mathcal C,\mathcal D)\times\mathcal Cat(\mathcal D,\mathcal E)\rightarrow\mathcal Cat(\mathcal C,\mathcal E)$ such that $\textgoth F_\mathcal O(\textgoth f\rightarrow_\times\textgoth g):=\textgoth g\circ\textgoth f$ and $\textgoth F_\mathcal A(\tau\rightarrow_\times\alpha):=\alpha\circ\tau$.\end{theorem}

\begin{proof}We will need to prove, first of all, that the unit natural transformation, $1_{\textgoth g\circ\textgoth f}$, is the same as $1_\textgoth g\circ1_\textgoth f$. To see that this is true, note (\ref{define nat compose}) implies $x;1_\textgoth g(\textgoth fx)*\textgoth g(1_{\textgoth fx}),1_\textgoth g\circ1_\textgoth f$ so that for any c-object $x$ in $\mathcal C$\begin{eqnarray}\nonumber x&;&1_{\textgoth g\circ\textgoth f}x,1_{\textgoth g\circ\textgoth f}\\\nonumber x&;&1_{(\textgoth g\circ\textgoth f)x},1_{\textgoth g\circ\textgoth f}\\\nonumber x&;&1_{(\textgoth g\circ\textgoth f)x}*1_{(\textgoth g\circ\textgoth f)x},1_{\textgoth g\circ\textgoth f}\\\nonumber x&;&1_{\textgoth g(\textgoth fx)}*\textgoth g1_{\textgoth fx},1_{\textgoth g\circ\textgoth f}\\\nonumber x&;&1_{\textgoth g}(\textgoth fx)*\textgoth g1_{\textgoth fx},1_{\textgoth g\circ\textgoth f}\\\nonumber x&;&(1_\textgoth g\circ1_\textgoth f)x,1_{\textgoth g\circ\textgoth f}.\end{eqnarray}We know $(\textgoth1_{\mathcal Cat(\mathcal C,\mathcal E)}\circ\textgoth F_\mathcal O)(\textgoth f\rightarrow_\times \textgoth g)$, results in $1_{\textgoth g\circ\textgoth f}$. On the other hand, $(\textgoth F_\mathcal A\circ\textgoth1_{\mathcal C})(\textgoth f\rightarrow_\times\textgoth g)$ is $1_\textgoth f\circ1_\textgoth g$, where we use $\mathcal C:=\mathcal Cat(\mathcal C,\mathcal D)\times\mathcal Cat(\mathcal D,\mathcal E)$.

Let $\tau:\textgoth f\rightarrow\textgoth h$ and $\alpha:\textgoth g\rightarrow\textgoth i$, which means $\tau\rightarrow_\times\alpha:\textgoth f\rightarrow_\times\textgoth g\longrightarrow\textgoth h\rightarrow_\times i$. The image of the arrow in the domain category is $\alpha\circ\tau:\textgoth g\circ\textgoth f\rightarrow\textgoth i\circ\textgoth h$.

The proof that $\textgoth F$ satisfies 3) is given by the lemma:\begin{eqnarray}\nonumber(\sigma\rightarrow
_\times\beta)\cdot(\tau\rightarrow_\times\alpha)&;&\textgoth F[(\sigma\rightarrow
_\times\beta)\cdot(\tau\rightarrow_\times\alpha)],\textgoth F\\\nonumber(\sigma\rightarrow
_\times\beta)\cdot(\tau\rightarrow_\times\alpha)&;&\textgoth F[\sigma\cdot\tau
\rightarrow_\times\beta\cdot\alpha],\textgoth F\\\nonumber(\sigma\rightarrow
_\times\beta)\cdot(\tau\rightarrow_\times\alpha)&;&(\beta\cdot\alpha)\circ(\sigma\cdot\tau
),\textgoth F\\\nonumber(\sigma\rightarrow
_\times\beta)\cdot(\tau\rightarrow_\times\alpha)&;&(\beta\circ\sigma)\cdot(\alpha\circ\tau
),\textgoth F\\\nonumber(\sigma\rightarrow
_\times\beta)\cdot(\tau\rightarrow_\times\alpha)&;&\textgoth F(\sigma\rightarrow_\times\beta)\cdot\textgoth F(\tau\rightarrow_\times\alpha
),\textgoth F.\end{eqnarray}\end{proof}

The functor is an object function $\textgoth F_\mathcal O:\mathcal O|\mathcal Cat(\mathcal C,\mathcal D)\times\mathcal O|\mathcal Cat(\mathcal D,\mathcal E)\rightarrow\mathcal O|\mathcal Cat(\mathcal C,\mathcal E)$ and an arrow function given by $\textgoth F_\mathcal A:\mathcal A|\mathcal Cat(\mathcal C,\mathcal D)\times\mathcal A|\mathcal Cat(\mathcal D,\mathcal E)\rightarrow\mathcal A|\mathcal Cat(\mathcal C,\mathcal E)$. These functions give operations. The operation for functors is $\textgoth F_\mathcal O:\mathcal O|\mathcal Cat(\mathcal C,\mathcal D)\rightarrow\{\mathcal O|\mathcal Cat(\mathcal D,\mathcal E)~f~\mathcal O|\mathcal Cat(\mathcal C,\mathcal E)\}$, and the operation for natural transformations is $\textgoth F_\mathcal A:\mathcal A|\mathcal Cat(\mathcal C,\mathcal D)\rightarrow\{\mathcal A|\mathcal Cat(\mathcal D,\mathcal E)~f~\mathcal A|\mathcal Cat(\mathcal C,\mathcal E)\}$.

	\section{Integer Systems}

We have proven the existence of a discrete number system, but there are many such systems. For any discrete number system $\mathcal Z$, consider the system obtained by adding transitive arrows. We conserve the property of having non-discernible arrows; add only one arrow for every pair $a\rightarrow b$ and $b\rightarrow c$. We denote such a system with $\mathcal Z_<$. If we add one reflexive arrow to every c-object of $\mathcal Z_<$, the system obtained is denoted by $\mathcal{Z}_{\leq}$. This new system is a partial order, in which we consider the arrows in the natural order, of axiom 3. When considering a partial order, we may write $a<b$ to make it clear $a$ and $b$ are not the same object.

Now that there is a partial order, associated to any discrete number system $\mathcal Z$, we can consider functors $\mathcal Z_\leq\rightarrow\mathcal Z_\leq$. We have the identity functor, which we assign to the object $0$, by $+0:=\textgoth I_{\mathcal Z_\leq}$. But, there is another functor to consider. We call it the functor \textit{sum 1}, denoted $+1:\mathcal Z_\leq\rightarrow\mathcal Z_\leq$. To prove that we have a functor, we will give an order preserving function, and then we define $+1$ as the functor associated to said function. It is not necessary to look far for our function, the arrows of $\mathcal Z$ form a bijective function. These arrows form a function, with $\mathcal Z$ in the domain, because every object is source of one arrow only. Since every object of $\mathcal Z$ is target of exactly one arrow, we can say that the object function of $+1$, is bijective. Take $a<b$ in $\mathcal Z_\leq$, then we also have $a+1\leq b$. Since we also have $b<b+1$, it turns out $a+1<b+1$. 

\begin{proposition}The arrows of a discrete number system $\mathcal Z$, are the components of an order bijectivity for the objects of $\mathcal Z_\leq$. The corresponding functor, $+1$, is an automorphism of the partial order $\mathcal Z_\leq$.\end{proposition}

Compose this automorphism with itself to form a new automorphism, and call it \textit{sum 2}. For any $x+1$ of the partial order $\mathcal Z_\leq$, such that $0<x+1$, we give the automorphism $+(x+1):=+1\circ+x$. Since we have given a functor $+x$, the object function is order preserving and the corresponding functor is precisely $+x$. We will prove that these functors are related as in proposition \ref{prop order funct order}.

\begin{lemma order funct1}If $x,x+1$ are two comparable objects, in a discrete number system, and $0<x+1$, then the functors satisfy $+x<+(x+1)$.\end{lemma order funct1}

\begin{proof}In terms of the order preserving functions, we have $a+x<(a+x)+1=a+(x+1)$ for any $a$ in $\mathcal Z_\leq$. This is the same as saying $+x_\mathcal O\preceq+(x+1)_\mathcal O$.\end{proof}

Since $+1$ is an automorphism, we have the inverse automorphism $+(-1):\mathcal Z_\leq\rightarrow\mathcal Z_\leq$. This is the functor \textit{minus 1}. If $x<0$, we define $+(x-1):=+(-1)\circ+x$, and $+(-1):=+1^{-1}$. The reader may verify that, in general, $+(-x)$ is the inverse of $+x$. Here, we give the second part of the last result.

\begin{lemma order funct2}For comparable objects $x,x-1$ with $x-1<0$, the functors satisfy $+(x-1)<+x$.\end{lemma order funct2}

\begin{proof}Note that $a+(x-1)=(a+x)-1<a+x$.\end{proof}

\begin{theorem}If $\mathcal Z$ is a discrete number system, then there is an isomorphism $\oplus:\mathcal Z_\leq\rightarrow\{+x\}_\leq$, where $\{+x\}_\leq$ is the partial order on automorphisms $+x$. Given any two $a\leq b$, in $\mathcal Z_\leq$, there is a natural transformation $+a\rightarrow+b$.\end{theorem}

\begin{proof}The lemmas give a discrete number system, $\cdots<+(-2)<+(-1)<+0<+1<+2<\cdots$. The object function of $\oplus$ is bijective; we have defined it to be onto, and we need to show it is monic. We verify this last, because any two functors $+a,+b$ are different. 

The arrow function is also bijective because partial orders consist of none discernible orders. All we need to do is assign arrows amongst respective objects.

Finally, if $a\leq b$, we know $+a\leq+b$. From the corollary of proposition \ref{nat tran 1}, we have a natural transformation of the form $+a\rightarrow+b$.\end{proof}

The following lemma will be extremely useful in proving our next theorem, where we wish to show the functors in the image of $\oplus$, commute under composition.

\begin{lemma commute funct}Let $f,g:\mathcal O\rightarrow\mathcal O$ be bijective functions that commute. Then $f^{-1},g$ also commute.\end{lemma commute funct}

\begin{proof}\begin{eqnarray}\nonumber g&;&g\circ f^{-1},f^{-1}\\\nonumber g&;&(f^{-1}\circ f)\circ(g\circ f^{-1}),f^{-1}\\\nonumber g&;&f^{-1}\circ (f\circ g)\circ f^{-1},f^{-1}\\\nonumber g&;&f^{-1}\circ (g\circ f)\circ f^{-1},f^{-1}\\\nonumber g&;&(f^{-1}\circ g)\circ (f\circ f^{-1}),f^{-1}\\\nonumber g&;&f^{-1}\circ g,f^{-1}\\\nonumber g,g&;&f^{-1},f^{-1}\end{eqnarray}\end{proof}

\begin{theorem}Let $\mathcal Z$ be a discrete number system. The object function, of the isomorphism $\oplus$, defines a full operation $+:\mathcal{O|Z}\rightarrow\{\mathcal{O|Z}f_{iso}\mathcal{O|Z}\}$. \begin{itemize}\item[1)]The functors of $\{+x\}_\leq$ form a commutative group, under composition; we represent it with $\mathbb Z^\dagger$. \item[2)]Using $+$, we construct a commutative group $\mathbb Z$, on the objects of the discrete number system.\item[3)]There is an isomorphism $\mathbb Z\rightarrow\mathbb Z^\dagger$.\end{itemize}\end{theorem}

\begin{proof}The object function $\oplus_\mathcal O$ sends $x$ into the automorphism $+x$, for $\mathcal Z_\leq$. The object function of $+x$ is $+x_\mathcal O$, we may also write it as $+x$. Thus, every object in $\mathcal Z_\leq$ is assigned a function $+x_\mathcal O:\mathcal{O|Z_\leq}\rightarrow\mathcal{O|Z_\leq}$. The fact we are sending every $x$ into an automorphism, implies the operation is full. The operation can be expressed as $+:\mathcal{O|Z}\rightarrow\mathcal{O|Z}f\mathcal{O|Z}$ because the objects of $\mathcal Z$ and $\mathcal Z_\leq$ are the same.\begin{itemize}\item[1)]We know that the collection of automorphisms is a group. Since $+0$ is the identity, we only need to notice the functors of $\{+x\}_\leq$ have inverse also in $\{+x\}_\leq$. Therefore, the functors $+x$ form a group, under $\circ$.

Now we would like to see the group is commutative. We wish to prove $+x\circ+y$ and $+y\circ+x$ are the same function, whoever $x,y$ of $\mathcal Z$ may be. If $0\leq x$, then $+1\circ +x$ is the same as $+x\circ +1$ because $+x$ is compositions of $+1$; using the lemma, one can prove $+(-x)\circ+1$ and $+1\circ+(-x)$ are the same. This proves $+1$ commutes with all. Now, suppose $+y$ commutes with all; we will prove $+(y+1)$ commutes with all other functors of the group.\begin{eqnarray}\nonumber+(y+1)&;&+(y+1)\circ+x,+x\\\nonumber+(y+1)&;&(+1\circ+y)\circ+x,+x\\\nonumber+(y+1)&;&+1\circ(+y\circ+x),+x\\\nonumber+(y+1)&;&+1\circ(+x\circ+y),+x\\\nonumber+(y+1)&;&(+1\circ+x)\circ+y,+x\\\nonumber+(y+1)&;&(+x\circ+1)\circ+y,+x\\\nonumber+(y+1)&;&+x\circ(+1\circ+y),+x\\\nonumber+(y+1)&;&+x\circ+(y+1),+x.\end{eqnarray}One can equally show $+(y-1)$ commutes with all $+x$, given $+y$ commutes with all.

\item[2)] We know the unit exists because $+0$ is identity. It will be useful to see how other objects act on 0. We know $0+1$ is the target in the arrow $0\rightarrow1$. Suppose $0+x$ is $x$, then we also have $0+(x+1)$ is $x+1$. The same is true for $x-1$, so that $0;x,x$ for every $x$.

Also, we see that every $x$ has an inverse. Since $+(-x):=+x^{-1}$ is true, we can conclude the inverse of $x$, is $-x$ because\begin{eqnarray}\nonumber x&;&x+(-x),-x\\\nonumber x&;&(0+x)+x^{-1},-x\\\nonumber x&;&(0+x)+x^{-1},-x\\\nonumber x&;&(+x^{-1}\circ+x)0,-x\\\nonumber x&;&0,-x\end{eqnarray}Commutativity of two objects $x,y$ follows from commutativity of their functors:\begin{eqnarray}\nonumber x&;&x+y,y\\\nonumber x&;&(0+x)+y\\\nonumber x&;&(+y\circ+x)0,y\\\nonumber x&;&(+x\circ+y)0,y\\\nonumber x&;&y+x,y\\\nonumber x,x&;&y,y.\end{eqnarray}Associativity is proven using commutativty of the operation and commutativity of the functors. To prove the operation is associative we will prove $x+$ and $+y$ commute. We use the fact $+x$ and $+y$ commute and that the operation itself commutes, which means left and right operations are the same functions.\begin{eqnarray}\nonumber +x,+x&;&+y,+y\\\nonumber x+,x+&;&+y,+y.\end{eqnarray}Up to this point we have proven that there is an commutative operation that satisfies the properties of group. However, we still need to build a group which consists of one c-object and objects of operation on it. We may consider the objects of a discrete number system as automorphisms on a category that is sufficiently large. We mean sufficiently large in the sense that we can make a category as large as we want to give as many automorphisms (arrows) as can be needed. We know that this is true for the group of ordered automorphisms $\mathbb Z^\dagger$. So, for any discrete array of numbers, we say each is an automorphism of a suitable category. This category is the c-object of the group, and we will give it the symbol $\mathbb Z_\dagger$. Our group $\mathbb Z$ is defined as the algebraic category with $\mathbb Z_\dagger$ as object of the automorphisms, which are the objects of the discrete number system $\mathcal Z$. We are defining $\mathcal A|\mathbb Z:=\mathcal{O|Z_\leq}$. The operation of $\mathbb Z$ is the one defined by the functors of $\mathbb Z^\dagger$, which are the inverses and compositions of the functor determined by the arrows of $\mathcal Z$.

 \item[3)]Lastly, consider the algebraic functor $\textgoth i:\mathbb Z\rightarrow\mathbb Z^\dagger$ such that $\textgoth i_\mathcal O\mathbb Z_\dagger:=\mathcal Z_\leq$, and $\textgoth i_\mathcal Ax:=+x$ for every automorphism $x:\mathbb Z_\dagger\rightarrow\mathbb Z_\dagger$. To show this is indeed a functor, we must prove the third condition of functors. That would be to verify $\textgoth i(x+y)$ is the same object of operation as $\textgoth ix\circ\textgoth iy$. In other words, one must prove $\textgoth i(x+y):=+(x+y)$ and $\textgoth ix\circ\textgoth iy:=+x\circ+y$. Start with associativity and apply commutativity to prove this:\begin{eqnarray}\nonumber a,x&;&a+y,y+x\\\nonumber a,x&;&a+y,x+y.\end{eqnarray}\end{itemize}\end{proof}

Every discrete number system generates a group of functors, that determine the objects of the discrete number system as a group; both groups involved will be isomorphic. At the same time, the discrete number system is a group of automorphisms of another category. Thus, automorphisms of the kind generated by discrete number systems, follow the pattern of a discrete number system. Every object is in the same position; every object has one arrow into it and one arrow from.

For practical purposes, $\begin{array}{rr} a \\ c \end{array}$ or $a~~c$ will express $a\leq c$. We have freedom to write $a~~b~~c\cdots$ or its vertical counterpart, in order to express all the possible relations $a\leq b$, $a\leq c$, $b\leq c$, \ldots. The question, now, is how we got from $\leq$ to this $+$. This is basically an explanation of (\ref{grid int sum}), below. We generalize the system of integers, in the following manner. Begin with one row, which represents $\mathcal{Z}_{\leq}$, and to each object of that row assign a function. The function $+x$, represented by the columns of $0$ and $x$, consists of the arrows from the column 0 into the column $x$. One object is chosen, $0$, for which the arrows of its function are all reflexive; the identity function. The functions $+x$, all preserve the natural order of $\mathcal{Z}_{\leq}$; every column is $\mathcal{Z}_{\leq}$. The functions are order bijectivities and since the functors are ordered, natural transformations define a partial order for them. If $0\leq x$, the natural transformation $+0\longrightarrow_{\tau}+x$ is the function that assigns to each object $a$ of $\mathcal{Z}_{\leq}$, the arrow $a\rightarrow a+x$, so $a\rightarrow_{\tau}\tau a:a\leq a+x$.

The fact that our bridge is a natural transformation, means
$\tau b,\tau a;+x(a\leq b),+0(a\leq b)$. This can otherwise be written as $b\leq b+x,a\leq a+x;a+x\leq b+x,a \leq b$. This last is equivalent to saying $b\leq b+x\circ a\leq b$ is the same composition as $a+x\leq b+x\circ a\leq a+x$; both are $a\leq b+x$.

The grid of objects (\ref{grid int sum}) is a representation of the constructions we have carried out. We can see it as a collection of arrows for the objects of $\mathcal{Z}_{\leq}$. In this case, we are providing no new information, to what is given by a discrete number system. We can also, and more meaningfully, interpret this as comparability of relations. 

\begin{eqnarray}\begin{array}{rrrrrrrr} & \vdots & \vdots & \vdots & \vdots & \vdots & \vdots & \\ \cdots & -5 & -4 & -3 & -2 & -1 & 0 & \cdots \\ \cdots & -4 & -3 & -2 & -1 & 0 & 1 &\cdots \\\cdots &-3 & -2 & -1 & 0 & 1 & 2 &\cdots\\ \cdots&-2 & -1 & 0 & 1 & 2 & 3&\cdots  \\\cdots& -1 & 0 & 1 & 2 & 3 & 4 &\cdots\\\cdots& 0 & 1 & 2 & 3 & 4 & 5&\cdots\\& \vdots & \vdots & \vdots & \vdots & \vdots & \vdots &\end{array}.\label{grid int sum}\end{eqnarray}

We say $a\leq c$ and $b\leq d$ are comparable if they appear in a rectangle: $\begin{array}{rr}a & b \\ c & d \end{array}$. This would of course be expressed as $a,c;b,d$. We have chosen this definition of comparability for good reason. Suppose $0$ is one of the four objects, and you have a rectangle of the form:\\1) 
$\begin{array}{rr} 0 & a \\ x & n\end{array}$~~~~~~~~~~~~~~~~2) $\begin{array}{rr} a & n \\ 0 & x\end{array}$~~~~~~~~~~~~~~~~3)  $\begin{array}{rr} n & x\\ a & 0\end{array}$~~~~~~~~~~~~~~~~4) $\begin{array}{rr} x & 0 \\ n & a\end{array}$.\\This means we have applied the functor $+a$ to 0, and $x$. After applying $+a$ to 0, we naturally get $a$. If we apply $+a$, to $x$, the result is $n$. In terms of the notation for $+$, we can assert $0,x;a,n$. The operation is defined by $a\rightarrow_{+x}n$ if 1) or 3) occur. If, on the contrary, we have the occurence of 2), the rectangle in which $a,x$ occur will be unique. The same is true for 4). One can see this is true in general; we need not take $0$ as a corner. Given any rectangle $\begin{array}{rr} a & b \\ c & d\end{array}$, we have $a,c;b,d$ in terms of the operation as well as the comparability of arrows. Although, we must be careful in using the notation $,;,$ for the following reason. When representing comparability for the order relation, or the operation, $a,c;b,d$ and $a,b;c,d$ mean the same thing, because we have said that $\begin{array}{rr}a\\c\end{array}$ means the same thing as $a~~c$. However, changing $a,c;b,d$ for $b,d;a,c$ is only valid if the notation is used in representation of the operation. This is due to anti-symmetry of the order. Defining the operation in terms of comparability, with arrows that have a 0, assures the operation is an ordered triple; we are eliminating one component, by fixing 0.

	\section{Rational Systems}

		\subsection{Product}

A new operation is defined, in terms of the sum, for the objects of $\mathcal Z$. For $x\geq0$ we define the \textit{product} as the operation $\cdot$ given by $a;(a\cdot x)+a,x+1$. This is to say $\cdot(x+1)a:=(+a\circ\cdot x)a$. It is in our best interest to define the operation in such a way that $x\cdot1$ means we add $x$ once. Said differently, we want $x;x,1$. To achieve this, we define $x;0,0$. If we have $0;0,x$, we also have $0;0,x+1$. Similarly, $1;x+1,x+1$ if $1;x,x$. From this point on, we will use the notation $a-x$ to express $a+(-x)$. The product for $x\leq0$, is defined by $a;(a\cdot x)-a,x-1$. It is immediate, from this definition, that $x;-x,-1$. Again, we have $0;0,x-1$ and $-1;-(x-1),x-1$ given $0;0,x$ and $-1;-x,x$. We thus have, for every integer $a$,
\[a;0,0~~~~~~~~~~0;0,a~~~~~~~~~~1;a,a~~~~~~~~~~a;a,1.\]

\begin{theorem}The product is a full operation $\cdot:\mathbb Z\rightarrow\{\cdot x\}$, where $\{\cdot x\}$ is a collection of functions $\cdot x$. Every $\cdot x$ determines a functor $\cdot x$ on $\mathbb Z$. Furthemore, the operation is commutative and associative.\end{theorem}

\begin{proof}To prove $\cdot x$ is a functor suppose, for $x\geq0$, that $a+b;(a\cdot x)+(b\cdot x),x$.\begin{eqnarray}\nonumber a+b&;&(a+b)\cdot x+(a+b),x+1\\\nonumber a+b&;&[(a\cdot x)+(b\cdot x)]+(a+b),x+1\\\nonumber a+b&;&[(a\cdot x)+a]+[(b\cdot x)+b],x+1\\\nonumber a+b&;&a\cdot(x+1)+b\cdot(x+1),x+1.\end{eqnarray}This implies $\cdot x$ is a functor from $\mathbb Z$ to itself. It is not, however, an automorphism. It is not difficult to see under what condidtions it is monic.

We know 0 commutes with any integer, so now we wish to know if the same is true for $x+1$, given that $x$ commutes with any integer. Using the last result we find
\begin{eqnarray}\nonumber a+1&;&(a+1)\cdot x+(a+1),x+1\\\nonumber a+1&;&[(a\cdot x)+x]+(a+1),x+1\\\nonumber a+1&;&[((a\cdot x)+x)+a]+1,x+1\\\nonumber a+1&;&[((a\cdot x)+a)+x]+1,x+1\\\nonumber a+1&;&[(a\cdot x)+a]+(x+1),x+1\\\nonumber a+1&;&[(x\cdot a)+a]+(x+1),x+1\\\nonumber a+1&;&(x+1)\cdot(a+1),x+1\\\nonumber a+1,a+1&;&x+1,x+1\end{eqnarray}

Of course we can prove that $x;x\cdot a+x\cdot b,a+b$ because of commutativity. We now turn to associativity for the product. If $a,b;a\cdot x,x\cdot b$ then,
\begin{eqnarray}\nonumber a&;&a\cdot[(x+1)\cdot b],(x+1)\cdot b\\\nonumber a&;&a\cdot[(x\cdot b)+b],(x+1)\cdot b\\\nonumber a&;&a\cdot(x\cdot b)+a\cdot b,(x+1)\cdot b\\\nonumber a&;&(a\cdot x)\cdot b+a\cdot b,(x+1)\cdot b\\\nonumber a&;&[(a\cdot x)+a]\cdot b,(x+1)\cdot b\\\nonumber a,b&;&a\cdot (x+1),(x+1)\cdot b\end{eqnarray}

Similar proofs hold for $x\leq0$. That is, we are able to prove that if the properties hold for $x$, they also hold for $x-1$.\end{proof}

		\subsection{Dual Orders}

The main concept of a ratio is that of comparing two relations. Therefore, we will take a more general view, by defining an order on arrows of $\mathbb{Z}\times\mathbb Z$. We will carry out a construction of two ordered systems, seperately, and then we will combine them. This process will be iterated three times; until we have a satisfactory system. At which point in time we will close out the construction by defining the two resulting systems as the same. 

We build a new system by taking away, from $\mathbb Z$, all the objects $x<0$, and relations of such objects. This new system is $\mathbb{N}_{0}$. If we take away the object of operation $0$ from $\mathbb{N}_{0}$, we are left with the system $\mathbb{N}$. The notation used here is a convenience in representation of arrows of the cartesian products; $a\rightarrow_{\times}c$ is $\frac{a}{c}$. An order will be defined for the c-objects of $\mathbb{N}\times\mathbb{N}$; make $\frac{a}{c}\leq\frac{b}{c}$ if $a\leq b$. The order with respect to the target objects is a bit different. We say $d\leq c$ if and only if $\frac{a}{c}\leq\frac{a}{d}$. The first order is represented by $\mathbb N_\dagger:=\mathbb N\times\{c\}$. The second is $\mathbb N^\dagger:=\{a\}\rightarrow_\times\mathbb N$, for any $0<a,c$.

\begin{proposition}The orders $\mathbb N_\dagger$ and $\mathbb N^\dagger$ are dual; explicitly $\mathbb N_\dagger,(\mathbb N^{\dagger})^{op};(\mathbb N_\dagger)^{op},\mathbb N^\dagger$. Therefore, there is a contravariant isomorphism $\mathbb N_\dagger\rightarrow\mathbb N^\dagger$.\end{proposition}

\begin{proof}If we hold the target fixed, then there is an order bijectivity to $\mathbb N$. In the case we hold fixed the source, the resulting order is dual to $\mathbb N$.

We will give the contravariant functor, to be more precise, and better understand the situation. Let $\mathbb N_\dagger:=\mathbb N\times\{c\}$ and $\mathbb N^\dagger:=\{a\}\times\mathbb N$. Then we make the object function $\frac xc\mapsto\frac ax$. For arrows, $\frac xc\leq\frac yc\mapsto\frac ay\leq\frac ax$.\end{proof}

In a similar manner, we form the system $-\mathbb{N}_{0}$, where we now take away the objects $x>0$ and relations of every $x>0$. Notice that the systems $\mathbb N^\dagger$ are order bijective with $-\mathbb N$. What we want then, is to define a partial order $-\mathbb N\times\mathbb N$ which is dual to $\mathbb N\times\mathbb N$. Consider a functor $\textgoth d:-\mathbb N\times\mathbb N\rightarrow\mathbb N\times\mathbb N$. The object function is defined as $\textgoth d\left(\frac{-a}{c}\right):=\frac ac$. So, we define the arrow function $\textgoth d(\frac{-a}{c}\leq\frac{-b}{d}):=\frac bd\leq\frac ac$. This defines a partial order that is dual to the order of $\mathbb N\times\mathbb N$ because we have defined the order in terms of a contravariant functor.

\begin{eqnarray}\begin{array}{rrrrrrrrrrrrrrrrrrrrrrrrrrrrrrrrrrrrrrrrrrrrrrrrrrrrrrrrrrrrrrrrrrrrrrrrrrrrrrrr}-\infty&\rightarrow &\cdots & \frac{-4}{1}&\rightarrow & \frac{-3}{1}&\rightarrow  & \frac{-2}{1}&\rightarrow  & \frac{-1}{1}&\rightarrow  & \frac{0}{1}&\rightarrow   & \frac{1}{1} &\rightarrow & \frac{2}{1}&\rightarrow  & \frac{3}{1}&\rightarrow  & \frac{4}{1} &\cdots&\rightarrow&\infty\\\\ &&&\downarrow&& \downarrow&&\downarrow&& \downarrow&&\updownarrow&&\uparrow &&\uparrow&&\uparrow&&
\uparrow&&& \\\\-\infty&\rightarrow&\cdots & \frac{-4}{2}&\rightarrow & \frac{-3}{2}&\rightarrow  & \frac{-2}{2}&\rightarrow  & \frac{-1}{2}&\rightarrow  & \frac{0}{2}&\rightarrow   & \frac{1}{2} &\rightarrow & \frac{2}{2}&\rightarrow  & \frac{3}{2}&\rightarrow  & \frac{4}{2}&\cdots &\rightarrow&\infty\\
\\ &&&\downarrow&& \downarrow&&\downarrow&& \downarrow&&\updownarrow&&\uparrow &&\uparrow&&\uparrow&&\uparrow&&& \\\\-\infty&\rightarrow&\cdots & \frac{-4}{3}&\rightarrow & \frac{-3}{3}&\rightarrow  & \frac{-2}{3}&\rightarrow  & \frac{-1}{3}&\rightarrow  & \frac{0}{3}&\rightarrow   & \frac{1}{3} &\rightarrow & \frac{2}{3}&\rightarrow  & \frac{3}{3}&\rightarrow  & \frac{4}{3} &\cdots&\rightarrow &\infty\\\\
&&&\vdots & &\vdots & & \vdots & & \vdots  && \vdots  && \vdots && \vdots && \vdots && \vdots & & &\end{array}.\end{eqnarray}

We combine the two systems already formed, into one system, represented above. We give a partial order $\mathbb Q_\dagger$, with $\mathcal O|\mathbb Q_\dagger:=\mathcal O|\mathbb Z\times\mathbb N$. We do this in such a manner that $\frac{-a}{c}<0$ and $0<\frac{a}{c}$, for every $0<a,c$. To achieve this, we define $\frac{0}{x}:=0$, for every $x$ in $\mathbb N$. Said differently, any order $\mathbb{Z}\times\{x\}$ is a partial order generated by a discrete number system, for $0<x$.

\begin{eqnarray}\begin{array}{rrrrrrrrrrrrrrrrrrrrrrrrrrrrrrrrrrrrrrrrrrrrrrrrrrrrrrrrrrrrrrrrrrrrrrrrrrrrrrrr}
-\infty&\leftrightarrow&\cdots & \frac{-4}{0}&\leftrightarrow& \frac{-3}{0}&\leftrightarrow& \frac{-2}{0}&\leftrightarrow& \frac{-1}{0}&&&& \frac{1}{0} &\leftrightarrow& \frac{2}{0}&\leftrightarrow& \frac{3}{0}&\leftrightarrow& \frac{4}{0} &\cdots&\leftrightarrow&\infty\\\\&&&\downarrow&& \downarrow&&\downarrow&& \downarrow&&&&\uparrow &&\uparrow&&\uparrow&&
\uparrow&&&\\\\-\infty&\rightarrow &\cdots & \frac{-4}{1}&\rightarrow & \frac{-3}{1}&\rightarrow  & \frac{-2}{1}&\rightarrow  & \frac{-1}{1}&\rightarrow  & \frac{0}{1}&\rightarrow   & \frac{1}{1} &\rightarrow & \frac{2}{1}&\rightarrow  & \frac{3}{1}&\rightarrow  & \frac{4}{1} &\cdots&\rightarrow&\infty\\\\ &&&\downarrow&& \downarrow&&\downarrow&& \downarrow&&\updownarrow&&\uparrow &&\uparrow&&\uparrow&&
\uparrow&&& \\\\-\infty&\rightarrow&\cdots & \frac{-4}{2}&\rightarrow & \frac{-3}{2}&\rightarrow  & \frac{-2}{2}&\rightarrow  & \frac{-1}{2}&\rightarrow  & \frac{0}{2}&\rightarrow   & \frac{1}{2} &\rightarrow & \frac{2}{2}&\rightarrow  & \frac{3}{2}&\rightarrow  & \frac{4}{2}&\cdots &\rightarrow&\infty\\\\ &&&\downarrow&& \downarrow&&\downarrow&& \downarrow&&\updownarrow&&\uparrow &&\uparrow&&\uparrow&&
\uparrow&&& \\\\-\infty&\rightarrow&\cdots & \frac{-4}{3}&\rightarrow & \frac{-3}{3}&\rightarrow  & \frac{-2}{3}&\rightarrow  & \frac{-1}{3}&\rightarrow  & \frac{0}{3}&\rightarrow   & \frac{1}{3} &\rightarrow & \frac{2}{3}&\rightarrow  & \frac{3}{3}&\rightarrow  & \frac{4}{3} &\cdots&\rightarrow &\infty\\\\&&&\downarrow&& \downarrow&&\downarrow&& \downarrow&&\updownarrow&&\uparrow &&\uparrow&&\uparrow&&
\uparrow&&& \\\\-\infty&\rightarrow&\cdots & \frac{-4}{4}&\rightarrow & \frac{-3}{4}&\rightarrow  & \frac{-2}{4}&\rightarrow  & \frac{-1}{4}&\rightarrow  & \frac{0}{4}&\rightarrow   & \frac{1}{4} &\rightarrow & \frac{2}{4}&\rightarrow  & \frac{3}{4}&\rightarrow  & \frac{4}{4} &\cdots&\rightarrow &\infty
\\\\&&&\vdots & &\vdots & & \vdots & & \vdots  && \vdots  && \vdots && \vdots && \vdots && \vdots & & &\\\\&&&\downarrow&& \downarrow&&\downarrow&& \downarrow&&\updownarrow&&\uparrow &&\uparrow&&\uparrow&&
\uparrow&&& \\\\
&&& 0&& 0&& 0&& 0&& 0 && 0&& 0&&0&&0&&&\\\\&&&\downarrow&& \downarrow&&\downarrow&& \downarrow&&\updownarrow&&\uparrow &&\uparrow&&\uparrow&&
\uparrow&&&\\\\&&&\vdots & &\vdots & & \vdots & & \vdots  && \vdots && \vdots && \vdots && \vdots && \vdots & & &\\\\\infty&\leftarrow&\cdots & \frac{-4}{-4}&\leftarrow & \frac{-3}{-4}&\leftarrow  & \frac{-2}{-4}&\leftarrow  & \frac{-1}{-4}&\leftarrow  & \frac{0}{-4}&\leftarrow   & \frac{1}{-4} &\leftarrow & \frac{2}{-4}&\leftarrow  & \frac{3}{-4}&\leftarrow  & \frac{4}{-4} &\cdots&\leftarrow &-\infty\\\\&&&\downarrow&& \downarrow&&\downarrow&& \downarrow&&\updownarrow&&\uparrow &&\uparrow&&\uparrow&&
\uparrow&&&\\\\\infty&\leftarrow&\cdots & \frac{-4}{-3}&\leftarrow & \frac{-3}{-3}&\leftarrow  & \frac{-2}{-3}&\leftarrow  & \frac{-1}{-3}&\leftarrow  & \frac{0}{-3}&\leftarrow   & \frac{1}{-3} &\leftarrow & \frac{2}{-3}&\leftarrow  & \frac{3}{-3}&\leftarrow  & \frac{4}{-3} &\cdots&\leftarrow &-\infty\\\\&&&\downarrow&& \downarrow&&\downarrow&& \downarrow&&\updownarrow&&\uparrow &&\uparrow&&\uparrow&&
\uparrow&&& \\\\\infty&\leftarrow&\cdots & \frac{-4}{-2}&\leftarrow & \frac{-3}{-2}&\leftarrow  & \frac{-2}{-2}&\leftarrow  & \frac{-1}{-2}&\leftarrow  & \frac{0}{-2}&\leftarrow   & \frac{1}{-2} &\leftarrow & \frac{2}{-2}&\leftarrow  & \frac{3}{-2}&\leftarrow  & \frac{4}{-2} &\cdots&\leftarrow &-\infty\\\\&&&\downarrow&& \downarrow&&\downarrow&& \downarrow&&\updownarrow&&\uparrow &&\uparrow&&\uparrow&&
\uparrow&&&\\\\\infty&\leftarrow&\cdots & \frac{-4}{-1}&\leftarrow & \frac{-3}{-1}&\leftarrow  & \frac{-2}{-1}&\leftarrow  & \frac{-1}{-1}&\leftarrow  & \frac{0}{-1}&\leftarrow   & \frac{1}{-1} &\leftarrow & \frac{2}{-1}&\leftarrow  & \frac{3}{-1}&\leftarrow  & \frac{4}{-1} &\cdots&\leftarrow &-\infty\\\\&&&\downarrow&& \downarrow&&\downarrow&& \downarrow&&&&\uparrow &&\uparrow&&\uparrow&&
\uparrow&&&\\\\\infty&\leftrightarrow&\cdots & \frac{4}{0}&\leftrightarrow& \frac{3}{0}&\leftrightarrow& \frac{2}{0}&\leftrightarrow& \frac{1}{0}&&&& \frac{-1}{0} &\leftrightarrow& \frac{-2}{0}&\leftrightarrow& \frac{-3}{0}&\leftrightarrow& \frac{-4}{0} &\cdots&\leftrightarrow&-\infty\end{array}.\label{diagramRationals}\end{eqnarray}

An \textit{extended rational system} is any collection of objects such that the order between them is given by (II.5). The objects that are connected to $\infty$, or $-\infty$, by double arrows, will all be denoted by $\infty$, or $-\infty$. If the collection of objects does not include the objects $\infty,-\infty$, then the system is simply \textit{rational}. We will call the objects of such systems, \textit{rationals}. Our objects $-\infty,+\infty$ are comparable in the order but not so much in the operations. The object $\frac00$ is not comparable even in the order; the reader may closely study (II.5) to see that there is no convenient definition.

\subsection{Involution}

The bottom part of the system (II.5) is $\mathbb Q^\dagger$, and $\mathbb Q^\dagger_\mathcal O:=\mathcal O|\mathbb Z\times-\mathbb N$. There is an order bijectivity $\mathbb Q_\dagger\rightarrow\mathbb Q^\dagger$. We shall prove that they are also dual orders. This means that (II.5) is the combination of two dual systems that are actually the same system. Thus, in this case we can consider the simplified version (II.4), or its dual, because it is self dual. 

Consider the bijective function $q:\mathbb Q_\dagger\rightarrow\mathbb Q^\dagger$ that sends $\frac ac\mapsto_q\frac {-a}{-c}$. The function $p$ has the same domain and and image, but $\frac xc\mapsto_p\frac{x}{-c}$. Also, $r_1,r_2:\mathbb N_0\times\mathbb N,\mathbb N_0\times-\mathbb N\rightarrow-\mathbb N_0\times\mathbb N,-\mathbb N_0\times-\mathbb N$ such that for any integer $x$, we have $\frac ax\mapsto_r\frac{-a}x$.

\begin{proposition}The function $q$ is a two part function

\begin{eqnarray}\nonumber q|_{\mathbb N_0\times\mathbb N}&=&r_2\circ p|_{\mathbb N_0\times\mathbb N}=p|_{-\mathbb N_0\times\mathbb N}\circ r_1\\\nonumber q|_{-\mathbb N_0\times\mathbb N}&=&r^{-1}_2\circ p|_{-\mathbb N_0\times\mathbb N}=p|_{\mathbb N_0\times\mathbb N}\circ r^{-1}_1.\end{eqnarray}\end{proposition}

In the following proposition we are saying that $p$ establishes an order bijectivity $\mathbb Q_\dagger,(\mathbb Q^{\dagger})^{op};(\mathbb Q_\dagger)^{op},\mathbb Q^\dagger$. On the other hand, $q$ establishes an order bijectivity $\mathbb Q_\dagger,\mathbb Q^\dagger;(\mathbb Q_\dagger)^{op},(\mathbb Q^\dagger)^{op}$.

\begin{proposition}$q:\mathbb Q_\dagger\rightarrow\mathbb Q^\dagger$ is an order bijectivity. What is more, both systems are dual orders if we consider the function $p$.\end{proposition}

We can say, more concisely, that $p$ is a bijective function from one collection onto itself, with the interesting characteristic that its action on a rational number, with respect to the order, is that of taking opposite order. Given our equivalence relation, we see that $p$ and $r$ are the same function $-$, that gives the inverse under $+$. The interesting thing to notice is that if we reverse the arrows two times, we have the same order; that is to say, the opposite, of the opposite order, is the original order. This is why we have an order bijectivity $q$, after applying the function $-$ two times (once with the function $p$ and once with the function $r$). This is a special case of something more general. An \textit{involution} is a function $f:\mathcal O\rightarrow\mathcal O$ such that $fx;x,f$ which means $f(fx)$ is $x$. We will encounter many kinds of involutions.

	\subsection{Product for Rational Systems}

Presently, we will define a relation for determining when we will regard two arrows of $\mathbb{Z}\times \mathbb{Z}$ as the same. This of course is given in terms of the product for integers. The arrows $a\rightarrow_\times c$ and $b\rightarrow_\times d$ are $=$ comparable if $a,c;b,d$ in terms of the product; we will say $a\rightarrow_{\times}c=b\rightarrow_{\times}d$. This has an important meaning in terms of the notation: $a,c;b,d$ is used for stating 1) $a\cdot d$ and $b\cdot c$ result in the same object of operation, and 2) $\frac ac=\frac bd$. First, we observe that $a,c;a,c$, is obtained from $a,a;c,c$ by using commutativity. Second, we see that $a,c;b,d$ can be re-written as $b,d;a,c$. We will later see that transitivity holds, because in terms of the operation we have $a,c;b,d$ and $b,d;e,f$ which we will prove implies $a,c;e,f$. Some arrows in the equivalence are 
$\frac{-a}c=\frac a{-c}$. This is proven by the rules for product:\begin{eqnarray}\nonumber -a&;&(-a)\cdot(-c),-c\\\nonumber -a&;&[(-1)\cdot a]\cdot[(-1)\cdot c],-c\\\nonumber -a&;&[(-1\cdot a)\cdot -1]\cdot c,-c\\\nonumber-a&;&[a\cdot(-1\cdot -1)]\cdot c,-c\\\nonumber-a&;&a\cdot c,-c\\\nonumber-a,c&;&a,-c.\end{eqnarray}

We turn to define the product for the objects of the dual systems in (II.4); notice we are excluding objects of the form $\frac{x}{0}$. We define the operation by $\frac{a}{c};\frac{a\cdot b}{c\cdot d},\frac{b}{d}$. We first have to find the unit for this operation. It is not difficult to verify $\frac11$ is unit. We see $0;0,\frac ab$ and $\frac ab;0,0$.

Application of associativity, commutativity for integers, and the relation of equality defined above, allow us to prove $\frac{a}{c};\frac{a}{c},\frac{x}{x}$. Proofs for associativity and commutativity are a direct application of the defintion of product for rationals. So now, we turn to find inverse, under the product. We readily verify that the inverse of $\frac{a}{b}$ is $\frac{b}{a}$.

There is one more thing we wish to prove in this section. That is, if we multiply $\frac{a}{b}$ by any two rationals that are the same with respect to $=$, then the results are the same. That is, given $\frac{c}{d}=\frac{x}{y}$, we have to verify $(a\cdot c)\cdot(b\cdot y)$ and $(a\cdot x)\cdot(b\cdot d)$ are the same object. These expressions are rewritten as $(a\cdot b)\cdot(c\cdot y)$ and $(a\cdot b)\cdot(x\cdot d)$ which are both the same, since $c\cdot y$ and $x\cdot d$ are the same object. Also, if $\frac{a\cdot c}{b\cdot d}=\frac{a\cdot x}{b\cdot y}$, then $\frac cd=\frac xy$. Let us now return to the matter of transitivty for the relation $=$ defined above. This means we have the equalities $\frac ac=\frac bd$ and $\frac bd=\frac ef$, and we wish to prove $\frac ac=\frac ef$. The equalities give $\frac ac\cdot f=\frac bd\cdot f=\frac ef\cdot f=e$.

\begin{theorem}The system $\mathbb Q$, with rational numbers as objects of operation, is a commutative group under the product. As with the sum, the product is also a functor; in this case $\cdot:\mathbb{Q}\rightarrow\mathbb{Q}\textgoth{F}\mathbb{Q}$.\end{theorem}\begin{proof}The condition $\cdot,\cdot;\circ(\cdot y),\cdot y$ is the statement that $\cdot(x\cdot y)$ is the same function as $\cdot x\circ\cdot y$. Again, this is proven by applying commutativity once to the expression $a,x;a\cdot y,y\cdot x$, of associativity for the product. The result is $a,x;a\cdot y,x\cdot y$.\end{proof}

	\subsection{Sum for Rational Systems}

Now, given that arrows of the form $\frac{a}{0}$ are not rationals, we do not define the sum for these. The sum for rationals is defined in the last step\begin{eqnarray}\nonumber \frac{a}{c}&;&\frac{a}{c}+\frac{b}{d},\frac{b}{d}\\\nonumber \frac{a}{c}&;&\frac{a}{c}\cdot\frac{d}{d}+\frac{b}{d}\cdot\frac{c}{c},\frac{b}{d}\\\nonumber \frac{a}{c}&;&\frac{a\cdot d}{c\cdot d}+\frac{b\cdot c}{d\cdot c},\frac{b}{d}\\\nonumber \frac{a}{c}&;&\frac{a\cdot d}{c\cdot d}+\frac{b\cdot c}{c\cdot d},\frac{b}{d}\\\nonumber\frac{a}{c}&;&\frac{a\cdot d+b\cdot c}{c\cdot d},\frac{b}{d}.\end{eqnarray}Once we multiply by $\frac{c}{c}$ and $\frac{d}{d}$, the rest is straightforward. But why did we choose to multiply by that, and not by $\frac{a}{a}$ and $\frac{b}{b}$? Because we have guarantee that $c,d$ are both not 0. To prove commutativity and associativity for this operation, we proceed as follows, first with commutativity:\begin{eqnarray}\nonumber\frac{a}{c}&;&\frac{a\cdot d+b\cdot c}{c\cdot d},\frac{b}{d}\\\nonumber \frac{a}{c}&;&\frac{b\cdot c+a\cdot d}{d\cdot c},\frac{b}{d}\\\nonumber\frac{a}{c}&;&\frac{b}{d}+\frac{a}{c},\frac{b}{d}\\\nonumber\frac{a}{c},\frac{a}{c}&;&\frac{b}{d},\frac{b}{d}.\end{eqnarray}For associativity,\begin{eqnarray}\nonumber\frac{a}{c}&;&\frac{a}{c}+\frac{x\cdot d+b\cdot y}{y\cdot d},\frac{x}{y}+\frac{b}{d}\\\nonumber\frac{a}{c}&;&\frac{a\cdot(y\cdot d)+(x\cdot d+b\cdot y)\cdot c}{c\cdot (y\cdot d)},\frac{x}{y}+\frac{b}{d}\\\nonumber\frac{a}{c}&;&\frac{a\cdot (y\cdot d)+[(x\cdot d)\cdot c+(b\cdot y)\cdot c]}{c\cdot(d\cdot y)},\frac{x}{y}+\frac{b}{d}\\\nonumber\frac{a}{c}&;&\frac{b\cdot(y\cdot c)+[(x\cdot c)\cdot d+(a\cdot y)\cdot d]}{d\cdot(y\cdot c)},\frac{x}{y}+\frac{b}{d}\\\nonumber\frac{a}{c}&;&\frac{b}{d}+\left(\frac{x}{y}+\frac{a}{c}\right),\frac{x}{y}+\frac{b}{d}\\\nonumber\frac{a}{c}&;&\frac{b}{d}+\left(\frac{a}{c}+\frac{x}{y}\right),\frac{x}{y}+\frac{b}{d}\\\nonumber\frac{a}{c}&;&\left(\frac{a}{c}+\frac{x}{y}\right)+\frac{b}{d},\frac{x}{y}+\frac{b}{d}\\\nonumber\frac{a}{c},\frac{b}{d}&;&\frac{a}{c}+\frac{x}{y},\frac{x}{y}+\frac{b}{d}.\end{eqnarray}As before, the object 0 which \textit{absorbs} with product, serves as unit for sum. This is verified by\begin{eqnarray}\nonumber\frac ac&;&\frac ac+\frac 0x,\frac0x\\\nonumber\frac ac&;&\frac{a\cdot x+0\cdot c}{c\cdot x},\frac0x\\\nonumber\frac ac&;&\frac{a\cdot x}{c\cdot x},\frac0x\\\nonumber\frac ac&;&\frac ac\cdot\frac xx,\frac0x\\\nonumber\frac ac&;&\frac ac,\frac0x.\end{eqnarray}In light of the fact that we consider $\frac{0}{x}$ to be the same element for all $x$, we represent each such rational by 0, and we verify that objects dual with respect to the unit $0$, are objects $\frac{a}{b}$ and $\frac{-a}{b}$. Of course, our equivalence relation makes $\frac{a}{b}$ and $\frac{a}{-b}$ dual. In the same manner, $\frac{-a}{-b}$ is dual with $\frac{-a}{b}$ and $\frac{a}{-b}$.

\subsection{Embedding}

Define a bijective function $\iota:\mathbb Z\rightarrow\mathbb Z\times\{1\}$ such that $x\mapsto\frac x1$. It is not difficult to prove $\iota$ is an isomorphism for the operations sum and product. We follow the definition of sum to prove $\frac x1+\frac y1$ is $\frac{x+y}{1}$. This is also true for the product because $\frac x1\cdot\frac y1$ and $\frac{x\cdot y}1$ are the same.

The order for objects of operation, in $\mathbb Q$, is defined by $\frac{a}{c}\leq\frac{b}{d}$ if and only if $\frac{a\cdot d}{c\cdot d}\leq\frac{b\cdot c}{d\cdot c}=\frac{b\cdot c}{c\cdot d}$. This means that $a\cdot d\leq b\cdot c$ if and only if we have $c\cdot d>0$, and $b\cdot c\leq a\cdot d$ if and only if $c\cdot d<0$. 

In this case, $\iota$ is a functor for $\mathcal Z_{\leq}$ as a partial order and not an algebraic category. Additionally, we will prove that the system $\mathbb{Q}$ is a natural order. We know exactly one arrow is assigned to any two objects of $\mathbb Q$; this is given from the definition. We must prove transitivity for the order of $\mathbb{Q}$.

\begin{proposition}$\frac ac\leq\frac bd\Longleftrightarrow\frac ac\frac xx\leq\frac bd$\end{proposition}

\begin{proof}We have $\frac ac\leq\frac bd$ if and only if $\frac{a\cdot d}{c\cdot d}\leq\frac{b\cdot c}{d\cdot c}$; this is the definition of the order in $\mathbb Q_{\leq}$. Then, $c\cdot d<0$ if and only if $b\cdot c\leq a\cdot d$. 

First, we take $x>0$ which means $(c\cdot x)\cdot d=d\cdot(c\cdot x)=(c\cdot d)\cdot x<0$. Since $b\cdot c\leq a\cdot d$ we also have $(b\cdot c)\cdot x\leq(a\cdot d)\cdot x$ which is the same as $b\cdot (c\cdot x)\leq(a\cdot x)\cdot d$. We conlude that $\frac{(a\cdot x)\cdot d}{(c\cdot x)\cdot d}\leq\frac{b\cdot(c\cdot x)}{d\cdot(c\cdot x)}$.

Now we let $x<0$ and we get $0<(c\cdot x)\cdot d=d\cdot(c\cdot x)=(c\cdot d)\cdot x$. Here, $x$ reverses inequalities, so that $(a\cdot x)\cdot d\leq b\cdot(c\cdot x)$. The conclusion is the same.

The reader can similarly treat the case for $c\cdot d>0$.\end{proof}

\begin{proposition}\makebox[5pt][]{}\mbox {}\begin{itemize}\item[1)]If $\frac xy>0$, then $\frac ac\leq\frac bd\Longleftrightarrow\frac ac\frac xy\leq\frac bd\frac xy$\item[2)]If $\frac xy<0$, then $\frac ac\leq\frac bd\Longleftrightarrow\frac bd\frac xy\leq\frac ac\frac xy$\end{itemize}\end{proposition}

\begin{proof}Suppose $c\cdot y,d\cdot y<0$ which means $0<(c\cdot y)\cdot(d\cdot y)$ and $0<c\cdot d$. If $\frac ac\leq\frac bd$ we then have $a\cdot d\leq b\cdot c$. From $0<\frac xy$ we can easily see $0<x\cdot y$. This means that $(a\cdot x)\cdot(d\cdot y)\leq(b\cdot x)\cdot(c\cdot y)$ and we may conclude $\frac{a\cdot x}{c\cdot y}\leq\frac{b\cdot x}{d\cdot y}$.

Let us make $0<d\cdot y$ so that $(c\cdot y)\cdot(d\cdot y)<0$ and $c\cdot d<0$. Hence, $(b\cdot x)\cdot(c\cdot y)\leq(a\cdot x)\cdot(d\cdot y)$. The conclusion is the same as before. The reader may follow a similar proof for 2).\end{proof}

\begin{proposition}The order of $\mathbb Q$ is transitive.\end{proposition}

\begin{proof}Supposing $\frac ac\leq\frac bd\leq\frac ef$, we have $\frac{a\cdot d}{c\cdot d}\leq\frac{b\cdot c}{d\cdot c}$ and $\frac{b\cdot f}{d\cdot f}\leq\frac{e\cdot d}{f\cdot d}$. Let $0<c\cdot d$ and $f\cdot d<0$, then $a\cdot d\leq b\cdot c$ and $e\cdot d\leq b\cdot f$. We also have $c\cdot f<0$. The inequality $a\cdot d\leq b\cdot c$ is the equivalent $(a\cdot f)\cdot(d\cdot f)\leq(b\cdot f)\cdot(c\cdot f)$. The inequality $e\cdot d\leq b\cdot f$ is true if and only if $(b\cdot f)\cdot(c\cdot f)\leq(e\cdot d)\cdot (c\cdot f)$. We apply transitivity for $\mathbb Z_{\leq}$, and get $(a\cdot f)\cdot(d\cdot f)\leq(e\cdot d)\cdot (c\cdot f)=(e\cdot c)\cdot(d\cdot f)$. This last expression is the same as $e\cdot c\leq a\cdot f$. Since $c\cdot f<0$ we conclude $\frac ac\leq\frac ef$.

The remainding cases 1) $c\cdot d<0$, and $f\cdot d>0$, 2) $0<c\cdot d,f\cdot d$, 3) $c\cdot d,f\cdot d<0$. Can be treated in a similar manner.\end{proof}

\begin{theorem}$\mathbb Q$ is a natural order and there is an embedding of $\mathbb Z$ into $\mathbb Q$ for $+,\cdot,\leq$.\end{theorem}

We can say more about the order of the rationals. Because, we have already seen that the orders of (III.2) are dual. Daulity appears in many forms in $\mathbb Q$, but the real duality here is what we did at the beginning by defining columns dually to rows at the beginning of the section Dual Orders. In the right side of system (III.2), we have a system such that rows and columns are dual. The rows are order isomorphic to $\mathbb N$ and the columns are dual to $\mathbb N$.

\chapter{Set Theory}
In the foregoing, and to maintain the rigor of mathematics, we will use the $\Rightarrow$ relation that is representative of implication. We will not use it in a rigorous sense as of yet; for now $\Rightarrow$ is representative for $if$. We will say the collection of all objects is $\textbf{Obj}$. Consider the system $\textbf{Obj}_\in$ whose collection of objects is $\textbf{Obj}$; the relations of the system are $x\in A$ if and only if $A$ is a collection and $x$ is an object of $A$. We say $x~is~element~of~A$; of course $\notin$ will be used in the opposite manner in such cases $x$ is not element of $A$. Two collections $A,B$ are related $A\subseteq B$ if for every arrow $x\in A$, we also have $x\in B$. If the arrow $\subseteq$ has been ruled out to be reflexive we will write $A\subset B$. In case we have $A\subseteq B$ and $B\subseteq A$ we will say $A=B$; this is saying that the arrow between $A,B$ is double ended and when that happens we regard the two collections to be the same.

Any collection or object is said to be $normal$ if it is not an elment of itself. First of all, we have a case because the collection of all collections is an object in itself. Furthermore, any collection can be turned into a normal collection; if $\mathcal O\in\mathcal O$ we make the new collection $\underline{\mathcal O}$ by taking away the object $\mathcal O$. Now consider the collection of normal objects, denote it by $\mathcal R$. This collection is elusive in nature. Suppose $\mathcal R\in\mathcal R$, then $\mathcal R\notin\mathcal R$. If $\mathcal R\notin\mathcal R$, then by definition, $\mathcal R\in\mathcal R$. We cannot establish if this collection contains itself or not; both conclusions are simultaneously true and, therefore, simultaneously not true.

Let $\mathcal U$ be a normal subcollection of $\textbf{Obj}$ that will be called the \textit{universe of sets}. We will use $\mathcal{V}$ as the collection which is obtained from $\mathcal{U}$ by adding one object: $\mathcal{U}$. If two sets are related $A\subseteq B$, we say \textit{$A$ is a subset of $B$}. If $x\in X$, for some set $X\in\mathcal{X}$, we say $\mathcal{X}$ is a \textit{family}. The \textit{union} of a family is the collection of objects that consists of objects that belong to any member of the family; we write $\bigcup\mathcal{X}$. The collection of all subsets of $A$ is $\textgoth{P}A$.

\begin{definition}The following poperties define $\mathcal{U}$.\begin{itemize}\item[1)]$\mathcal{A|}\mathbb{Z}\subseteq\mathcal{U}$\item[2)]$x
\in A\in \mathcal{U}\Rightarrow x\in\mathcal{U}$\item[3)]$A\subseteq B\in\mathcal{U}\Rightarrow A\in\mathcal{U}$\item[4)]$A,B\in\mathcal{U}\Rightarrow A\rightarrow_{\times}B\in\mathcal{U}$
\item[5)]$A\in\mathcal{U}\Rightarrow\textgoth{P}A\in\mathcal{U}$\item[6)]
$\mathcal{X}\in\mathcal{U}\Rightarrow\bigcup\mathcal{X}\in\mathcal{U}$\item[7)]
If a function $f:A\rightarrow B\subseteq\mathcal{U}$ is onto and $A\in\mathcal{U}$, then $B\in\mathcal{U}$.
\end{itemize}\end{definition}

We ask 3) and 7) hold because we do not want sets to have \textit{too many objects}; we want to keep them small so as to distinguish them from arbitrary collections, which can be unimaginably large. 

\begin{definition}A set is any object of $\mathcal{U}$. A set function is any function $f:A\rightarrow B$, where $A,B\in\mathcal{U}$.\end{definition}

We see that the properties defining the universe allow us to find that any set function is an object in the universe because $f\subseteq A\rightarrow_{\times}B$. This allows us to define a category where the collection of objects is $\mathcal{U}$ and the collection of arrows is the set of all set functions. This is the \textit{category of small sets} and it will be written as $\textbf{Set}$. A set $A$ may be represented by the notation $A=\{x\}_{x\in A}$. If, for example a set consists of two objects, we then have $A=\{x,y\}=\{y,x\}$. A set that consists of one element is called a \textit{singleton}; if the element of the singleton is $x$, then the singleton is $\{x\}$.

	\section{Set Operations}

		\subsection{First Generation}

\subsubsection{Union} Take $A,B\in\mathcal{V}$. We will establish $\bigcup$ as an operation $\mathcal{V}\rightarrow\mathcal{V}f\mathcal{V}$. When we are considering such a union, we will write $A\cup B$ in place of $\bigcup\{A,B\}$.

\begin{proposition}For any $A,B,C\in\mathcal{V}$ and $\mathcal{X}\subseteq\mathcal{Y}\in\mathcal{V}$ we verify\begin{itemize}\item[1)]$\bigcup\mathcal{X}\subseteq\bigcup\mathcal{Y}$\item[2)]$A\cup B=B\cup A$
\item[3)]$A\cup(B\cup C)=(A\cup B)\cup C=A\cup B\cup C$.
\end{itemize}\end{proposition}

Notice that we have a special case for 1); namely that $A\subseteq\bigcup\mathcal{X}$, for every $A\in\mathcal{X}$.

\subsubsection{Difference} The union has an inverse operation, just as the sum has. The function $\cup B$ adds the objects of $B$ to $A$. Well, in view that systems are created by adding and taking objects, it is natural to give a function $-B$ that takes the objects of $B$ from $A$. Of course, $A-B\subseteq A$.

It is easily verified that the difference is not symmetric, just as the operation $-$ is not symmetric; we recall that $a-b$ and $b-a$ are not the same. Also, the difference is not associative, just as $(a-b)-c$ is not the same as $a-(b-c)$:

\begin{eqnarray}\nonumber(A-B)-C&=&A-(B\cup C)\\\nonumber A-(B-C)&=&(A-B)\cup[A-(A-C)].\end{eqnarray}

To reason this, let us examine the expression $A-(A-C)$. It is the collection of objects obtained by taking from $A$ the objects of $A-C\subseteq A$. We are leaving only those objects that are in $C$. That is, $x\in A-(A-C)$ if and only if $x$ is in $A$ and $C$ because $A-C$ consists of all objects in $A$ and not in $C$.

	\subsection{Second Generation}

		\subsubsection{Intersection} In set theory there will be an analogy between sum and product. The corresponding operations to sum and product, here, are union and intersection.

\begin{eqnarray}\nonumber A\cap B=A-(A-B).\end{eqnarray}

In general we say $x\in\bigcap\mathcal{X}$ if and only if $x\in A$, for every $A\in\mathcal{X}$.

\begin{proposition}For any $A,B,C\in\mathcal{V}$ and $\mathcal{X}\subseteq\mathcal{Y}\in\mathcal{V}$ we verify\begin{itemize}\item[1)]$\bigcap\mathcal{Y}\subseteq\bigcap\mathcal{X}$\item[2)]$A\cap B=B\cap A$\item[3)]$A\cap(B\cap C)=(A\cap B)\cap C=A\cap B\cap C$.\end{itemize}\end{proposition}

And we see that 1) gives for a special case; $\bigcap\mathcal{X}\subseteq A$, for every $A\in\mathcal{X}$.

		\subsubsection{Complement} 

Now, we consider the function that is left operation $\mathcal{U}-$. The image of $A$ may be represented by $A^c$.

\begin{proposition}For any $A\in\mathcal{V}$ we have\begin{itemize}\item[1)]$A\cup A^c=\mathcal{U}$\item[2)]$(A^c)^c=A$\end{itemize}\end{proposition}

We see that 2) takes the form $A=\mathcal{U}-(\mathcal{U}-A)=\mathcal{U}\cap A$

		\subsection{Properties}

			\subsubsection{Unit}

We now study the units of these operations, starting with the union. Notice that $A\cup B=A$ if and only if $B\subseteq A$. Since $\emptyset$ is a subset of any set, we have found the unit of the union operation to be $\emptyset$.

If we consider the intersection, we see that $A\cap B=A-(A-B)=A$ if and only if $A\subseteq B$. We need not look any further to find the unit, $\mathcal{U}$.

It is easily seen that $A=A-B$ if and only if $A\subseteq A-B$ and this is true if and only if $\emptyset=A-(A-B)=A\cap B$. Thus, $\emptyset$ is a right unit for the difference. In the case of the left operation, we are unable to find a unit; clearly there is no set $B$ such that $B-A=A$.

		\subsubsection{Inverse} 

It is now desireable to find the inverse set of $A$, for each of the operations. However, after a quick inspection we see that this is impossible. For the union we must find $B$ such that $A\cup B=\emptyset$, and for the intersection $C$ must satisfy $A\cap C=\mathcal{U}$. 

We will find a way around this, but the developments here will not give us an inverse set, in the strict sense. We will say $A$ and $A^{-1}$ are \textit{inverse sets for the union} if $A\cup A^{-1}=\mathcal{U}$. We say they are \textit{inverse sets for intersection} if $A\cap A^{-1}=\emptyset$. We say that $A^c$ is the \textit{exact inverse} of $A$.

		\subsubsection{Other Representations}

We have an interesting consequence of our definition for the operations $-$ and $\cap$. Each of these operations takes away objects from the source and this is in such a way that

\begin{eqnarray}A&=&(A\cap B)\cup(A-B)\label{decoset1}\\\nonumber\emptyset&=&(A\cap B)\cap(A-B).\end{eqnarray}

We have a similar decomposition for $A\cup B$:

\begin{eqnarray}A\cup B&=&A\cup(B-A)\label{decoset2}\\\nonumber\emptyset&=&A\cap(B-A).\end{eqnarray}

We also have \begin{eqnarray}\nonumber A-B&=&A-(\mathcal{U}-B^c)\\\nonumber&=&(A-\mathcal{U})\cup[A-(A-B^c)]\\\nonumber&=&A-(A-B^c)\\&=&A\cap B^c\end{eqnarray}

We are stating that the functions $-A,\cap A^c$ are the same, for any $A\in\mathcal{U}$. Of course we can also say \begin{eqnarray}\nonumber A\cap B=A-B^c\end{eqnarray}

\subsubsection{Distributions} In the foregoing we will accept a notation that will allow us to express some general relations for set operations:

\begin{eqnarray}\nonumber\bigcup\mathcal{X}&=&\bigcup_{A\in\mathcal{X}}A
\\\nonumber\\\nonumber\bigcap\mathcal{X}&=&\bigcap_{A\in\mathcal{X}}A
\end{eqnarray}

Of course there is a natural distribution for the union and intersection; as we had mentioned, there would be a clear parallelism with sum and product. However, the distribution here is both ways:

\begin{eqnarray} A\cup\bigcap_{B\in\mathcal{X}}B&=&\bigcap_{B\in\mathcal{X}}(A\cup B)\\\nonumber\\\nonumber A\cap\bigcup_{B\in\mathcal{X}}B&=&\bigcup_{B\in\mathcal{X}}(A\cap B)\end{eqnarray}

Now we give the distributions that involve the difference.

\begin{eqnarray} A-\bigcup_{B\in\mathcal{X}}B&=&\bigcap_{B\in\mathcal{X}}(A-B)
\\\nonumber\\\nonumber A-\bigcap_{B\in\mathcal{X}}B&=&\bigcup_{B\in\mathcal{X}}(A-B)
\end{eqnarray}

Naturally, the complement has a distribution rule, following that it is a special case of difference.

\begin{eqnarray}\nonumber\mathcal{U}-\bigcup_{B\in\mathcal{X}}B&=&\bigcap_{B\in\mathcal{X}}(\mathcal{U}-B)
\\\nonumber\\\nonumber \mathcal{U}-\bigcap_{B\in\mathcal{X}}B&=&\bigcup_{B\in\mathcal{X}}(\mathcal{U}-B)\end{eqnarray}

This is re-written as,

\begin{eqnarray} \left(\bigcup_{A\in\mathcal{X}}A\right)^c&=&\bigcap_{A\in\mathcal{X}}A^c
\\\nonumber\\ \left(\bigcap_{A\in\mathcal{X}}A\right)^c&=&\bigcup_{A\in\mathcal{X}}A^c\end{eqnarray}

These last expressions are the \textit{laws of DeMorgan}. The following relations are not difficult to prove.

\begin{eqnarray}\nonumber\bigcup_{A\in\mathcal{X}}A-\bigcup_{B\in\mathcal{Y}}B
&\subseteq&\bigcup(\mathcal{X}-
\mathcal{Y})\\\nonumber\\\nonumber\bigcap_{A\in\mathcal{X}}A-\bigcap_
{B\in\mathcal{Y}}B&\subseteq&\bigcap(\mathcal{X}-
\mathcal{Y}).\end{eqnarray}

	\section{Categories}

We have constructed the category of sets \textbf{Set}, and the main objective in this section is to give other descriptions of sets in terms of categories. That is, any set can be viewed as either a collection category or a partial order.

\subsection{Collection Category}

Unless it is otherwise specified, we define $\{\{A\}\}:=\{\{a\}\}_{a\in A}$, for any set $A$. We will build a category $\textbf{A}$ whose collection of c-objects is the family $\{\{A\}\}$. Consider a selection function for $\{\{A\}\}$; this function must send every set $\{a\}$ into the object $a$. We have exactly one selection function $f$ and it is defined by $Dom~f=\{\{A\}\}$, $Range~f=A$ and $\{a\}\mapsto_fa$. This is clearly a category; the unit arrow of each object is the only arrow corresponding to it.

Let $\textbf{A},\textbf{B}$ be two \textit{collection categories} and suppose there is an operation $*$ on each of the collections. We have a natural way of defining a functor $\textgoth{F}:\textbf{A}\rightarrow\textbf{B}$, given a function $f:A\rightarrow B$ such that $f,f;*fx,*x$. When considering the collection category of a power set we will write $\textbf{P}A$.

\subsection{Partial Order, Under Inclusion}

Let us consider the partial order defined on collection $\mathcal{V}$; we form a category $\textbf{Set}_{\subseteq}$ where the objects are ordered by inclusion. In other words, arrows are $A\subseteq B$. We see that we have a category because $A\subseteq A$ and if $A\subseteq B\subseteq C$, then $A\subseteq C$. Recall that in a partial order we are dealing with non-discernible arrows. In light of this, associativity of the composition holds. We note that a partial order may have objects which are not related; and in this order, that is the case. There are no arrows between $\{1\}$ and $\{2\}$, for example. Given a set $A$, we can identify it with a partial order, where the collection of objects is $\textgoth{P}A$; this order is a simplified version of $\textbf{Set}_{\subseteq}$, and we denote it by $\mathcal{P}A$.

		\subsection{Concrete Category}

We had the idea that given a discrete number system, there is a category such that the objects of the system are automorphisms of the category. After all, categories are quite large. So, we will consider the universe of sets in order to formailze the concept of a category whose arrows are functions. A category $\mathcal C$ is \textit{concrete} if we can provide a faithful functor $\textgoth C:\mathcal C\rightarrow\textbf{Set}$. We are sending the c-objects of the category into sets, while the arrows are sent into functions. Parallel arrows are sent into different functions of the same form. There have been several instances in which we have used arrows of categories as functions of one component. We have really been trying to represent the arrows as functions to define an equality of compositions in terms of natural pair of functions. A particular case was the second request in the definition of functors. What happens if the category is concrete?

We are giving a way of studying absract categories in terms of a well defined category. We had previously said that some algebraic categories can be seen as consisting of a category as c-object, and automorphisms of that category as arrows. For a concrete category, all the c-objects $x,y,z,...$ can be seen as sets, and all the arrows can be seen as set functions, in such a way that the transformation is a functor. This means we will not study a general abstract category with objects and arrows. We are going to study a collection of domains and ranges and set functions. This is a concrete concept because sets are well defined. The composition in set functions can be described as an operation $\circ:\mathcal A|\textbf{Set}\rightarrow\mathcal A|\textbf{Set}f\mathcal A|\textbf{Set}$ such that $f\mapsto_\circ\circ f$, where $g\mapsto_{\circ f}g\circ f$.

\subsubsection{Hom Set}Recall $\{a\rightarrow c\}$ is the collection of arrows in $\mathcal C$, such that $a$ is source and $c$ is target. Also, we have defined $\{a\rightarrow\}$ and $\{ 			\rightarrow c\}$ as the collections of arrows from $a$, and arrows into $c$, respectively. If $\{a\rightarrow c\}$ is a set, for every pair of c-objects in the category, we say $\mathcal C$ \textit{has all Hom sets}.

\begin{definition}Let $x$ be any c-object of $\mathcal C$, and define the contravariant functor $\textgoth R^{\ltimes}_x:\mathcal C\rightarrow\textbf{Set}$, where $a\mapsto\{a\rightarrow x\}$; every arrow $f:a\rightarrow c$ is sent into the set function $*f:\{c\rightarrow x\}\rightarrow\{a\rightarrow x\}$, of Hom sets. Define the covariant functor $\textgoth L_x:\mathcal C\rightarrow\textbf{Set}$ such that $a\mapsto\{x\rightarrow a\}$, so that $f$ is sent into a set function $f*:\{x\rightarrow a\}\rightarrow\{x\rightarrow c\}$.\end{definition}

\begin{lemma nat t1}Let $\textgoth F:\mathcal C_1\times\mathcal C_2\rightarrow\mathcal D$ be a functor. Then, every arrow $f:a\rightarrow c$, in $\mathcal C_1$, determines a natural transfomation $\tau_f:\textgoth F_a\rightarrow\textgoth F_c$ that sends $x\mapsto\textgoth F(f,1_x)$. The functors are of the form $\textgoth F_a,\textgoth F_c:\mathcal C_2\rightarrow\mathcal D$. A similar result can be formulated if $\mathcal C_2$ takes the place of $\mathcal C_1$.\end{lemma nat t1}

\begin{proof}The functors of the natural transformation are defined by the object functions $\textgoth F_ax:=\textgoth F(a,x)$ and $\textgoth F_cx:=\textgoth F(c,x)$. Let $g$ be an arrow in $\mathcal C_2$, then the arrow functions are $\textgoth F_ag:=\textgoth F(1_a,g)$ and $\textgoth F_cg:=\textgoth F(1_c,g)$. We will show $\tau y,\tau x;\textgoth F_cg,\textgoth F_ag$ for any arrow $g:x\rightarrow y$ in $\mathcal C_2$.\begin{eqnarray}\nonumber\tau y&;&\tau y*\textgoth F_ag,\textgoth F_ag\\\nonumber\tau y&;&\textgoth F(f,1_y)*\textgoth F(1_a,g),\textgoth F_ag\\\nonumber\tau y&;&\textgoth F(f*1_a,1_y*g),\textgoth F_ag\\\nonumber\tau y&;&\textgoth F(1_c*f,g*1_x),\textgoth F_ag\\\nonumber\tau y&;&\textgoth F(1_c,g)*\textgoth F(f,1_x),\textgoth F_ag\\\nonumber\tau y&;&\textgoth F_cg*\tau x,\textgoth F_ag\\\nonumber\tau y,\tau x&;&\textgoth F_cg,\textgoth F_ag.\end{eqnarray}\end{proof}

\begin{lemma nat t2}Let $\mathcal C_1,\mathcal C_2,\mathcal D$ be categories. For every $x$ in $\mathcal C_1$, and $y$ in $\mathcal C_2$, let $\textgoth R _x,\textgoth L_y:\mathcal C_2,\mathcal C_1\rightarrow\mathcal D$ be functors such that $y,\textgoth L_y;x,\textgoth R_x$. There exists a functor $\textgoth F:\mathcal C_1\times\mathcal C_2\rightarrow\mathcal D$, such that $\textgoth F_x$ is $\textgoth R_x$ and $\textgoth F_y$ is $\textgoth L_y$, if and only if for every arrow $f\rightarrow_\times g:a\rightarrow_\times b\longrightarrow c\rightarrow d$, in $\mathcal C_1\times\mathcal C_2$, we verify \begin{equation}\textgoth L_df,\textgoth L_bf;\textgoth R_cg,\textgoth R_ag.\label{eq lem cat}\end{equation} We define $\textgoth F(f,g):=\textgoth L_df*\textgoth R_ag$, and $\textgoth F(a,b):=\textgoth R_ab$.\end{lemma nat t2}

\begin{proof}We will first suppose the condition (\ref{eq lem cat}). First of all, $\textgoth1_\mathcal D\circ\textgoth F_\mathcal O$ is the same as $\textgoth F_\mathcal A\circ\textgoth 1_{\mathcal C_1\times\mathcal C_2}$:\begin{eqnarray}\nonumber a\rightarrow_\times b&;&(\textgoth1_\mathcal D\circ\textgoth F_\mathcal O)(a\rightarrow_\times b),\textgoth1_\mathcal D\circ\textgoth F_\mathcal O\\\nonumber a\rightarrow_\times b&;&\textgoth1_\mathcal D(\textgoth F(a,b)),\textgoth1_\mathcal D\circ\textgoth F_\mathcal O\\\nonumber a\rightarrow_\times b&;&\textgoth1_\mathcal D(\textgoth R_ab),\textgoth1_\mathcal D\circ\textgoth F_\mathcal O\\\nonumber a\rightarrow_\times b&;&1_{(\textgoth R_ab)},\textgoth1_\mathcal D\circ\textgoth F_\mathcal O\\\nonumber a\rightarrow_\times b&;&1_{(\textgoth R_ab)}*1_{(\textgoth R_ab)},\textgoth1_\mathcal D\circ\textgoth F_\mathcal O\\\nonumber a\rightarrow_\times b&;&1_{(\textgoth L_ba)}*1_{(\textgoth R_ab)},\textgoth1_\mathcal D\circ\textgoth F_\mathcal O\\\nonumber a\rightarrow_\times b&;&\textgoth L_b1_a*\textgoth R_a1_b,\textgoth1_\mathcal D\circ\textgoth F_\mathcal O\\\nonumber a\rightarrow_\times b&;&\textgoth F_\mathcal A(1_a,1_b),\textgoth1_\mathcal D\circ\textgoth F_\mathcal O\\\nonumber a\rightarrow_\times b&;&(\textgoth F_\mathcal A\circ\textgoth1_{\mathcal C_1\times\mathcal C_2})(a\rightarrow_\times b),\textgoth1_\mathcal D\circ\textgoth F_\mathcal O.\end{eqnarray} Given an arrow $f\rightarrow_\times g:a\rightarrow_\times b\longrightarrow c\rightarrow_\times d$, we have $\textgoth F(f,g):\textgoth F(a,b)\rightarrow\textgoth F(c,d)$. To prove this, notice $\textgoth R_ag,\textgoth L_df:\textgoth R_ab,\textgoth L_da\rightarrow\textgoth R_ad,\textgoth L_dc$. Since $d,\textgoth L_d;a,\textgoth R_a$ and $d,\textgoth L_d;c,\textgoth R_c$ are true, we may conlcude $\textgoth F(f,g):\textgoth F_ab\rightarrow\textgoth R_cd$. We move on to prove the third condition of functors. Let $h:c\rightarrow x$ and $i:d\rightarrow y$, \begin{eqnarray}\nonumber h\rightarrow_\times i*f\rightarrow_\times g&;&\textgoth F(h*f,i*g),\textgoth F\\\nonumber h\rightarrow_\times i*f\rightarrow_\times g&;&\textgoth L_y(h*f)*\textgoth R_a(i*g),\textgoth F\\\nonumber h\rightarrow_\times i*f\rightarrow_\times g&;&\textgoth L_yh*(\textgoth L_yf*\textgoth R_ai)*\textgoth R_ag,\textgoth F\\\nonumber h\rightarrow_\times i*f\rightarrow_\times g&;&\textgoth L_yh*\textgoth F(f,i)*\textgoth R_ag,\textgoth F\\\nonumber h\rightarrow_\times i*f\rightarrow_\times g&;&\textgoth L_yh*(\textgoth R_ci*\textgoth L_df)*\textgoth R_ag,\textgoth F\\\nonumber h\rightarrow_\times i*f\rightarrow_\times g&;&(\textgoth L_yh*\textgoth R_ci)*(\textgoth L_df*\textgoth R_ag),\textgoth F\\\nonumber h\rightarrow_\times i*f\rightarrow_\times g&;&\textgoth F(h,i)*\textgoth F(f,g),\textgoth F\\\nonumber h\rightarrow_\times i*f\rightarrow_\times g&;&\textgoth F(h\rightarrow_\times i)*\textgoth F(f\rightarrow_\times g),\textgoth F\end{eqnarray}
Finally, we must show $\textgoth F_x$ is the same functor as $\textgoth R_x$, and $\textgoth F_y$ is the same as $\textgoth L_y$. The observation is trivial for the object function; $\textgoth F_xy$ is $\textgoth F(x,y)$, whom we have defined as $\textgoth R_xy$. This implies that $\textgoth L_yx$ is $\textgoth F_yx:=\textgoth F(x,y)$. For the arrow functions, we have\begin{eqnarray}\nonumber i*g&;&\textgoth F_x(i*g),\textgoth F_x\\\nonumber i*g&;&\textgoth F(1_x,i*g),\textgoth F_x\\\nonumber i*g&;&\textgoth L_y1_x*\textgoth R_x(i*g),\textgoth F_x\\\nonumber i*g&;&1_{(\textgoth L_yx)}*\textgoth R_x(i*g),\textgoth F_x\\\nonumber i*g&;&1_{(\textgoth R_xy)}*\textgoth R_x(i*g),\textgoth F_x\\\nonumber i*g&;&\textgoth R_x1_y*\textgoth R_x(i*g),\textgoth F_x\\\nonumber i*g&;&\textgoth R_x[1_y*(i*g)],\textgoth F_x\\\nonumber i*g&;&\textgoth R_x(i*g),\textgoth F_x.\end{eqnarray}\begin{eqnarray}\nonumber h*f&;&\textgoth F_y(h*f),\textgoth F_y\\\nonumber h*f&;&\textgoth F(h*f,1_y),\textgoth F_x\\\nonumber h*f&;&\textgoth R_x(1_y)*\textgoth L_y(h*f),\textgoth F_y\\\nonumber h*f&;&1_{(\textgoth R_xy)}*\textgoth L_y(h*f),\textgoth F_y\\\nonumber h*f&;&1_{(\textgoth L_yx)}*\textgoth L_y(h*f),\textgoth F_y\\\nonumber h*f&;&\textgoth L_y1_x*\textgoth L_y(h*f),\textgoth F_y\\\nonumber h*f&;&\textgoth L_y[1_x*(h*f)],\textgoth F_y\\\nonumber h*f&;&\textgoth L_y(h*f),\textgoth F_y.\end{eqnarray}

Now, suppose the contrary and prove (\ref{eq lem cat})\begin{eqnarray}\nonumber \textgoth L_df&;&\textgoth L_df*\textgoth R_ag,\textgoth R_ag\\\nonumber \textgoth L_df&;&\textgoth F_df*\textgoth F_ag,\textgoth R_ag\\\nonumber \textgoth L_df&;&\textgoth F(f,1_d)*\textgoth F(1_a,g),\textgoth R_ag\\\nonumber \textgoth L_df&;&\textgoth F(f*1_a,1_d*g),\textgoth R_ag\\\nonumber \textgoth L_df&;&\textgoth F(1_c*f,g*1_b),\textgoth R_ag\\\nonumber \textgoth L_df&;&\textgoth F(1_c,g)*\textgoth F(f,1_b),\textgoth R_ag\\\nonumber \textgoth L_df&;&\textgoth F_cg*\textgoth F_bf,\textgoth R_ag\\\nonumber \textgoth L_df&;&\textgoth R_cg*\textgoth L_bf,\textgoth R_ag.\end{eqnarray}\end{proof}

\begin{theorem}Given any category with all Hom sets, we can form a bifunctor $Hom:\mathcal C\times\mathcal C\rightarrow\textbf{Set}$; referred to as the Hom bifunctor. 

The two functors that constitute Hom, are called the contravariant and covariant Hom functors. Given an arrow $f:a\rightarrow c$ in $\mathcal C$, we have natural transformations $f^\dagger:\textgoth L_c\rightarrow\textgoth L_a$ and $f_\dagger:\textgoth R^\ltimes_a\rightarrow\textgoth R^\ltimes_c$.\end{theorem}

\begin{proof}For every $x$ in $\mathcal C$, there are covariant functors $\textgoth R_x:\mathcal C^{op}\rightarrow\textbf{Set}$ and $\textgoth L_x:\mathcal C\rightarrow\textbf{Set}$; the functor $\textgoth R_x$ is defined by $\textgoth R_x(g^{op}):=(\textgoth R_x^\ltimes g)^{op}$. Notice that $\textgoth L_df$ is the function $f*:\{d\rightarrow a\}\rightarrow\{d\rightarrow c\}$ that sends $h:d\rightarrow a$ into $f*h$. On the other hand, $\textgoth R_a(g^{op})$ is $(\textgoth R^\ltimes_ag)^{op}:\{b\rightarrow a\}\rightarrow\{d\rightarrow a\}$ that sends $i^{op}:b\rightarrow_{op}a$ into $i^{op}*g^{op}$. Therefore, $\textgoth L_df\circ\textgoth R_a(g^{op})$ is the function $(f*)\circ(*g^{op}):\{b\rightarrow a\}\rightarrow\{d\rightarrow c\}$; the arrow function sends $i$ into $f*(i^{op}*g^{op})$, which is $f*(g*i)$.

One can just as easily prove $\textgoth R_c(g^{op})\circ\textgoth L_bf$ is the function $(*g^{op})\circ(f*):\{b\rightarrow a\}\rightarrow\{d\rightarrow c\}$ that sends $i$ into $(f*i^{op})*g^{op}$ which is $f*(i^{op}*g^{op})$. Using the second lemma, we find a functor $\textgoth F:\mathcal C^{op}\times\mathcal C\rightarrow\textbf{Set}$; consequently, we have a bifunctor $Hom:\mathcal C\times\mathcal C\rightarrow\textbf{Set}$.

Finally, apply the first lemma to the functor $\textgoth F$. For any arrow $f^{op}:a\rightarrow_{op}c$ in $\mathcal C^{op}$, we have a natural transformation $f^\dagger:\textgoth L_c\rightarrow\textgoth L_a$. The reader can find $f_\dagger:\textgoth R_a\rightarrow\textgoth R_c$ in the same way.\end{proof}

\subsubsection{Yoneda's Lemma}Here we will see that the construction provided for the natural numbers is not a coincidince. It is a specific application of the following result that generalizes even results from group theory (Cayley's Theorem). We will try to see why Yoneda's Lemma generalizes this result. The Yoneda embedding is provided as a corollary to the lemma and it enables us to prove Cayley's theorem. 

Let $E:\mathcal Cat(\mathcal C,\textbf{Set})\times\mathcal C\rightarrow\textbf{Set}$, be defined by $E(\textgoth C,x):=\textgoth Cx$ as object function. If $\tau\rightarrow_\times f$ is an arrow in the domain, with $\tau:\textgoth C\rightarrow\textgoth D$ as a natural transformation and $f:a\rightarrow c$, then the arrow function of $E$ is defined by the element given in $\tau c,\tau a;\textgoth Df,\textgoth Cf$. That is to say, $E(\tau,f):=\tau c\circ\textgoth Cf$. 

Now let $N:\mathcal Cat(\mathcal C,\textbf{Set})\times\mathcal C\rightarrow\textbf{Set}$ such that $N(\textgoth C,x):=Nat(\textgoth L_x,\textgoth C)$. Suppose $\textgoth C\rightarrow_\times a$ and $\textgoth D\rightarrow_\times c$ are c-objects in the domain such that $\tau:\textgoth C\rightarrow\textgoth D$ and $f:a\rightarrow c$. If we want for $N$ to be a functor, we have to give an arrow function that sends $\tau\rightarrow_\times f$ into a function $N(\tau,f):Nat(\textgoth L_a,\textgoth C)\rightarrow Nat(\textgoth L_c,\textgoth D)$. Take any $\alpha$ in $Nat(\textgoth L_a,\textgoth C)$ and define $N(\tau,f)\alpha:=\tau\cdot\alpha\cdot f^\dagger$; this is the vertical composition of natural transformations where we recall $f^\dagger:\textgoth L_c\rightarrow\textgoth L_a$.

We say that a natural transformation is a \textit{natural isomorphism} if all the components $\tau x$ are isomorphisms of the category in the range of the functors. This serves as a natural transformation from each functor to the other so that we can jump from one functor to the other and back.

\begin{Y Lemma}Let $a$ be any c-object in a category $\mathcal C$, with all Hom sets.\begin{itemize}\item[1)]Let $\textgoth C:\mathcal C\rightarrow\textbf{Set}$ be a functor, then there is a bijective function $\phi_{\textgoth Ca}:Nat(\textgoth L_a,\textgoth C)\rightarrow\textgoth Ca$.\item[2)]$E,N$ are functors and the functions $\phi_{\textgoth Ca}$ are components of a natural isomorphism $\Phi:N\rightarrow E$.\end{itemize}\end{Y Lemma}

\begin{proof}\makebox[5pt][]{}\mbox {}\begin{itemize}\item[1)]Let us define the function $\phi$ so that $\tau:\textgoth L_a\rightarrow\textgoth C$ is transformed by $\tau\mapsto\tau a(1_a)$. We know $\tau$ sends any c-object $x$, in $\mathcal C$, into the function $\tau x:\{a\rightarrow x\}\rightarrow\textgoth Cx$. This means $\tau a(1_a)$ is an element of $\textgoth Ca$. We also know that for any $f:a\rightarrow c$, the relation $\tau c,\tau a;\textgoth Cf,f*$ holds in terms of composition of functions; recall $\textgoth L_af$ is a function $f*:\{a\rightarrow a\}\rightarrow\{a\rightarrow c\}$. Therefore, we may say $\tau c(f*1_a)$ is the same as $\textgoth Cf[\tau a(1_a)]$. If we use the notation for applying functions, we express $f;\textgoth Cf[\tau a(1_a)],\tau c$. Let us suppose $\tau $ and $\sigma$ are different natural transformations in $Nat (\textgoth L_a,\textgoth C)$; that is, there exists an object $c$ such that $\tau c$ and $\sigma c$ are different set functions. As a consequence, we can give an arrow $f\in\{a\rightarrow c\}$ and we verify $\tau cf\neq\sigma cf$. Thus, $\textgoth Cf(\phi\tau)\neq\textgoth Cf(\phi\sigma)$, which implies $\phi\tau\neq\phi\sigma$, and we conlcude $\phi$ is monic. 

Let $x\in\textgoth Ca$, and define $\tau_x:\mathcal{O|C}\rightarrow\mathcal A|\textbf{Set}$ such that $\tau_xc:\{a\rightarrow c\}\rightarrow\textgoth Cc$. The function $\tau_xc$ transforms $f\mapsto\textgoth Cf(x)$. We will show this defines a natural transformation in $Nat(\textgoth L_a,\textgoth C)$. Apply the function $\tau_xc\circ f*$ to an arrow $g\in\{a\rightarrow a\}$:\begin{eqnarray}\nonumber g&;&(\tau_xc\circ f*)g,\tau_xc\circ f*\\\nonumber g&;&\tau_xc(f*g),\tau_xc\circ f*\\\nonumber g&;&\textgoth C(f*g)(x),\tau_xc\circ f*\\\nonumber g&;&(\textgoth Cf\circ\textgoth Cg)(x),\tau_xc\circ f*\\\nonumber g&;&\textgoth Cf[\textgoth Cg(x)],\tau_xc\circ f*\\\nonumber g&;&\textgoth Cf[\tau_xa(g)],\tau_xc\circ f*\\\nonumber g&;&(\textgoth Cf\circ\tau_xa)g,\tau_xc\circ f*\\\nonumber g&;&(\textgoth Cf\circ\tau_xa)g,\tau_xc\circ f*.\end{eqnarray}We have thus proven $\tau_xc,\tau_xa;\textgoth Cf,\textgoth L_af$ and we conclude $\phi$ is bijective.

\item[2)]Now, we would like to show $E,N$ are functors. We begin with $E$, verifying $\textgoth1_\textbf{Set}E(\textgoth C,x)$ is the identity function. The same results from sending $\textgoth C\rightarrow_\times x$ into its unit arrow, and transforming that with $E$; this results in $(1_\textgoth Cx)\circ( \textgoth C1_x)$, which is the composition of the identity function $I_{\textgoth Cx}$ with itself. Take a c-object in the domain, say $\textgoth C\rightarrow_\times a$ and $\textgoth D\rightarrow_\times c$, then the second condition of functors is given by $E(\tau,f):=\tau c\circ\textgoth Cf:\textgoth Ca\rightarrow\textgoth Cc\rightarrow\textgoth Dc$. The last condition is not difficult to prove either; let $\tau\rightarrow_\times f$ and $\sigma\rightarrow_\times g$ be composable arrows in the domain:\begin{eqnarray}\nonumber\sigma\rightarrow_\times g*\tau\rightarrow_\times f&;&E(\sigma\rightarrow_\times g*\tau\rightarrow_\times f),E\\\nonumber\sigma\rightarrow_\times g*\tau\rightarrow_\times f&;&E(\sigma\cdot\tau,g*f),E\\\nonumber\sigma\rightarrow_\times g*\tau\rightarrow_\times f&;&(\sigma\cdot\tau)c\circ\textgoth C(g*f),E\\\nonumber\sigma\rightarrow_\times g*\tau\rightarrow_\times f&;&\sigma c\circ\tau c\circ\textgoth Cg\circ\textgoth Cf,E\\\nonumber\sigma\rightarrow_\times g*\tau\rightarrow_\times f&;&\sigma c\circ\textgoth Dg\circ\tau b\circ\textgoth Cf,E\\\nonumber\sigma\rightarrow_\times g*\tau\rightarrow_\times f&;&E(\sigma,g)\circ E(\tau,f),E.\end{eqnarray}

To show the first condition of functros is valid for $N$, notice $(\textgoth1_\textbf{Set}\circ N)(\textgoth C,x)$ results in the identity function of $Nat(\textgoth L_x,\textgoth C)$. On the other hand, if we apply $N$ to the unit arrow of $\textgoth C\rightarrow_\times x$, we get a function $N(1_\textgoth C,1_x):Nat(\textgoth L_x,\textgoth C)\rightarrow Nat(\textgoth L_x,\textgoth C)$ such that $\alpha\mapsto1_\textgoth C\cdot\alpha\cdot1_x^\dagger$. Since $1_x^\dagger$ is the natural transformation $1_{\textgoth L_x}:\textgoth L_x\rightarrow\textgoth L_x$, we say 1) for functors is true. The second condition for $N$ to be a functor has been shown to be true, by construction. We are left to prove 3); we must show $N(\sigma\cdot\tau,g*f)$ is the same function as $N(\sigma,g)\circ N(\tau,f)$.\begin{eqnarray}\nonumber\alpha&;&N(\sigma\cdot\tau,g*f)\alpha,N(\sigma
\cdot\tau,g*f)\\\nonumber\alpha&;&(\sigma\cdot\tau)\cdot\alpha\cdot(g*f)^\dagger,
N(\sigma\cdot\tau,g*f)\\\nonumber\alpha&;&(\sigma\cdot\tau)\cdot\alpha\cdot
(f^\dagger\cdot g^\dagger),N(\sigma\cdot\tau,g*f)\\\nonumber\alpha&;&
\sigma\cdot(\tau\cdot\alpha\cdot f^\dagger)\cdot g^\dagger,N(\sigma\cdot\tau,g*f)
\\\nonumber\alpha&;&N(\sigma,g)(\tau\cdot\alpha\cdot f^\dagger),N(\sigma\cdot\tau,g*f)\\\nonumber\alpha&;&[N(\sigma,g)\circ N(\tau,f)]\alpha,N(\sigma\cdot\tau,g*f).\end{eqnarray}Here, we are considering that for composable arrows $f,g:a,b\rightarrow b,c$ we have the natural transfomations $f^\dagger:\textgoth L_b\rightarrow\textgoth L_a$ and $g^\dagger:\textgoth L_b\rightarrow\textgoth L_c$.

Now we prove there is a natural transformation from $\Phi:N\rightarrow E$. To verify this, let $\tau\rightarrow_\times f$ be an arrow $\textgoth C\rightarrow_\times a$ into $\textgoth D\rightarrow_\times b$ and let $\alpha:\textgoth L_a\rightarrow\textgoth C$:\begin{eqnarray}\nonumber \alpha&;&[\Phi(\textgoth C,a)\circ N(\tau,f)]\alpha,\Phi(\textgoth C,a)\circ N(\tau,f)\\\nonumber\alpha&;&\phi_{\textgoth Ca}(\tau\cdot\alpha\cdot f^\dagger),\Phi(\textgoth C,a)\circ N(\tau,f)\\\nonumber\alpha&;&[(\tau \cdot\alpha\cdot f^\dagger)b](1_b),\Phi(\textgoth C,a)\circ N(\tau,f)\\\nonumber\alpha&;&[(\tau\cdot\alpha)b\circ f^\dagger b](1_b),\Phi(\textgoth C,a)\circ N(\tau,f)\\\nonumber\alpha&;&[(\tau\cdot\alpha)b]f,\Phi(\textgoth C,a)\circ N(\tau,f)\end{eqnarray}Since $f^\dagger$ is a natural transformation that sends c-objects in $\mathcal C$, into set functions, we know $f^\dagger b$ is a function $\{b\rightarrow_{op} b\}\rightarrow\{b\rightarrow_{op}a \}$, so that $f^\dagger b(1_b)$ is defined as $f^{op}*1_b^{op}$, which is $f$. We continue,\begin{eqnarray}\nonumber\alpha&;&(\tau b\circ\alpha b)f,\Phi(\textgoth C,a)\circ N(\tau,f)\\\nonumber\alpha&;&\tau b(\alpha bf),\Phi(\textgoth C,a)\circ N(\tau,f)\\\nonumber\alpha&;&\tau b[\textgoth Cf(\alpha a(1_a))],\Phi(\textgoth C,a)\circ N(\tau,f)\\\nonumber\alpha&;&(\tau b\circ\textgoth Cf)(\alpha a1_a),\Phi(\textgoth C,a)\circ N(\tau,f)\\\nonumber\alpha&;&E(\tau,f)(\phi_{\textgoth Ca}\alpha),\Phi(\textgoth C,a)\circ N(\tau,f)\\\nonumber\alpha&;&E(\tau,f)[\Phi(\textgoth C,a)\alpha],\Phi(\textgoth C,a)\circ N(\tau,f)\\\nonumber\alpha&;&[E(\tau,f)\circ\Phi(\textgoth C,a)]\alpha,\Phi(\textgoth C,a)\circ N(\tau,f).\end{eqnarray}Notice we are using the fact that $\alpha bf$ is the same as $\textgoth Cf[\alpha a(1_a)]$. With this we have proven $\Phi$ is natural. Now, we only need to notice that the components are bijective functions, to conclude $\Phi$ is a natural isomorphism.\end{itemize}\end{proof}

The result presented below has been proven above, for the most part; it is commonly known as the Yoneda Embedding.

\begin{corollary}There is a contravariant functor $Y^\dagger:\mathcal C\rightarrow\mathcal Cat(\mathcal C,\textbf{Set})$ such that $Y^\dagger x:=\textgoth L_x$. There is a covariant functor $Y_\dagger:\mathcal C\rightarrow\mathcal Cat(\mathcal C^{op},\textbf{Set})$ that makes $Y_\dagger x:=\textgoth R_x$. The functors $Y^\dagger$ and $Y_\dagger$ are full embeddings, with the arrow functions $f\mapsto_{Y^\dagger}f^\dagger$ and $f\mapsto_{Y_\dagger} f_\dagger$.\end{corollary}

\begin{proof}We only need to prove the arrow function is bijective. In the lemma, make $\textgoth C:=\textgoth L_b$, so that we have a bijective function $\phi:Nat(\textgoth L_a,\textgoth L_b)\rightarrow\textgoth L_ba$; this is a bijection $\{b\rightarrow a\}\rightarrow Nat(\textgoth L_a,\textgoth L_b)$. Given any arrow $f$, we know there is a natural transformation $f^\dagger$ in $Nat(\textgoth L_a,\textgoth L_b)$. Additionally, we have proven $\phi f^\dagger:=f^\dagger a(1_a)=f$.\end{proof}

\begin{corollary}Suppose $G$ is a group such that the collection of objects of operation is a set, $\mathcal A|G\in\textbf{Set}$. Then $G$ is isomorphic to a group of transformations $G_\dagger,G^\dagger:\{e\rightarrow e\}\rightarrow\{e\rightarrow e\}$.\end{corollary}

\begin{proof}Let $\mathcal C:=G$, be the group. The functor $\textgoth C:=\textgoth L_e:G\rightarrow\textbf{Set}$ is as we have described before; the c-object $e$ is sent into a set $S$, while objects of operation are sent into set functions $S\rightarrow_{iso}S$. The Yoneda embedding provides a functor $Y^\dagger:G\rightarrow\mathcal Cat(G,\textbf{Set})$, where the c-object is transformed into the Hom functor $\textgoth L_e$. The existence of inverse objects of operation, in the group $G$, implies $f*$ and $*f$ are isomorphisms (bijective function) for the set $\mathcal A|G$; this was proven in proposition (\ref{prop lr can grp}). This means we indeed have a group of transformations for $\mathcal A|G$, if we replace the image by a category where $S$ is c-object, instead of $\textgoth L_e$.\end{proof}

In building the group of integers, we took a category $\mathbb Z_\dagger$ as the c-object of the group. The objects of operation are of course the integers which we considered as automorphisms for the category mentioned. Then, the group of functors $\mathbb Z^\dagger$ was constructed as compositions and inverses of $+1$. This is the functor we used to later prove that we can view the discrete number system as a group. This is the functor that gave an operation.

			\subsubsection{Representable Functors}An object in the collection $\mathcal C\rightarrow_\times Nat(\textgoth L_x,\textgoth C)$, is said to be a \textit{representation of} $\textgoth C:\mathcal C\rightarrow\textbf{Set}$. The functor is said to be \textit{representable} and $x$ is called a \textit{representing object}.

We see that given two representations of a functor, it is possible to find an isomorphism, in $\mathcal C$, that relates the natural transformations.

\begin{proposition}Let $(x,\beta)$ and $(y,\gamma)$ be two representations of a functor $\textgoth F:\mathcal C\rightarrow\textbf{Set}$. Then there is a unique isomorphism $f:x\rightarrow y$ such that $\gamma=\beta\cdot f^\dagger$.\end{proposition}

\begin{proof}We can define a composition of natural isomorphisms, $\beta^{-1}\circ\gamma:\textgoth L_y\rightarrow\textgoth C\rightarrow \textgoth L_x$, where $\beta^{-1}$ is the natural isomorphism of inverse arrows. From Yoneda's embedding, we know there is an arrow $f$ such that $\beta^{-1}\circ\gamma=f^\dagger$. We have proven that functors send isomorphisms into isomorphisms; proposition \ref{iso map}.\end{proof}

	\section{Set Function}

		\subsection{Image}

Here we will give relations for the image and preimage of set functions. If we have an object in $fA$, for some $A\subseteq Dom~f$, then there is an $x\in A$ such that our original object, in $fA$, is the object $fx$. For this reason, we will be justified in denoting our object of interest with $(fx)\in fA$. The notation says that $(fx)$ is an object in $Range~f$, such that $(fx)=fx$.

When considering a family of sets, we also consider another family, to be regarded as the image of the original. Let $\mathcal{X}=\{A\}_{A\in\mathcal{X}}$ be a family of subsets of $Dom~f$. The \textit{image of family $\mathcal{X}$}, is the family that consists of the images of sets in $\mathcal{X}$. This means, $f[[\mathcal{X}]]=\{fA\}_{A\in\mathcal{X}}$. Let us suppose we have a family $\mathcal{Y}=\{B\}_{B\in\mathcal{Y}}$ of subsets $B\subseteq Range~f$. We define the \textit{inverse image of family $\mathcal{Y}$}, as the family consisting of the sets that are inverse image of sets in $\mathcal{Y}$. Let $f^{-1}[[\mathcal{Y}]]$ be the inverse image of $\mathcal{Y}$, and take an object of it. Then there exists $B\in\mathcal{Y}$ such that our original object is the inverse image of $B$. We are justified in representing an arbitrary object in $f^{-1}[[\mathcal{Y}]]$, with $(f^{-1}B)$.

		\subsubsection{Image and Inclusions}

The image and preimage preserve subsets. That is, $fA\subseteq fB$ and $f^{-1}A\subseteq f^{-1}B$; the first relation holds given $A\subseteq B\subseteq Dom~f$ and the second holds given $A\subseteq B\subseteq Range~f$.

\begin{proposition}Let $f:Dom~f\rightarrow Range~f$ be a set function. Then, for every $A\subseteq Dom~f$ and $B\subseteq Range~f$ we verify \begin{itemize}\item[1)]$A\subseteq f^{-1}fA$, and $A=f^{-1}fA$ if $f$ is monic.\item[2)]$B\supseteq ff^{-1}B$, and $B=ff^{-1}B$ if $f$ is onto.\end{itemize}\end{proposition}

\begin{proof}$x\in A$ implies $fx\in fA$, which is true if and only if $x\in f^{-1}fA$. If $f$ is monic, then $x\in A$ if and only if $fx\in fA$, which proves the equality holds.

We have $(fx)\in ff^{-1}B$ if and only if there exists $x\in f^{-1}B$ such that $(fx)=fx\in B$. Let $f$ be onto, then for every $(fx)\in B$ there exists $x\in f^{-1}B$ such that $(fx)=fx$, which proves $(fx)\in ff^{-1}B$.\end{proof}

Using the results so far given, we get \begin{eqnarray}\nonumber fA\subseteq B &\Longleftrightarrow&A\subseteq f^{-1}B\end{eqnarray}

which was to be expected since we define the inverse image by $x\in f^{-1}A\Leftrightarrow fx\in A$. If $f$ is onto,

 \begin{eqnarray}\nonumber  f^{-1}B\subseteq A&\Longrightarrow&B\subseteq fA.\end{eqnarray}

If $f$ is monic, then

 \begin{eqnarray}\nonumber B\subseteq fA&\Longrightarrow&f^{-1}B\subseteq A.\end{eqnarray}

		\subsubsection{Image and Set Operations}

The difference is preserved under preimage; $x\in f^{-1}(B-A)~\Leftrightarrow~fx\in B-A~\Leftrightarrow~fx\in B$ and $fx\notin A~\Leftrightarrow~x\in f^{-1}B$ and $x\notin f^{-1}A~\Leftrightarrow~x\in f^{-1}B-f^{-1}A$. We get the following result, as a consequence of this:

\begin{eqnarray}f^{-1}(Range~f-A)=Dom~f-f^{-1}A.\nonumber\end{eqnarray}

We will generally say $A^c$ is $X-A$ if we accept that all work will be done in a certain set $X$. If we are considering the sets $Dom~f$ and $Range~f$, we may re-write the above expression as $f^{-1}A^c=(f^{-1}A)^c$. Let $B\subseteq Range~f$, then $f|_{A}^{-1}B=f|_{A}^{-1}[(B^c)^c]=A-f|_A^{-1}B^c=A-(f|^{-1}_AB)^c=A\cap f|_A^{-1}B$. It is easy to prove $A\cap f|_A^{-1}B=A\cap f^{-1}B$, so we conclude \begin{eqnarray}f|_A^{-1}B=A\cap f^{-1}B.\end{eqnarray}This means the inverse image of $B$, under $f|_A$, is equal to $A\cap f^{-1}B$.

\begin{proposition}Let $\mathcal{X}=\{A\}_{A\in\mathcal{X}}$ a family of subsets $A\subseteq Dom~f$ and $\mathcal{Y}=\{B\}_{B\in\mathcal{Y}}$ a family of subsets $B\subseteq Range~f$. Then \begin{eqnarray}f\bigcup_{A\in\mathcal{X}}A&=&\bigcup_{fA\in f[[\mathcal{X}]]} fA\\\nonumber\\\nonumber f\bigcap_{A\in\mathcal{X}}A&\subseteq&\bigcap_{fA\in f[[\mathcal{X}]]}fA\\\nonumber\\\nonumber f^{-1}\bigcup_{B\in\mathcal{Y}}B&=&\bigcup_{f^{-1}B\in f^{-1}[[\mathcal{Y}]]} f^{-1}B\\\nonumber\\\nonumber  f^{-1}\bigcap_{B\in\mathcal{Y}}B&=&\bigcap_{f^{-1}B\in f^{-1}[[\mathcal{Y}]]}f^{-1}
B.\end{eqnarray}We verify equality in the second relation, given $f$ is monic.\end{proposition}

\begin{proof}\makebox[5pt][]{}\mbox {}\begin{itemize}\item[1)]We know $(fx)\in f\bigcup_{A\in\mathcal{X}}A$ if and only if there exists $x\in\bigcup_{A\in\mathcal{X}}A$ such that $(fx)=fx$. This is the same as saying $(fx)\in fA$, for some $fA\in f[[\mathcal{X}]]$.\item[2)]The relation $(fx)\in f\bigcap_{A\in\mathcal{X}}A$ holds if and only if $(fx)=fx$, for some $x\in \bigcap_{A\in\mathcal{X}}A$. Now, let $(fA)\in\ f[[\mathcal{X}]]$. Then $(fA)=fA$, for some $A\in\mathcal{X}$ such that $fx\in fA$. In other words, $(fx)\in(fA)$, for every $(fA)\in f[[\mathcal{X}]]$.

Take $(fx)\in\bigcap_{fA\in f[[\mathcal{X}]]}fA$. This is stating that $(fx)$ is in the image of every $A\in\mathcal{X}$. This holds if and only if, for every $A\in\mathcal{X}$, there exists $x_a\in A$ such that $(fx)=fx_a$. If $f$ is monic, then  $x_a=x$, for some $x$ that is the same for all $A\in\mathcal{X}$. We conclude $(fx)=fx\in f\bigcap_{A\in\mathcal{X}}A$ because $x\in\bigcap_{A\in\mathcal{X}}A$.\item[3)]It is from the definition of inverse image that we have $fx\in\bigcup_{B\in\mathcal{Y}}B$, for every $x\in f^{-1}\bigcup_{B\in\mathcal{Y}}B$. Then, $x\in f^{-1}B$, for some $B\in\mathcal{Y}$. We conclude $x\in\bigcup_{B\in\mathcal{Y}}f^{-1}B$. All the implications in the argument are $\Leftrightarrow$.\item[4)]$x\in f^{-1}\bigcap_{B\in\mathcal{Y}}B$ if and only if $fx\in\bigcap_{B\in\mathcal{Y}}B$. It follows from this that $x\in f^{-1}B$, for every $B$.\end{itemize}\end{proof}

		\subsubsection{Fiber}

We say that the preimage of a singleton $\{z\}\subseteq Im~f$ is the \textit{fiber for $z$, under $f$}. We will express this by $f^{-1}\{z\}$ or $f^{-1}[z]$, in place of the strict notation $f^{-1}[\{z\}]$. We notice that $x\in f^{-1}[z]\Leftrightarrow x\in f^{-1}\{z\}\Leftrightarrow fx\in\{z\}\Leftrightarrow fx=z$. It is clear that $f^{-1}[z]=\emptyset$ implies $z\in Range~f-Im~f$. From the definition of function we know every object in the domain belongs to exactly one fiber of $f$. If $fA=z\in Im~f$, then $A\subseteq f^{-1}[z]$. A function with $z\in Im~f$ such that $f^{-1}[z]=Dom~f$ is called \textit{constant function to z} and we write $\rightarrow z$.

A function is onto if and only if every fiber of $f$ is non-empty. Let $f$ be a bijective function and let $g$ be the inverse function of $f$, then $x\in f^{-1}[z]$ if and only if $z\in g^{-1}[x]$. Said differently, this last means $x\in f^{-1}\{z\}$ if and only if $z\in(f^{-1})^{-1}\{x\}$; we take notice that $f^{-1}\{z\}$ is the inverse image of $\{z\}$, under $f$, while $(f^{-1})^{-1}\{x\}$ is representative of the inverse image of $\{x\}$ under $f^{-1}$.

Let us consider the family $f^{-1}\{\{B\}\}:=f^{-1}[[\{\{b\}\}_{b\in B}]]=\{f^{-1}\{b\}\}_{b\in B}$, which represents inverse image of the family $\{\{b\}\}_{b\in B}$.

\begin{proposition}Let $f$ be a monic function and take $A\subseteq Dom~f$. Then   $fA=B\Leftrightarrow A=\bigcup f^{-1}\{\{B\}\}$.\label{FIB}\end{proposition}

\begin{proof}Suppose $A=\bigcup f^{-1}\{\{B\}\}=\bigcup\{f^{-1}\{b\}\}_{b\in B}$, then \begin{eqnarray}\nonumber fA&=&f\bigcup\{f^{-1}\{b\}\}_{b\in B}\\\nonumber&=&\bigcup f\{f^{-1}\{b\}\}_{b\in B}\\\nonumber&=&\bigcup\{\{b\}\}_{b\in B}\\\nonumber&=&\bigcup_{b\in B}\{b\}.\end{eqnarray}

If on the contrary, $fA=B$, then

\begin{eqnarray}\nonumber A&=&\bigcup_{a\in A}\{a\}\\\nonumber&=&\bigcup_{a\in A}f^{-1}f\{a\}\\\nonumber&=&\bigcup_{b\in fA}f^{-1}\{b\}\\\nonumber&=&\bigcup\{f^{-1}\{b\}\}_{b\in B}.\end{eqnarray}

\end{proof}

	\subsection{Quotient Sets and Decomposition of Functions}

We start this section by giving a result that characterizes onto functions as functions that have and right inverse. Consider the family $\{\{Im~f\}\}=\{\{x\}\}_{x\in Im~f}$, which means it consists of the sets $\{x\}$, where $x\in Im~f$. The inverse image, $f^{-1}[[\{Im~f\}]]$, is a family of sets in $Dom~f$ which turn out to be the fibers of $f$. We will view $f^{-1}[[\{\{Im~f\}\}]]$ as a simple set, in which we will not take into account what the objects of the sets in the family are; we only care for the sets of the family.

A function can always be expressed as a composition of an inmersion and an onto function; a composition of one onto function and one monic function. Let $f|^{Im~f}:Dom~f\rightarrow Im~f$, denote the function $f$ restricted to the image.

\begin{lemma2.1}For any $f$ and $\iota_f:Im~f\rightarrow Range~f$ we verify $f=\iota_f\circ f|^{Im~f}$.\end{lemma2.1}

Further on, we will complete this description by expressing a function as the composition of three functions, one of each type: onto, bijective, and monic.

Given a set function $f$, we will give an equivalence relation $E_f$ defined for $Dom~f$. We say two objects $x,y\in Dom~f$ are related, $xE_fy$, if $fx=fy$. We call such a relation the \textit{image equivalence of f}. Another way of seeing this is $$xE_fy\Longleftrightarrow x,y\in f^{-1}[z],$$for some $z\in Im~f$. We may conclude that the fibers of $f$ form the simple equivalence reations of $E_{f}$.

Define a new domain $Dom~f/E_f:=f^{-1}[[\{\{Im~f\}\}]]$, for the function $$f/E_f:Dom~f/E_f\rightarrow Im~f.$$ Naturally, we define $f^{-1}[z]\mapsto_{f/E_f}z$, because for every $x\in f^{-1}[z]$, we have $x\mapsto_fz$.

\begin{lemma2.2}The function $f/E_f$ is said to be the function $f$ module $E_{f}$ and it is bijective.\end{lemma2.2}

\begin{proof}The function is onto because every fiber of $Im~f$ is non-empty. To see that it is also monic, take two different fibers $f^{-1}[w]$ and $f^{-1}[z]$. These two fibers consist of objects $x\mapsto_f w$ and $y\mapsto_f z$ and they form simple equivalence relations under $E_f$. Therefore, $w\neq z$.\end{proof}

Let $p_f:Dom~f\rightarrow Dom~f/E_f$ that sends an object to its fiber. We can do this because we have already stated that every object in $Dom~f$ is associated one fiber. We know the function is onto because every fiber in $f^{-1}[[\{Im~f\}]]$ is non-empty.

\begin{theorem}Given a function $f$, we can decompose it as $f=\iota_f\circ f/E_f\circ p_f$\end{theorem}

	\subsection{Sequence}

From this point on, we may use $\mathbb N$ to refer either to the collection or $\mathbb N_{\leq}$. A sequence is a set function $s:\mathbb{N}\rightarrow A$. A \textit{general sequence} is a sequence in which the image $A$ is a collection, not necesarilly a set. We will say a set function is a sequence, if $Dom~s$ is $\mathbb{N}_{0}$ instead of $\mathbb{N}$. Let $n\in\mathbb{N}$, and define $\textbf{n}$ as the set of all $x\in\mathbb{N}$ such that $1\leq x\leq n$. A \textit{finite sequence} is a function $\textbf{n}\rightarrow A$. Also, a set $A$ is said to be finite if there exists $n\in\mathbb{N}$ such that there is a bijective function $\textbf{n}\rightarrow A$.

Consider now the category $\mathbb{N}$ which stands for the partial order. Let $\textgoth{S}:\mathbb{N}\rightarrow\mathbb{N}$ be a functor on the order. Then $s\circ\textgoth{S}$ is a sequence and we will call it a subsequence of 
$s$. We request $\textgoth{S}$ to be a functor so that the the order of $\mathbb{N}$ is preseved.

		\subsubsection{Representation}

A sequence is, informally, an ordered collection. It makes sense to, for example, write a sequence of two numbers as $(a,b)$ and not $\{a,b\}$ which is a simple set. In general, to express a sequence with components $n\mapsto x_n$, we will write $(x_i)_i$. If the sequence is finite, we will write $(x_i)_{i=1}^n$. When we write $(x)_{i}$, we are expressing that the sequence is a constant function, $\rightarrow x$.

		\subsubsection{Monotonic Sequence}

Suppose we define a functor $\mathbb{N}\rightarrow A$, where $A$ is a partial order. Then, the object function of such a functor is called a increasing sequence in $A$. If there is a functor $\mathbb{N}\rightarrow A^{op}$, we will call the object function of such a functor, a decreasing sequence in $A$.

		\subsubsection{Sequence of Objects for an Operation}
 
We can give a general form of expressing the operation union and intersection for sets. Let $\mathcal{I}$ be a set, the \textit{index set}, and let $\mathcal{I}\rightarrow\mathcal{X}$ such that $i\mapsto A_i$. The following may be used as a common notation

$$\bigcup\mathcal{X}=\bigcup_{i\in\mathcal{I}}A_i.$$

Similar remarks hold for the intersection. We will use the set of integers as an index set, in order to express general operations. Let $\oplus$ be an operation $\oplus:\mathcal{O}\rightarrow\mathcal{O}f\mathcal{O}$ and let $\mathcal{O}^\mathbb{N}$ represent the collection of sequences 
$\mathbb{N}\rightarrow\mathcal{O}$. We will give, for $n+1\in\mathbb{N}_0$, a function $\Lambda_{i=1}^{n+1}:\mathcal{O}^\mathbb{N}\rightarrow\mathcal{O}$ such that $(x_i)_i\mapsto\Lambda_{i=1}^{n}(x_i)_i\oplus x_{n+1}$, where $\Lambda_{i=1}^1(x_i)_i=x_1$. These functions are called the \textit{finite operators}. The sequence $\Lambda:\mathbb{N}\rightarrow\mathcal{O}^\mathbb{N}f\mathcal{O}$ is defined by $n\mapsto_{\Lambda}\Lambda_{i=1}^n$; we call it the \textit{general operator}. For any sequence of sets, we can give union and intersection and this is expressed by $\bigcup_{i\in\mathbb{N}}A_i$ or $\bigcap_{i\in\mathbb{N}}A_i$. We define $\Lambda_\mathcal{O}=\bigcup_{n\in\mathbb{N}}Im~\Lambda n$. If there is a sensible way of defining a function $\Lambda\mathbb N:A\subseteq\mathcal O^\mathbb N\rightarrow\mathcal O$, we may say the \textit{series} of $(x_i)_i$ is the result in $(x_{i})_i;\Lambda_{i\in\mathbb N}x_i,\Lambda\mathbb N$. We may also represent the series by $\Lambda_{i=1}^\infty x_i$.

We may use this development to express the general sum of multiple rational numbers. Just as we use the notation $\mathcal{O}^{\mathbb{N}}$, we will say $\mathcal{O}^{\textbf{n}}$ stands for the collection of all finite sequences $\textbf{n}\rightarrow\mathcal{O}$. Thus the sum of a finite sequence of rational numbers, $(x_i)_i\in\mathbb{Q}^\textbf{n}$, is given by $$\sum_{i=1}^nx_i=\sum_{i=1}^{n-1}x_i+x_n$$

and from this we get

$$a\cdot n=\sum_{i=1}^na.$$

We are making it clear that $\cdot n$ and $\sum_{i=1}^n$ are the same function when applied, respectively, to $a\in\mathbb{Q}$ and the sequence $(a)_i\in\mathbb{Q}^\textbf{n}$. Note that every natural number $a\in\mathbb{N}$ is represented as $a=\sum_{i=1}^a1$.

		\subsubsection{Sequence of Functions}

We have just defined a sequence of functions, that can be applied to a sequence, and the result is another sequence. We study this with a general view. A sequence of objects, in a collection of functions, is called a \textit{sequence of functions}. A sequence of functions from $\mathcal O$ into $\mathcal Q$ is a function $\Phi:\mathbb N\rightarrow\mathcal Of\mathcal Q$. Let $n\mapsto_\Phi f_n$ be the components of $\Phi$, and define $\phi:\mathcal O\rightarrow\mathcal Q^\mathbb N$ such that $x\mapsto_\phi(f_ix)_i$; it is the \textit{sequence function}. 

In the case of series, the general operator $\Lambda:\mathbb N\rightarrow\mathcal O^\mathbb Nf\mathcal O$ is the sequence of functions $\Phi$. The sequence function is $\phi:\mathcal O^\mathbb N\rightarrow\mathcal O^\mathbb N$. The image of $(x_i)_i$, under $\phi$, is the \textit{sequence of partial operations for $(x_i)_i$} and we represent it with $(\phi_i)_i$. Clearly, every $\phi_n$ is $\Lambda_{i=1}^n(x_i)_i$, so we have stated $\phi(x_i)_i$ is $\left(\Lambda_{i=1}^n(x_i)_i\right)_n$. We will consider an important special case of this. Let $(a)_i\in\mathcal O^\mathbb N$ be a constant sequence, and suppose there is an operation on the collection, called the product operation. Then the \textit{sequence of powers} is defined as the sequence of partial products. For the operation of product $\cdot$ in $\mathbb Q$, we use $\prod$ as the general operator and

$$a^n=\prod_{i=1}^na,$$

in representing the finite operator of product, applied to $(a)_i$.

		\subsubsection{Sequence for Composition} 

There is something important to be noticed in the last paragraph. For a fixed sequence $(x_i)_i\in\mathcal{O}^{\mathbb{N}}$, we have a sequence of functions $(\oplus x_i)_i$ which we compose to find $\Lambda_{i=1}^{n}(x_i)_i$. This sequence of functions has something particular about it. If the operation is full, each one of these functions is of the form $\mathcal{O}\rightarrow\mathcal{O}$. Let us make the case for a more general situation.

Let $(f_i)_i$ be a sequence of functions such that $Im~f_n=Dom~f_{n+1}$. This is a composable sequence of functions and to better understand this definition we will give a more specific description of the composition of functions. Let $f:Dom~f\rightarrow Range~f$ and $g:Dom~g\subseteq Range~f\rightarrow Range~g$. Let $D=f^{-1}[Im~f\cap Dom~g]$, then the composition is $g\circ f:D\rightarrow Range~g$; we are assured $x\mapsto_{f}fx\mapsto_{g}g(fx)$. The composition is onto if and only if g is onto, and it is monic if and only if both $f,g$ are monic.

The compositions of a sequence are $\bigcirc_{i=1}^{n+1}(f_i)_i\mapsto\bigcirc_{i=1}^n(f_i)_i\circ f_{n+1}$, remembering that $\bigcirc_{i=1}^1(f_i)_i$ is $f_1$. We have a special case, when we take a constant sequence $(f)_i$, consisting of only one function $f$. The composition $\bigcirc_{i=1}^n(f)_i$ is the function $f^n$. We see that $Im~f^2=f[Im~f]$, $Im~f^3=f[f[Im~f]]$, etc...

If $x\mapsto_{f}x$, then we say \textit{x is an invariant object under f}. If $fA=A$, then \textit{A is an invariant subset under f}. Notice that it is not the same to say $A$ is an invariant subset under $f$ as opposed to saying $A$ is a set of invariant objects under $f$. In this last case we say that $A$ is a \textit{strongly invariant subset}; these are subsets of $Inv~f$ which is the set of invariant objects. If $Im~f$ is strongly invariant, $Im~f\subseteq Inv~f$, then the function is said to be \textit{once-effective}. The situation is that $f^2:Im~f\rightarrow f[Im~f]=Im~f$ and $x\mapsto_{f^2}fx$ because $fx$ is invariant under $f$. This means the sequence of compositions $(f^{i})_{i}$ is the constant sequence of the function $f|_{Im~f}$.

Let $a$ be invariant under $f$ and suppose that for every $x\in Dom~f$ there exists $n\in\mathbb{N}$ such that $x\mapsto_{f^n}a$; we say \textit{f stableizes at a}. If there is $n\in\mathbb{N}$ such that $x\mapsto_{f^n}a$, for every $x\in Dom~f$, then \textit{f is nilpotent} into $a$. 

		\subsubsection{Sequence and Cartesian Product}

From this point on, we make the convention of writing $A\times B$ in the place of $A\rightarrow_\times B$. The cartesian product of two collections $A,B$ is a collection of arrows $a\rightarrow_\times b$, with $a,b\in A,B$ respectively. Let $\textbf{2}f(A\sqcup B)$ be the collection of two part functions $\textbf{2}\rightarrow A\cup B$ that send $1$ to an object in $A$ and send 2 into object of $B$. We have a bijective function of the form $A\times B\rightarrow\textbf{2}f(A\sqcup B)$.

Let us now consider the cartesian product of a sequence of sets; that is, $A=\times_{i=1}^{n+1}A_i$. This is the set of all arrows $a_1\longrightarrow a_{n+1}:(a_1\longrightarrow a_{n})\rightarrow_\times a_{n+1}$, where $a_1\longrightarrow a_{2}$ is $a_1\rightarrow_\times a_{2}$ and $a_i\in A_i$. We see that each of these arrows $a_1\longrightarrow a_{n+1}$ is a finite sequence $\textbf{n+1}\rightarrow\bigsqcup_{i=1}^{n+1}A_i$ with $i\mapsto a_i$.

If we limit ourselves to $A_i=A$, for all $i\in\textbf{n+1}$, then we have a finite sequence on $A$. In making such considerations, we have a bijective function $\times_{i=1}^{n+1}A\rightarrow A^{n+1}$, where $A^{n+1}$ is the collection of all sequences $\textbf{n+1}\rightarrow A$. We define $A^\mathbb{N}=\times_{i\in\mathbb{N}}A$ as the collection of all sequences for $A$.

	\subsection{Net}

We will be working with functions that have a kind of special domain. This will lead to a useful generalization of sequence. Sequence is generalized because we generalize the concept of ordered index set; where we have used $\mathbb{N}$ we will now use a general kind of set called directed set.

		\subsubsection{Bounds}

Let $X$ be a set that forms a partial order with some relation $\leq$. A subset $A\subseteq X$ has a unique object $\max A\in A$ that is \textit{maximum} if $a\leq\max A$, for every $a\in A$. Now consider the set $\downarrow A$ of all $x\in X$ such that $a\leq x$, for every $a\in A$. This is the set of \textit{upper bounds of $A$} and it need not have a \textit{minimum}, defined dually to the maximum. Therefore, we cannot always say that there is an upper bound; if $\downarrow A\neq\emptyset$ we still cannot say their is a least upper bound $\min \downarrow A$; even if $\min \downarrow A$ does exist we cannot say that it is in $A$. When this most important object $\min\downarrow A$, exists in $X$, we say it is the \textit{supremum} and denote it by $\sup A$. The \textit{infimum}, denoted $\inf A$ is defined, when it exists, as $\max \uparrow A$. Of course, if the supremum is  in $A$, then it is the same as the maximum. 

If $\downarrow A=\emptyset$, then \textit{A is not bounded by above}. Similarly, we define sets that are \textit{not bounded by below}. If a set is not bounded by below and is not bounded by above it is simply not bounded.

		\subsubsection{Directed Set}

A set $I$ that has defined a preorder is said to be a directed set if for every $i,j\in I$ there is a $k\in I$ such that $i,j\leq k$. We give an equivalent definition for directed set, in terms of bounds.

\begin{proposition}A partial order is directed if and only if every finite subset has an upper bound.\end{proposition}

\begin{proposition}$\mathbb{Q}\supseteq\mathbb{Z}\supseteq\mathbb{N}$ are directed sets. In general, a natural order is a directed set.\end{proposition}
We already know of an important directed set that is not the set of integers or rationals. The category $\mathcal{P}\mathbb{N}$ is directed, under inclusion.

\begin{proposition}For any set $A$, the category $\mathcal{P}A$ gives a directed set. Let $x\uparrow x_0$ represent the set of objects $x\in X$ such that $x\leq x_0\in X$. Then $x\uparrow x_0$ is a directed set.\end{proposition}

Given preorders $\leq_1,\leq_2$ on sets $I,J$, we can form the new order $\leq$ by defining $(i_1,j_1)\leq(i_2,j_2)$ if and only if $i_1\leq i_2$ and $j_1\leq j_2$.

\begin{proposition}If $I,J$ are two directed sets with orders $\leq_1,\leq_2$, then the cartesian product $I\times J$ is also a directed set under the corresponding product order.\end{proposition}

		\subsubsection{Definition}

A \textit{net} is a function $\eta:I\rightarrow\mathcal{O}$, where $I$ is a directed set. The concept of net generalizes that of a sequence because we know $\mathbb{N}$ is a directed set. 

When considering a finite sequence, we now think in a broader sense. A finite sequence is a function $N\rightarrow \mathcal{O}$, where $N\subseteq\mathbb{N}$ is bijective to $\textbf{n}$, for some $n\in\mathbb{N}$. A general definition of a sequence, but less general than a net, is to replace the set $N$ with any directed subset $I\subseteq\mathbb{N}$. But, it turns out that any subset of $\mathbb{N}$ is directed. So, a sequence is a function $I\subseteq\mathbb{N}\rightarrow\mathcal{O}$. A net generalizes this definition. The first thing we will do to put to use the conpcept of net, is to give a definition of of matrix.

	\subsection{Matrix}

The reason why we did not consider the object $\frac{0}{0}$ in our grid of the rational numbers, (\ref{diagramRationals}), is that we did not have any way of considering this strange object in terms of our operation. So we leave it out, and in fact, we make changes to the normal behavior of the grid, along the border. This does not mean, however, that we are not able to define a good order in $\mathbb{Z}\times\mathbb{Z}$, if we just ignore the operation in this process. Again, representing $a\rightarrow_\times b$ with $\frac{a}{b}$, we say $\mathbb{Z}^2_{\leq}$ is the matrix order, defined by $\frac{a}{c}\leq\frac{b}{d}$ if and only if $a<b$ or, if this fails to be true, $c\leq d$. This order is a natural order.

\begin{definition}Let $I,J\subseteq\mathbb{Z}$, then we say the function $(a_i^j)_{i,j}:J\times I\subseteq\mathbb{Z}^2_\leq\rightarrow A$ is a matrix for the objects of $A$. If $i,j\in I,J$ respectively, then the image of $\frac{j}{i}$, under the matrix, is represented by $a^j_i$.

A matrix is said to be finite if $I,J$ are bijective to sets $\textbf{m},\textbf{n}\subseteq{N}$ respectively. The matrix is regular if it is of the special form $\textbf{m}\times\textbf{n}\rightarrow A$, that is, if $I=\textbf{m}$ and $J=\textbf{n}$. In any case, we say the matrix is of order $m\times n$.

If we restrict the matrix to one of the sets $\textbf{m}\times \{j\}$, then we have a matrix $(a_i^j)_i=(a_i^j)_{i,j}|_{\textbf{m}\times \{j\}}$, of the form $\textbf{m}\times \{j\}\rightarrow A$ and we say $(a_i^j)_i$ is the j-th column of the matrix $(a_i^j)_{i,j}$. A similar remark holds for the restriction of $(a_i^j)_{i,j}$ to a set $\{i\}\times\textbf{n}$; in this case we say $(a_i^j)_j$ is the i-th row of $(a_i^j)_{i,j}$.\end{definition}

\begin{proposition}\makebox[5pt][]{}\mbox {}
\begin{itemize}\item[1)]Every column or row is a matrix\item[2)]Every matrix is uniquely identified with a sequence of rows and with a sequence of columns.\item[3)]Any sequence can be  uniquely identified with one column matrix and one row matrix.\end{itemize}\end{proposition}

\section{Special Families}

	\subsection{Power Set}

		\subsubsection{Power Functor} 

We move on to study the power set $\textgoth{P}$. We will find a functor $\textgoth{P}:\textbf{Set}\rightarrow\textbf{Set}$ such that its object function is given by $B\mapsto\textgoth{P}B$. Since we are to give a functor, we need to specify an arrow function such that $f:B\rightarrow C$ is assigned a function $\textgoth{P}f:\textgoth{P}B\rightarrow\textgoth{P}C$. To this end, we define the arrow function such that $\textgoth{P}f$ is the function that makes $A\mapsto f[A]$. We verify that $\textgoth{P}$ is indeed a functor. We must first prove that $1_{\textbf{Set}},1_{\textbf{Set}};\textgoth{P},\textgoth{P}$, where $1_{\textbf{Set}}B$ is the identity function for $B$. That is, $(1_{\textbf{Set}}\circ\textgoth{P})B$ and $(\textgoth{P}\circ1_{\textbf{Set}})B$ are the same functions, for every $B\in\textbf{Set}$. We also observe that $\textgoth{P}(g\circ f)$ is the same function as $\textgoth{P}g\circ\textgoth{P}f$. Let $A\subseteq B$, then $\textgoth{P}(g\circ f)A=(g\circ f)[A]=g[f[A]]=g[\textgoth{P}fA]=\textgoth{P}g(\textgoth{P}fA)=(\textgoth{P}g\circ
\textgoth{P}f)A$.

\textbf{Properties and Relations} Here we find equalities that involve expressions of power sets. We first note that for any family of sets $\mathcal{X}=\{A\}_{A\in\mathcal{X}}$,

$$\mathcal{X}\subseteq\textgoth{P}\bigcup_{A\in\mathcal{X}}A.$$

We verify the validity of this. If $A\in\mathcal{X}$, then every object of $A$ is in $\bigcup\mathcal{X}$. Therefore, $A$ is an object in the power set of such union. 

In particular, we have \begin{equation}\{A\}\subseteq\textgoth{P}A.\label{eqnpwr1}\end{equation}  We know $A=\bigcup\{A\}\subseteq\bigcup\textgoth{P}A$. Also, $X\subseteq A$, for every $X\in\textgoth{P}A$, we may conlcude $A=\bigcup\textgoth{P}A$.  

The arrow in (\ref{eqnpwr1}) leads to the trivial relation $A\in\textgoth{P}A$. The relation $A\in\textgoth{P}A$ implies $\textgoth{P}A\in\textgoth{P}\textgoth{P}A$. We will give a generalization of this last relation.

\begin{proposition}Let $A$ be any set and $\mathcal{X}_{A}\subseteq\textgoth{P}A$ any family of subsets of $A$. Then $\mathcal{X}_{A}\in\textgoth{P}\textgoth{P}A$.
\end{proposition}

Consider a family $\mathcal{X}$, and an element $A\in\mathcal{X}$. We wish to prove $\textgoth{P}A\subseteq\textgoth P\bigcup\mathcal{X}$. This is clear because $\textgoth{P}A$ is a family of subsets of $A\subseteq\bigcup\mathcal{X}$. Applying the last proposition, 

\begin{proposition}For $A\in\mathcal{X}\in\mathcal{U}$, we verify $\textgoth{P}A\in\textgoth{P}\textgoth{P}\bigcup\mathcal{X}$.\end{proposition}

\begin{proposition}The object function of the power functor, satisfies $$\textgoth{P}(A\cap B)=\textgoth{P}A\cap\textgoth{P}B$$
$$\textgoth{P}A\cup\textgoth{P}B
\subseteq\textgoth{P}(A\cup B).$$
\end{proposition}

		\subsubsection{Direct Image}

Now we define the \textit{direct image} of a family of sets as the collection $f[\mathcal{X}]$ of all $B\subseteq Range~f$ such that $f^{-1}B\in\mathcal{X}$. Of course, the \textit{direct inverse image} of a family $\mathcal{Y}$, is the collection $f^{-1}[\mathcal{Y}]$ of sets $A\subseteq Dom~f$ such that $fA\in\mathcal{Y}$. Let $f^\rightarrow A$ be the collection of all subsets of the range, whose inverse image is $A$, and $f^\leftarrow B$ be the collection of all subsets of the domain, whose image is $B$.

Let $\textgoth{P}f^{-1}\{B\}$ be the fiber of $B\subseteq Range~f$, for the function $\textgoth{P}f$. Then the fiber has a maximum $M=\bigcup\textgoth{P}f^{-1}\{B\}$. All we need to verify is that $B=\textgoth{P}fM$. This is straightforward, $fM=f[\bigcup\textgoth{P}f^{-1}\{B\}]=\bigcup_{X\in\textgoth{P}f^{-1}\{B\}}fX=
\bigcup_{X\in\textgoth{P}f^{-1}\{B\}}B=B$. 

\begin{lemma3.1}Let $B\in Im~\textgoth{P}f$, then $f^{-1}B
=\bigcup\textgoth{P}f^{-1}\{B\}$.\end{lemma3.1}

\begin{proof}We know $M\subseteq f^{-1}fM=f^{-1}B$. Also, $f^{-1}B\in\textgoth{P}f^{-1}\{B\}$ because $\textgoth{P}f(f^{-1}B)=B$. This implies $f^{-1}B\subseteq\bigcup\textgoth{P}f^{-1}\{B\}$.\end{proof}

\begin{lemma3.2}\makebox[5pt][]{}\mbox {}\begin{itemize}\item[1)]Let $f$ be an onto function, then $B\in f[\mathcal{X}]$ if and only if $\bigcup\textgoth{P}f^{-1}\{B\}\in\mathcal{X}$\item[2)]Let $f$ be a monic function, $A\in f^{-1}[\mathcal{Y}]$ if and only if $A=\bigcup f^{-1}\{\{B\}\}$, for some $B\in\mathcal{Y}$.\end{itemize}\label{lemma3.2}\end{lemma3.2}

\begin{proof}By definition of direct image and lemma I, $B\in f[\mathcal{X}]$ if and only if $M=f^{-1}B\in\mathcal{X}$.

On the other hand, $A\in f^{-1}[\mathcal{Y}]$ if and only if $fA\in\mathcal{Y}$. From proposition $\ref{FIB}$, we get $A=\bigcup f^{-1}\{\{B\}\}$, for $B=fA\in\mathcal{Y}$.\end{proof}

Define $f\emptyset:=Range~f-Im~f$, and let $f(A\cup\emptyset)$ be the collection of all $Y$ such that $fA\subseteq Y\subseteq fA\cup f\emptyset$.

\begin{lemma3.3}\makebox[5pt][]{}\mbox {}\begin{itemize}\item[1)]If $f$ is onto, then $f^\leftarrow B=\textgoth{P}f^{-1}\{B\}$\item[2)]$f^\rightarrow A\subseteq f(A\cup\emptyset)$ and equality holds given $f$ is monic.\end{itemize}\end{lemma3.3}

\begin{proof}The first result is a direct consequence of the definitions of $f^\leftarrow B$ and $\textgoth{P}f$.

To prove the second result we first take $B$ such that $A=f^{-1}B$. We get $fA=ff^{-1}B\subseteq B$. Suppose that $(fx)\in B$ and $(fx)\notin fA\cup f\emptyset$. This means $(fx)\in (fA\cup f\emptyset)^c=(fA)^c\cap Im~f=Im~f-fA$. But, we have $A=f^{-1}B$ which is equivalent to saying $fx\in B$ if and only if $x\in A$. This is a clear contradiction, therefore $B\subseteq fA\cup f\emptyset$. We conlcude $B\in f(A\cup\emptyset)$.

If we wanted to prove $f(A\cup\emptyset)\subseteq f^\rightarrow A$ by taking $Y$ that satisfies the conditions of being an object of $f(A\cup\emptyset)$, then $A\subseteq f^{-1}fA\subseteq f^{-1}Y\subseteq f^{-1}(fA\cup f\emptyset)=f^{-1}fA$. We would get $x\in f^{-1}Y~\Rightarrow~fx\in fA$, but $fx\in fA$ does not imply $x\in A$. We request the function be monic.\end{proof}

\begin{theorem}Suppose $\mathcal{X}$ is a family of non-empty subsets of $Dom~f$ and $\mathcal{Y}$ is a family of non-empty subsets of $Range~f$. If $B\subseteq Im f$, then $$B\in f[\mathcal{X}]~\Longleftrightarrow~\bigcup f^\leftarrow B\in\mathcal{X}.$$

If $f$ is monic, we have $$A\in f^{-1}[\mathcal{Y}]~\Longleftrightarrow~\bigcap f^\rightarrow A\in\mathcal{Y}.$$\end{theorem}

\begin{proof}The first result follows from lemmas II,III.

We know that $A\in f^{-1}[\mathcal{Y}]$ if and only if $B=fA=\bigcap f(A\cup\emptyset)=\bigcap f^{\rightarrow}A$, for some $B\in\mathcal{Y}$.\end{proof}

We see that there is a duality relationship for fiber and image in the forms of $f^\leftarrow B$ and $f^\rightarrow A$.

		\subsubsection{Another Functor}

We form yet another functor, using the power set $\textgoth{P}$ as object function of this functor $\textbf{Set}_{\subseteq}\rightarrow\textbf{Set}_{\subseteq}$. This means the second condition is satisfied because the functor applied to any arrow $A\subseteq B$ results in $\textgoth{P}A\subseteq\textgoth{P}B$. Since the first condition is trivial, we are left to give a proof of condition 3) for functors. It means that transitivity is preserved. This is verified by

\begin{eqnarray}\nonumber\textgoth{P}[(A\subseteq B)\circ(B\subseteq C)]&=&\textgoth{P}(A\subseteq C)\\\nonumber&=&\textgoth{P}A\subseteq\textgoth{P}C\\\nonumber&=&(\textgoth{P}A
\subseteq\textgoth{P}B)\circ(\textgoth{P}B\subseteq\textgoth{P}C)
\\\nonumber&=&\textgoth{P}(A\subseteq B)\circ\textgoth{P}(B\subseteq C).\end{eqnarray}

	\subsection{Family of Families}

We will now give two basic results in applying the operations of sets to families of families of sets. We will be using $\mathbb{X}=\{\mathcal{X}\}_{\mathcal{X}\in\mathbb{X}}$ to represent the family that consists of families, such that $\mathcal{X}=\{A\}_{A\in\mathcal{X}}$. This is to say, we have a two generation family of sets.

\begin{equation}\bigcup\bigcap\mathbb{X}\subseteq\bigcup\bigcup\mathbb{X}\label{ex1st}
\end{equation}

\begin{equation}\bigcap\bigcup\mathbb{X}\subseteq\bigcap\bigcap \mathbb{X}\label{ex2st}.\end{equation}

Both of these follow from the fact that $\bigcap\mathbb{X}\subseteq\bigcup\mathbb{X}$.

	\subsection{Nest}

We have pointed out that the power set is a partial order under inclusion. The objects of study in the present section are those families of sets that form a natural order, under inclusion. In particular, there are two cases of main interest. We will use this context to extend results (\ref{decoset1}) and (\ref{decoset2}).

		\subsubsection{Increasing Nest}

The first case happens when we have an infinite chain going forward. This means that there is a sequence of sets $\{A_i\}_{i\in\mathbb{N}}$ such that $A_n\subseteq A_{n+1}$. When such a family is encountered, one usually needs to find the union; the intersection is clearly $A_{1}$. The union will be expressed in terms of a family of disjoint sets. Consider the family $B_{i}=A_{i}-A_{i-1}$. We know $\bigcup_{i}B_i\subseteq\bigcup_iA_i$ because each $B_i\subseteq A_i$. Also, if $x\in\bigcup_iA_i$, then $x\in B_m$, where $m$ is the smallest integer such that $x\in A_{m}$. We conclude

\begin{eqnarray}\bigcup_iA_i&=&\bigcup_iB_i\\\nonumber
\\\nonumber\emptyset&=&\bigcap_iB_i.\end{eqnarray}

\subsubsection{Decreasing Nest}

Now we analyze a sequence of sets such that $A_{n+1}\subseteq A_n$. We are able to express \begin{equation}A_{1}=\bigcup_i(A_{i}-A_{i+1})\end{equation}

To prove this, we begin with $A_{1}=(A_{1}\cap A_{2})\cup(A_{1}-A_{2})= A_{2}\cup(A_{1}-A_{2})$. Then, since $A_{2}=A_{3}\cup(A_{2}-A_{3})$ we can say $A_{1}=(A_{1}-A_{2})\cup(A_{2}-A_{3})\cup A_{3}$. We continue in this manner and see that this is the union of a disjoint family.

Often, when using nests, one needs to find the intersection. It is expressed by

\begin{eqnarray}\nonumber\bigcap_iA_{i}&=&\left(\bigcup_iA_i^c\right)^c\\&=&
\left(\bigcup_i(A_{i}^c-A_{i-1}^c)\right)^c\end{eqnarray}

This is obtained by using the sequence of $A_i^c$ as a growing nest.

	\subsection{$\sigma$-Algebra}

A set of subsets of some $A\in\mathcal{U}$ can be such that it is closed for the operations of union, intersection and complement. Formally, a $\sigma-algebra$, $\mathcal{A}$, is an object of $\textgoth{P}\textgoth{P}A$ such that 1) $A\in\mathcal{A}$, 2) If $X\in\mathcal{A}$, then $X^c\in\mathcal{A}$, and 3) For any sequence $\mathcal{B}\subseteq\mathcal{A}$ we verify $\bigcup\mathcal{B}\in\mathcal{A}$. This is similar to having a group, where the objects of operation are the objects of the $\sigma-$algebra. Remember, however, that the complement is not the inverse in the sense defined for categories.

For the same subset $\mathcal{B}$, we have $\bigcap\mathcal{B}\in\mathcal{A}$. The proof is $\bigcap\mathcal{B}=\bigcup\mathcal{B}^c$, since the objects of $\mathcal{B}^c$ are also in $\mathcal{A}$.

		\subsubsection{Generated $\sigma-$algebra}

Let $\mathcal{B}\in\textgoth{P}\textgoth{P}A$, then the $\sigma-$algebra generated by $\mathcal{B}$ is represented by $\sigma(\mathcal{B})$. We define it as the smallest $\sigma-$algebra that contains all the objects of $\mathcal{B}$. In other words, if $\mathcal{B}\subseteq\mathcal{A}$, where $\mathcal{A}$ is a $\sigma-$algebra, then $\mathcal{B}\subseteq\sigma(\mathcal{B})\subseteq\mathcal{A}$.

Let $\Sigma(\mathcal{B})$ be the collection of all $\sigma-$algebras that contain all the objects of $\mathcal{B}$.

\begin{proposition}The $\sigma-$algebra generated by a collection $\mathcal{B}$, of subsets of $A$, is expressed by

\begin{eqnarray}\nonumber\sigma(\mathcal{B})&=&\bigcap\Sigma(\mathcal{B})
\end{eqnarray}\end{proposition}

\begin{proof}We know $\sigma(\mathcal{B})\subseteq\bigcap\Sigma(\mathcal{B})$ because $\sigma(\mathcal{B})\subseteq\mathcal{A}$, for every $\mathcal{A}\in\Sigma(\mathcal{B})$. Also, $\bigcap\Sigma(\mathcal{B})\subseteq\sigma(\mathcal{B})$ because $\sigma(\mathcal{B})\in\Sigma(\mathcal{B})$.\end{proof}

		\subsection{Set Filter}

A \textit{filter} $\mathcal{F}$ is a non-empty family of non-empty subsets (there is at least one set in the filter and each set has at least one object) of $X$ such that for every $F,G\in\mathcal{F}$ and any $H\supseteq F$:

\begin{itemize}\item[A)]$F\cap G\in\mathcal{F}$\item[B)] $H\in\mathcal{F}$\end{itemize}

A \textit{filter base} is a non-empty family of non-empty subsets of $X$ such that $A)$ holds. We say a) holds if for every $F,G\in\mathcal{F}$ there exists an $H\in\mathcal{F}$ such that

\begin{itemize}\item[a)]$H\subseteq F\cap G$.\end{itemize}

Condition A) can be taken to be the \textit{finite intersection property} because for any finite family of sets $\mathcal X\subseteq\mathcal F$ we have $\bigcap\mathcal X\in\mathcal F$. A base filter can be defined equivalently, if we ask for the condition a) instead of condition A). So, a filter may also be defined as a family that verifies conditions a) and B). 

		\subsubsection{Filter Generated by a Base}

We will now give a construction that sheds light on the terms selected. We will at times be justified in saying that a certain filter base, is \textit{base} of some filter. Let $\mathcal{B}\in\textbf{FB}X$ be an element in the set of all filter bases for $X$, and say $\langle\mathcal{B}\rangle\in\textgoth{P}\textgoth{P} X$ consists of those sets $A$ such that $F\subseteq A$, for some $F\in\mathcal{B}$. We note that $\mathcal{B}\subseteq\langle\mathcal{B}\rangle$.

\begin{proposition}The family $\langle\mathcal{B}\rangle$ is a filter and we will say the filter base is $\mathcal{B}$.\end{proposition}

\begin{proof}We only need to prove that condition B) in the definition of filter, holds. Take $B\supseteq A\in\langle\mathcal{B}\rangle$. Then, there is $F\in\mathcal{B}$ such that $F\subseteq A\subseteq B$. This means $B\in\langle\mathcal{B}\rangle$.\end{proof}

Let $\textbf{FB}X$ and $\textbf{F}X$ represent the collections of filter bases and filters for $X$. Then, the last proposition is simply stating that there is a function from $\textbf{FB}X$ onto $\textbf{F}X$, because every filter is also a filter base and the filter generated by any filter is itself. Two base filters that generate the same filter are called equivalent bases. For any $X\in\mathcal{U}$, we have $\textbf{Filt}_X:(\textbf{FB}X)\rightarrow(\textbf{F}X)$, and $\textbf{Filt}_X\mathcal{B}=\langle\mathcal{B}\rangle$. All we have said with respect to this is that for every $\mathcal{F}\in\textbf{F}X$ we also have $\mathcal{F}\in\textbf{FB}X$ and $\textbf{Filt}_X\mathcal{F}=\mathcal{F}$.

\begin{proposition}For any $\mathcal{B}\in\textbf{FB}X$, we have $\langle\mathcal{B}\rangle=\mathcal{B}$ if and only if $\mathcal{B}\in\textbf{F}X$.\end{proposition}

What is more, for any filter base $\mathcal{B}$ we find that the filter $\textbf{Filt}_X\mathcal{B}$ is the smallest filter that contains the filter base.

\begin{proposition}Let $\mathcal{B}\in\textbf{FB}X$ and $\mathcal{F}\in\textbf{F}X$. If $\mathcal{B}\subseteq\mathcal{F}$, then $\mathcal{B}\subseteq\langle\mathcal{B}\rangle\subseteq\mathcal{F}$.\end{proposition}

		\subsubsection{Principal Filter and Filter Generated by a Point}

It is clear that a family $\{F\}$ consisting of one non-empty set is a filter base. The filter generated from such a filter base is called a \textit{principal filter}. Of course, if we consider the inclusion order, we have $\langle\{F\}\rangle=\downarrow F$.

\begin{lemma4.1}Let $F\in\mathcal{F}$, then $\textbf{Filt}_X\{F\}\subseteq\mathcal{F}$.\end{lemma4.1}

From this, we get the following result, where we consider the filter as an index set by noting the function $\mathcal{F}\rightarrow\textgoth{P}\textgoth{P}X$ that makes $F\mapsto \langle F\rangle=\textbf{Filt}_X{\{F\}}$.

\begin{theorem}Any filter is the union of principal filters. In particular,$$\mathcal{F}=\bigcup_{F\in\mathcal{F}}\langle F\rangle.$$\end{theorem}

In case we have $F=\{x\}$, for some $x\in X$, then we write $\langle x\rangle$ instead of the strict notation $\textbf{Filt}_X\{\{x\}\}=\langle\{\{x\}\}\rangle$. These are called \textit{point generated filters}.

\subsubsection{Cofinite Filter} Consider the family of all finite subsets, of $X$, and denote it by $\mathcal{F}^c$. Then, the family $\mathcal{F}=\{F\}_{F^c\in\mathcal{F}^c}$, is called the \textit{cofinite filter} of $X$. We are using $F$ to represent the complement of $F^c$.

\begin{proposition}The cofinite filter of $X$ is a filter of $X$.\end{proposition}

\begin{proof}First, we see that $F^c\cap G^c=(F\cup G)^c$, where $F^c,G^c\in\mathcal{F}^c$. We know $F\cup G$ is finite because $A,B$ are finite. This proves that $F^c\cap G^c\in\mathcal{F}$.

Now, let $F\subseteq H$. This means $H^c\subseteq F^c$, thus proving $H^c$ is finite.\end{proof}

\subsubsection{Fr\'echet Filter Base and Filter} Let $\mathcal I$ be a directed set and consider, for every $i\in\mathcal I$, the set of objects $\textbf{i}$ that consists of all $x\in\mathcal I$ such that $x\leq i$. Let $i,j\in\mathcal I$, then we can find $k\in\mathcal I$ such that $i,j\leq k$. This means that $\textbf{k}^c$ is a subset of $\textbf{i}^c$ and $\textbf{j}^c$.  This proves that $\{\textbf{i}^c\}_{i\in\mathcal{I}}$ is the \textit{Fr\'echet filter base} of $\mathcal I$.

If $\{\textbf{i}\}_{i\in\mathcal{I}}$ is a family of finite sets, then the cofinite filter of $\mathcal{I}$ is called the \textit{Fr\'echet filter} of $\mathcal{I}$ and we will write $Fr(\mathcal{I})=\{\textbf i^c\}_{i\in\mathcal I}$.

\subsubsection{Filter as a Directed Set} A filter can easily be seen as a directed set, if we consider the partial order $\mathcal{P}X$, with arrows reversed. This new partial order can of course be written as $\mathcal{P}^{-1}X$. That is why we may say that a filter is \textit{downward directed}.

We now give a proof for our assertion. Let $\mathcal{F}$ be our filter and define an order for this family, where $F\leq G$ if $G\subseteq F$. We know we have a partial ordering and all that needs to be shown is that given $F,G\in\mathcal{F}$, we have $H\in\mathcal{F}$ such that $F,G\leq H$. This follows from condition a) for filters.

\subsubsection{Image} The image of a filter base behaves in a good manner. By this we mean that functions send bases into bases. However, filters are not sent into filters. They are sent into base filters, but this is not a problem because we already have a construction that sends base filters into filters!

\begin{proposition}Let $f:X\rightarrow Y$ be a set function. Also, let  $\mathcal{B}\in\textbf{FB}X$ and $\mathcal{F}\in\textbf{F}X$. Then $f[[\mathcal{B}]],f[[\mathcal{F}]]\in\textbf{FB}Y$.\end{proposition}

\begin{proof}Take $(fF),(fG)\in f[[\mathcal{B}]]$, then there exist $F,G\in\mathcal{B}$ for which $fF=(fF)$ and $fG=(fG)$. Therefore, $(fF)\cap(fG)=fF\cap fG\subseteq f(F\cap G)\in f[[\mathcal{B}]]$. In particular, this conclusion also holds for filters in place of bases.\end{proof}

We will say that the \textit{generated image} $\langle f\mathcal{F}\rangle$, of the filter $\mathcal{F}$, is the filter generated by the corresponding image, which is a filter base; this is $\langle f\mathcal{F}\rangle=\textbf{Filt}_Y{f[[\mathcal{F}]]}$.

The preimage of a filter base behaves well, under certain conditions.

\begin{proposition}Let $f:X\rightarrow Y$ be a set function and $\mathcal{B}\in\textbf{FB}Y$. Then $f^{-1}[[\mathcal{B}]]\in\textbf{FB}X$ if and only if $F\cap fX\neq\emptyset$, for every $F\in\mathcal{B}$.\end{proposition}

\begin{proof}If we suppose $f^{-1}[[\mathcal{B}]]$ is a filter base of $X$, then for every $F\in\mathcal{B}$ we have $f^{-1}F\neq\emptyset$. This implies that there exists an $x\in X$ such that $fx\in F$.

If $(f^{-1}F),(f^{-1}G)\in f^{-1}[[\mathcal{B}]]$, then there exist $F,G\in\mathcal{B}$ such that $(f^{-1}F)\cap(f^{-1}G)=f^{-1}F\cap f^{-1}G=f^{-1}(F\cap G)\in f^{-1}[[\mathcal{B}]]$.\end{proof}

The \textit{generated inverse image} of $\mathcal{F}\in\textbf{F}Y$ is defined as $\langle f^{-1}\mathcal{F}\rangle=\textbf{Filt}_Xf^{-1}[[\mathcal{F}]]$; given, of course, $F\cap fX\neq\emptyset$, for every $F\in\mathcal{F}$.

\subsubsection{Elementary Filter of a Net} Given a net $\eta:\mathcal I\rightarrow X$, we can give a filter associated to it, $\eta[\textbf{Filt}_X\{\textbf{i}^c\}_{i\in\mathcal{I}}]$. This is nothing more than the direct image, under $\eta$, of the filter generated by $\{\textbf{i}^c\}_{i\in\mathcal{I}}$. In case we have a Fr\'echet filter, we will of course write $\eta[Fr(\mathcal{I})]$. In any case, the filter given is called the \textit{elementary filter of $\eta$}.

\subsubsection{Ultrafilter}

We have seen how a filter is a partial order, that is filters have an internal order. Well it turns out that the concept of filter gives rise to another type of order, an external order. This has already been manifest in that we have enconutered a minimal filter $\langle\mathcal{B}\rangle$ that contains $\mathcal B$. In this paragraph we will dedicate ourselves to the study of a special kind of maximal filters. In the next section we will study the order defined on $\textbf{F}X$, using a more general definition of filter.

\begin{definition}If $\mathcal B,\mathcal C\in\textbf{FB}X$ are such that for every $B\in\mathcal B$ there is $C\in\mathcal C$ such that $C\subseteq B$ then we say $\mathcal C$ is finer than $\mathcal B$, or that $\mathcal C$ is a refinement of $\mathcal B$. This is represented by $\mathcal B\preceq\mathcal C$.

A filter $\mathcal{F}$ is an \textit{ultrafilter} if for every $A\subseteq X$ we have $A\in\mathcal{F}$ or $A^c\in\mathcal{F}$. The collection of ultrafilters on a set $X$ is $\textbf{UF}X$.\end{definition}

\begin{proposition}The relation of refinement $\preceq$ is a preorder for $\textbf{FB}X$.\end{proposition}

\begin{proof}This follows from reflexivity and transitivty for set inclusion.\end{proof}

\begin{proposition}Let $\mathcal F,\mathcal G\in\textbf{F}X$. Then $\mathcal F\subseteq\mathcal G$ if and only if $\mathcal F\preceq\mathcal G$.\end{proposition}

\begin{proof}Take $F\in\mathcal F\subseteq\mathcal G$, then for $F\in\mathcal G$ we have $F\subseteq F$.

Now suppose $\mathcal G$ is finer than $\mathcal F$ and let $F\in\mathcal F$. We have $G\subseteq F$ for some $G\in\mathcal G$. Since $\mathcal G$ is a filter we have $F\in\mathcal G$.\end{proof}

\begin{corollary}The relation $\preceq$ for filters is a partial order.
\end{corollary}

We do not generally have anti-symmetry for the refinement relation in $\textbf{FB}X$ but we do have the following important result.

\begin{proposition}Let $\mathcal B,\mathcal C\in\textbf{FB}X$. Then $\mathcal B\preceq\mathcal C$ and $\mathcal C\preceq\mathcal B$ if and only if $\langle\mathcal B\rangle=\langle\mathcal C\rangle$.\end{proposition}

\begin{proof}Let $F\in\langle\mathcal B\rangle$. Then $\mathcal B\preceq\mathcal C$ is true if and only if there exists $B\in\mathcal B$, and consequently $C\in\mathcal C$, such that $C\subseteq B\subseteq F$. Therefore $F\in\langle\mathcal C\rangle$ and we may conlcude $\langle\mathcal B\rangle\subseteq\langle\mathcal C\rangle$.\end{proof}

\begin{theorem}A filter $\mathcal F$ is ultrafilter if and only if for every $\mathcal G\in\textbf{F}X$ such that $\mathcal F\leq\mathcal G$, then $\mathcal G=\mathcal F$.\end{theorem}

\begin{proof}Suppose $\mathcal F$ is an ultrafilter. Since $\preceq$ is a partial order relation, we only need to prove that $\mathcal G\preceq\mathcal F$. As we have just seen, this is equivalent to proving $\mathcal G\subseteq\mathcal F$. Let $G\in\mathcal G$, then $G$ or $G^c$ are in $\mathcal F$; we cannot have both in $\mathcal F$ because that would imply $\emptyset=G\cap G^c\in\mathcal F$ which is a contradiction to the definition of filter. So, we suppose $G^c\in\mathcal F$. Since $\mathcal G$ is finer than $\mathcal F$ we have $H\in\mathcal G$ such that $H\subseteq G^c$. This leads to a contradiction because $\emptyset=G\cap H\in\mathcal G$.

Now, we would like to show that $\mathcal F$ is an utrafilter, given the second condition. Let $\mathcal G$ be a filter that contains $\mathcal F$ in the strict sense, $\mathcal F\prec\mathcal G$. That is, there exists $G\in\mathcal G$ such that $G\notin \mathcal F$. If it were the case $G^c\in\mathcal F$ we would have $H\in\mathcal G$, such that $H\subseteq G^c$, because $\mathcal G$ is finer than $\mathcal F$, but this would imply $\emptyset=G\cap H\in\mathcal G$. Therefore, no such $\mathcal G$ exists and $\mathcal F$ is maximal.\end{proof}

We see that the subcollection of ultrafilters absorbs under union of sets. Let us make a general and fomal statement, for this last observation.

\begin{theorem}Let $\mathcal F,\mathcal G$ be two filters on a set $X$ such that $\mathcal F\cup\mathcal G\in\textbf{UF}X$, then $\mathcal F\in\textbf{UF}X$ or $\mathcal G\in\textbf{UF}X$. Also, if $\mathcal F\in\textbf{UF}X$ and $\mathcal G\in\textbf{F}X$, then $\mathcal F\cup\mathcal G\in\textbf{UF}X$ given $\mathcal F\cup\mathcal G\in\textbf{F}X$.\end{theorem}

\begin{proof}Suppose $\mathcal F\cup\mathcal G$ is an ultrafilter and $\mathcal G$ is not an ultrafilter. So we take $F\subseteq X$ such that $F,F^c$ are not both $\mathcal G$; we have $F\in\mathcal F\cup\mathcal G$ or $F^c\in\mathcal F\cup\mathcal G$. From this it follows that $F\in\mathcal F$ or $F^c\in\mathcal F$.

Now let $\mathcal F$ be an ultrafilter and $\mathcal G$ a filter. Take $H\subseteq X$ and suppose $H\notin\mathcal F\cup\mathcal G$, then $H^c\in\mathcal F\subseteq\mathcal F\cup\mathcal G$. This means $\mathcal F\cup\mathcal G$ is an ultrafilter.\end{proof}

We have a similar result on the internal structure of filters.

\begin{proposition}$\mathcal F\in\textbf{UF}X$ if and only if $F\cup G\in\mathcal F$ implies $F\in\mathcal F$ or $G\in\mathcal F$.\end{proposition}

\begin{proof}Let $\mathcal F$ be an ultrafilter and let $F,G\subseteq X$ such that $F\cup G\in\mathcal F$. We know $F$ or $F^c$ is in $\mathcal F$, the same is true of $G,G^c$. Of course we cannot have $F^c,G^c$ both in $\mathcal F$.

Now suppose $F\cup G\in\mathcal F\Rightarrow$ $F\in\mathcal F$ or $G\in\mathcal F$ for any $F,G\subseteq X$ and any $\mathcal F\in\textbf{F}X$. Consider the special case $F\cup F^c=X\subseteq X$ which means $F\in\mathcal F$ or $F^c\in\mathcal F$. We conclude $\mathcal F$ is an ultrafilter.\end{proof}

We give the easiest example of an ultrafilter in the following result.

\begin{proposition}If $x\in X$ then $\langle x\rangle=\downarrow\mathcal\{x\}\in\textbf{UF}X$ and we call it a principal ultrafilter.\end{proposition}

This is to say that a point generated filter is an ultrafilter.

\begin{proposition}If $\mathcal F$ is an ultrafilter on a finite set $X$, then $\mathcal F$ is a principal ultrafilter.\end{proposition}

	\section{Zorn's Lemma and the Axiom of Choice}

Here, we will formulate the axiom in terms of partial order. Essentially, we will suppose that given a partial order, we are able to take away objects and relations so as to build a new system that is a natural order. Of course we can take away so few objects and relations so as to be left with an order that is not necessarily natural. The important aspect of the way in which we will take away objects and relations is that we can take away many enough to get a natural order but just enough so that if we take away any less, then we don't get a natural order.

\begin{AOC}Given a non-trivial partial order $\mathcal P$ and a natural order $\mathcal N\subseteq\mathcal P$, there exists at least one maximal natural order $\mathcal M$ such that $\mathcal N\subseteq \mathcal M\subseteq\mathcal P$.\end{AOC}

Maximality of $\mathcal M$ means that if $\mathcal Q\supseteq\mathcal M$ is a natural order then $\mathcal Q=\mathcal M$. If we have a maximal natural order then any upper bound of the natural order is the maximum. Having a natural order that is maximal means we are not able to add any objects to it and conserve the status of natural order. This means that any object we add will not be comparable with the objects of the order. Therefore, if $M$ is an upper bound we must have $M$ in the natural order and we conclude it is the maximum.

\begin{Zlemma}Suppose for every natural order $\mathcal N\subseteq\mathcal P$ we have $\downarrow\mathcal N\neq\emptyset$. Then there exists a maximal element in $\mathcal P$.\end{Zlemma}

\begin{proof}If the partial order is trivial then we have maximal elements. Suppose it is not trivial, then there is a natural order consisting of the comparabale objects $x,y$. We know that there is a maximal natural order $\mathcal M\supseteq\{x,y\}$ and $\mathcal M$ has an upper bound, call it $M$. This object $M$ is the maximum of $\mathcal M$ and it is maximal in $\mathcal P$ because $\mathcal M$ is maximal.\end{proof}

	\chapter{Universal Concepts}

The topics we have covered in sequences and filters are of relation to a concept known as limit. The notion of limit has been found to be most general in the setting of categories, and this is a case of \textit{universal properties}. We will see how instances of this occur when defining the functor of common domain, and a bifunctor. We will discuss here, how filter limits arise from the general definition. We will build the real number system, with the tools provided. Then, we will retake this subject in a later chapter for topologies.

Although the general concept of comma category was not trivial to develop, it has arisen aleady in certain occasions; we call them arrow categories. The instances in which it has appeared have been very particular cases. Let $\textgoth f,\textgoth g:\mathcal C_1,\mathcal C_2\rightarrow\mathcal D$ be two functors with common range. An \textit{arrow category} $\mathcal D(\textgoth f,\textgoth g)$ is a category such that the c-objects are all the arrows $\textgoth fa\rightarrow\textgoth gb$, for any $(a,b)\in\mathcal{O|C}_1\times\mathcal{O|C}_2$. The arrows of $\mathcal D(\textgoth f,\textgoth g)$ are $\tau\rightarrow\sigma:\textgoth fa\rightarrow\textgoth gb\longrightarrow\textgoth fc\rightarrow\textgoth gd$, such that $\sigma,\tau;\textgoth gi,\textgoth fh$ for $h,i:a,b\rightarrow c,d$.

	\chapter{Lattice}

In this chapter we will study certain partial orders that satisfy certain properties in regards to the their bounds or supremum/infimum.

Lattices are a special type of order and one of the main reasons why it is special is that it can be viewed as an algebraic structure. That is, the supremum concept of the order is described in terms of an operation. Here, we will denote partial orders with $\mathcal L$, or the like.

In a sense, we may think that the main idea behind this chapter is that the procedure of saying \textit{a is greater than b, c is greater than a,...} can be viewed as an operation.

	\section{Supremum and Infimum}

We recall that in the study of bounds we define a special bound, whether it be for upper/lower bounds. These bounds are the supremum and infimum. Here we will write the supremum of $\{x,y\}$ as $x\vee y$, and the infimum as $x\wedge y$. More generally, $\bigvee A=\min\downarrow A$ and $\bigwedge A=\max\uparrow A$. As we know, a set can be written as $A=\{x\}_{x\in A}$ and in such cases we may write $\bigvee_{x\in A}x=\min\downarrow A$ and $\bigwedge_{x\in A}x=\max\uparrow A$.

Before moving on, we take on the problem of defining the supremum/infimum of empty set/universe. Let us consider the empty set, first. The supremum is, by definition, the least among upperbounds of $\emptyset$. We take the posture that, by default, any object is an upper bound of $\emptyset$, so that $\bigvee\emptyset=\min\mathcal L=m$. We can just as easily say $\bigwedge\emptyset=\max\mathcal L=M$. Now, consider the set of upper bounds of $\mathcal L$; it is obviously the maximum of $\mathcal L$. Therefore, $\bigvee\mathcal L=M$, while $\bigwedge\mathcal L=m$.

 We provide some properties for supremum and infimum. The proof to the first three is trivial and always follows directly from the definition. Result 4), in the next proposition, requires little knowledge from sets, to reason.

\begin{proposition}Let $\mathcal L$ be a partial order and $A\subseteq B$ be subsets of $\mathcal L$ such that their supremum and infimum exist.

\begin{itemize}

	\item[1)]For every $a\in A$ we have $\bigwedge A\leq a\leq\bigvee A$

	\item[2)]For every $x\in\mathcal L$, we have $x\leq\bigwedge A$ if and only if $x\leq a$, for every $a\in A$.

	\item[3)]For every $x\in\mathcal L$, we have $\bigvee A\leq x$ if and only if $a\leq x$, for every $a\in A$.

	\item[4)]$\bigvee A\leq\bigvee B$ and $\bigwedge B\leq\bigwedge A$.

\end{itemize}\label{prop sup 1}\end{proposition}

		\section{Lattice}

Consider a partial order with a full operation $\bigvee:\mathcal L\rightarrow\mathcal L f\mathcal L$ that sends $x\mapsto_{\vee y}x\vee y=\min\downarrow\{x,y\}$. All this means is that the supremum exists for every pair of objects in the order and we make a commutative operation from this. If we are also able to build $\bigwedge:\mathcal L\rightarrow\mathcal Lf\mathcal L$, for the infimum, we say that $\mathcal L$ is a \textit{lattice}. Let us be more specific. Take a lattice $\mathcal L$ and $\mathcal H\subseteq\mathcal L$, where $x\vee y\in\mathcal H$, for every $x,y\in\mathcal H$. Then $\mathcal H$ is called an \textit{upper sublattice}. If, instead, the infimum exists for every pair, we have a \textit{lower sublattice}. In the case $\mathcal H$ is both upper and lower sublattice of $\mathcal L$, we simply say $\mathcal H$ \textit{is sublattice of} $\mathcal L$.

Consider a bounded partial order; we will seek properties of unit. For the supremum, we have as unit, the minimum. That is, $x\vee m=x$. In case of the infimum operation, the operation has  $M$ for unit, this is, $x\wedge M=x$. Also, for the supremum, we have $x\vee M=M$, and for the infimum operation, $x\wedge m=m$. The reader should be able to relate this with the situation encountered in sets.

Recall that the left and right operations of a given object are functions. The following result leaves it clear that a given object may be invariant under more than one right/left operation.

\begin{proposition}Let $\mathcal L$ be a lattice, then $x\leq y$ if and only if $x\vee y=y$ and $x\wedge y=x$.\label{sup td}\end{proposition}

\begin{proof}If $x\leq y$, then $y$ is a the maximum of $\{x,y\}$ and $x$ is the minimum. We conclude $x\vee y=y$ and $x\wedge y=x$. On the contrary, if $x\vee y =y$ we know that $y$ is an upper bound of $\{x,y\}$ which gives $x\leq y$.\end{proof}

\begin{proposition}The functions $\vee x$, $\wedge x$, of a lattice, preserve the order.\label{sup sb}\end{proposition}

\begin{proof}Suppose $x\leq a\leq b$, then $a\vee x=a\leq b=b\vee x$. If $a\leq x\leq b$, then $a\vee x=x\leq b=b\vee x$. Finally, $a\leq b\leq x$ implies $a\vee x=x=b\vee x$.\end{proof}

The next proposition gives a special case of proposition (\ref{sup td}). We will say that a constant function, for some object $x\in Range~f$, is the function \textit{onto $x$}. Also, the compositions $\vee x\circ\wedge x$ and $\wedge x\circ\vee x$ are the \textit{constant function} that sends every $a\in\mathcal L$ into $x$ and manifests that the supremum and infimum operations are commutative and associative.  If we have an operation $*:\mathcal{O}\rightarrow\mathcal{O}f\mathcal{O}$, such that for some object $e$ of $\mathcal{O}$, the notation yields $e;e,e$, then we will say $e$ is a \textit{primitive unit} for $*$.

\begin{proposition}\label{sup alg}For any lattice $\mathcal L$, and any $a,b\in\mathcal L$, we have

\begin{itemize}\item[1)]$a$ is primitve unit under both operations.

\item[2)]The functions $\vee a\circ\wedge a$ and $\wedge a\circ\vee a$ are the function onto $a$.

\item[3)]Both operations are commutative.

\item[4)]Both operations ae associative.

\end{itemize}

\end{proposition}

\begin{proof}

From 1), we have $a=a\vee a=\bigvee\{a,a\}\leq\bigvee\{a,b\}=a\vee b$, because $\{a,a\}=\{a\}\subseteq\{a,b\}$. Using proposition (\ref{sup td}) in the last inequality proves the first part of 2); the second part is similar.

Proposition (\ref{sup sb}) gives $a\vee y\leq(x\vee a)\vee y$ and $x\leq x\vee y\leq(x\vee a)\vee y$. We use (\ref{sup sb}) and the first inequality, then we use (\ref{sup td}) and the second inequality to prove

\begin{eqnarray}\nonumber x\vee(a\vee y)&\leq&x\vee[(x\vee a)\vee y]\\\nonumber&=&(x\vee a)\vee y\end{eqnarray}

We can just as easily say $(x\vee a)\vee y\leq x\vee(a\vee y)$.
\end{proof}

The following result states that the supremum of a finite union of sets, is supremum of supremums.

\begin{proposition}If $(A_i)_{i=1}^n$ is a sequence of bounded subsets of a lattice, then their union is also bounded and \begin{itemize}\item[1)]$\bigvee\bigcup\limits_{i=1}^n A_i=\bigvee\limits_{i=1}^n\bigvee A_i$\item[2)]$\bigwedge\bigcup\limits_{i=1}^n A_i=\bigwedge\limits_{i=1}^n\bigwedge A_i$.\end{itemize}\end{proposition}

\begin{proof}Let $A,B\subseteq\mathcal L$, then $\bigvee A,\bigvee B\leq \bigvee A\vee\bigvee B$; this follows from 4) in proposition (\ref{prop sup 1}). Therefore, every object in the union is less than or equal to $\bigvee A\vee\bigvee B$. This means that $\bigvee A\vee\bigvee B$ is an upper bound of $A\cup B$, and we conclude $\bigvee(A\cup B)\leq\bigvee A\vee\bigvee B$. 

We wish to prove the other inequality and it is sufficient to prove $\bigvee(A\cup B)$ is greater than or equal to $\bigvee A$ and $\bigvee B$; this also follows from 4).\end{proof}

We give a characterization of the supremum in terms of the 3-diagram; the result will be rather useful dealing with supremum and infimum. This result inverts 2) and 3) in proposition (\ref{prop sup 1}).

\begin{proposition}Let $\mathcal L\subseteq\mathcal N$, be a lattice, subset of a natural order $\mathcal N$.

\begin{itemize}

	\item[1)]For every $x\in\mathcal L$, we have $x\leq\bigvee A$ if and only if there exists $a\in A$ such that $x\leq a\leq\bigvee A$.

	\item[2)]For every $x\in\mathcal L$, we have $\bigwedge A\leq x$ if and only if there exists $a\in A$ such that $\bigwedge A\leq a\leq x$.

\end{itemize}

 \end{proposition}

\begin{proof}Let $\mathcal L\subseteq\mathcal N$, be a lattice, subset of a natural order $\mathcal N$.

$x\leq\bigvee A$ means $x$ is not an upperbound of $A$. So then we know that there is some element $a\in A$ such that $x$ is not greater than it. But, $a,x$ are comparable so that $x\leq a$.

The proof of 2) is similar.
\end{proof}

One can see that it is crucial to have a natural order, otherwise we cannot assure comparability for the objects in the proof.

		\section{Semilattice and Algebraic Aspects}

A \textit{semilattice} is a, commutative, associative category with every object as a primitive unit. A semilattice with unit is a commutative algebraic category with every object as a primitive unit. When we describe a lattice as an algebraic structure, we will consider two, dual, semilattices. 

			\subsection{Order of a Semilattice}

A semilattice always has an order defined in terms of the operation. Let $\diamond$ be the operation of the semilattice. We define an order on the collection of objects, $x\leq y$ if and only if $x;y,y$. In this case, our operation will act as the supremum of the order. We could have just as well said that the order is defined by $y\leq x$ if and only if $x;y,y$ and in such instance the operation is infimum. 

\begin{proposition}A semilattice $\mathcal L$ defines two, opposite, partial orders and the operation acts as supremum or infimum, respectively.\label{sup semi}\end{proposition}

\begin{proof}We will first prove that the semilattice defines the supremum operation. Let $x, y\in\mathcal L$, we wish to prove $x\vee y=x\diamond y$. It is quite obvious that $x\diamond(x\diamond y)=x\diamond y$ means $x\leq x\diamond y$. Of course, we also have $y\leq x\diamond y$ and therefore $x\vee y\leq x\diamond y$. On the other hand, $(x\vee y)\diamond(x\diamond y)=[(x\vee y)\diamond x]\diamond y=x\vee y$ because of associativity and the definition of the order; this means that $x\diamond y\leq x\vee y$. We have associated an order to $\diamond$, such that $\diamond$ is the supremum operation of the order.

The order just used is defined as $x\leq y$ if $x;y,y$ and we denote it by $\mathcal L$, as well. The order that is defined as $y\leq x$, under the same conditions, is the opposite order $\mathcal L^{op}$. Now we wish to prove that $\diamond$ acts as infimum for $\mathcal L^{op}$. From $x\diamond(x\diamond y)=x\diamond y$ we get $x\diamond y\leq x$, and similarly $x\diamond y\leq y$. Finally, to prove that $x\wedge y=x\diamond y$ we need only show that $x\wedge y\leq x\diamond y$. We again have $(x\wedge y)\diamond(x\diamond y)=x\wedge y$, so we are done.\end{proof}

In this last result we see a strong relation between supremum and infimum; namely, supremum and infimum are the same operation on opposite orders. Further on, we shall establish a more precise statement to this observation, in terms of a functor. The following theorem states every lattice determines two semilattices.

\begin{proposition}Let $\mathcal L$ be a partial order such that the least upper bound exists for every pair of objects, then the supremum operation defines a semilattice.

A similar remark holds if the greatest lower bound exists for every pair.\label{sup12}\end{proposition}

\begin{proof}We know that the operation of supremum is commutative and associative. Also, it is easily seen that any object in the partial order is a primitive unit.\end{proof}

			\subsection{Lattice as an Algebraic Structure}

What we have done is to show that an order can induce a pair of operations that have a certain algebraic aspect to them. Two semilattices, on the same collection of objects, with operations $\diamond,\diamondsuit$, are dual if they satisfy condition 2) of proposition (\ref{sup alg}).

\begin{proposition}Two dual semilattices of operations $\diamond,\diamondsuit$ verify $x\diamond y=y$ if and only if $x\diamondsuit y=x$.\label{sup ops}\end{proposition}

\begin{proof}Let the conditions of proposition (\ref{sup alg}) hold for the operations. We have $x\diamondsuit y=x\diamondsuit(x\diamond y)=x$. If $x\diamondsuit y=x$, then we can say $x\diamond y=(x\diamondsuit y)\diamond y=y$ becuase of commutativity in the operations.\end{proof}

Now we will see that the algebraic implications of the lattice are actually determinant in the structure of a lattice. That is, if we have an associative category, for two operations that satisfy 1)-4) of (\ref{sup alg}), then we have a lattice. This following result holds if $\mathcal L$ is replaced with an associative category, but in such cases the order will not be bounded.

\begin{theorem}Let $\diamond,\diamondsuit$ be two dual semilattices, both with unit and  both on the collection $\mathcal L$, and define an order such that $x\leq y$ if and only if $x\diamond y=y$. Then the order is a lattice with $x\vee y=x\diamond y$ and $x\wedge y=x\diamondsuit y$. Additionally, the units are maximum and minimum of the order.

If $\mathcal L$ is a bounded lattice, then we have two dual semilattices. The operations for these are supremum and infimum; the maximum and minimum are units.\end{theorem}

\begin{proof}From proposition (\ref{sup semi}) we can say that $\diamond$ is supremum, on the order defined by $x\leq y$ if $x;y,y$. By the same proposition, the operation $\diamondsuit$ is the supremum or infimum, on the order already defined.
Because of proposition (\ref{sup ops}) and duality of operations, we know that $x\diamondsuit y=x$; therefore $\diamondsuit$ is the infimum. This proves that supremum and infimum are full operations because the operation of a semilattice is full.

Using proposition (\ref{sup12}) we have two sublattices defined by supremum and infimum. Since supremum and infimum satisfy condition 2) of proposition (\ref{sup alg}), which means the sublattices are dual. It is not difficult to see that minimum and maximum of orders are units for supremum and infimum, respectively. In fact, one can also prove that minimum absorbs infimum and maximum absorbs supremum.\end{proof}

		\section{Completeness}

In this section we develop notions that guarantee some existence of supremum.

		\subsection{Complete Partial Order}

The concept of completeness is best studied for a lattice in terms of the order, and not the operation. Therefore, before we introduce completeness in the context of lattice, we will give definitions for partial order. First, we establish a notion for diected subsets of a partial order. After this, we do so for bounded subsets of the order. The following proposition will help prove some results in this section.

\begin{proposition}\label{sup max}The following are equivalent statements\begin{itemize}\item[1)]The maximum of $A$ exists\item[2)]The supremum of $A$ exists and $\bigvee A=\max A$\item[3)]The supremum of $A$ exists and $\bigvee A\in A$.\end{itemize}\end{proposition}

\begin{proof}The maximum is upper bound and if any other object is greater than the maximum, it fails to be in the set. This means that 1) implies 2). Of course 2) implies 1), trivially. To verify 2) implies 3), we simply note that $\max A\in A$. If 3), then 2) is true; the supremum is greater than every $x\in A$.\end{proof}

			\subsubsection{Directed Complete}

If $\mathcal L$ is a directed set, it is a non-empty partial order such that for every pair of objects $x,y\in\mathcal L$ we have $\downarrow\{x,y\}\neq\emptyset$. The definition of semilattice, is more specific, requiring that $\min\downarrow\{x,y\}$ exist for every pair. We can conclude that a partial order which is not a directed set, cannot be a semilattice either.

 Let $D_\mathcal L$ be the non-empty collection of directed subsets of some partial order $\mathcal L$. Suppose we have a function $D_\mathcal L\subseteq\mathcal{PL}\rightarrow\mathcal L$ that sends every object in the domain into its supremum; we are of course making it implicit that the supremum exists for every directed subset of the order. In such cases we say that $\mathcal L$ is \textit{directed complete}. If, additionally, $\mathcal L$ has a minimum, we say it is a \textit{complete partial order.} We restate this: a partial order is complete if and only if it has a minimum and there is a function $\textbf{sup}:D_\mathcal L\rightarrow\mathcal L$ such that $I \mapsto_\textbf{sup}\bigvee I$. This can be stated in terms of the infimum operation if we consider the opposite order. Since the set $I$ is directed, for every $x,y\in\mathcal L^{op}$ we have $k\in\mathcal L^{op}$ such that $k\leq x,y$. By proposition (\ref{sup semi}) we conclude that the order has a maximum and there is a function $\textbf{inf}:D_{\mathcal L^{op}}\rightarrow\mathcal L^{op}$ such that $I\mapsto_{\textbf{inf}}\bigwedge I$.

We recall that a set $\mathcal L$ is finite if there a bijective function $\textbf n\rightarrow\mathcal L$, for some $n\in\mathbb N$.

\begin{proposition}If $\mathcal L$ is a finite partial order, then it is directed complete.\end{proposition}

\begin{proof}To prove that the order is directed complete, we take any directed subset $I\subseteq\mathcal L$. The set is finite so we know that there is finite sequence of the form $\textbf n\rightarrow\mathcal L$. We represent $I$ with $(x_i)_{i=1}^n$ and since it is directed, for every pair $x_i,x_j\in I$, we have $x_i,x_j\leq x_k\in I$. 

Take $x_1,x_2$ and find $x_{k_1}$ such that $x_1,x_2\leq x_{k_1}$. Suppose $x_{k_1}$ is not maximum of $I$. Then there is $x_{k_2}$ such that $x_{k_1},x_3\leq x_{k_2}$ and $x_{k_1}<x_{k_2}$. If $x_{k_2}\neq\max I$, then there is $x_{k_3}$ such that $x_{k_2},x_4\leq x_{k_3}$ and $x_{k_1}<x_{k_2}<x_{k_3}$. We can continue in this manner and if we get to $x_n$, then $x_{k_1}<x_{k_2}<\cdots<x_{k_{n-1}}<x_n=\max I$. There is, however, the possibility that we find $\max I=x_{k_m}$ at some step $m<n$. This means that there is a function $\textbf{sup}:D_\mathcal L\rightarrow\mathcal L$ that makes $I\mapsto_\textbf{sup}\max I=\bigvee I$; use proposition (\ref{sup max}).\end{proof}

			\subsubsection{Naturally Complete}

We say that a partial order $\mathcal P$ is \textit{naturally complete} if and only if every natural order $\mathcal N\subseteq\mathcal P$ has supremum. That is, if $\min\downarrow\mathcal N$ exists for every $\mathcal N$. Zorn's lemma is applied to a partial order whose natural orders are bounded; $\downarrow\mathcal N\neq\emptyset$ for every $\mathcal N$. Thus, we can apply Zorn's lemma to any naturally complete order and conclude that it has a maximal element.

There is a more particular definition, and it will be useful since it will help to take the concept of completeness into the realm of sequences. Define the collection $IS_\mathcal L$ of all natural order $\mathcal N\subseteq\mathcal L$ such that there exists a functor $s:\mathbb N\rightarrow\mathcal N$; this means that we are considering the collection of natural orders that can be arranged in a growing sequence. If there is a function $\textbf{sup}:IS_\mathcal L\subseteq\mathcal{PL}\rightarrow\mathcal L$, we will say $\mathcal L$ is \textit{IS-complete}.

\begin{proposition}If $\mathcal L$ is directed complete, then it is also naturally complete. This last implies that $\mathcal L$ is $IS$-complete.\end{proposition}

No proof is needed for this last result; this is not the case is the next result, however.

\begin{proposition}A partial order $\mathcal L$ is naturally complete if and only if $\min Inv~\textgoth F$ exists for every functor $\textgoth F:\mathcal L\rightarrow\mathcal L$.\end{proposition}

			\subsubsection{Bounded Complete}

If we are to consider a lattice that is not bounded, the question may arise: \textit{does supremum and infimum exist for the bounded subsets of the lattice?} We are compelled to make definitions in such direction, and we do so in the same manner as we did in the last subdivision. A partial order is \textit{upper bound complete} if and only if $\downarrow A\neq\emptyset$ implies the existence of $\min\downarrow A$, for any $A\subseteq\mathcal P$. We also define the dual concept of \textit{lower bound complete} for partial orders such that $\uparrow A\neq\emptyset$ implies the existence of $\max\uparrow A$.

\begin{proposition}A partial order is upper bound complete if and only if it is lower bound complete.\end{proposition}

\begin{proof}Let us prove $\mathcal P$ is lower bound complete, given that it is upper bound complete. Let $A\neq\emptyset$ such that $\uparrow A\neq\emptyset$. We know $\bigvee\uparrow A$ exists given $\uparrow A$ is bounded above. Since $A\neq\emptyset$ we know that $\uparrow A$ is bounded above by some $a\in A$. Now we show that $\bigwedge A=\bigvee\uparrow A$. First of all, if $x$ is upper bound of $\uparrow A$, then $\bigvee\uparrow A\leq x$. Since any object of $A$ is upper bound of $\uparrow A$, we have $\bigvee\uparrow A\leq x$, for all $x\in A$. This is equivalent to saying $\bigvee\uparrow A\in\uparrow A$, and if we recall proposition (\ref{sup max}) we see that it is the same as $\bigvee\uparrow A=\max\uparrow A$.\end{proof}

This means that there is maximum and minimum if the order is upper or lower bound complete. In light of the previous proposition, we say $\mathcal L$ is \textit{bounded complete} if it is upper or lower bound complete.

Given a set $\mathcal L$, let us define an order, for the objects of $\mathcal Lf\mathcal L$, by $f\leq g$ if there exists $A\subseteq\mathcal L$ such that $f$ is the function $g|_{A}$; we are defining an order in terms of complicated versions of systems.

\begin{proposition}The order defined for $\mathcal Lf\mathcal L$ is complete and bounded complete.\end{proposition}

\begin{proof}If we are to show that the order is complete we must prove it is directed complete and there is a minimum. Let $I\subseteq\mathcal Lf\mathcal L$ be directed; this is $I\in D_{\mathcal Lf\mathcal L}$. If $f,g\in I$ we have $h\in I$ such that $f,g\leq h$, which means $f,g$ are $h|_A$, $h|_B$, respectively for sets $A,B\subseteq\mathcal L$. So, if $a\in Dom~f\cap Dom~g$ we have $a,g;a,f$. Thus, for any $a\in\bigcup_{i\in I}Dom~i$ we can write $Ia$ in place of the image, under any function of $I$ that is defined for $a$. Consider the set function $\textbf{sup}~I:\bigcup_{i\in I} Dom~i\rightarrow\mathcal L$ that makes $a\mapsto Ia$.

We shall prove that $\textbf{sup}~I=\bigvee I$. It is not difficult to see that $\textbf{sup}~I$ is an upper bound of $I$. This is, every $i\in I$ is a simplified version of $\textbf{sup}~I$. If $f<\textbf{sup}~I$, then there exists $a\in\bigcup_{i\in I}Dom~i- Dom~f$. This object is in some $g\in I$. For this, $f$ is not an upper bound of $I$ and we may conclude that there is a function $\textbf{sup}:D_{\mathcal Lf\mathcal L}\rightarrow\mathcal Lf\mathcal L$ that sends every object in the domain into its supremum. The minimum of this order is the function $\min\mathcal Lf\mathcal L:\emptyset\rightarrow\mathcal L$.

We are left to prove the order is bounded complete. Take a set $X\subseteq \mathcal Lf\mathcal L$ that is bounded by above; there is $h\in\downarrow X$ such that every function $x\in X$ is a simplified version of $h$. We can restate this saying every $x\in X$ is of the form $h|_{A_x}$. Denote with $\textbf{sup}~X:\bigcup_{x\in X}A_x\rightarrow\mathcal L$ the function that is a simplified version of $h$; the arrows are $a\mapsto ha$. Proving $\textbf{sup}~X=\bigvee~X$ is similar to what was done above.\end{proof}

		\subsection{Complete Lattice}

The main idea in the concept of complete lattice is that any subset of a lattice has supremum and infimum. We now give a duality for supremum and infimum. Consider any function $\textbf{sup}\in(\mathcal{PL})f\mathcal L$ such that $A\mapsto_{\textbf{sup}}\bigvee A$. The function is order preserving; this follows from 4) in proposition (\ref{prop sup 1}).

\begin{theorem}Given a function $\textbf{sup}:\mathcal{PL}\rightarrow\mathcal L$, such that $A\mapsto_{\textbf{sup}}\bigvee A$, we are able to construct $\textbf{inf}:\mathcal{PL}\rightarrow\mathcal L$, such that $A\mapsto_{\textbf{inf}}\bigwedge A$, for every $A\subseteq\mathcal L$. In other words, $\textbf{inf}~A=\textbf{sup}\uparrow A$ and $\textbf{sup}~A=\textbf{inf}\downarrow A$.

Moreover, $\textbf{sup}$ is order preserving and $\textbf{inf}$ is order reversing.\end{theorem}

\begin{proof}We will show that infimum is $\textbf{sup}\uparrow A$.  Let $x\in A$, then $x$ is an upper bound of $\uparrow A$. This means $\textbf{sup}\uparrow A\leq x$. Therefore, $\textbf{sup}\uparrow A$ is a lower bound of $A$; this is equivalent to $\bigvee\uparrow A=\textbf{sup}\uparrow A\in\uparrow A$. Recall that $\bigvee A\in A$ if and only if $\bigvee A=\max A$, so that $\bigwedge A=\max\uparrow A=\textbf{sup}\uparrow A$.

The reader can just as easily prove that $\textbf{sup}~A=\textbf{inf}\downarrow A$. The second part of this result is 4) from proposition (\ref{prop sup 1}).\end{proof}

A partial  order that has functions $\textbf{sup}$ and $\textbf{inf}$ is a \textit{complete lattice}. Of course, any complete lattice is a lattice. We see that complete lattices are, trivially, bounded complete.

\begin{proposition}Let $\mathcal L$ be a complete lattice, then $\mathcal L$ and $\mathcal L^{op}$ are directed complete.\end{proposition}

\begin{proof}By definition, a complete lattice is obviously directed complete. Since $\mathcal L$ is complete, we know that the infimum of every subset exists. But, this is the same as saying that the supremum exists for every subset of $\mathcal L^{op}$. In particular, the supremum exists for every directed subset and therefore $\mathcal L^{op}$ is directed complete.\end{proof}












	\chapter{Group}

		\section{Subcollections}

We will sometimes express $G$ instead of the collection of objects of operation, $\mathcal{A}|G$. Given a group $G$, we can consider systems obtained from $G$ by taking away objects of operation. These systems are the object of study in the present discussion.

			\subsection{Power operation}

We define an operation such that source and target objects are the objects of operation of $G$. It is a non-commutative operation; that is why sometimes we decide to treat it as an ordered collection of functions. We define this operation in a way that generalizes the definition of product in terms of sum.

Given a group $G$ and any object $x$, in $\mathbb{Z}$, let $^{\uparrow} x:G\rightarrow G$. For $x\geq0$, we define $a;a^{\uparrow}x*a,x+1$ and $a;(a^{x})^{-1},-x$. We accept $a;1,0$ which means $a^{\uparrow}0$ is the unit of the group, for every object $a$ of $G$ that is not the unit. In case $G$ is abelian, we have a functor $^{\uparrow}x:G\rightarrow G$. To prove this, we must prove $^{\uparrow}x,^{\uparrow}x;*(b^{\uparrow}x),*b$. We see that the relation holds for 0, so now we shall verify it holds for $x+1$, given it holds for $x$. \begin{eqnarray}\nonumber a*b&;&[(a*b)^{\uparrow}x]*(a*b),x+1\\\nonumber a*b&;&[(a^{\uparrow}x)*(b^{\uparrow}x)]*(a*b),x+1\\\nonumber a*b&;&
[(a^{\uparrow}x)*a]*[(b^{\uparrow}x)*b],x+1\\\nonumber a*b&;&
[a^{\uparrow}(x+1)]*[b^{\uparrow}(x+1)],x+1.\end{eqnarray}

So, we are defining a function $\mathbb{Z}\rightarrow G\textgoth{F}G$, which is an operation that we will call the \textit{power operation of $G$}. Notice that this operation is almost full, in fact, for it to be full we only need to define $1\uparrow0$ but there is no consistent way so we leave the operation as it is. If $e$ is a primitve unit, then for all $x\in\mathbb{N}$, we have $e;e,x$.

The product operation, defined for integers, is the power for $\mathbb{Z}$. We note this because $a;a\cdot x+a,x+1$. We now wish to find the power operation of $\mathbb{Q}$. We see that this is the operation $^{\uparrow}:\mathbb{Z}\rightarrow\mathbb{Q}\textgoth{F}\mathbb{Q}$, where we define $a;a^{x} \cdot a,x+1$.

		\subsection{Subgroup}

For any subcollection $A$, of a group $G$, and $n\in \mathbb{Z}$, we will write $A^{n}$ to represent the collection of objects of the form $a^{\uparrow}n$, where $a$ is any object of $A$. We write $AB$ to express the collection of all objects of the form $a*b$, where $a,b\in A,B$. For any group, $G=G^{-1}$. A subgroup of $G$ is a group $H$ such that $H\subseteq G$. The collection of all subgroups of $G$ is represented by $\mathcal{G}_{\subseteq}$.

\begin{theorem}The following are equivalent statements\begin{itemize}\item[1)] $H\in \mathcal{G}_{\subseteq}$\item[2)]$HH^{-1}\subseteq H$\item[3)]$H^{2},H^{-1}\subseteq H$\end{itemize}\end{theorem}

\begin{proof}It is clear from the definition of a group, that $1)\Rightarrow2),3)$. We shall prove $2),3)\Rightarrow1)$. Supposing 3) to be true, we find $HH^{-1}\subseteq H^2\subseteq H$. So, all we need to prove is 2) implies 1). Suppose $H$ has at least one object $x$:

\begin{eqnarray}\nonumber x&;&1,x^{-1}\\\nonumber1&;&x^{-1},x^{-1} \\\nonumber x&;&x\cdot y,(y^{-1})^{-1}\end{eqnarray}

These three statements indicate $1$, $x^{-1}$, $x\cdot y$ are all objects of $HH^{-1}\subseteq H$.

\end{proof}

\begin{theorem}Let $a^{x}\subseteq G$ represent the collection of objects of such form, for fixed $a\in G$ and every $x\in\mathbb{Z}$. Then $a^{x}$ is an abelian subgroup of $G$.

\end{theorem}

\begin{proof}

 First we will prove cummutativity with $a$, and then we will prove $a^{x}(a^{x})^{-1}\subseteq a^{x}$. Due to the last theorem, this proves $a^{x}$ is a group. After that we will verify commutatvity.

\begin{itemize}

\item[1)] Suppose that for the product we have $a,a;a^{n},a^{n}$:

\begin{eqnarray}\nonumber a&;&a*a^{n+1},a^{n+1}\\\nonumber a&;&a*(a^{n}*a),a^{n+1}\\\nonumber a&;&(a*a^{n})*a,a^{n+1}\\\nonumber a&;&(a^{n}*a)*a,a^{n+1}\\\nonumber a&;&a^{n+1}*a,a^{n+1}\\\nonumber a,a&;&a^{n+1},a^{n+1}\end{eqnarray}

\item[2)] Given $n$, suppose $a^{n};a^{n-m},a^{-m}$, for any $m$:

\begin{eqnarray}\nonumber a^{n+1}&;&(a^{n}*a)*a^{-m},(a^{m})^{-1}\\\nonumber a^{n+1}&;&(a*a^{n})*a^{-m},(a^{m})^{-1}\\\nonumber a^{n+1}&;&a*(a^{n}*a^{-m}),(a^{m})^{-1}\\\nonumber a^{n+1}&;&a*a^{n-m},(a^{m})^{-1}\\\nonumber a^{n+1}&;&a^{n-m}*a,(a^{m})^{-1}\\\nonumber a^{n+1}&;&a^{n-m+1},(a^{m})^{-1}\end{eqnarray}

\item[3)]\begin{eqnarray}\nonumber a^{n}&;&a^n*a^{-m},a^{-m}\\\nonumber a^{n}&;&a^{n-m},a^{-m}\\\nonumber a^{n}&;&a^{-m+n},a^{-m}\\\nonumber a^{n}&;&a^{-m}*a^n,a^{-m}\\\nonumber a^{n},a^n&;&a^{-m},a^{-m}\end{eqnarray}
\end{itemize}

\end{proof}

	\subsection{Congruence Class}

Given an equivalence relation $\leftrightarrow$ for a collection $A$, carry out a seperation of this system, without losing information. The systems obtained are the subcollections of all equivalent objects. That is, we have all simple equivalence relations. For objects of the same simple equivalence relation we will express $$a\equiv b~~~~mod~\leftrightarrow.$$

Let $H\in\mathcal{G}_{\subseteq}$, and define $a\sim x$ if $h;a,x$ for some $h\in H$. This is, the collection of objects $x$ such that we have an arrow $h\rightarrow_{*x} a$. In other words, $x$ such that its right operation has an arrow with source object in $H$ and target is $a$.

\begin{theorem}The relation $\sim$ is an equivalence relation and $a\sim b$ will be expressed by $$a\equiv b~~~~mod~Hx.$$ Each simple equivalence relation resulting from $\sim$ will be called a right congruence class. Every right congruence class is the subcollection $Hx$, for any $x$ in the class. For every $x\in G$, the subcollection $Hx$ is the right congruence class that contains $x$ and, in particular, $H$ is a congruence class.\end{theorem}

\begin{proof}\makebox[5pt][]{}\mbox {}
\begin{itemize}

\item[1)]We will prove $\sim$ is indeed an equivalence relation. It is obvious that the reflexive property holds; if we take $h$ to be the unit of $G$ we have $h;a,a$. To prove symmetry, suppose $h;a,b$:

\begin{eqnarray}\nonumber h^{-1}&;&h^{-1}*(h*b),a\\\nonumber h^{-1}&;&(h^{-1}*h)*b,a\\\nonumber h^{-1}&;&b,a.\end{eqnarray}

The result follows from $h^{-1}\in H$. Finally, let $h_{1},h_{2}\in H$ such that $h_{1};a,b$ and $h_{2};b,c$. From this,

\begin{eqnarray}\nonumber h_{1}*h_{2}&;&(h_{1}*h_{2})*c,c\\\nonumber h_{1}*h_{2}&;&h_{1}*(h_{2}*c),c\\\nonumber h_{1}*h_{2}&;&h_{1}*b,c\\\nonumber h_{1}*h_{2}&;&a,c.\end{eqnarray}

We may conclude because $h_{1}*h_{2}\in H$.

\item[2)] Let $x$ be an element of a right congruence class. Then, $a\in Hx~\Leftrightarrow ~h;a,x~\Leftrightarrow~h^{-1};x,a~\Leftrightarrow ~a$ is in the right congruence class of $x$.

If $x\in G$, then $x$ is in some right equivalence class; this is guaranteed by the reflexive property. But this class is the subcollection $Hx$, because $x$ is an object of it.

\end{itemize}
\end{proof}

\begin{theorem}There is a bijective function $H\rightarrow Hx$, for every $x\in G$.\end{theorem}

\begin{proof}Our bijective function is $*x$. Let $y\in Hx$, then there is a unique solution to $h;y,x$, for $h$. This means $*x$ is monic and onto $Hx$.

\end{proof}

It is clear that we still have work to do in the subject matter of congruence classes. We can take an analogous path to prove similar theorems for the \textit{left congruence class}, $xH$. The equivalence relation in this context is given by $a\sim x$ if $x;a,h$. Similar theorems hold for the left congruence class. Of course this is the collection of objects $x\in G$ such that $x\rightarrow_{*h}a$, for some $h\in H$.

	\subsection{Invariant subspace}

A subgroup $N\subseteq G$ is called \textit{invariant} or \textit{normal} if $xN=Nx$, for every $x\in G$. Simply put, it is any subgroup such that the equivalence relation it defines yields the same right and left congruence classes. The collection of all invariant subgroups of $G$ is denoted by $\hat{\mathcal{G}}_{\subseteq}$.

\begin{theorem}The following are equivalent:\begin{itemize}\item[1)]$xN=Nx$, for every $x\in G$.\item[2)]$xNx^{-1}
=N$, for every $x\in G$\item[3)] $xNx^{-1}\subseteq N$, for every $x\in G$. \end{itemize}\end{theorem}

\begin{proof} \makebox[5pt][]{}\mbox {}

\begin{itemize}\item[1)$\Rightarrow$2)] For $n_1\in N$, we find $n_2\in N$ such that \begin{eqnarray}\nonumber x*n_{1}&;&(x*n_{1})*x^{-1},x^{-1}\\\nonumber x*n_{1}&;&(n_{2}*x)*x^{-1},x^{-1}\\\nonumber x*n_{1}&;&n_{2}*(x*x^{-1}),x^{-1}\\\nonumber x*n_{1}&;&n_{2},x^{-1} \end{eqnarray}which proves that $xNx^{-1}\subseteq N$. Now, to prove $N\subseteq xNx^{-1}$:

\begin{eqnarray}\nonumber n_{1}&;&n_{1}*(x*x^{-1}),1\\\nonumber n_{1}&;&(n_{1}*x)*x^{-1},1\\\nonumber n_{1}&;&(x*n_{2})*x^{-1},1.\end{eqnarray}

\item[$3)\Rightarrow1)$] We wish to prove $xN\subseteq Nx$:

\begin{eqnarray}\nonumber x&;&(x*n_{1})*1,n_{1}\\\nonumber x&;&(x*n_{1})*(x^{-1}*x),n_{1}\\\nonumber x&;&[(x*n_{1})*x^{-1}]*x,n_1\\\nonumber x&;&n_{2}*x,n_1.\end{eqnarray}

Similarly, we verify $Nx\subseteq xN$.

\end{itemize}

\end{proof}

\section{Quotient group}

Given an invariant subgroup $N$, consider the collection of congruence classes. We say congruence classes because the classes on the right and left are the same. Thus, we can write $$a\equiv b~~~~mod~N$$ if there is an $n\in N$ such that $n;a,b$. We see that every class $Nx$ is related to $N$ in the same manner, so we will consider each class to be a reflexive arrow for $N$. Let $G/N$ be the algebraic category with one c-object, $N$, and objects of operation are the congruence classes. The operation for the classes is defined by $Nx;N(x*y),Ny$. It is left to the reader to verify that the definition is consistent.

\begin{theorem}The algebraic category $G/N$ is a group. Moreover, if $G$ is abelian so is $G/N$.\end{theorem}

\begin{proof}\makebox[5pt][]{}\mbox {}
\nonumber\begin{itemize}\item[1)]We first verify that the operation has a unit and the objects are dual to it.  After that we will prove associativity. The unit in $G/N$ is the invariant subspace $N=Ne$:

\begin{eqnarray}\nonumber Nx&;&Nx*Ne,N\\\nonumber Nx&;&N(x*e),N\\\nonumber Nx&;&Nx,N.\end{eqnarray}

Duality amongst the objects is determined by

\begin{eqnarray}\nonumber Nx&;&Nx*Nx^{-1},Nx^{-1}\\\nonumber Nx&;&N(x*x^{-1}),N\\\nonumber Nx&;&Ne,Nx^{-1}.\end{eqnarray}

We turn to associativity: 

\begin{eqnarray}\nonumber Na&;&Na*N(b*c),Nb*Nc\\\nonumber Na&;&N[a*(b*c)],Nb*Nc\\\nonumber Na&;&N[(a*b)*c],Nb*Nc\\\nonumber Na&;&N(a*b)*Nc,Nb*Nc\\\nonumber Na&;&(Na*Nb)*Nc,Nb*Nc\\\nonumber Na,Nc&;&Na*Nb,Nb*Nc.\end{eqnarray}

\item[2)] Supposing $G$ is abelian:

\begin{eqnarray}\nonumber Na&;&N(a*b),Nb\\\nonumber Na&;&N(b*a),Nb\\\nonumber Na&;&Nb*Na,Nb\\\nonumber Na,Na&;&Nb,Nb\end{eqnarray}

\end{itemize}\end{proof}

We will now construct the two trivial quotient groups of $G$, obtained by considering $N$ to be the subgroup that consists of the unit, and the subgroup $G$. The quotient group $G/1$ is the same group as $G$. We verify this because we only have one solution to $1;x,y$, for $y$. The solution is $x$ and that means all the relations are reflexive:$$x\equiv x~~~~mod ~1.$$Another way of seeing this is by finding the class of $x$ by $xe$. In other words, every $x\in G$ is its congruence class. 

If we turn to find $G/G$, we see its the subgroup $e$. This is true due to the fact that we always have a solution $h\in G$ for $h;x,y$. That is any $x,y\in G$ are related. We also see that this is due to $Gx=G$.

	\section{Commutator}

We go on to study the more general cases of an abelian group, in which not necessarilly all objects of operation commute with all the rest.

We know that the inverse of an object $a*b$ is unique and we find it by \begin{eqnarray}\nonumber a*b&;&(a*b)*(b^{-1}*a^{-1}),(b^{-1}*a^{-1})\\\nonumber a*b&;&a*[b*(b^{-1}*a^{-1})],(b^{-1}*a^{-1})\\\nonumber a*b&;&a*[(b*b^{-1})*a^{-1}],(b^{-1}*a^{-1})\\\nonumber a*b&;&a*a^{-1},(b^{-1}*a^{-1})\\\nonumber a*b&;&e,(b^{-1}*a^{-1}).\end{eqnarray}

For the following definition, note that $a,b$ commute if and only if $(a*b);e,(a^{-1}*b^{-1})$.

\begin{definition}Define the operation $\square:G\rightarrow GfG$ such that $a\square b\mapsto(a*b)*(a^{-1}*b^{-1})$, for $a,b\in G$. The image of $b$, under $a\square$ is $[a,b]$ and we call it the commutator of $a,b$. In other words, the notation for $*$ expresses $a*b;[a,b],a^{-1}*b^{-1}$. The object $a$ is said to be commutable if $a\square:G\rightarrow G$ is the function into $e$. Define $Comm(G)$ as the subcollection of commutable objects and call it the center of $G$.\end{definition}

Notice that the center of $G$ is, by definition, the fiber corresponding to the function into $e$. That is, if we represent the function into $e$ with $\rightarrow e$, then $Comm(G)=\square^{-1}[\rightarrow e]$

\begin{proposition}The center of $G$ verifies the relation $Comm(G)\in\mathcal{G}_{\subseteq}$ and we say it is the central subgroup of $G$.\end{proposition}

\begin{proof}Take $a\in Comm(G)$ and $b^{-1}\in[Comm(G)]^{-1}$. We shall verify $[a*b^{-1},x]$ is $e$, for every $x\in G$.

\begin{eqnarray}\nonumber (a*b^{-1})*x&;&[(a*b^{-1})*x]*[(b*a^{-1})*x^{-1}],(a*b^{-1})^{-1}*x^{-1}\\\nonumber (a*b^{-1})*x&;&[(a*b^{-1})*b]*[(x*a^{-1})*x^{-1}],(a*b^{-1})^{-1}*x^{-1}\\\nonumber (a*b^{-1})*x&;&(a*x)*(a^{-1}*x^{-1}),(a*b^{-1})^{-1}*x^{-1}\\\nonumber (a*b^{-1})*x&;&(x*a)*(a^{-1}*x^{-1}),(a*b^{-1})^{-1}*x^{-1}\\\nonumber (a*b^{-1})*x&;&e,(a*b^{-1})^{-1}*x^{-1}.\end{eqnarray}

 We have shown that $a*b^{-1}$ is commutable; this proves $Comm(G)\in\mathcal{G}_{\subseteq}$.\end{proof}

\begin{proposition}If $C\subseteq Comm(G)$ is a subgroup, then $C\in\hat{\mathcal{G}}_{\subseteq}$; we say $C$ is a central invariant subgroup of $G$.\end{proposition}

\begin{proof}Observe $xC=Cx$, for any $x\in G$.\end{proof}

We now build the collection generated by the commutators of $G$. Take the collection of all commutators $[a,b]$, this is the subcollection $[x,y]_{x,y\in G}$. We recall the notation $\Lambda([x,y]_{x,y\in G})$, where $\Lambda$ is the generailzed operation function of the operation of the group, represents the image of all the operators $\Lambda_{i=1}^n$ in the image of $\Lambda$. We will represent it with $[G,G]$. Of course, if a group contains all the commutators, it also contains $[G,G]$.

\begin{proposition}The subcollection $[G,G]\subseteq G$ is called the commutant of $G$ and $[G,G]\in\hat{\mathcal{G}}_{\subseteq}$.\end{proposition}

\begin{proof}We first want to prove that $[G,G]$ is a subgroup. We will give a proof that $[G,G]^2$ and $[G,G]^{-1}$ are subcollections of $[G,G]$. If $c,d\in [G,G]$ then they can be represented by $\prod_{i=1}^mC_{i}$ and $\prod_{j=1}^nD_j$, where $C_i,D_j$ are all commutators of $G$. This means $c*d$ is of the same form as each of the factors $c,d$. For the second part, we can easily find the inverse of $\prod_{i=1}^nx_i$, for $x_i\in G$, by the following. Suppose $\prod_{i=1}^nx_i;e,\prod_{i=1}^nx^{-1}_{n+1-i}$. Then, we find\begin{eqnarray}\nonumber\prod_{i=1}^{n+1}x_i&;&\left[\left(\prod_{i=1}^{n}x_i
\right)*x_{n+1}\right]*\left[x^{-1}_{n+1}*\left(\prod_{i=2}^{n+1}x_{n+2-i}^{-1}\right)\right],\prod_{i=1}^{n+1}x_{n+2-i}^{-1}\\\nonumber\prod_{i=1}^{n+1}x_i&;&\left(\prod_{i=1}^{n}x_i
\right)*\left(\prod_{i=1}^{n}x_{n+1-i}^{-1}\right),\prod_{i=1}^{n+1}x_{n+2-i}^{-1}\\\nonumber\prod_{i=1}^{n+1}x_i&;&e,\prod_{i=1}^{n+1}x_{n+2-i}^{-1}.\end{eqnarray}

If $[a,b]\in[x,y]_{x,y\in G}$, then $[a,b]^{-1}\in[x,y]_{x,y\in G}$. In fact,

\begin{eqnarray}\nonumber (a*b)*(a^{-1}*b^{-1})&;&[(a*b)*(a^{-1}*b^{-1})]*[(b*a)*(b^{-1}*a^{-1})],(b*a)*(b^{-1}*a^{-1})\\\nonumber (a*b)*(a^{-1}*b^{-1})&;&[(a*b)*a^{-1}]*[a*(b^{-1}*a^{-1})],(b*a)*(b^{-1}*a^{-1})\\\nonumber (a*b)*(a^{-1}*b^{-1})&;&(a*b)*(b^{-1}*a^{-1}),(b*a)*(b^{-1}*a^{-1})\\\nonumber (a*b)*(a^{-1}*b^{-1})&;&a*a^{-1},(b*a)*(b^{-1}*a^{-1})\\\nonumber (a*b)*(a^{-1}*b^{-1})&;&e,(b*a)*(b^{-1}*a^{-1})\end{eqnarray}

We can now say $c^{-1}\in[G,G]$
 because $c^{-1};\prod_{i=1}^nC^{-1}_{n+1-i},e$.

Next we need to prove $[G,G]$ is invariant. It is easy to see that $x*c*x^{-1};\prod_{i=1}^nx*C_i*x^{-1},e$. Also,

\begin{eqnarray}\nonumber x*[a,b]&;&(x*[a,b])*x^{-1},x^{-1}\\\nonumber x*[a,b]&;&[x*(a*b)*(a^{-1}*b^{-1})]*x^{-1},x^{-1}\\\nonumber x*[a,b]&;&x*(a*x^{-1})*(x*b)*(x^{-1}*x)*(a^{-1}*x^{-1})*(x*b^{-1})*x^{-1},x^{-1}
\\\nonumber x*[a,b]&;&(x*a*x^{-1})*(x*b*x^{-1})*(x*a^{-1}*x^{-1})*(x*b^{-1}*x^{-1}),x^{-1}
\\\nonumber x*[a,b]&;&(x*a*x^{-1})*(x*b*x^{-1})*(x*a*x^{-1})^{-1}*(x*b*x^{-1})^{-1},x^{-1}
\end{eqnarray}

This means that $x*[a,b]*x^{-1}$ is the same as $[xax^{-1},xbx^{-1}]\in[x,y]_{x,y\in G}$. This is, $x*c*x^{-1}\in[G,G]$.

\end{proof}

The commutant of $G$ turns out to be a special subgroup. We cannot say $[G,G]$ is abelian. However, it does generate an abelian quotient group and it is the smallest subgroup of $G$ to do so.

\begin{theorem}Let $N\in\hat{\mathcal{G}}_{\subseteq}$. Then\begin{itemize}\item[1)] $G/[G,G]\in\textbf{Ab}\mathcal{G}_{\subseteq}$\item[2)]$G/N\in\textbf{Ab}
\mathcal{G}_{\subseteq}\Rightarrow[G,G]\subseteq N$\item[3)]$[N,N]\in\hat{\mathcal{G}}_{\subseteq}$.\end{itemize}\end{theorem}

	\begin{proof}\makebox[5pt][]{}\mbox {}
	
		\begin{itemize}

\item[1)] To prove that $G/[G,G]$ is a commutative subgroup of $G$ we take $[G,G]a$ and $[G,G]b$, for some $a,b\in G$. We will show that their commutator is $[G,G]$, the unit in $G/[G,G]$.

\begin{eqnarray}\nonumber[G,G]a*[G,G]b&;&[G,G](a*b)*[G,G](a^{-1}*b^{-1}),[G,G]a^{-1}*[G,G]b^{-1}\\\nonumber[G,G]a*[G,G]b&;&[G,G][a,b],[G,G]a^{-1}*[G,G]b^{-1}\\\nonumber[G,G]a*[G,G]b&;&[G,G],[G,G]a^{-1}*[G,G]b^{-1}.\end{eqnarray}

\item[2)]Suppose $[G,G]$ is not a subset of $N$, then there is $[a,b]\in[x,y]_{x,y\in G}$ such that $[a,b]\notin N$. The class $[G,G][a,b]$ is not $N$ and therefore $[[G,G]a,[G,G]b]$ do not commute.

		\end{itemize}

	\end{proof}

\section{Transformation}

We have proven there is a functor $+:\mathbb{Z}\rightarrow\mathbb{Z}\textgoth{F}\mathbb{Z}$, for groups. The main objective of the present section is to build the tools necessary in order to generalize and clarify that situation.

		\subsection{Homomorphism.}

\begin{definition} We will say that a functor $\textgoth{h}:G_{1}\rightarrow G_{2}$ is a \textit{homomorphism} if $G_{1},G_{2}$ are groups.\end{definition} 

Condition 1) for functors means $1;1,\textgoth{h}$. Also, for the product in $G_{2}$,

\begin{eqnarray}\nonumber \textgoth{h}x^{-1}&;&\textgoth{h}x^{-1}*[\textgoth{h}x*(\textgoth{h}x)^{-1}],1
\\\nonumber \textgoth{h}x^{-1}&;&[\textgoth{h}x^{-1}*\textgoth{h}x]*(\textgoth{h}x)^{-1},1
\\\nonumber\textgoth{h}x^{-1}&;&\textgoth{h}(x*x^{-1})*(\textgoth{h}x)^{-1},1\\\nonumber\textgoth{h}x^{-1}&;&\textgoth{h}
1*(\textgoth{h}x)^{-1},1\\\nonumber\textgoth{h}x^{-1}&;&(\textgoth{h}x)^{-1}
,1.\end{eqnarray}

Here, $1$ represents the unit, for both groups. A homomorphism with monic arrow function is called \textit{monomorphism}. If the arrow function of the homomorphism is onto, we will say it is an \textit{epimorphism}.

\begin{theorem}Let $\textgoth{h}:G_{1}\rightarrow G_{2}$ be a homomorphism and $\textgoth{g}:G_{1}\rightarrow G_{2}$ an epimorphism.\makebox[5pt][]{}\mbox {}\begin{itemize}\item[1)]If $H_{1}\in\mathcal{G}_{1_{\subseteq}}$, then $\textgoth{h}H_{1}\in\mathcal{G}_{2_{\subseteq}}$. And, if $N_{1}\in\hat{\mathcal{G}}_{1_{\subseteq}}$, then $\textgoth{g}N_{1}\in\hat{\mathcal{G}}_{2\subseteq}$. \item[2)]If $H_{2}\in\mathcal{G}_{2_{\subseteq}}$ then $\textgoth{h}^{-1}H_{2}\in\mathcal{G}_{1_{\subseteq}}$. And, if $N_{2}\in\hat{\mathcal{G}}_{2_{\subseteq}}$, then $\textgoth{h}^{-1}N_{2}\in\hat{\mathcal{G}}_{1\subseteq}$.\end{itemize}\end{theorem}

\begin{proof}\makebox[5pt][]{}\mbox {}

	\begin{itemize}

\item[1)] We take two objects in $\textgoth{h}H_{1}$ and we wish to see if their product is, again in $\textgoth{h}H_{1}$. That is,

\begin{eqnarray}\nonumber\textgoth{h}h_{1}&;&\textgoth{h}h_{1}
*\textgoth{h}h_{2},\textgoth{h}h_{2}\\\nonumber\textgoth{h}h_{1}&;&\textgoth{h}(h_{1}*h_{2}),\textgoth{h}h_{2}\\\nonumber\textgoth{h}h_{1}&;&\textgoth{h}h,\textgoth{h}h_{2}\end{eqnarray}

for some $h\in H_{1}$. Now, we suppose $N_{1}$ is an invariant subspace; that is $xN_{1}x^{-1}\subseteq N_{1}$, for every $x\in G_{1}$. We shall prove $\textgoth{g}N_{1}$ is invariant as well. We wish to verify $\textgoth{g}x(\textgoth{g}N_{1})\textgoth{g}x^{-1}\subseteq \textgoth{g}N_{1}$, for every $x\in G_{2}$. Any object of $G_{2}$ can be represented by $\textgoth{g}x$, for some $x\in G_{1}$, because $\textgoth{g}$ is onto.Take an object in $\textgoth{g}N_{1}$, say $\textgoth{g}n_{1}$ where $n_{1}\in N_{1}$. 

\begin{eqnarray}\nonumber \textgoth{g}x*\textgoth{g}n_{1}&;&(\textgoth{g}x*\textgoth{g}n_{1})
*\textgoth{g}x^{-1},(\textgoth{g}x)^{-1}\\\nonumber\textgoth {g}x*\textgoth{g}n_{1}&;&
\textgoth{g}[(x*n_{1})*x^{-1}],
(\textgoth{g}x)^{-1}\\\nonumber\textgoth{g}x*\textgoth{g}n_{1}&;&
\textgoth{g}n_{2},(\textgoth{g}x)^{-1}\end{eqnarray}

for some $n_{2}\in N_{1}$. Therefore, $(\textgoth{g}x*\textgoth{g}n_{1})*\textgoth{g}x^{-1}\in\textgoth{g}N_{1}$.

\item[2)] We follow the same line of thought as in 1). Thus, we take $h_{1},h_{2}\in\textgoth{h}^{-1}H_{2}$ and we see that

\begin{eqnarray}\nonumber \textgoth{h}h_{1}&;&\textgoth{h}h_{1}
*\textgoth{h}h_{2},\textgoth{h}h_{2}\\\nonumber \textgoth{h}h_{1}&;& \textgoth{h}(h_{1}*h_{2}),\textgoth{h}h_{2}\end{eqnarray}

which means $h_{1}*h_{2}\in \textgoth{h}^{-1}H_{2}$, because $\textgoth{h}h_{1}*\textgoth{h}h_{2}\in H_{2}$.

We now prove $\textgoth{h}^{-1}N_{2}$ is invariant given $N_{2}$ is invariant. Let $n_{1}\in \textgoth{h}^{-1}N_{2}$, then

\begin{eqnarray}\nonumber (x*n_{1})*x^{-1}&;&\textgoth{h}[(x*n_{1})
*x^{-1}],\textgoth{h}\\\nonumber (x*n_{1})*x^{-1}&;&(\textgoth{h}x*\textgoth{h}n_{1})*\textgoth{h}x^{-1},
\textgoth{h}\\\nonumber (x*n_{1})*x^{-1}&;&
n_{2},\textgoth{h}\end{eqnarray}

for some $n_{2}\in N_{2}$. This is the same as $(x*n_{1})*x^{-1}\in\textgoth{h}^{-1}N_{2}$.

	\end{itemize}

\end{proof}

	\subsection{Isomorphism}

\begin{theorem}Let $\textgoth{h}:G_{1}\rightarrow G_{2}$ a homomorphism, and let $Nul~\textgoth{h}=\textgoth{h}^{-1}1$. Then,

\begin{itemize}\item[1)]$Nul~\textgoth{h}\in\hat{\mathcal{G}}_{1_{\subseteq}}$
\item[2)]There exists an isomorphism $\phi:G_{1}/Nul~\textgoth{h}\rightarrow Im~\textgoth{h}$ such that for every congruence, we have $(Nul~\textgoth{h})x\mapsto_{\phi}\textgoth{h}x$. This isomorphism is called the natural isomorphism for $G_{1}/Nul~\textgoth{h}$ and $Im~\textgoth{h}$.\end{itemize}\label{thGRPdecomp}\end{theorem}

\begin{proof}\makebox[5pt][]{}\mbox {}

\begin{itemize}

\item[1)]We first prove $Nul~\textgoth{h}$ is a subgroup of $G_{1}$. For $e_1,e_2\in Nul~\textgoth{h}$, we have

\begin{eqnarray}\nonumber e_{1}*e_{2}^{-1}&;&\textgoth{h}(e_{1}*e_{2}^{-1}),\textgoth{h}\\\nonumber e_{1}*e_{2}^{-1}&;&\textgoth{h}e_{1}*\textgoth{h}
e_{2}^{-1},\textgoth{h}\\\nonumber e_{1}*e_{2}^{-1}&;&1*(\textgoth{h}
e_{2})^{-1},\textgoth{h}\\\nonumber e_ {1}*e_{2}^{-1}&;&1*1,\textgoth{h}\end{eqnarray}

which means $e_{1}*e^{-1}_{2}\in Nul~\textgoth{h}$. To verify that $Nul~\textgoth{h}$ is invariant we must prove $(x*e)*x^{-1}\in Nul~\textgoth{h}$, for any $e\in Nul~\textgoth{h}$.

\begin{eqnarray}\nonumber (x*e)*x^{-1}&;&\textgoth{h}[(x*e)*x^{-1}],\textgoth{h}\\\nonumber(x*e)*x^{-1}&;&(\textgoth{h}x*\textgoth{h}e)*
\textgoth{h}x^{-1},\textgoth{h}\\\nonumber(x*e)*x^{-1}&;&(\textgoth{h}x*1)*(
\textgoth{h}x)^{-1},\textgoth{h}\\\nonumber(x*e)*x^{-1}&;&\textgoth{h}x*(
\textgoth{h}x)^{-1},\textgoth{h}\\\nonumber(x*e)*x^{-1}&;&1,\textgoth{h}\end{eqnarray}

\item[2)] To see that $\phi$ is a functor:

\begin{eqnarray}\nonumber(Nul~\textgoth{h})a*(Nul~\textgoth{h})b&;&
\phi[Nul~\textgoth{h}(a)*Nul~\textgoth{h}(b)],\phi\\\nonumber(Nul~\textgoth{h})a*(Nul~\textgoth{h})b&;&
\phi[Nul~\textgoth{h}(a*b)],\phi\\\nonumber(Nul~\textgoth{h})a*(Nul~\textgoth{h})b&;&
\textgoth{h}(a*b),\phi\\\nonumber(Nul~\textgoth{h})a*(Nul~\textgoth{h})b&;&\textgoth{h}
a*\textgoth{h}b,\phi
\\\nonumber(Nul~\textgoth{h})a*(Nul~\textgoth{h})b&;&
\phi[(Nul~\textgoth{h})a]*\phi[(Nul~\textgoth{h})b],\phi\end{eqnarray}

We now prove that $\phi$ is onto. Every object $Im~\textgoth{h}$ is of the form $\textgoth{h}x$ such that $x\in G_{1}$. So we have, for every $\textgoth{h}x\in Im~\textgoth{h}$ an object $(Nul~\textgoth{h})x\in G_{1}/Nul~\textgoth{h}$ such that $(Nul~\textgoth{h})x\mapsto_{\phi}\textgoth{h}x$.

To verify we have an isomorphism, we look at the explicit relation between two objects in the same congruence class. We have 

$$a\equiv b~~~~mod~Nul~\textgoth{h}$$

if and only if $e;a,b$ where $e;1,\textgoth{h}$. So, supposing $a,b\in G_1$ are related, we get

\begin{eqnarray}\nonumber \textgoth{h}a&;&\textgoth{h}a*\textgoth{h}b^{-1},(\textgoth{h}b)^{-1}\\\nonumber \textgoth{h}a&;&\phi[(Nul~\textgoth{h})a]*\phi[(Nul~\textgoth{h})b^{-1}],(\textgoth{h}b)^{-1}\\\nonumber \textgoth{h}a&;&\phi
[(Nul~\textgoth{h})a*(Nul~\textgoth{h})b^{-1}],(\textgoth{h}b)^{-1}\\\nonumber\textgoth{h}a&;&\phi[(Nul~\textgoth{h})(a*b^{-1})],(\textgoth{h}b)^{-1}\\\nonumber\textgoth{h}a&;&\phi[(Nul~\textgoth{h})((e*b)*b^{-1})],(\textgoth{h}b)^{-1}\\\nonumber\textgoth{h}a&;&\phi[(Nul~\textgoth{h})e],(\textgoth{h}b)^{-1}\\\nonumber\textgoth{h}a&;&\textgoth{h}e,(\textgoth{h}b)^{-1}\\\nonumber\textgoth{h}a&;&1,(\textgoth{h}b)^{-1}\end{eqnarray}

That is, $a,b$ are related if and only if $\textgoth{h}a,\textgoth{h}b$ are the same object. Therefore, if $(Nul~\textgoth{h})a$ and $(Nul~\textgoth{h})b$ are not the same class, then the image under $\phi$ are not the same object. We have thus proven the functor is monic.

\end{itemize}\end{proof}

\begin{theorem}An epimorphism $\textgoth{g}:G_{1}\rightarrow G_{2}$ such that $Nul~\textgoth{g}=\{1\}\subseteq G_{1}$, is an isomorphism.\end{theorem}

\begin{proof}We have an isomorphism $G_{1}/Nul~\textgoth{g}\rightarrow G_{2}$. All we need is an isomorphism $G_{1}\rightarrow G_{1}/Nul~\textgoth{g}$. As we have already seen, $G_{1}/Nul~\textgoth{g}=G_{1}/1=G_{1}$.\end{proof}

	\subsection{Automorphism}

Given a group $G$, and an object $x\in G$, we will define an \textit{internal automorphism} for $G$. It is defined by $a\mapsto (x*a)*x^{-1}$, and we will express it with $\textgoth{a}_{x}$.

\begin{theorem}Let $G$ be any group. Then,

\begin{itemize}\item[1)]The collection of internal automorphisms is a subcollection of $\{G\textgoth F_{iso}G\}$, and we denote it by $\textgoth{a}_{G}$. \item[2)] The functor $G\rightarrow\textgoth{a}_{G}$ that makes $x\mapsto\textgoth{a}_{x}$, is an epimorphism.\item[3)]$\textgoth{a}_{G}$ is an invariant subgroup of $\{G\textgoth F_{iso}G\}$.\end{itemize}
\end{theorem}

\begin{proof}\makebox[5pt][]{}\mbox {}

\begin{itemize}\item[1)]We see that, for a given $x\in G$, the internal automorphism $\textgoth{a}_{x}$ is bijective. Let $\textgoth{a}_{x}a$ and $\textgoth{a}_{x}b$ be the same object in the image of $\textgoth{a}_{x}$, then:

\begin{eqnarray}\nonumber1,\textgoth{a}_{x}a&;&1,\textgoth{a}_{x}b\\\nonumber1,(x*a)*x^{-1}&;&1,(x*b)*x^{-1}\\\nonumber 1,[(x*a)*x^{-1}]*x&;&1,[(x*b)*x^{-1}]*x\\\nonumber 1,x*a&;&1
,x*b\\\nonumber 1,1&;&x*a
,x*b\\\nonumber x*b,1&;&x*a,1\\\nonumber x^{-1}*(x*b),1&;&x^{-1}*(x*a),1\\\nonumber b,1&;&a,1\\\nonumber 1,1&;&a,b\\\nonumber 1,a&;&1,b\end{eqnarray}

Now, to prove $\textgoth{a}_{x}$ is onto, we see that for any $a\in G$, there exists $x^{-1}*(a*x)\in G$ such that 

\begin{eqnarray}\nonumber (x^{-1}*a)*x&;&(x*[x^{-1}*(a*x)])*x^{-1},\textgoth{a}_{x}\\\nonumber (x^{-1}*a)*x&;&(a*x)*x^{-1},\textgoth{a}_{x}\\\nonumber (x^{-1}*a)*x&;&a,\textgoth{a}_{x}.\end{eqnarray}

We see that we have a functor because\begin{eqnarray}\nonumber \textgoth{a}_{x}a&;&[(x*a)*x^{-1}]*[(x*b)*x^{-1}],\textgoth{a}_{x}b\\\nonumber \textgoth{a}_{x}a&;&[(x*a)*x^{-1}]*[x*(b*x^{-1})],\textgoth{a}_{x}b\\\nonumber \textgoth{a}_{x}a&;&(x*a)*(b*x^{-1}),\textgoth{a}_{x}b\\\nonumber \textgoth{a}_{x}a&;&[(x*a)*b]*x^{-1},\textgoth{a}_{x}b\\\nonumber \textgoth{a}_{x}a&;&[x*(a*b)]*x^{-1},\textgoth{a}_{x}b\\\nonumber \textgoth{a}_{x}a&;&\textgoth{a}_{x}(a*b),\textgoth{a}_{x}b.\end{eqnarray}

\item[2)] We note that $\textgoth{a}_{G}$ is an algebraic category. As a result, we can define a functor $G\rightarrow\textgoth{a}_{G}$, such that the function for objects of operation is defined by $x\mapsto\textgoth{a}_{x}$. We assert that this is a functor because, for every $a\in G$:

\begin{eqnarray}\nonumber a&;&[(x*y)*a]*(x*y)^{-1},\textgoth{a}_{x*y}\\\nonumber a&;&[x*(y*a)]*(y^{-1}*x^{-1}),\textgoth{a}_{x*y}\\\nonumber a&;&([x*(y*a)]*y^{-1})*x^{-1},\textgoth{a}_{x*y}\\\nonumber a&;&(x*[(y*a)*y^{-1}])*x^{-1},\textgoth{a}_{x*y}\\\nonumber a&;&(\textgoth{a}_{x}\circ\textgoth{a}_{y})a,\textgoth{a}_{x*y}\end{eqnarray}

 The facts proven thus far, and theorem 7, imply that $\textgoth{a}_{G}$ is a subgroup of $\{G\textgoth{F}_{iso}G\}$.
\end{itemize}\end{proof}

	\subsection{Action Group}

\begin{definition}A group $G$, with operation $*$, is an \textit{action group} of $X$ if there exists a homomorphism $\bar{*}$ from G into the group of transformations of X. We will write $a\mapsto a\bar{*}$ to represent the images. The null space of $\bar*$ is called the nucleus of non-effectivity for $G$. If $\bar{*}$ is an isomorphism, then we will say $G$ is an effective action group of $X$.

An action group is called transitive if for every $x,y\in X$ there exists $a\in G$ such that $x;y,a\bar*$.

Let $f:X\rightarrow Y$ be a bijective set function and $\textgoth{f}:G\rightarrow H$ an isomorphism such that $G,H$ is an action group of $X,Y$, and the operations are $*_1,*_2$. If $f,f;\textgoth{f}a\bar{*}_2,a\bar*_1$ we say $(X,G)$ and $(Y,H)$ are similar pairs. All that is being said is $f(a\bar*_1x)$ and $(\textgoth{f}a)\bar*_2(fx)$ are the same object, for every $x\in X$.

\end{definition}

The definition of action group is one of the main reasons why defining the operation as a function $\mathcal{O}_1\rightarrow\mathcal{O}_2f\mathcal{O}_3$ is convenient; to some extent it is easier to see the deep relation between operation and group. Recall that the objects of $\mathcal{O}_1$ were called the actions of the operation; they are the objects that act on the objects of $\mathcal{O}_2$. We conclude with the observation that providing $\bar*$ is, in a way, extending the operation of the group, so as to let objects of the collection be acted upon by objects of $G$.

\subsubsection{Group, Action Group and Group of Transformations} First of all, any action group of $X$ is homomorphic to the group of transformations of $X$. Also notice that any group of transformations of $X$ is an action group of $X$. 

Any group $G$ may be seen as the action group of any one object set $\{1\}$; we define the homomorphism $\bar*$ as the trivial $x\mapsto I$, where $I:\{1\}\rightarrow \{1\}$, for all $x\in G$. We see that the images allowed for a homomorphism $\bar*$, depends on the objects of the collection $X$ because this is what defines the diversity of functions $X\rightarrow X$. For example, we have the extreme case in which a group $G$ is an action group of a set with one object, \{x\}. In this case $Nul~\bar*=G$ because all functions are the same, by definition.

\subsubsection{Similarity} We see that a similar pair is the best thing that can happen, when considering two action groups. One is entirely justified in considering these two as the same group, to far extent. The situation is as follows, in considering a similar pair. Not only are the groups isomorphic, the collections for which they are defined are also bijective. And on top of this, the function and isomorphism commute well, in the following sense. We can first apply $a\bar*_1$ and then $f$, or we can first apply $f$ and then $(\textgoth{f}a)\bar*_2$. 

\subsubsection{Transitive Action Groups} In the following we will consider the collection of left congruence classes. Such a collection of classes, with respect to $H\subseteq G$, is represented by $G/xH$.

\begin{theorem}Let $H\in\mathcal{G}_{\subseteq}$, with operation $*$. Let $\bar*:G\rightarrow(G/xH)f_{iso}(G/xH)$ be such that $a\mapsto_{\bar*}a\bar*$ and $xH\mapsto_{a\bar*}(a*x)H$, for every $a\in G$. Then,\begin{itemize}\item[1)]$G$ is a transitive action group of $G/xH$, by the homomorphism $\bar*$\item[2)]$H;H,x\bar*\Leftrightarrow x\in H$\item[3)]$N\in\hat{\mathcal{H}}_\subseteq\Rightarrow N\subseteq Nul~\bar*\subseteq H$.\end{itemize}\end{theorem}

	\begin{proof}\makebox[5pt][]{}\mbox {}

		\begin{itemize}

\item[1)]We will see that $\bar*$ is a homomorphism. We are left to prove $a*b;(a\bar*)\circ(b\bar*),\bar*$. The domain and image of the functions $(a*b)\bar*$ and $(a\bar*)\circ(b\bar*)$ are the same, respectively. Let $xH\in G/xH$, for some $x\in G$,

\begin{eqnarray}\nonumber xH&;&(a*b)\bar*xH,(a*b)\bar*\\\nonumber xH&;&[(a*b)*x]H,(a*b)\bar*\\\nonumber xH&;&[a*(b*x)]H,(a*b)\bar*\\\nonumber xH&;&a\bar*(b*x)H,(a*b)\bar*\\\nonumber xH&;&a\bar*(b\bar*xH),(a*b)\bar*\end{eqnarray}

Now our aim is to give a proof that $G$ is transitive. Let $aH,bH\in G/xH$, since $G$ is a group, $b*a^{-1}\in G$ so that 

\begin{eqnarray}\nonumber aH&;&(b*a^{-1})\bar*aH,(b*a^{-1})\bar*\\\nonumber aH&;&[(b*a^{-1})*a]H,(b*a^{-1})\bar*\\\nonumber aH&;&bH,(b*a^{-1})\bar*.\end{eqnarray}

\item[2)] $x\in H$ if and only if $x*h\in H$, for any $h\in H$. This last condition is true if and only if: a) $x\bar*H=xH\subseteq H$, and b) $H\subseteq xH=x\bar*H$.

Suppose $x*h\in H$, for every $h\in H$, then $xH\subseteq H$. We can equivalently say $x\in H$, and this means that $h\in H$ is the same as $x*( x^{-1}*h)\in xH$ because $x*(x^{-1}*h)$ is the same as $(x*x^{-1})*h\in H$.

Let us now suppose that conditions a) and b) are true, then $x*h\in xH=H$, for any $h\in H$. We may now conclude $x\in H$ if and only if $x\bar*H=H$.

\item[3)] If we prove $Nul~\bar*=\bigcap_{x\in G}xHx^{-1}$, we are done proving 3), for the following reasons. From this we get $Nul~\bar*\subseteq H$, for any $x\in G$, because $H=eHe^{-1}$. Also, consider any invariant subgroup $N$, of $H$, and let $n\in N\subseteq H$. Let $x\in G$, then there exists $m\in N\subseteq H$ such that $n$ is $(x*m)*x^{-1}$. Therefore, $n\in xHx^{-1}$ and we can say $N\subseteq\bigcap_{x\in G}xHx^{-1}$.

We know $a\in Nul~\bar*$ if and only if $(a*x)H=a\bar*xH=xH$, for every $x\in G$. We have\begin{eqnarray}\nonumber H&=&eH\\\nonumber&=&(x^{-1}*x)H\\\nonumber&=&(x^{-1}*x)\bar*H
\\\nonumber&=&x^{-1}\bar*(x\bar*H)\\\nonumber&=&x^{-1}\bar*xH\\\nonumber &=&x^{-1}\bar*(a*x)H
\\\nonumber&=&[x^{-1}*(a*x)]H,\end{eqnarray} and because of the first theorem in the section for congruence classes, $x^{-1}*(a*x)\in H$. From this we have $a*x\in xH$ and after that, $a\in xHx^{-1}$. All the implications are reversible, so we also have $\bigcap_{x\in G}xHx^{-1}\subseteq Nul~\bar*$.\end{itemize}\end{proof}

Suppose $G$ is a transitive action group of $X$ and take $a\in X$. Define, $G_{a\mapsto b}$, for any $b\in X$, as the subcollection of all $x\in G$ such that $a;b,x\bar*$; transitivty of $G$ assures that $G_{a\mapsto b}\neq\emptyset$.

\begin{proposition}If $a\in X$ and $G$ is an action group of $X$, we will say that $G_{a\mapsto a}$ is the \textit{stable subgroup for a}, and represent it by $Inv(a)$, in view of $Inv(a)\in\mathcal{G}_{\subseteq}$.\end{proposition}

\begin{proof}All we need to prove is that $Inv(a)Inv(a)^{-1}\subseteq Inv(a)$. Let $x\in Inv(a)$ and $y^{-1}\in Inv(a)^{-1}$, for some $y\in Inv(a)$. We must verify $(x*y^{-1})\bar*$ applies $a$ to $a$. 

\begin{eqnarray}\nonumber a&;&(x*y^{-1})\bar*a,(x*y^{-1})\bar*\\\nonumber a&;&x\bar*(y^{-1}\bar*a),(x*y^{-1})\bar*\\\nonumber a&;&x\bar*[(y\bar*)^{-1}a],(x*y^{-1})\bar*\\\nonumber a&;&x\bar*a,(x*y^{-1})\bar*\\\nonumber a&;&a,(x*y^{-1})\bar*\end{eqnarray}Notice we are using the inverse function $(y\bar*)^{-1}$ in virtue that $y\bar*$ is bijective. We use the fact that $(y\bar*)^{-1}$, just as $y$, applies $a$ to $a$.\end{proof}

We get a another representation of the nucleus of non-effectivity for $G$; again, as an intersection.

$$Nul~\bar*=\bigcap_{x\in X}Inv(x).$$

Now, consider the operation $\phi:X\rightarrow Xf\textgoth{P}G$ that makes $a\mapsto_{\phi}a\phi$ and define $a\phi:X\rightarrow \textgoth{P}G$ as the function that makes $b\mapsto_{a\phi} G_{a\mapsto b}$. The notation for this operation yields $a;G_{a\mapsto b},b$. Another way of saying this is in terms of the function is $b;G_{a\mapsto b},a\phi$.

\begin{theorem}Let $G$ be a transitive action group of $X$ and $\textgoth{I}:G\rightarrow G$ be the identity functor.

	\begin{itemize}

\item[1)] $Im~a\phi=G/x~Inv(a)$ and $a\phi|^{Im~a\phi}=a\phi|^{G/x~Inv(a)}$ is bijective

\item[2)] The pairs $(X,G)$ and $(G/x~Inv(a),G)$ are similar, in virtue of $a\phi|^{G/x~Inv(a)},\textgoth I$

\item[3)]$x\in G_{a\mapsto b}\Rightarrow Inv(b)=x~Inv(a)~x^{-1}$.

	\end{itemize}

\end{theorem}

\begin{proof}\makebox[5pt][]{}\mbox {}

	\begin{itemize}

\item[1)]

\item[2)]

\item[3)]

\item[4)]

	\end{itemize}

\end{proof}

\chapter{Linear Space}

\begin{definition}Let $\mathbb K(+,\cdot)$ be a field and $V_\mathcal G(*)$ an abelian group. Represent the underlying set of $V_\mathcal G$ with $V$ and recall that $\{V_\mathcal G\textgoth F_{iso}V_\mathcal G\}$ is the group of automorphisms for $V_\mathcal G$. Let $\mathbb K(\cdot)$ be the group under the product, for the field, and suppose there is a homomorphism $\bar{*}:\mathbb K(\cdot)\rightarrow\{V_\mathcal G\textgoth F_{iso}V_\mathcal G\}$, such that \begin{equation}a+b;(a\bar{*}u)*(b\bar{*}u),u.\label{lspace1}\end{equation} This last condition means $(a+b)\bar{*}u$ is the same as $(a\bar{*}u)*(b\bar{*}u)$ where $a,b\in\mathbb K$ and $u\in V_\mathcal G$. The abelian group $V_\mathcal G$, together with the homomorphism $\bar{*}:\mathbb K(\cdot)\rightarrow\{V_\mathcal G\textgoth F_{iso}V_\mathcal G\}$ forms a \textit{linear space}. \end{definition}In other words we have  a linear space iff we have a field that under product is homomorphic to the group of automorphisms of an abelian group, in such a manner that (\ref{lspace1}) holds. The linear space is represented by $\mathbb V$. The following establishes equivalence with the usual definition of linear space.

\begin{proposition}A linear space $\mathbb V$ exists iff there is an abelian group $V_\mathcal G(*)$ and a field $\mathbb K(+,\cdot)$ such that there is an operation $\bar*:\mathcal{A|}\mathbb{K}\rightarrow\{\mathcal{A|}V_\mathcal G~f~\mathcal{A|}V_\mathcal G\}$ defined so that (\ref{lspace1}) holds, as well as\begin{eqnarray}a&;&(a\bar*u)*(a\bar*v),u*v\\1&;&u,u\\a\cdot b&;&a\bar*(b\bar*u),u,\end{eqnarray}in terms of the notation for $\bar*$ and for $1\in\mathbb K(\cdot)$.\end{proposition}

\begin{proof}Since the image of $\bar*$ consists of automorphisms, we can say $a\bar*(u*v)$ is the same as $(a\bar*u)*(a\bar*v)$ for every $a$ in the field and $u,v$ in the abelian group $V_\mathcal G$. Secondly, since $\bar*$ is a homomoprhism, we can say $1\bar*:V_\mathcal G\rightarrow V_\mathcal G$ is the identity automorphism. Finally, the fact that $\bar*$ is a homomorphism implies that $(a\cdot b)\bar*$ is the same automorphism as $a\bar*\circ b\bar*$, which is the same as saying $(a\cdot b)\bar*u$ is equal to $a\bar*(b\bar*u)$.\end{proof}

\chapter{Topological System}

Let $X$ be a set and let $\textgoth X$ be an algebraic category with the c-object defined as $\mathcal O|\textgoth X:=X$. For every subset $A\subseteq X$, we define an arrow in $\textgoth X$, also denoted $A$, and we call it the arrow subset. If there is no fear for confusion we may say it is a subset of $X$. To make sure $\textgoth X$ is a category, we must provide a composition operation for the arrows. We know union is associative, and the emptyset acts as unit. Also, the union of any two subsets of $X$, is also a subset of $X$. This means that $\textgoth X$ is an algebraic category with subsets as objects of operation, and $X$ as c-object.

Now we want to consider a subcategory of $\textgoth X$. This will be the category that consists of the arrow subset corresponding to the singletons of $X$. The category will be represented by $\{\{\textgoth X\}\}$. The arrows of this category are called \textit{points} of $X$. We will include $\emptyset$ and $X$ in the category of points. We are identifying each subset (in the strict sense) of $X$, with the arrow subset, in $\textgoth X$. We are doing this so that we can view subsets as arrows in a category.

		\section{Two Descriptions, One System}
	
		\subsection{Closure}

Let $\textgoth T$ be a subcategory of $\textgoth X$, such that $\emptyset,X$ are arrow subsetes in $\textgoth T$. We say $\textgoth T$ is the \textit{topological space}, while a \textit{topological system} is a functor $Cl_X:\textgoth T\rightarrow\textgoth T$ so that for any point $x$, and any subset $A$, in $\textgoth T$:

\begin{itemize}\item[1)]$~x~;~x,Cl_X$\item[2)]$~Cl_XA~;~Cl_XA,
Cl_X$.\end{itemize}Both 1) and 2) can be expressed as one condition. We are saying $\{\{\textgoth T\}\}$ is strongly invariant under $Cl_X$, and $Cl_X$ is a once effective functor. Of course, we use $\{\{\textgoth T\}\}$ to represent the points of $\textgoth T$. Simply put, $Im~Cl_X\cup\{\{\textgoth{T}\}\}$ is strongly invariant. The concept of stongly invariant and once effective function is presented in the subdivision Sequence for Composition. In view of the last statements, the two conditions are replaced by the condition that any arrow subset $A$, in $Im~Cl_X\cup\{\{\textgoth{T}\}\}$, satisfies $A;A,Cl_X$. A topological system is an algebraic functor that leaves $Im~Cl_X\cup\{\{\textgoth{T}\}\}$ strongly invariant.

Any $D$ in $\textgoth T$ such that $D;D,Cl_X$ is said to be \textit{closed}. The collection of all closed subsets of $X$ is represented by $\mathcal{C}$. A subset $V\subseteq X$, complement of a closed subset, is in the collection of \textit{open} sets, $\mathcal{O}$.

\begin{proposition}For any $A,B\subseteq X$, we have

\begin{itemize}\item[1)]$A\subseteq Cl_XA$\item[2)]$A\subseteq B\Rightarrow Cl_XA\subseteq Cl_XB$.\end{itemize}

\label{topth1}\end{proposition}

\begin{proof}A point $x$, of the space is \textit{adherent} to $A\subseteq X$ if $x\in Cl_XA$. We say $Cl_XA$ is the adherence of $A$.

\item[1)]Let $x\in A$, then\begin{eqnarray}\nonumber A&;&Cl_X(A\cup \{x\}),Cl_X\\\nonumber A&;&Cl_XA\cup Cl_X\{x\},Cl_X\\\nonumber A&;&Cl_XA\cup \{x\},Cl_X.\end{eqnarray}

This last expression is $Cl_XA=Cl_XA\cup \{x\}$, which is equivalent to $x\in Cl_XA$.

\item[2)]If $A\subseteq B$, we have\begin{eqnarray}\nonumber B&;&Cl_X(A\cup B),Cl_X\\\nonumber B&;&Cl_XA\cup Cl_XB,Cl_X\\\nonumber\Leftrightarrow Cl_XB&=&Cl_XA\cup Cl_XB\\\nonumber\Leftrightarrow Cl_XA&\subseteq&Cl_XB.\end{eqnarray}

\end{proof}

\begin{theorem}For any finite subcollections $\{D_{i}\}^{n}_{i=1}\subseteq\mathcal{C}$ or $\{V_{i}\}^{n}_{i=1}\subseteq\mathcal{O}$, we have
\begin{itemize}\item[1)]$\bigcup_{i}D_{i}\in\mathcal{C}$\item[2)]$\bigcap_{i}
V_{i}\in\mathcal{O}$.\end{itemize}\end{theorem}

\begin{proof}\begin{itemize}\item[1)]\begin{eqnarray}\nonumber D_{1}\cup D_2&;&Cl_X(D_1\cup D_2),Cl_X\\\nonumber D_{1}\cup D_2&;&Cl_XD_1\cup Cl_XD_2,Cl_X\\\nonumber D_{1}\cup D_2&;&D_1\cup D_2,Cl_X.\end{eqnarray}
The ressult follows for $n+1$ subsets if we suppose that it holds for $n$.\item[2)] \begin{eqnarray}\nonumber \left(\bigcap_i V_i\right)^c&;&Cl_X\bigcup_i V^c_i,Cl_X\\\nonumber \left(\bigcap_i V_i\right)^c&;&Cl_X\bigcup_i D_i,Cl_X\\\nonumber \left(\bigcap_i V_i\right)^c&;&\bigcup_i D_i,Cl_X\\\nonumber \left(\bigcap_i V_i\right)^c&;&\bigcup_i V_i^c,Cl_X\\\nonumber \left(\bigcap_i V_i\right)^c&;&\left(\bigcap_i V_i\right)^c,Cl_X.\end{eqnarray}This means $(\cap_i V_i)^c\in\mathcal{C}$, which proves $\cap_i V_i\in\mathcal{O}$.\end{itemize}\end{proof}

\begin{theorem}For any subcollections $\{D_{i}\}_{i}\subseteq\mathcal{C}$ or $\{V_{i}\}_{i}\subseteq\mathcal{O}$, we have
\begin{itemize}\item[1)]$\bigcap_{i}D_{i}\in\mathcal{C}$\item[2)]$\bigcup_{i}
V_{i}\in\mathcal{O}$.\end{itemize}\end{theorem}

\begin{proof}We know that $\cap_iD_i\subseteq Cl_X\cap_iD_i$. Thus, all we need to prove is $Cl_X\cap_iD_i\subseteq \cap_iD_i$. From 2) in proposition \ref{topth1}, we get $Cl_X\cap_iD_i\subseteq Cl_XD$, for every $D\in\{D_i\}_i$. Then, $Cl_X\cap_iD_i\subseteq\cap_i Cl_XD_i=\cap_i D_i$.

We can prove 2) as we did in the last theorem.\end{proof}

\begin{theorem}Let $A\subseteq X$ and $\mathcal{C}_{A}$ the collection of all $D$ such that $A\subseteq D\in\mathcal{C}$. Then $A;\bigcap\mathcal{C}_{A},Cl_X$.\end{theorem}

\begin{proof}We know $Cl_XA\in\mathcal{C}$, and since $A\subseteq Cl_XA$, we also have $Cl_XA\in\mathcal{C}_{A}$. This implies $\cap\mathcal{C}_{A}\subseteq Cl_XA$.

On the other hand, $A\subseteq D$, for every $D\in\mathcal{C}_{A}$, means $A\subseteq\cap\mathcal{C}_{A}$. From this, and the fact that $\cap\mathcal{C}_{A}\in\mathcal{C}$, we get $Cl_XA\subseteq Cl_X\cap\mathcal{C}_{A}=\cap\mathcal{C}_{A}$.\end{proof}

	\subsection{Interior}

To establish a topological system, a \textit{closure} operation must be defined for the subsets of $X$. However, it is not absolutely necessary to define the closure for all subsets. If we define the collection of closed sets, we will be able to say what the closure of any subset is; this is provided by the last theorem. Simply put, to every family of subsets of $X$, that includes $\emptyset,X\in\mathcal{C}$, there corresponds a topological sysytem. Now, in view of the fact that a topological system can be established by defining the closed sets, we see that it is perfectly acceptable to give the topological system by defining the dual objects of these. That is, a topological system is defined by defining the collection of open sets for $X$.

\section{Neighborhoods}For any point $x\in X$, the family $\mathcal{O}_{x}$ consists of all $V$ such that $x\in V\in\mathcal{O}$. Also, $\mathcal{O}^{A}$ is the family of all $V$ such that $A\supseteq V\in\mathcal{O}$.

\begin{definition}A family $\mathcal{B}\subseteq\mathcal{O}$ is called a base of the topological system $Cl_X$ if for every $\emptyset\neq V\in\mathcal{O}$ we have $V=\bigcup(\mathcal{O}^V\cap\mathcal{B})$. If $x\in U\in\mathcal{B}$, we say $U$ is a base-neighborhood of $x$. The subset $N$ is called a neighborhood of $x$ if there exists $V\in\mathcal{O}$ such that $x\in V\subseteq N$. A family $\mathcal{B}_{x}\subseteq\mathcal{O}$ is a point base of $x$ if for every neighborhood $N$ of $x$, there exists $U\in\mathcal{B}_{x}$ such that $x\in U\subseteq N$.\end{definition}

Notice that in order to speak of base-neighborhoods, we must have a fixed  base. So, anytime base-neighborhoods are discussed, we will presuppose the choice of a base. 

We wish to see that every point of $X$ is in some base-neighborhood. We can easily prove this for any non-trivial case in which their are at least two points. In this case, we have proven $X$ is open and therefore $X=\bigcup\mathcal{B}$. From this it follows that every point $x$ is in some $N\in\mathcal{B}$.

\begin{proposition}For any $x\in X$ we verify

\item[1)]Any base of $\textgoth X$ is a point base of $x$.\item[2)]The collection of all base-neighborhoods of $x$ is a point base of x.\end{proposition}

\begin{proof}We will not give a proof for 1); this is a trivial observation. Let $N$ be a neighborhood of $x$, then there is an open set $V$ such that $x\in V\subseteq N$. We know $V=\bigcup\mathcal{V}$, where $\mathcal{V}$ is a subfamily of $\mathcal{B}$. Therefore, $x\in U\subseteq N$, for some subset $U\in\mathcal{V}$.\end{proof}

\begin{theorem}A subset is open if and only if it is a neighborhood of every point it contains.\end{theorem}

\begin{proof}Clearly, $V\in\mathcal{O}$ is a neighborhood of any $x\in V$ because $V\subseteq V$.

Suppose $V$ is a neighborhood of every $x\in N$. Consider, for each $x$, the open set $x\in U_{x}\subseteq V$. It is easily verified that $V=\bigcup_x U_x\in\mathcal{O}$.\end{proof}

\begin{theorem}$\mathcal{B}$ is a base of $CL_X$ if and only if $$x\in V\in \mathcal{O}~\Rightarrow~(\exists U\in\mathcal{B})(x\in U\subseteq V).$$\end{theorem}

\begin{proof}$\mathcal{B}$ is a base of $\textgoth X$ if and only if $V=\bigcup(\mathcal{O}^{V}\cap\mathcal{B})$. This last holds if and only if $x\in U$, for some $U\in(\mathcal{O}^V\cap\mathcal{B})$.\end{proof}

\textbf{Closure and Neighborhoods.} Of course, once we have specified a base, we have also, unequivocally, specified a topological system. This is due to the fact that a base determines a collection $\mathcal{O}$ that consists of $\bigcup\beta$, for every $\beta\subseteq\mathcal{B}$. Let $\bar{}:\textgoth T\rightarrow\textgoth T$ be the functor that assigns $A\mapsto\bar A$, where $x\in\bar{A}$ if and only if $x\in U\in\mathcal{B}$ implies the existence of some $a\in A\cap U$. When such a functor is given, we call it a \textit{closure in terms of a base}.

\begin{theorem}The closure operation, defined in terms of a base, forms a topological system which is the same topological system determined by the open sets, generated from the base.\end{theorem}

\begin{proof}We must verify that $\bar{A}=\bigcap\mathcal{C}_A$, where the closed subsets are those corresponding to open sets given by the base. Let $x\in \bar{A}$, and take $D\in\mathcal{C}_A$. We show $x\notin D^c\in\mathcal{O}$ in order to prove $x\in\bigcap\mathcal C_A$. If it were true that $x\in D^c$, then there exists $U\in\mathcal{B}$ such that $x\in U\subseteq D^c$. But this means $a\in U$, for some $a\in A$. This is a clear contradiction beacause $A\subseteq D\subseteq U^c$.

For the opposite inclusion, suppose $x\in\bigcap\mathcal C_A$ and let $x\in U\in\mathcal{B}$. If $A\cap U=\emptyset$, then $x\notin A\subseteq U^c\in\mathcal{C}$. which implies $x\notin\bigcap\mathcal{C}_A$. Therefore, $a\in A\cap U$, for some $a$. Thus, $x\in\bar{A}$.\end{proof}

The collection of all the bases of a topological system is represented by $\textbf{B}$; in an analogous manner we denote by \textbf{O} the collection of open sets that determine a topological system. Also, let $\textbf{T}$ be the collection of all topological systems of $X$. We will say $CL:\textbf{B}\rightarrow\textbf{T}$ is the function that assigns to each base of $X$ the unique closure in terms of the base. Represent by $Cl:\textbf{B}\rightarrow\textbf{O}\rightarrow\textbf{T}$, the function that assigns to each base of $X$, the family of open sets, for which it is base, and then to that family, assigns its corresponding closure function. The theorem says both are the same.






\newpage\section*{Acknowledgements}

The present work began as a project to study several areas of mathematics using the perspective of systems; objects and relations. The initial intention on the part of the author was to use this only as a conceptual tool for personal understanding. In a series of talks with a fellow student, we came to the conclusion that this focus was similar to the concept of category, where relations take the special form of arrows. This enabled me to start writing out the present document. My thanks go out to Marco Armenta, and several other costudents whom I had enriching conversations with.

The initial thoughts came about after reading a recommendation from my professor, Victor P\'erez Abreu. After making a project on axiomatic systems, for a class I was taking with the professor, he recommended I read, An Eternal Golden Braid by Douglas R. Hofstadter. This book had profound implications in my life and philosophical views. Since I am an undergraduate student of mathematics, I decided to use this philosophy to start learning mathematics from scratch. Of course, lead by the works of others in all areas of mathematics, I have tried to make an account of some basic principle of mathematics. The professor was always giving great insight in his class and at talks in his office. His patience and dedication to all his students is admirable.

After making some progress, I came across the work of Jouni J$\ddot{a}$rvinen. This helped me to materialize the constructions of the integers and rationals. Although, I initially came across an article on information systems, this article was not ultimately of much help but his article [VIII] turned out to be of tremendous help. I went back and tried to properly define the construction using the definitions provided in [VIII]. I have also made use of the online encyclopedia, Wikipedia. 

Any error or misrepresentation of mathematical results that goes against established mathematical standards are sole responsibility of the author, whom does not wish to involve any of the above mentioned, on negative feedback, that shall certainly come. Of course, all feedback will be more than welcome by the author.

\section*{Bibliography}

\begin{itemize}
\item[I.]Mac Lane, Saunders. Categories for the Working Mathematician. New York: Springer-Verlag, 1971.\item[II.]Asperti, Andrea and Longo, Giuseppe. \textit{Categories Types and Structures}: An Introduction to Category Theory for the Working Computer Scientist. MIT Press, 1991.\item[III.]R. M. Dudley. Real Analysis and Probability. Cambridge: Cambridge University Press, 2004.\item[IV.]Mac Lane, Saunders and  Birkhoff, Garrett. Algebra. Providence, Rhode Island: AMS Chelsea Publishing, Third edition 2004.\item[V.]L. S. Pontriaguin. Grupos Continuos. Mosc\'u: Editorial Mir, 1978.\item[VI.] J. N. Sharma. Krishna's Topology. Meerut: Krishna Prakashan Media (P) Ltd, 1979. \item[VII.]Mariusz Wodzicki. Notes on Topology. December 3, 2010.
\item[VIII]J$\ddot{a}$rvinen, Jouni. Lattice Theory for Rough Sets. Turku Centre for Computer Sciences. FI-20014 University of Turku, Finland.
\item[IX]Bifunctor. V.E. Govorov (originator), Encyclopedia of Mathematics.
\item[X]A Concrete Introduction to Categories. William R. Schmitt. Department of Mathematics, George Washington University.
\item[XI]Representable Functors and the Yoneda Lemma. Brown, Gordon. Spring 2015
\item[XII]math.stackexchange.com 
\end{itemize}

\end{document}